\documentclass[a4paper,10pt]{amsart}
\usepackage[utf8]{inputenc}
\usepackage{amsthm}
\usepackage{amsmath}

\usepackage[dvipsnames]{xcolor}

\usepackage{comment}
\usepackage{bm}
\usepackage{amsfonts}
\usepackage{amssymb}
\usepackage{mathtools}
\usepackage{stmaryrd}
\usepackage{mathrsfs}
\usepackage{setspace}
\usepackage{textgreek}
\usepackage{todonotes}
\usepackage{caption}
\usepackage{tikz}
\usepackage{pgfplots}
    \pgfplotsset{
        compat=1.15,
        width=8cm,
    }
\usepackage[a4paper, left=2.8cm, right=2.8cm, top=3cm, bottom=3cm]{geometry}

\usepackage{thmtools} % solves problem of shared counters in statements and autoref
\usepackage{bm}
\usepackage{keyval}

\usepackage{esint}

\usepackage[pdfdisplaydoctitle,colorlinks,breaklinks,urlcolor=blue,linkcolor=blue,citecolor=blue]{hyperref} %must always be the last
\mathtoolsset{showonlyrefs}

\newcommand{\R}{\mathbb{R}}

\newcommand{\rmd}{{\rm d}}

\renewcommand{\leq}{\leqslant}
\renewcommand{\geq}{\geqslant}

\newcommand{\eee}{equation}
\newcommand{\be}{\begin{\eee}}
\newcommand{\ee}{\end{\eee}}

\let\div\relax %The default for ``div'' is the division symbol

\DeclareMathOperator{\div}{div}

\newcommand{\dd}{\,\mathrm{d}}
\newcommand{\dif}{\mathrm{d}}
\newcommand{\eps}{\varepsilon}

\renewcommand{\oint}{\fint}

\newcommand{\CC}{\mathbb{C}}

\newcommand{\EE}{\mathbb{E}}

\newcommand{\NN}{\mathbb{N}}

\newcommand{\PP}{\mathbb{P}}

\newcommand{\RR}{\mathbb{R}}
\renewcommand{\SS}{\mathbb{S}}
\newcommand{\TT}{\mathbb{T}}

\newcommand{\ZZ}{\mathbb{Z}}

\newcommand{\cC}{\mathcal{C}}
\newcommand{\cD}{\mathcal{D}}

\newcommand{\cF}{\mathcal{F}}

\newcommand{\cH}{\mathcal{H}}

\newcommand{\cM}{\mathcal{M}}

\newcommand{\cP}{\mathcal{P}}

\newcommand{\cV}{\mathcal{V}}

\newcommand{\var}{{\rm Var}}

\newcommand{\vertiii}[1]{{\left\vert\kern-0.25ex\left\vert\kern-0.25ex\left\vert #1 
    \right\vert\kern-0.25ex\right\vert\kern-0.25ex\right\vert}}

\numberwithin{equation}{section}
\newtheorem{theorem}{Theorem}[section]

\newtheorem{cor}[theorem]{Corollary}

\newtheorem{lemma}[theorem]{Lemma}

\newtheorem{prop}[theorem]{Proposition}

\theoremstyle{definition}

\newtheorem{assumption}[theorem]{Assumption}
\theoremstyle{remark}
\newtheorem{remark}[theorem]{Remark}

\title{Anomalous Dissipation and Regularization in isotropic Gaussian  turbulence}

\author{Theodore D. Drivas}
\address[T. D. Drivas]{Department of Mathematics, Stony Brook University, Stony Brook, NY, 11790}
\email{tdrivas(at)math.stonybrook.edu}

\author{Lucio Galeati}
\address[L. Galeati]{Dipartimento di Ingegneria e Scienze dell’Informazione e Matematica, Università degli Studi
dell’Aquila, Italy}
\email{lucio.galeati(at)univaq.it}

\author{Umberto Pappalettera}
\address[U. Pappalettera]{Fakult\"at f\"ur Mathematik, Universit\"at Bielefeld, D-33501 Bielefeld, Germany}
\email{upappale(at)math.uni-bielefeld.de}

\begin{document}

\begin{abstract}
In this work we rigorously establish a number of properties of ``turbulent'' solutions to the stochastic transport and the stochastic continuity equations constructed by Le Jan and Raimond in [Ann. Probab. 30(2): 826-873, 2002].
The advecting velocity field, not necessarily incompressible, is Gaussian and white-in-time, space-homogeneous and isotropic, with $\alpha$-H\"older regularity in space, $\alpha\in (0,1)$.
We cover the full range of compressibility ratios giving spontaneous stochasticity of particle trajectories. For the stochastic transport equation, we prove that generic $L^2_x$ data experience anomalous dissipation of the mean energy, and study basic properties of the resulting anomalous dissipation measure.
Moreover, we show that starting from such irregular data, the solution immediately gains regularity and enters into a fractional Sobolev space $H^{1-\alpha-}_x$.
The proof of the latter is obtained as a consequence of a new sharp regularity result for the degenerate parabolic PDE satisfied by the associated two-point self-correlation function, which is of independent interest.
In the incompressible case, a Duchon-Robert-type formula for the anomalous dissipation measure is derived, making a precise connection between this self-regularizing effect and a limit on the flux of energy in the turbulent cascade.
Finally, for the stochastic continuity equation, we prove that solutions starting from a Dirac delta initial condition undergo an average squared dispersion growing with respect to time as $t^{1/(1-\alpha)}$, rigorously establishing the analogue of Richardson’s law of particle separations in fluid dynamics.\\

\vspace{-2mm}

\vspace{-2mm}
\end{abstract}

\date{\today}
\maketitle

\vspace{-10mm}

\tableofcontents

\section{Introduction}
In this work, we consider stochastic advection by a  Gaussian random field $W:=W(t,x): \R_+ \times \mathbb{R}^d \to \mathbb{R}^d$, $d \geq 2$, which is spatially colored and white-in-time correlated. The evolution is governed by a stochastic partial differential equation for a scalar field $\theta$ of the form
\begin{equation} \label{eq:Kraichnan}  \tag{STE}
\dd  \theta + \circ \dd W \cdot \nabla \theta = 0,
\end{equation}
where the symbol $\circ$ denotes that we are interpreting the stochastic integral in the Stratonovich sense.  The velocity field, being Gaussian, is completely prescribed by its mean (taken zero for simplicity) and covariance
\begin{align*}
    \mathbb{E}[ W(t,x) \otimes W(t',x') ] = (t \wedge t') C(x,x') ,
 \quad
 C : \R^d \times \R^d \to \R^{d \times d}.
\end{align*}
We are primarily interested in the case where the velocity field is spatially  \textit{homogeneous}, i.e. $C(x,x')= C(x-x')$, and  \textit{non-degenerate}, i.e. $C(0)$ is a positive definite matrix.  When $W$ is spatially smooth the analysis of \eqref{eq:Kraichnan} is classical \cite{Ku82,Ku97}, and relatively detailed information about the behaviour of solutions is available \cite{LeJ85,BaHa86}. Moreover, exponential mixing and enhanced dissipation have been established in \cite{DKK2004,GesYar2025}.

However, an important qualitative feature of the velocity fields that we consider is that they are spatially rough: exactly $\gamma$--H\"{o}lder in space for every $\gamma<\alpha\in (0,1)$ and not better. This amounts to considering a covariance matrix with the following small-distance asymptotic behavior
\be\label{covassum}
Q(z) := C(0) - C(z) =O( | z |^{2 \alpha}).
\ee
Slightly abusing terminology, we refer to such a velocity field as being  $\alpha$--regular. 
Such Gaussian fields are characterized by the Fourier transform of their covariance, see \cite[\S 12]{MoYa75} and \autoref{app:auxiliary}.

Our main motivation to study the stochastic transport equation \eqref{eq:Kraichnan} comes from homogeneous isotropic fluid turbulence. 
In fluid-dynamics, turbulence is said to be homogeneous and isotropic if the fluid velocity is a homogeneous and isotropic random field.
This concept is a mathematical idealization which,
however, is very convenient for a statistical description of certain turbulent flows \cite{MoYa75}.
A very popular model of homogeneous and isotropic turbulence is the so-called \textit{Kraichnan model} \cite{Kr68}, corresponding to the particular choice
\begin{equation}\label{eq:kraichnan_covariance_fourier_intro}
    \hat{C}(\xi)
\propto  \frac{1}{(|\xi|^2 + m^2)^{\frac{d}{2}+\alpha}}\bigg[ a \,\hat \xi\otimes \hat\xi + \frac{b}{d-1} (I_d- \hat \xi\otimes \hat \xi)\bigg] ,
\end{equation}
where the parameters $a,b\geq 0$ determine the compressibility of the noise (which is divergence-free when $a=0$ and a gradient when $b=0$), and $m>0$ is an infra-red cut-off parameter at low frequency. We assume moreover $a+b>0$, so that $W \neq 0$. 

First investigations on the Kraichnan model have been done by physicists in the particular \textit{self-similar} case $m=0$, i.e. assuming the exact scaling $C(z) - C(0) \propto |z|^{2\alpha}$ for every $z \in \R^d$.
Despite the fact that solutions to \eqref{eq:Kraichnan} cannot be rigorously defined in the self-similar case (e.g. $C(0)$ would be infinite), the simple and explicit structure of the covariance allows many formal computations suggesting the validity of some characteristic features of turbulent fluids, such as the \textit{anomalous dissipation} of $L^2_x$ norm of solutions, \textit{spontaneous stochasticity} of particle ``trajectories'', and Richardson's law for the growth with respect to time of the mean square separation of fluid particle pairs.
 Mathematically, the latter can be made precise by studying the stochastic continuity equation
\begin{align} \label{eq:Kraichnan.continuity} \tag{SCE}
\dd \mu_t + \nabla \cdot (\mu_t \circ \dd W)  = 0
\end{align}
starting from a Dirac delta initial condition $\mu_0 = \delta_x$, $x \in \R^d$, which represents the distribution of a single particle, initially located at the point $x$.
See the lectures \cite{Ga02} for a thorough discussion. 

Arguing formally, it is possible to clearly identify regimes of regularity $\alpha$ and \textit{compressibility} $\eta:= \frac{b}{a+b}$ in which the aforementioned anomalous phenomena happen \cite{GawVerg00}.
 One of our aims will be to disentangle this formal self-similarity assumption from many of the key phenomena possessed by the model, giving rigorous proofs when $m>0$.
We shall also allow $W$ to be non-divergence free in many of our results, covering the full \emph{diffusive regime} in which spontaneous stochasticity happens (see below), although some results will hold only in the special, but very important, divergence-free setting.  
Moreover, the particular covariance structure of the Kraichnan model is not strictly necessary as long as the key qualitative properties of the noise (space regularity and relative intensity of divergence-free and gradient part) are sufficiently good. Precise assumptions are stated in the relevant sections in the main body of the paper.

\subsection{Main results}
The stochastic transport equation \eqref{eq:Kraichnan} displays a rich array of phenomena in common with hydrodynamic turbulence.  One such phenomenon is \emph{anomalous dissipation} (or diffusion).
To describe this behavior, consider the following diffusive approximation\footnote{The coefficient $\sqrt{1-\kappa}$ appearing in front of $\dd W$ in \eqref{eq:intro_viscous_approx}, which might look unusual, is quite natural in the Kraichnan model, see the discussion after \autoref{rem:Lp_bound_divergence_free}, as well as \autoref{lem:approx} and its proof.
In full generality, the correct viscous approximation we will consider for \eqref{eq:Kraichnan} is given by \eqref{eq:spde_viscous_approx} below. 
Throughout the Introduction, for the sake of simplicity we implicitly assume $C(0)=2 I_d$, so that \eqref{eq:spde_viscous_approx} reduces to
\eqref{eq:intro_viscous_approx}.} of \eqref{eq:Kraichnan}
\begin{equation}\label{eq:intro_viscous_approx}
    \dd \tilde\theta^\kappa + \sqrt{1-\kappa} \circ \dd W \cdot \nabla \tilde{\theta}^\kappa = \kappa \Delta \tilde\theta^\kappa \dd t,
\end{equation}
for which the following energy balance holds for $\kappa \in (0,1)$:
\begin{align} \label{eq:contraction_L2k}
\frac{1}{2} \mathbb{E} \| \tilde\theta_t^\kappa \|_{L^2_x}^2 = \frac{1}{2}  \| \theta_0 \|_{L^2_x}^2 - \kappa \int_0^t   \mathbb{E} \|\nabla \tilde\theta_s^\kappa\|_{L^2_x}^2\rmd s,
\qquad
\forall t \geq 0.
\end{align}
Despite $W$ is possibly not divergence-free, this energy balance holds in expectation due to the fact that the energy input from $\div W$ is a martingale. 
The last term in \eqref{eq:contraction_L2k} describes to mean energy that is dissipated by $\tilde\theta^\kappa$ at time $t$. Persistence of mean energy dissipation as $\kappa \to 0$, namely 
\begin{align*}
    \liminf_{\kappa \to 0} \kappa \int_0^t   \mathbb{E} \|\nabla \tilde\theta_s^\kappa\|_{L^2_x}^2\rmd s >0 ,
\end{align*}
clearly necessitates that spatial roughness of $\tilde{\theta}^\kappa$ emerges as $\kappa\to 0$. Next we discuss what are the implications of the above to solutions of \eqref{eq:Kraichnan}.

For early mathematical investigations on the Kraichnan model \eqref{eq:Kraichnan}, we refer to \cite{EVan2000,EVan2001,EyiXin1996,EyiXin2000}. 
A very thorough mathematical theory has then been established by Le Jan and Raimond in their seminal papers \cite{LeJRai2002,LeJRai2004}; their approach, based on Wiener chaos expansions, has subsequently been expanded in \cite{LotRoz2004,LotRoz2006}. 
Le Jan and Raimond developed a general solution theory for (backward) stochastic equations on manifolds, not limited to equations of transport type, showing existence of solutions for every initial condition $\theta_0 \in L^2_x$, and uniqueness within the class of $\{\mathcal{F}_t\}_{t \geq 0}$-adapted solutions satisfying a suitable energy bound, cf. \cite[Theorem 3.2]{LeJRai2002}.
Thus, neglecting for a moment the discrepancy between backward and forward formulations of \eqref{eq:Kraichnan} (see \autoref{prop:solution_LJR} for more details),  uniqueness of adapted solutions to \eqref{eq:Kraichnan} holds and the vanishing diffusivity procedure selects the unique adapted solution.  Anomalous dissipation can be interpreted as non-conservation of the average $L^2_x$ norm of the unique adapted solution to the stochastic transport equation, without taking the vanishing diffusivity limit. It is defined in general as the strict inequality
\begin{align} \label{eq:anom_diss_def}
\mathbb{E} \| \theta_t \|_{L^2_x}^2
<
\| \theta_0 \|_{L^2_x}^2, 
\qquad
\mbox{ for some } t > 0.
\end{align}
Bernard, Gawedzki, and Kupiainen, in their celebrated paper \cite{bernard1998slow}, linked anomalous dissipation \eqref{eq:anom_diss_def} to non-uniqueness of characteristics.  This was subsequently rigorously established by  Le Jan and Raimond.  
Specifically, they prove that the solution map $S : \theta_0 \mapsto \theta$ is generally given by composition with a \textit{flow of Markovian kernels}, rather than a flow of maps. 
In particular, particle ``trajectories'' do not always make sense and must be replaced by a suitable probabilistic notion.  Le Jan and Raimond link energy conservation (that is, equality in \eqref{eq:anom_diss_def} for every initial condition $\theta_0 \in L^2_x$) to $S$ being or not a flow of maps \cite[Lemma 6.5]{LeJRai2002}: energy conservation holds if and only if $S$ is a flow of maps.

In \cite[Section 10]{LeJRai2002}, Le Jan and Raimond focus on the Kraichnan model in $\R^d$ and show that the qualitative behaviour of the solution map $S$, and therefore of anomalous dissipation, depends only upon the space regularity $\alpha$ and the compressibility ratio $\eta$ of the noise.
For values of $\alpha>1$ particle trajectories are classically defined, and $S$ is a flow of maps. 
When $\alpha \in (0,1)$ three different scenarios arise, depending on the compressibility ratio:
\begin{itemize}
    \item[1)]  For $\eta < 1-\frac{d}{4\alpha^2}$, forward particle trajectories are classically defined and coalesce together when they hit each other; $S$ is a (semi)flow of maps.
    \item[2)] In the \textit{diffusive regime} $\eta > 1- \frac{d}{4\alpha^2}$ particle trajectories are not classically defined and $S$ is the composition with a genuine Markovian kernel. Morally, from any initial point of $\R^d$ infinitely many particle trajectories emanate. Moreover:
    \item[2a)] For $\eta\in (1- \frac{d}{4\alpha^2}, \frac12 - \frac{d-2}{4\alpha})$, trajectories intersect at positive times with positive probability.
    \item[2b)] For $\eta> \frac12 - \frac{d-2}{4\alpha}$, such intersections happen with probability $0$.
\end{itemize}

In view of \cite[Lemma 6.5]{LeJRai2002} of Le Jan and Raimond, in the diffusive regime for the Kraichnan model on $\mathbb{R}^d$ anomalous dissipation is known to occur for some initial data.
But one could ask how generic the phenomenon is.
Our first result, valid in the general framework of Gaussian homogeneous noise, is that anomalous dissipation is generic whenever it occurs.
\begin{theorem} \label{thm:dissipation_intro} 
Let $\mathfrak{D}$ be the set of initial conditions $\theta_0 \in L^2_x$ such that \eqref{eq:anom_diss_def} holds. Then either $\mathfrak{D}$ is empty or $\mathfrak{D} = L^2_x \setminus \{0\}$,  and  for every $\theta_0 \in \mathfrak{D}$ energy dissipation happens continuously in time:
\begin{align*}
\mathbb{E} \| \theta_t \|_{L^2_x}^2 
<
\mathbb{E} \| \theta_s \|_{L^2_x}^2,
\qquad
\forall s<t.
\end{align*}
In particular, specifically for the Kraichnan model:
\begin{itemize}
    \item If $d\in \{2,3\}$ and $\eta<1-\frac{d}{4\alpha^2}$, then all solutions conserve the mean energy.
    \item If $d\geq 4$ or $\eta>1-\frac{d}{4\alpha^2}$, then all non-zero solutions continuously dissipate the mean energy.
    \end{itemize}
\end{theorem}

Constructing deterministic examples of vector fields with H\"older regularity which give rise to anomalous scalar dissipation has received interest lately  \cite{DrElIyJe22, CoCrSo23,ArVi23+,BuSzWu23+,ElLi23+,JoSo24+,HeRo25,HeRo25+}. 
For the incompressible Kraichnan model on the torus, genericity of anomalous dissipation has been proved by Rowan in \cite{rowan2023} using PDE techniques.

The proof of \autoref{thm:dissipation_intro} is obtained proving that the set $\mathfrak{C} := \mathfrak{D}^c$ of initial conditions giving energy conservation is: $i$) contained in the subspace of $L^2_x$ orthogonal to a certain non-zero $h \in L^1_x \cap L^2_x$, as a consequence of Riesz Representation Theorem, and $ii$) dense in the space $\mathcal{C}_0$ of continuous functions vanishing at $\infty$, as a consequence of Stone-Weierstrass Theorem.
Details are in \autoref{sec:dissipation}.

We further study the local energy balance for \eqref{eq:Kraichnan}, showing that the solution $\theta$ admits a local energy balance that is governed by the stochastic PDE
\begin{equation*}
 \dd  |\theta_t|^2   + \circ \, \dd W \cdot \nabla | \theta_t |^2  + \dd \mathcal{D}[\theta] = 0,
\end{equation*}
where $\mathcal{D}[\theta]$ is a non-negative distribution that, analogously to the Duchon-Robert distribution for the Euler equations \cite{DuRo00}, describes exactly how the kinetic energy of $\theta$ is dissipated in space and time.
Then we show that
\begin{align*}
    \mathcal{D}[\theta](\rmd t, \rmd x) = \lim_{\kappa\to 0} \kappa |\nabla \tilde\theta_t^\kappa(x)|^2 \rmd t \dd x,
\end{align*}
where the equality is interpreted in the sense of distributions, and the limit in probability.  In view of this identification $\mathcal{D}[\theta]$ is, in fact, almost surely a non-negative Radon measure (see Theorem \ref{thm:energy_balance_inviscid}).  Assuming that the noise is divergence-free, we can derive an explicit formula for  $\mathbb{E}\mathcal{D}[\theta]$.
\begin{theorem} \label{thm:dissipation_measure_intro}
In the divergence-free case $\eta = 1$, there is a positive constant $\tilde c=\tilde c(\alpha,d)$ so that   
\begin{align}\label{idealdefect}
\mathbb{E} \mathcal{D}[\theta] (\rmd t, \rmd x) 
= \tilde c
%\tr C(0) (1-\alpha)
\lim_{r \rightarrow 0} 
\oint_{\mathbb{S}^{d-1}} 
\frac{\mathbb{E}| \delta_{r \hat{y}} \theta_t (x) |^2}{r^{2- 2\alpha}} \sigma ( \mathrm{d} \hat{y}) \dd t \dd x,
\end{align}
where $\delta_{r} \theta_t(x):= \theta_t(x+r)-\theta_t(x) $ and the identity holds as space-time distributions.
\end{theorem}

The  proportionality constant is explicit, see \autoref{rem:proportionality_constant}. Interestingly, $\tilde c(\alpha,d)$ does not vanish or explode as $\alpha$ approaches the endpoints $\alpha\downarrow 0^+$ and $\alpha\uparrow 1^-$.

Formula \eqref{idealdefect} is a manifestation of Yaglom's law for scalar turbulence, and can be regarded as an analogue of Duchon-Robert's result for the Kraichnan model.
It shows that the dissipation measure is directly connected to the scaling properties of $\theta$, controlled by the parameter $\alpha$ from the spatial regularity of the velocity field.  From \eqref{idealdefect}, one immediately deduces that
\begin{align*}
    \int_I\mathbb{E} \| \theta_t \|_{H^{1 - \alpha}_x}^2\dd t < \infty \text{ or }\int_I\mathbb{E} \| \theta_t \|_{B^{(1 - \alpha)+}_{2,\infty}}^2\dd t < \infty \quad \Rightarrow\quad
    \mathbb{E} [\| \theta_t \|_{L^2_x}^2] 
=
\mathbb{E} [\| \theta_s \|_{L^2_x}^2] \text{ for all } s, t \in I.
\end{align*}
The threshold regularity exponent $1-\alpha$ corresponds to the Obukhov-Corrsin regularity theory \cite{Ob49, Co51} for the stochastic transport equation, representing the minimal requirement for anomalous dissipation to take place, cf. \autoref{rmk:OC} below.  In fact, despite the above heuristic is obtained in the divergence-free setting, solutions possess exactly the borderline regularity allowing dissipation of mean energy in the whole compressibility regime inducing the latter, as we shall now further discuss.
The result below is stated for the Kraichnan model only, but holds true more generally under \autoref{ass:sharpness_reg} and \autoref{assumption}.

\begin{theorem} \label{thm:regularity_Kraichnan}
Let $W$ be given by the Kraichnan noise in the diffusive regime $\eta>1-\frac{d}{4\alpha^2}$. 
Then for every $\theta_0 \in L^2_x$ and small $\delta>0$ the unique solution of \eqref{eq:Kraichnan} satisfies
\begin{align} \label{eq:regularization_intro}
\int_0^t \mathbb{E} \| \theta_s \|^2_{H^{1-\alpha-\delta}_x} \dd s
\lesssim (1+t)\, \| \theta_0 \|_{L^2_x}^2,
\qquad
\forall\, t \geq 0,
\end{align}
\begin{align}\label{eq:regularization_intro3}
    \EE \| \theta_t\|_{H^{1-\alpha-\delta}_x}^2 \lesssim (1\wedge t)^{-\frac{1}{1-\alpha}-\delta} \| \theta_0\|_{L^2_x}^2,\qquad \forall\, t > 0,
\end{align}
and
\begin{align}\label{eq:regularization_intro4}
    \EE \| \theta_t\|_{L^\infty_x}^2 \lesssim (1\wedge t)^{-\frac{d}{2(1-\alpha)}-\delta} \| \theta_0\|_{L^2_x}^2, \qquad \forall\, t > 0.
\end{align}
If in addition $\theta_0\in H^1_x$, then
\begin{align} \label{eq:regularization_intro2}
 \mathbb{E} \| \theta_t \|^2_{H^{1-\alpha-\delta}_x} \lesssim (1+t)\, \| \theta_0 \|_{H^1_x}^2,
\qquad
\forall\, t \geq 0.   
\end{align}
Moreover, the regularity gain up to $1-\alpha$ is sharp. Indeed, for every non-zero $\theta_0 \in L^2_x$ it holds that
\begin{align} \label{eq:sharpness_intro}
\int_s^t \mathbb{E} \| \theta_r \|^2_{H^{1-\alpha}_x} \dd r
= \infty,
\qquad
\forall\, s<t.
\end{align}
Finally, if the noise $W$ is divergence-free ( $\eta = 1$), then for every non-zero $\theta_0 \in L^2_x$ and $t \geq 0$ 
\begin{align} \label{eq:regularization_besov_intro}
\tilde{c} \lim_{r \rightarrow 0}
\int_0^t
\oint_{\mathbb{S}^{d-1}} 
 \frac{\mathbb{E}\| \delta_{r \hat{y}} \theta_s \|^2_{L^2_x}}{r^{2- 2\alpha}} \sigma ( \mathrm{d} \hat{y}) \dd s
= \|\theta_0\|_{L^2_x}^2 - \mathbb{E} \|\theta_t\|_{L^2_x}^2
\in (0,\infty).
\end{align}
\end{theorem}

Following \cite{GaGrMa24+}, we call this phenomenon \emph{anomalous regularization}, since every $L^2_x$ initial datum immediately jumps into a positive regularity Sobolev space
and gains $L^\infty_x$-integrability.
This behavior clearly breaks time reversal symmetry, just as anomalous dissipation does.

In the particular case of the incompressible Kraichnan model ($\eta=1$), anomalous regularization in Sobolev spaces of \emph{negative regularity}, for both \eqref{eq:Kraichnan} and the two-dimensional Euler equations in vorticity form, has been recently proved by Coghi and Maurelli in \cite{CoMa23+}.
Their approach has been subsequently extended to other fluid dynamics SPDEs, see \cite{JiLu24,BaGaMa24,BaGrMa24}. 
For the linear incompressible Kraichnan model, the result was strengthened in \cite{GaGrMa24+}, which obtained Sobolev regularization in the full range of exponents $(-d/2+1-\alpha,1-\alpha)$, which in particular contains the strictly positive interval $(0,1-\alpha)$.
In comparison to the aforementioned literature, we cover for the first time the full diffusive regime $\eta > 1- \frac{d}{4\alpha^2}$, %of the Kraichnan model,
(the natural regime where regularization is expected), and we establish \eqref{eq:regularization_intro} under less restrictive assumptions on the initial condition $\theta_0$, cf. also \cite{CLP25+}.
The pointwise-in-time estimates \eqref{eq:regularization_intro3}-\eqref{eq:regularization_intro4}-\eqref{eq:regularization_intro2}, as well as the identities \eqref{idealdefect}-\eqref{eq:regularization_besov_intro} in the divergence-free case, are novel contributions of this work, too.

Since the proof of \autoref{thm:regularity_Kraichnan} relies on novel techniques, different than those employed in the aforementioned works, let us explain the source of the regularization stated therein. 
In fact, we obtain \eqref{eq:regularization_intro} and \eqref{eq:regularization_besov_intro} as a consequence of different arguments.

As for \eqref{eq:regularization_intro}, under the \autoref{assumption} on the short-distance behavior of the incompressible and compressible parts of the covariance $C$, we study the two-point self-correlation function 
\begin{align*}
F_t^\kappa(z)
:=
\mathbb{E} \left[ \int_{\R^d} \tilde\theta_t^ \kappa (x+ z) \tilde\theta_t^ \kappa(x) \dd x \right],
\end{align*}
which is readily related to the second-order structure function $S^\kappa_t(z) := \EE\|\delta_z \tilde\theta_t^\kappa\|^2_{L^2_x}$ via the relation $S^\kappa_t(z) = 2(F_t^\kappa(0)-F_t^\kappa(z))$. In particular, there is a direct correspondence between Sobolev regularity of $\tilde{\theta}^\kappa$ and H\"older regularity of $F^\kappa$, see \autoref{lem:Sobolev-Holder}. Then, we notice that $F^\kappa$  solves the parabolic PDE in non-divergence form
\begin{align} \label{eq:PDE_F_intro_tilde}
\partial_t F^ \kappa  = (1-\kappa) Q(z) : D_z^2 F^ \kappa  + 2\kappa \Delta F^ \kappa .
\end{align}
We show that solutions of this problem (which is degenerate at $z=0$ when $\kappa=0$, since $Q(z) \sim |z|^{2\alpha}$) have the appropriate uniform-in-$\kappa$ fractional regularity.
More specifically, for every $\delta>0$ sufficiently small it holds that
\begin{align*}
\sup_{\kappa \in (0,1/2)}
\int_0^t \|F^\kappa_s\|_{C^{2-2\alpha-\delta}_x} \dd s
\lesssim
(1+t) \|\hat{F}_0\|_{L^1_x}
=
(1+t) \|\theta_0\|_{L^2_x}^2,
\end{align*}
where $\hat{F}_0$ denotes the Fourier transform of $F_0$.
This is a consequence of the more general \autoref{thm:reg_F_intro}, which is of independent interest.

\begin{remark}[On the regularity theory for degenerate parabolic equations]
    It is clear from our analysis that there is a correspondence between the regularity results for \eqref{eq:Kraichnan} and those for the associated two-point self-correlation $F$, solving
    \begin{align} \label{eq:PDE_F_intro_tilde_bis}
    \partial_t F  = Q(z) : D_z^2 F.
    \end{align}
    In particular, any further progress in the regularity theory for \eqref{eq:PDE_F_intro_tilde} may help improving the statements of \autoref{thm:regularity_Kraichnan}.
    Unfortunately, \eqref{eq:PDE_F_intro_tilde_bis} belongs to a class of degenerate parabolic PDEs (due to $Q(0)=0$) in non-divergence form which lack an existing satisfactory regularity theory.
    While other kinds of degeneracies have been successfully treated (see \cite{AuFiVi24,AuFiVi24+} and the references therein), even in the divergence-form analogue of \eqref{eq:PDE_F_intro_tilde_bis} there are several questions left open since the works \cite{ChiSer1984,ChSe87,GuWh91}.
    The main difficulty with such PDEs comes from the elliptic and parabolic regularity theories displaying very different behaviours, in contrast with the non-degenerate case.
    Due to the asymptotics $Q(z)\sim |z|^{2\alpha}$ as $z\sim 0$, $Q$ belongs to a class of elliptic PDEs with Muckenhoupt weights for which local H\"older regularity was established in \cite{FKS1982}, by means of De Giorgi-Nash-Moser arguments; however \cite{ChiSer1984} provided explicit counterexamples to similar results, for the same class of weights and solutions, for the parabolic case in divergence form.
    Later \cite{ChSe87,GuWh91} successfully established Harnack inequalities, with \cite[Thm. 3.6]{ChSe87} reaching a small $C^\eps$ regularity result, for some not easily quantifiable $\eps>0$; but the impossibility to perform standard bootstrap arguments did not allow to further improve the H\"older regularity exponent. To the best of our knowledge, to this day the state of the art for such PDEs remains as in \cite{ChSe87}.

    In contrast to the above, for our specific PDE \eqref{eq:PDE_F_intro_tilde_bis} we are able to obtain rather satisfactory regularity results, see in particular \autoref{thm:reg_F_intro} and \autoref{cor:reg_F_improved}.
    The $C^{2-2\alpha-\delta}_x$-regularity for $F$ is essentially sharp, as a consequence of \eqref{eq:sharpness_intro} and the deep link between $F$ and $\theta$.
    We believe our success is due to previous approaches being too general, as they only took in the account the degree of degeneracy of $Q$ at the origin (dictated by $\alpha$), while neglecting its precise structure (the incompressibility parameter $\eta$) which also plays a crucial role.    
\end{remark}

Let us move to discussing \eqref{eq:regularization_besov_intro}.
As mentioned before, it can be obtained by leveraging on the representation formula \eqref{idealdefect} for the dissipation measure, as we do in \autoref{sec:dissipation} below.
Let us also mention that, in order to derive \eqref{idealdefect} for general $\alpha \in (0,1)$, we use some degree of Sobolev regularity of $\theta$ implied by \eqref{eq:regularization_intro}, cf. \autoref{rem:required_regularity}. 
However, for $\alpha>1/2$, the derivation of \eqref{idealdefect} works under the sole assumption of $L^2_x$-regularity: 
The balance law, once derived,  provides a regularizing effect on the solution, immediately putting it into a space where the flux $\rmd \mathcal{D}[\theta]$ is finite.
Such a mechanism holds for the Burgers equation (and for more general nonlinear one-dimensional scalar conservation laws) \cite{goldman2015new} and has been proposed as a source of regularization in Navier-Stokes \cite{drivas2022self}.

Let us present an heuristic computation to give an intuition about the anomalous regularization mechanism. Let $\tilde\theta^\kappa$ be a solution to \eqref{eq:intro_viscous_approx} and consider its second-order structure function $S^\kappa_t(z)=\EE[\| \delta_z \tilde{\theta}^\kappa_t\|_{L^2_x}^2]$.
For simplicity, let the noise $W$ be divergence-free and statistically self-similar ($\eta=1$, $m=0$).
Applying It\^o's formula one can see that $S^\kappa$ formally solves the linear SPDE
\begin{align*}
    \partial_t S^\kappa_t(z)
    & = (1-\kappa) Q(z): D^2 S^\kappa_t(z) -2\kappa \EE[\|\nabla\delta_z\tilde\theta^\kappa_t\|_{L^2_x}^2]\\
    & = (1-\kappa) \nabla\cdot (Q \nabla S^\kappa_t)(z)-2\kappa \EE[\|\nabla\delta_z\tilde\theta^\kappa_t\|_{L^2_x}^2],
\end{align*}
where the second equality follows from $Q$ being divergence free, since $W$ is so. Integrating over the ball $B_r=\{|z|\leq r\}$, one finds
\begin{align*}
    \frac{\dd}{\dd t} \fint_{B_r} S^\kappa_t(z) \dd z
    =(1-\kappa) \fint_{B_r}\nabla\cdot (Q \nabla S^\kappa_t)(z) \dd z -2\kappa \fint_{B_r} \EE[\|\nabla\delta_z\tilde\theta^\kappa_t\|_{L^2_x}^2] \dd z
    =: (1-\kappa) I^1_r - 2\kappa I^2_r.
\end{align*}
Applying the Divergence Theorem and exact self-similarity $Q(z) \propto |z|^{2\alpha}$, one finds
\begin{align*}
    I^1_r
    &\sim 
    r^{-d} \int_{\partial B_r} \hat z\cdot Q(z)\nabla S^\kappa_t(z) \sigma(\rmd z)
    \sim 
    r^{2\alpha-d} \int_{\partial B_r} \hat z\cdot \nabla S^\kappa_t(z) \sigma(\rmd z)\\
    & \sim r^{2\alpha-d} \frac{\rmd}{\rmd r} \int_{\partial B_r} S^\kappa_t(\hat y) \sigma(\rmd \hat y)\\
    & \sim
    r^{2\alpha-1} \frac{\rmd}{\rmd r} \fint_{\partial B_r} S^\kappa_t(\hat y) \sigma(\rmd \hat y)
    +(d-1)r^{2\alpha-2} \fint_{\partial B_r} S^\kappa_t(\hat y) \sigma(\rmd \hat y).
\end{align*}
As the last term on the right-hand side is positive, we can obtain a lower bound on $I^1_r$ by dropping it.
Rearranging the above relations and 
integrating in time, uniformly over $\kappa\leq 1/2$, we get
\begin{align*}
    \frac{\rmd}{\rmd r}\int_0^T \oint_{\mathbb{S}^{d-1}} S^\kappa_t(r\hat y) \sigma(\rmd \hat y) \dd t
    &\lesssim 
    r^{1-2\alpha} \Big[ \fint_{B_r} S^\kappa_T(z) \dd z - \fint_{B_r} S^\kappa_0(z) \dd z\Big] + 2\kappa \int_0^T I^2_t \dd t \\
    &\lesssim r^{1-2\alpha} \sup_{z\in\R^d} \Big( S^\kappa_T(z) + 2\kappa \fint_{B_r} \int_0^T \EE[\|\nabla\delta_z\tilde\theta^\kappa_t\|_{L^2_x}^2] \dd t \dd z\Big).
\end{align*}
Integrating in $r$ and using the definition of $S^\kappa_t$ together with triangular inequality, we arrive at
\begin{align*}
    \int_0^T \oint_{\mathbb{S}^{d-1}} \frac{\EE[\| \delta_{r \hat{y}}  \tilde\theta^\kappa_t\|_{L^2_x}^2]}{r^{2-2\alpha}} \sigma(\rmd \hat y)
    \lesssim \EE[\|\tilde\theta^\kappa_T\|_{L^2_x}^2] + 2\kappa \int_0^T \| \nabla \tilde{\theta}^\kappa_t\|_{L^2_x}^2] \dd t
    = \|\theta_0\|_{L^2_x}^2,
\end{align*}
where in the last step we used \eqref{eq:contraction_L2k}.
The resulting (formal) estimate is uniform in $\kappa$, thus allowing to take the limit as $\kappa\to 0^+$ and obtain a corresponding statement for $\theta$ solving \eqref{eq:Kraichnan}.

Clearly the argument is only formal, for a number of reasons; however it nicely features many of the key quantities we want to use in order to get rigorous bounds on $\theta$: i) the behaviour of $Q$ around the origin; ii) the use of structure functions to quantify Sobolev regularity; iii) the importance of considering spherical averages of increments and exploiting isotropy of $Q$.
Even though the argument assumed $W$ to be divergence free, such tools ultimately allow to carry out the proof of \autoref{thm:regularity_Kraichnan} in the full weakly compressible regime $\eta>1-\frac{d}{4\alpha^2}$.
The rigorous proof of \eqref{eq:regularization_besov_intro} in the divergence free case, based on a different argument, will be presented in \autoref{section:dissipation_measure}.

We now discuss the final aspect of the present work, which concerns the stochastic behavior of ``particles'' in the Kraichnan velocity field. Rigorously, particles are described by the two-point motion $(X,Y)$ defined by Le Jan and Raimond in \cite{LeJRai2002}, cf. also \autoref{ssec:two_point} for additional details.
In classical theory of incompressible turbulent fluid, the statistics of the separation between two particle trajectories has historically been addressed for the first time by Richardson in \cite{Ri26}.
Notwithstanding the subsequent century of investigations, a sound mathematical formalization is still missing.
According to Richardson, fluid particles in a turbulent fluid separate at rate proportional to $t^{3/2}$, on average, at least for small times $t \ll 1$. 
Moreover, this rate is insensitive to the initial separation of fluid particles (i.e. independent of $|X_0-Y_0|$).

Recall that in \cite[Section 10]{LeJRai2002}, Le Jan and Raimond study non-uniqueness of Kraichnan trajectories, depending on the space regularity $\alpha$ and the compressibility ratio $\eta$ of the noise. The behavior they uncover in the diffusive regime is reminiscent of Richardson dispersion in turbulence, where insensitivity to initial condition is captured by stochastic splitting of trajectories in rough fields.  This phenomenon has been termed \emph{spontaneous stochasticity} and is a subject of intense investigation \cite{bernard1998slow, EVan2000, eyink2015spontaneous, drivas2017lagrangian, mailybaev2016spontaneous, eyink2020renormalization, mailybaev2023spontaneous, mailybaev2023spontaneously, drivas2024statistical, bandak2024spontaneous}.  In addition, it is worth remarking that the analysis in \cite[Section 10]{LeJRai2002} can be extended without difficulties to more general homogeneous isotropic noises, not necessarily given by the Kraichnan covariance. The role of the compressibility in the different phase transitions is mirrored by the relative weight of the longitudinal and normal components of the covariance $C$ close to zero, as encompassed in our \autoref{assumption}.

The above picture can be made precise via studying the stochastic continuity equation \eqref{eq:Kraichnan.continuity}
\begin{align*} 
\dd \mu_t + \nabla \cdot (\mu_t \circ \dd W)  = 0,
\end{align*}
starting from a Dirac delta. Indeed, the mean square particle separation is captured by
\begin{align*}
\mathbb{E}[|X_t-Y_t|^2] = 2 \mathbb{E}[\var(\mu_t)],
\end{align*}
where the expected variance of $\mu_t$ is defined as
\begin{align*}
\mathbb{E}[\var(\mu_t)]
:=
\mathbb{E} \left[ \int_{\R^d} \left| y - \int_{\R^d} x \mu_t( \mathrm{d} x) \right|^2 \mu_t( \mathrm{d} y) \right]
=
\frac12 \mathbb{E} \left[ \int_{\R^d} \int_{\R^d} 
|x-y|^2 \mu_t( \mathrm{d} x) \mu_t( \mathrm{d} y) \right].
\end{align*}

In general, this quantity can grow with respect to time at most like $t^{\frac{1}{1-\alpha}}$, essentially independently of the compressibility regime under consideration.
However, a lower bound is more delicate to obtain, and is generally false (for instance for the Kraichnan model in the strongly compressive regime).
Nonetheless, in the diffusive regime, spontaneous stochasticity happens and we are able to establish a precise asymptotics for $\mathbb{E}[\var(\mu_t)]$ at short time, quantifying precisely the stochastic splitting phenomena.

\begin{theorem} \label{thm:richardson}
Let $\mu$ be the unique solution of \eqref{eq:stochastic.continuity} with initial condition $\delta_x$, $x \in \R^d$, then we have
\begin{equation}\label{eq:richardson_upper}
\mathbb{E}[\var(\mu_t)]
\lesssim
t^{\frac{1}{1-\alpha}},
\quad
\forall t \geq 0,
\end{equation}
where the implicit constant depends only on the covariance $C$.
Moreover, in the diffusive regime $\eta > 1 - \frac{d}{4\alpha^2}$, there exists a positive constant $K_{\mathsf{Ric}}$ such that for every $\eps>0$ there exists a positive time $t_\eps>0$ such that
\begin{equation}\label{eq:richardson_lower}
\mathbb{E}[{\rm Var}(\mu_t)]
\geq (1-\eps)K_{\mathsf{Ric}}
t^{\frac{1}{1-\alpha}},
\quad
\forall t \in (0,t_\eps).
\end{equation}
\end{theorem}

We mention that \autoref{thm:richardson} is stated for the particular case of Kraichnan noise, but similarly to other results here presented it is valid under the more general \autoref{ass:sharpness_reg} and \autoref{assumption}. 
The precise value of $K_{\mathsf{Ric}}>0$ is explicitly computable and depends only on the covariance of the noise. 
For the Kraichnan model, $K_{\mathsf{Ric}}$ tends to zero as the intensity of the noise tends to zero or the compressibility ratio $\eta$ approaches the threshold $1-\frac{d}{4\alpha^2}$. 
The proof of \autoref{thm:richardson} is obtained by a combination of It\^o's formula and elementary inequalities, and can be found in \autoref{sec:richardson}.
In \autoref{sec:dirac} we additionally establish, in the divergence-free case, pathwise lower bounds for $\mu_t$ and a sharp $L^2_x$-regularization result of the form (see \autoref{thm:main.thm.dirac} for additional details)
\begin{align*}
    \mathbb{E} [ \| \mu_t \|_{L^2_x}^2] \lesssim t^{-\frac{d}{2(1-\alpha)}}, \quad \forall t \in (0,1).
\end{align*}
Functional inequalities directly relate the above to a lower bound for Richardson's law (cf. \autoref{lem:deterministic.lower.bound}), further providing pathwise versions of it (cf. \autoref{cor:richardson_lower_functional}).

\begin{remark}[On the exponent $1-\alpha$] \label{rmk:OC}
It is worth comparing the exponent $1-\alpha$, dictating threshold Sobolev regularity when the noise is spatially $\alpha$-regular, with the classical Obukhov-Corrsin theory.
The latter postulates that the criticality threshold for energy conservation/dissipation in deterministic transport (replacing $\dd W$ with $u \dd t$) is
$u\in C^0_t C^\alpha_x$ and $\theta \in L^2_t H^\beta_x$, with $\beta := \frac{1-\alpha}{2}$, see e.g. \cite{DrElIyJe22}.
The discrepancy between the regularity exponent $\beta$ in Obukhov-Corrsin theory and our $1-\alpha$ is essentially rooted in the fact that we are dealing with stochastic objects that are rough in time: as the velocity $\dd W_t$ is only a distribution in $t$, commutator estimates à la Constantin-E-Titi \cite{CoETi94} can fail for some $\gamma>\frac{1-\alpha}{2}$.
Roughly speaking, $\rmd W\in C^{-1/2}_t C^\alpha_x$, which requires the space regularity for $\theta$ to be doubled and the critical condition to become $\theta\in L^2_t H^{2\beta}_x$.
A similar counting can be observed in Richardson's law of particle dispersion: for $u\in C^0_t C^\alpha_x$, denoting by $\mu_t$ the variance of the solution to the continuity equation $\partial_t \mu + \nabla\cdot (u \mu)=0$ started at a Dirac delta, one expects $\var (\mu_t)\sim t^{-2/(1-\alpha)}=t^{-1/\beta}$, see \cite{HeRo25+} for a rigorous result; instead in the Kraichnan model one finds $\EE\var (\mu_t)\sim t^{-1/(1-\alpha)}=t^{-1/(2\beta)}$, cf. \autoref{thm:richardson}.
Let us finally note that the Richardson-Kolmogorov scaling of cascade of energy (corresponding to $u\in C^0_t C^{1/3}_x$) is reproduced in the Kraichnan model by taking $\alpha=2/3$ (cf. \cite[p. 23]{CFG2008}), displaying yet another consistent doubled exponent.
\end{remark}

\begin{remark}[Extensions to the torus]
It is natural to wonder whether the same results hold on the torus $\TT^d$ as well, given the expected local nature of homogeneous, isotropic turbulence.
On $\TT^d$, it is still possible to define $W$ by prescribing its covariance $C$ via its discrete Fourier transform as in \eqref{eq:kraichnan_covariance_fourier_intro}; the resulting noise is still homogeneous and somewhat isotropic \cite{Gal2020,rowan2023}.
All the results from \autoref{sec:preliminaries} and \autoref{sec:rigidity} immediately transfer to that setting as well.
Instead \autoref{thm:dissipation} as stated fails on $\TT^d$ and must be appropriately modified (homogeneity is no longer enough), see \autoref{rem:counterexample_torus} for a deeper discussion.
Most importantly, the proofs of Theorems \ref{thm:dissipation_measure_intro} and \ref{thm:regularity_Kraichnan} currently do not work on $\TT^d$. Indeed, we crucially exploit radial symmetry of $\R^d$ several times across the paper, most notably: i) in the analysis of the PDE \eqref{eq:PDE_F_intro_tilde} and its reduction to a one-dimensional problem, see \eqref{eq:PDE_f}; ii) in the derivation of asymptotic expansions of $Q$ around $0$, cf. \autoref{lem:asymptotics_kraichnan}.
A similar limitations occurs in \cite[Sec. 9-10]{LeJRai2002}, where the authors can only consider rotationally symmetric domains.
Finally, concerning Richardson's law, the proof of the upper bound \eqref{eq:richardson_upper}         still holds, while the lower bound  \eqref{eq:richardson_lower} does not. However in the incompressible case $\eta=1$, using \cite[Prop. 3.1]{rowan2023}, one can still prove an analogue of \autoref{thm:main.thm.dirac}, yielding \eqref{eq:richardson_lower} via \autoref{lem:deterministic.lower.bound}.
\end{remark}

\subsection{Structure of the paper}
The paper is structured as follows.

In \autoref{sec:preliminaries} we collect the notation and basic facts that will be used throughout this work.
In \autoref{ssec:preliminaries_transp} we recall the construction of Le Jan and Raimond and the various equivalent formulations of the stochastic transport-diffusion equation, together with some approximations of it. \autoref{ssec:continuity} treats the stochastic continuity equation
and contains the duality between the solution map $S$ and the
continuity equation. \autoref{subsec:two_point_selfcor} introduces the two-point self-correlation function and records basic regularity facts for random distributions, while in \autoref{subsec:dissipation.measure} we define the
dissipation measure and prove its basic properties.

\autoref{sec:anomalous_dissipation} is devoted to anomalous dissipation of the mean energy. After proving that anomalous dissipation is a generic phenomenon whenever it occurs (\autoref{thm:dissipation}), we show that anomalous dissipation imposes an upper bound on regularity of solutions (\autoref{prop:regularity_implies_conservation}), establishing \eqref{eq:sharpness_intro} of \autoref{thm:regularity_Kraichnan}.

\autoref{sec:regularity} contains the anomalous regularization results.
The first subsection therein is devoted to the proof of the bounds \eqref{eq:regularization_intro} and \eqref{eq:regularization_intro2} in \autoref{thm:regularity_Kraichnan}. Some consequences, such as \eqref{eq:regularization_intro3} and \eqref{eq:regularization_intro4}, are discussed in \autoref{ssec:consequences.regularization}.
Finally, in \autoref{sec:regularity_PDE} we discuss the regularity theory for the class of degenerate parabolic PDEs satisfied by self-correlation functions, providing sharp spatial H\"older regularity (\autoref{thm:reg_F_intro}).

In \autoref{section:dissipation_measure} we study the structure of the dissipation measure in the incompressible Kraichnan model and derive the explicit representation \eqref{idealdefect} that links the dissipation to spherical averages of increments of $\theta$ (\autoref{thm:exact.formula.dissipation.measure}).
This particularly implies \autoref{thm:dissipation_measure_intro} and \eqref{eq:regularization_besov_intro} of \autoref{thm:regularity_Kraichnan} from the Introduction.

\autoref{sec:richardson} is concerned with particle dispersion.
In \autoref{ssec:Richardson's law} we prove \autoref{thm:richardson} for the expected variance of solutions to the stochastic continuity equation; we relate this growth to the two-point motion defined by Le Jan and Raimond in \autoref{ssec:two_point}, thus justifying the name \emph{Richardson's law}.
Then, in \autoref{sec:dirac} we discuss instantaneous $L^2_x$-regularization of Dirac delta initial conditions in the incompressible case  (\autoref{thm:main.thm.dirac}).

Auxiliary material in collected in the Appendix; \autoref{app:auxiliary} gathers facts on the covariance of the Kraichnan model, and \autoref{app:preliminaries} contains complements and deferred proofs from \autoref{sec:preliminaries}.

\section*{Acknowledgements}

We warmly thank Prof. Raul Serapioni for useful discussions on degenerate parabolic PDEs and the results from \cite{FKS1982,ChiSer1984,ChSe87}.

The research of TDD is partially supported by the NSF DMS-2106233 grant, NSF CAREER award \#2235395, a Stony Brook University Trustee’s award as well as an Alfred P. Sloan Fellowship.

LG is supported by the SNSF Grant 182565 and by the Swiss State Secretariat for Education, Research and Innovation (SERI) under contract number MB22.00034 and by the Istituto Nazionale di Alta Matematica (INdAM) through the project GNAMPA 2025 ``Modelli stocastici in Fluidodinamica e Turbolenza'' -CUP E5324001950001.

UP is supported by the European Research Council (ERC) under the European Union’s Horizon 2020 research and innovation programme (grant agreement No. 949981).

\section{Preliminaries}
\label{sec:preliminaries}

This section is devoted to a collection of more or less standard results on the stochastic transport equation \eqref{eq:Kraichnan}, its diffusive approximation \eqref{eq:spde_viscous_approx}, and the stochastic continuity equation \eqref{eq:stochastic.continuity}.

Equations are posed in the full space $\R^d$, but most of the results contained in \autoref{ssec:preliminaries_transp}, \autoref{ssec:continuity}, and \autoref{subsec:dissipation.measure} can be easily adapted to equations defined on the $d$-dimensional torus $\mathbb{T}^d$. In particular, notice that we do not require isotropy of the noise at the moment.  
\begin{assumption} \label{ass:well_posed}
$W$ is a space-homogeneous Gaussian noise with covariance $C(x,y) = C(x-y)$ satisfying: $i$) $\hat C$ is a nonnegative, symmetric matrix such that $\hat C\in L^1(\R^d)\cap L^\infty(\R^d)$,
$ii$) it holds $\sup_{\xi \in \R^d} \xi \cdot \hat{C}(\xi) \xi < \infty$, and $iii$) $C(0)$ is non-degenerate. 
\end{assumption}

Under \autoref{ass:well_posed}, we can represent $C$ as 
\begin{equation}\label{eq:representation_covariance}
	C(x-y)=\sum_{k \in \mathbb{N}} \sigma_k(x)\otimes \sigma_k(y),
\end{equation}
where $\sigma_k : \R^d \to \R^d$ belongs $\cap_{s\geq 0} H^s(\R^d)$ 
%$L^2(\R^d) \cap C^\infty(\R^d)$
for every $k \in \mathbb{N}$ and the series is uniformly convergent on compact sets, see for instance \cite[Prop. 2.6]{GalLuo2023} or \cite[Lemma 2.5]{CoMa23+}.
Correspondingly, the noise $W$ can be represented as $W_t(x)=\sum_{k \in \mathbb{N}} \sigma_k(x) W^k_t$, where $\{W^k\}_{k\in \mathbb{N}}$ is a collection of i.i.d. Brownian motions defined on a filtered probability space $(\Omega, \mathcal{F}, \{\mathcal{F}_t\}_{t \geq 0}, \mathbb{P})$.

Associated to $C$, we can consider the convolutional operator $\mathcal{C} f := C\ast f$, which is symmetric and nonnegative. It admits a square root $\mathcal{C}^{1/2}$, which corresponds to the Fourier multiplier associated to $\hat{C}^{1/2}$, which by our assumptions belongs to $L^2(\R^d)\cap L^\infty(\R^d)$.

In order to define weak solutions of \eqref{eq:Kraichnan} and \eqref{eq:Kraichnan.continuity}, we need to make sense of the stochastic integrals therein. 
Given a Banach space $B$, we say that a stochastic process $f : \Omega \times \R_+ \to B$ is progressively measurable if it is $\{\mathcal{F}_t\}_{t \geq 0}$-progressive measurable as a $B$-valued stochastic process.

\begin{lemma}\label{lem:stoch_integr}
Suppose \autoref{ass:well_posed}. Then given two progressively measurable processes $f=(f^1,\dots,f^d)$ and $g$ such that $\PP$-almost surely for every $t<\infty$
\begin{align*}
\int_0^t \| \cC^{1/2} f_s\|_{L^2_x}^2 \dd s <\infty,
\quad
\int_0^t \| g_s\|_{L^2_x}^2 \dd s <\infty,
\end{align*}
the stochastic It\^o integrals
\begin{align*}
M^f_t 
:= 
\int_0^t \langle  f_s , \dd W_s \rangle
:=
\sum_{k \in \NN} \int_0^t \langle f_s , \sigma_k \rangle \dd W^k_s,
	\quad
	t \geq 0,
\end{align*}
and
\begin{align*}
M^g_t 
:= 
\int_0^t \langle \nabla g_s, \dd W_s \rangle 
=: 
- \int_0^t \langle g_s, \dd(\div W)_s \rangle
:=
-\sum_{k \in \NN} \int_0^t \langle g_s , \div \sigma_k \rangle \dd W^k_s,
	\quad
	t \geq 0,
\end{align*}
are well-defined local martingales with respect to the filtration $\{\mathcal{F}_t\}_{t \geq 0}$ with continuous trajectories $\mathbb{P}$-almost surely and quadratic variation given by
\begin{gather*}
[M^f]_t 
=
\int_0^t \| \cC^{1/2} f_s\|_{L^2_x}^2 \dd s 
= 
\int_0^t \langle C\ast f_s,f_s\rangle \dd s,
\\
[M^g]_t 
= 
\int_0^t \int_{\R^d} \xi\cdot \hat C(\xi)\xi |\hat g_s(\xi)|^2 \dd\xi \dd s 
\leq 
\sup_\xi \xi\cdot \hat C(\xi)\xi\, \int_0^t \| g_s\|_{L^2_{x}}^2 \dd s.
\end{gather*}
\end{lemma} 
\begin{proof}
The statement on $M^f$ can be proved arguing as in \cite[Lem. 2.8]{GalLuo2023}. As for $M^g$, let us suppose preliminarily that $g \in L^2_t H^1_x$ almost surely and let $f := \nabla g$.
By the previous result and Parseval's identity, it holds
\begin{align*}
[M^g]_t
= 
\int_0^t \langle C\ast \nabla g_s, \nabla g_s\rangle \dd s
= 
\int_0^t \int_{\R^d}  \hat C(\xi) \widehat{\nabla g_s}(\xi) \cdot \overline{\widehat{\nabla g_s}(\xi)} \dd\xi \dd s
= 
\int_0^t \int_{\R^d} \xi\cdot \hat C(\xi)\xi |\hat g_s(\xi)|^2 \dd\xi \dd s,
\end{align*}
from which the conclusion readily follows.
The general case can be treated by standard approximations and localization arguments.
\end{proof}

\begin{remark}\label{rem:bound_C_TV}
    Let $B=(\mathcal{M}(\mathbb{R}^d)^{\otimes d}, \|\cdot\|_{TV})$ denote the Banach space of finite, vector-valued, signed Radon measures, endowed with the total variation norm. Given $\mathfrak{m} \in \mathcal{M}(\mathbb{R}^d)^{\otimes d}$, we denote by $|\mathfrak{m}|=|\mathfrak{m}_1|+\ldots+|\mathfrak{m}_d|$ its total variation measure.
    It holds
    \begin{equation}\label{eq:covariance_TV_inequaity}
        \| \cC^{1/2} \mathfrak{m}\|_{L^2_x}^2 \leq {\rm Tr}(C(0))\,\| \mathfrak{m}\|_{TV}^2 \leq \| \hat C\|_{L^1_x} \,\| \mathfrak{m}\|_{TV}^2.
    \end{equation}
    The second estimate immediately follows from $C$ being a positive definite function;
    for the first one, by \eqref{eq:representation_covariance} we have
    \begin{align*}
        \| \cC^{1/2} \mathfrak{m}\|_{L^2_x}^2
        & = \langle C\ast \mathfrak{m},\mathfrak{m}\rangle
        \leq \sum_{k\in\NN} \int_{\R^d\times \R^d} |\sigma_k(x)||\sigma_k(y)| |\mathfrak{m}|(\dif x) |\mathfrak{m}|(\dif y)\\
        & \leq \frac{1}{2} \sum_{k\in\NN} \int_{\R^d\times \R^d} |\sigma_k(x)|^2 |\mathfrak{m}|(\dif x) |\mathfrak{m}|(\dif y)
        + \frac{1}{2} \sum_{k\in\NN} \int_{\R^d\times \R^d} |\sigma_k(y)|^2 |\mathfrak{m}|(\dif x) |\mathfrak{m}|(\dif y)\\
        & = \| \mathfrak{m}\|_{TV} \int_{\R^d} \Big(\sum_{k\in\NN} |\sigma_k(x)|^2\Big) |\mathfrak{m}|(\dif x)
        = {\rm Tr}(C(0))\,\| \mathfrak{m}\|_{TV}^2.
    \end{align*}
\end{remark}

In the following two subsections we provide preliminaries concerning the stochastic transport and continuity equations.
In order to keep this section as concise as possible, we postpone some of the proofs to \autoref{app:preliminaries}.

\subsection{Stochastic transport-diffusion equation} \label{ssec:preliminaries_transp}
Let $\theta_0 \in L^2_x$ be given.
We start by considering the stochastic transport equation \eqref{eq:Kraichnan}, whose It\^o form, using the contraction $C(0) : D^2 := \sum_{i,j = 1}^d C(0)^{ij} \partial_{x_i} \partial_{x^j}$, reads as
\begin{align} \label{eq:STE_ito}
\dd \theta
+ 
\dd W\cdot \nabla \theta
&= 
\frac12 C(0) : D^2 \theta\dd t.
\end{align}
In order to comprehend into our analysis the stochastic transport-diffusion equation presented in the introduction, where a diffusive term $\kappa \Delta$ is added to the right-hand side of \eqref{eq:Kraichnan}, we shall consider more generally equations of the form
\begin{align}\label{eq:spde_viscous}
\dd \theta
+ 
\dd W\cdot \nabla \theta
&= 
A \theta\dd t,
\end{align}
where $A := \frac12 C(0) : D^2 + \kappa \tilde{C}:D^2$, $\kappa \geq 0$, and $\tilde{C}$ is a uniformly elliptic matrix modelling diffusion (e.g. $\tilde{C}:D^2 = \Delta$ for $\tilde{C}=I_d$).  
The It\^o formulation \eqref{eq:spde_viscous} is more convenient for the analysis of the equation, and will be preferentially used throughout this section.
Let $P_t$ denote the semigroup associated to the operator $A$. 
For \eqref{eq:spde_viscous} we adopt the the following notion of solution, valid both when $\kappa = 0$ (in which case \eqref{eq:spde_viscous} reduces to \eqref{eq:STE_ito}) and $\kappa > 0$.
In the next statement, $\cF^W=\{\cF^W_t\}_{t\geq 0}$ denotes the (standardly augmented) natural filtration of $W$.

\begin{prop} \label{prop:solution}
Suppose \autoref{ass:well_posed} and let $\theta_0 \in L^2_x$ be given.
Let $\theta :\Omega \times \R_+ \to L^2_x$ be a $\cF^W$-progressively measurable stochastic process satisfying 
\begin{align} \label{eq:contraction_L2}
\mathbb{E} \| \theta_t \|_{L^2_x}^2 \leq   \| \theta_0 \|_{L^2_x}^2 \qquad
\forall t \geq 0.
\end{align}
Then the following are equivalent:
\begin{enumerate}
\item
For every $\varphi \in C^\infty_c(\R^d)$ and $t \geq 0$ the random variable $\langle \theta_t, \varphi \rangle$ satisfies $\PP$-almost surely
\begin{align} \label{eq:weak_sol}
\langle \theta_t , \varphi \rangle
=
\langle \theta_0 , \varphi \rangle
+
\sum_k
\int_0^t \langle \theta_s , \div (\sigma_k \varphi) \rangle \dd W^k_s 
+
\int_0^t
\langle \theta_s ,  A \varphi \rangle \dd s,
\end{align}
where the stochastic integral is well-defined by \autoref{lem:stoch_integr}.
\item
For every $\varphi \in C^\infty_c(\R^d)$ and $t \geq 0$ the random variable $\langle \theta_t, \varphi \rangle$ satisfies $\PP$-almost surely
\begin{align} \label{eq:mild_sol}
\langle \theta_t , \varphi \rangle
&=
\langle \theta_0 , P_t \varphi \rangle
+
\sum_k
\int_0^t \langle \theta_s , \div (\sigma_k P_{t-s}\varphi) \rangle \dd W^k_s,
\end{align}
where the stochastic integral is well-defined by \autoref{lem:stoch_integr}.
\item
For every $\varphi \in C^\infty_c(\R^d)$ and $t \geq 0$ the random variable $\langle \theta_t, \varphi \rangle$ satisfies $\PP$-almost surely
\begin{align} \label{eq:wiener_sol}
\langle \theta_t ,\varphi \rangle
=
\langle \theta_0 , P_t \varphi \rangle
+
\sum_{n \geq 1}
J^n_t(\theta_0,\varphi),
\end{align}
where for every $n \geq 1$ the $n$-th Wiener chaos stochastic integral is well-defined as $L^2_\omega$ random variable given by
\begin{align*}
J^n_t(\theta_0,\varphi)
&:=
\sum_{k_1,\dots,k_n}
\int_0^t \int_0^{t_1} \dots \int_0^{t_{n-1}}
\langle 
\theta_0 , P_{t_n} \mathcal{J}^{k_n}_{t_{n-1}-t_n}\dots\mathcal{J}^{k_1}_{t-t_1} \varphi \rangle   \,
\dd W^{k_1}_{t_1} \dots \dd W^{k_n}_{t_n},
\end{align*}
with $\mathcal{J}^{k}_{t}: \psi \mapsto \div ( \sigma_{k} P_t \psi)$, and the sum over $n \geq 1$ converges in $L^2_\omega$.
\end{enumerate}

Moreover, any two processes $\theta$ and $\tilde{\theta}$ satisfying the properties above are modifications of each other.
We say that a stochastic process $\theta$ satisfying the above conditions is the unique  solution of \eqref{eq:spde_viscous} (up to modifications).
\end{prop}

The previous equivalent formulations are particularly meaningful in the inviscid case $\kappa=0$.
Let us stress that in this case \autoref{prop:solution} only concerns $\cF^W$-adapted solutions and as such provides \emph{Wiener uniqueness}, cf. \cite{Fla2017open}; outside this class, uniqueness can fail, see \cite{LeJRai2004}.
In contrast, when \eqref{eq:spde_viscous} has a strictly positive diffusivity $\kappa > 0$, the analysis of the SPDE becomes easier and pathwise unique strong solutions can be constructed by rather classical arguments (cf. \autoref{prop:wellposedness_viscous_SPDE} in the Appendix).

In light of the above, in the rest of the paper, we will always take $\cF=\cF^W$ (equivalently, work on the canonical probability space associated to $W$) and restrict our analysis to the unique solutions coming from \autoref{prop:solution}. 

Let us now contextualize \autoref{prop:solution} within the theory developed by  Le Jan and Raimond.
Let us recall \cite[Theorem 3.2]{LeJRai2002}, where the authors construct a $\{\mathcal{F}_t\}_{t\geq 0}$-adapted solution map $S$ to the stochastic transport-diffusion equation \eqref{eq:spde_viscous} interpreted in a backward sense.
By points c) and d) of \cite[Theorem 3.2]{LeJRai2002}, given any $f \in L^2_x$ the solution map $S$ satisfies
\begin{align*}
\sup_{t \geq 0} \mathbb{E} [\|S_t f \|_{L^2_x}^2] \leq \|f\|_{L^2_x}^2
\end{align*}
and for every $\varphi \in L^2_x$ and $t \geq 0$ it holds $\PP$-almost surely
\begin{align*}
\langle S_t f , \varphi \rangle
=
\langle f, P_t \varphi \rangle
+
\sum_{k \in \mathbb{N}} 
\int_0^t \langle S_s \left( \sigma_k \cdot \nabla P_{t-s} f \right) , \varphi \rangle \, \dd W^k_s.
\end{align*}

Denote $\mathcal{I}^k_t : \psi \mapsto \sigma_k \cdot \nabla P_t \psi$.
By iteration of the formula above, the random variable $\langle S_t f , \varphi \rangle \in L^2_\omega$ has the following Wiener chaos decomposition, valid for every $t \geq 0$ $\PP$-almost surely 
\begin{align*}
%\label{eq:wiener_chaos_LJR}
\langle S_t f , \varphi \rangle 
&=
\langle f , P_t \varphi \rangle 
+
\sum_{n \geq 1}
I^n_t(f,\varphi),
\\
I^n_t(f,\varphi)
&:=
\sum_{k_1, \dots, k_n}
\int_0^t \int_0^{t_1} \dots \int_0^{t_{n-1}}
\langle
P_{t_n} \mathcal{I}^{k_n}_{t_{n-1}-t_n} \dots
\mathcal{I}^{k_1}_{t-t_1} f ,  \varphi\rangle  \,
\dd W^{k_1}_{t_1} \dots \dd W^{k_n}_{t_n}.
\\
&= 
(-1)^n
\sum_{k_1, \dots, k_n}
\int_0^t \int_0^{t_1} \dots \int_0^{t_{n-1}}
\langle f , P_{t-t_1} \mathcal{J}^{k_1}_{t-t_1} \dots   \mathcal{J}^{k_{n-1}}_{t_{n-1}-t_n}  \mathcal{J}^{k_n}_{t_n} \varphi \rangle \,
\dd W^{k_1}_{t_1} \dots \dd W^{k_n}_{t_n}.
\end{align*}

The next proposition states that the unique solution of the forward stochastic transport-diffusion equation \eqref{eq:spde_viscous} can be obtained by time reversal of the Le Jan-Raimond solution map $S$.

\begin{prop} \label{prop:solution_LJR}
Suppose \autoref{ass:well_posed} and let $\theta_0 \in L^2_x$ be given.
For every $t \geq 0$, consider the Brownian motions $\tilde{W}^k_s := W^k_t-W^k_{t-s}$ on the time interval $[0,t]$ and let $\tilde{S}^t$ be the Le Jan-Raimond solution map with respect to Wiener process $\tilde{W}$ and its natural filtration $\{\tilde{\mathcal{F}}_s\}_{s \in [0,t]}$.
Then $\theta :\Omega \times \R_+ \to L^2_x$ defined as $\theta_t := \tilde{S}^t_t \theta_0$ is the unique solution of \eqref{eq:spde_viscous}, up to modifications.
Moreover, $\theta$ satisfies \eqref{eq:wiener_sol} for every $\varphi \in L^2_x$. 
Finally, for every $t \geq 0$ the $L^2_x$-valued random variables $\theta_t := \tilde{S}^t_t \theta_0$ and $S_t \theta_0$ have the same law.
\end{prop}

\begin{proof}
Clearly $\theta$ is $\{\mathcal{F}_t \}_{t \geq 0}$-adapted and satisfies \eqref{eq:contraction_L2}.
Fix $t>0$ and $\varphi \in L^2_x$.  
By the Wiener chaos expansion of the Le Jan-Raimond solution map, the projection of $\langle \theta_t , \varphi \rangle = \langle \tilde{S}^t_t \theta_0 , \varphi \rangle$ onto the $n$-th chaos is given by
\begin{align*}
(-1)^n
&\sum_{k_1, \dots, k_n}
\int_0^t \int_0^{t_1} \dots \int_0^{t_{n-1}}
\langle \theta_0 , P_{t-t_1} \mathcal{J}^{k_1}_{t-t_1} \dots   \mathcal{J}^{k_{n-1}}_{t_{n-1}-t_n}  \mathcal{J}^{k_n}_{t_n} \varphi \rangle \,
\dd \tilde{W}^{k_1}_{t_1} \dots \dd \tilde{W}^{k_n}_{t_n}
\\
=
&\sum_{h_1,\dots,h_n}
\int_0^t \int_0^{s_1} \dots \int_0^{s_{n-1}}
\langle 
\theta_0 , P_{s_n} \mathcal{J}^{h_n}_{s_{n-1}-s_n} \dots   \mathcal{J}^{h_1}_{t-s_1}  \varphi \rangle  \,
\dd W^{h_1}_{s_1} \dots \dd W^{h_n}_{s_n},
\end{align*}
where we have used the change of variables $s_i := t-t_{n-i+1}$, $h_{i} := k_{n-i+1}$ and interpret the above as a $\PP$-almost sure equality. Therefore the random variable $J^n_t(\theta_0,\varphi) \in L^2_\omega$ is well-defined and $\langle \theta_t , \varphi \rangle$ satisfies \eqref{eq:wiener_sol} $\PP$-almost surely. 
Moreover, since the right-hand side of \eqref{eq:wiener_sol} defines a $\{\mathcal{F}_t\}_{t\geq 0}$-progressively measurable process for every $\varphi \in L^2_x$, $\theta$ has a weakly $\{\mathcal{F}_t\}_{t\geq 0}$-progressively measurable modification, still denoted $\theta$ for simplicity.
Since the space $L^2_x$ is separable, Pettis measurability Theorem implies that $\theta$ is progressively measurable as $L^2_x$-valued stochastic process. 
Since $C^\infty_c(\R^d) \subset L^2_x$, $\theta$ satisfies point $3$) of \autoref{prop:solution}, and thus $\theta$ is the unique solution of \eqref{eq:spde_viscous} up to modifications.

Finally, the $L^2_x$-valued random variables $S_t \theta_0$ and $\tilde{S}^t_t \theta_0$ have the same law, as $\langle S_t \theta_0, \varphi \rangle$ has the same law as $\langle \tilde{S}^t_t \theta_0, \varphi \rangle$ for every $\varphi \in L^2_x$ (looking at chaos expansion) and the latter determine the law of $S_t \theta_0$ and $\tilde{S}^t_t \theta_0 = \theta_t$ as $L^2_x$-valued random variables. 
\end{proof}

The next proposition gives a summary of some of the useful properties of the unique solution of \eqref{eq:spde_viscous}. It is valid both when $\kappa=0$ and $\kappa > 0$ (see \autoref{app:preliminaries} for the proof).

\begin{prop}\label{prop:properties.inviscid.Kraichnan}
Suppose \autoref{ass:well_posed} and let $\theta_0 \in L^2_x$ be given. Denote $\theta$ the unique solution of \eqref{eq:spde_viscous}. Then
\begin{enumerate}
    \item 
    $\theta$ is a Markov process and $\EE [\theta_t|\cF_s]=P_{t-s}\theta_s$ for any $s\leq t$.
    \item 
    If $\theta_0 \in L^2_x \cap L^p_x$ for some $p \in [1,\infty]$, then
    \begin{equation}\label{eq:contraction_Lp}
\sup_{t \geq 0}  \| \theta_t \|_{L^p_\omega L^p_x} 
\leq
\| \theta_0 \|_{L^p_x}.
\end{equation}
    \item  
    If $\theta_0 \in L^2_x \cap L^p_x$ for some $p \in (1,\infty)$, then the map $t \mapsto \mathbb{E}[\| \theta_t \|_{L^p_x}^p]$ is uniformly continuous. 
\end{enumerate}
\end{prop}

\begin{remark}\label{rem:Lp_bound_divergence_free}
    If $W$ is divergence-free, then estimate \eqref{eq:contraction_Lp} can be improved to a pathwise bound, as shown in \cite[Prop. 3.1]{GalLuo2023}: it holds
    \begin{equation}\label{eq:Lp_bound_divergence_free}
        \sup_{t \geq 0}  \| \theta_t \|_{L^p_x} \leq \| \theta_0 \|_{L^p_x}\quad \PP\text{-almost surely.}
    \end{equation}   
\end{remark}

%\subsubsection{Approximation of Equation \eqref{eq:Kraichnan}}
Let us now consider two different approximations of the transport equation \eqref{eq:Kraichnan}, formulated in its It\^o form \eqref{eq:spde_viscous} with $\kappa=0$. The first one consists in taking a particular kind of decreasing-diffusivity approximation, obtained by reducing the intensity of the noise.
Namely we denote $\tilde{\theta}^\kappa$ the solutions of the equation in It\^o form 
\begin{align} \label{eq:spde_viscous_approx}
\dd \tilde{\theta}^\kappa_t
+
\sqrt{1-\kappa}\, \dd W \cdot \nabla \tilde{\theta}^\kappa
=
\frac12 C(0):D^2 \tilde{\theta}^\kappa\dd t,
\quad
\kappa \in (0,1/2).
\end{align}
This approximation scheme falls into the setting of \autoref{prop:solution} since
\begin{align*}
\dd \tilde{\theta}^\kappa_t
+
\sqrt{1-\kappa}\, \dd W \cdot \nabla \tilde{\theta}^\kappa
=
\frac{1-\kappa}2 C(0):D^2 \tilde{\theta}^\kappa\dd t
+
\frac{\kappa}{2} C(0):D^2 \tilde{\theta}^\kappa\dd t
\end{align*}
can be written in the form \eqref{eq:spde_viscous} with noise $\sqrt{1-\kappa}\,W$, covariance $(1-\kappa)C$, and diffusive term $\tilde{C} = C(0)/2$.
The reason for considering \eqref{eq:spde_viscous_approx} comes from \cite[pp. 858-859, eq.(9.23)]{LeJRai2002} and Malliavin calculus considerations, as $\tilde\theta^\kappa_t$ can be obtained by an Ornstein-Uhlenbeck operator applied to $\theta_t$. This allows to establish strong convergence of such approximations in $L^p_\omega L^p_x$, see \autoref{lem:approx} below, which in turn will be a crucial ingredient in the proof of \autoref{thm:energy_balance_inviscid}.

The second approximation we are going to consider consists in taking a spatially smooth truncation of the noise, by letting $W^\kappa := \sum_{k \leq \kappa^{-1}} \sigma_k W^k$ with covariance $C^\kappa := \sum_{k \leq \kappa^{-1}} \sigma_k \otimes \sigma_k$.
The equation becomes
\begin{align} \label{eq:spde_inviscid_approx}
\dd \hat{\theta}^\kappa_t
+
\dd W^\kappa \cdot \nabla \hat{\theta}^\kappa
=
\frac12 C^\kappa(0):D^2 \hat{\theta}^\kappa\dd t,
\quad
\kappa > 0.
\end{align}

Let $\theta$ be the unique solutions of \eqref{eq:spde_viscous} with $\kappa=0$ and for $\kappa>0$ let $\tilde{\theta}^\kappa$ and  $\hat{\theta}^\kappa$ be the unique solutions of \eqref{eq:spde_viscous_approx} and \eqref{eq:spde_inviscid_approx} respectively.
The convergence of $\tilde{\theta}^\kappa$ and $\hat{\theta}^\kappa$ towards $\theta$ is the content of the following result, proved in \autoref{app:preliminaries}.

\begin{lemma} \label{lem:approx}
Suppose \autoref{ass:well_posed} and let $\theta_0 \in L^2_x \cap L^p_x$, $p \in (1,\infty)$, and $t \geq 0$ be given. Then the random variable $\hat{\theta}^\kappa_t$ converges weakly in $L^2_\omega L^2_x \cap L^p_\omega L^p_x$ towards $\theta_t$ as $\kappa \downarrow 0$.
Moreover, the random variable $\tilde{\theta}^\kappa_t$ converges strongly in $L^2_\omega L^2_x \cap L^p_\omega L^p_x$ as $\kappa \downarrow 0$ towards $\theta_t$.
\end{lemma}

\subsection{Stochastic continuity equation} \label{ssec:continuity}
Let us move to considering the stochastic continuity equation \eqref{eq:Kraichnan.continuity}, written in It\^{o} form as
\begin{equation}
\dd \mu_t 
+ 
\nabla \cdot (\mu_t \dd W_t) 
= 
\frac{1}{2} C(0) : D^2 \mu_t\dd t, \label{eq:stochastic.continuity}
\end{equation}
and its diffusive counterpart (recall the operator $A$ from previous subsection)
\begin{equation}
\dd \mu_t 
+ 
\nabla \cdot (\mu_t \dd W_t) 
= 
A \mu_t\dd t. \label{eq:stochastic.continuity_diffusive}
\end{equation}

The natural setting to study \eqref{eq:stochastic.continuity_diffusive} is that of finite signed Radon measures $\mathcal{M}:=\mathcal{M}(\mathbb{R}^d)$.
% , which is a  Banach space when endowed with the total variation norm
% \begin{align*}
% \| \mu \|_{TV} :=
% \sup_{B_1 \subset \R^d} \mu(B_1) - \inf_{B_2 \subset \R^d} \mu(B_2).
% \end{align*}
We consider progressively measurable processes $\mu : \Omega \times \R_+ \to \mathcal{M}$ satisfying the bound
\begin{align} \label{eq:contraction_TV}
\sup_{t \geq  0} \| \mu_t \|_{\mathrm{TV}} \leq
\| \mu_0 \|_{\mathrm{TV}},
\quad
\mathbb{P}\mbox{-almost surely},
\end{align}
and define solutions of \eqref{eq:stochastic.continuity_diffusive} as follows.
\begin{prop} \label{prop:solution.continuity}
Let $\mu_0 \in \mathcal{M}$ be given.
Let $\mu : \Omega \times \R_+ \to \mathcal{M}$ be a progressively measurable stochastic process satisfying \eqref{eq:contraction_TV}. 
Then the following are equivalent:
\begin{enumerate}
\item
For every $\varphi \in C^\infty_c(\R^d)$ and $t \geq 0$ the random variable $\langle \mu_t, \varphi \rangle$ satisfies $\PP$-almost surely
\begin{align} \label{eq:weak_sol.continuity}
\langle \mu_t , \varphi \rangle
=
\langle \mu_0 , \varphi \rangle
+
\sum_k
\int_0^t \langle \mu_s ,\sigma_k \cdot \nabla \varphi  \rangle \dd W^k_s 
+
\int_0^t
\langle \mu_s ,  A \varphi \rangle\dd s,
\end{align}
where the stochastic integral is well-defined by \autoref{lem:stoch_integr} and \autoref{rem:bound_C_TV}.
\item
For every $\varphi \in C^\infty_c(\R^d)$ and $t \geq 0$ the random variable $\langle \mu_t, \varphi \rangle$ satisfies $\PP$-almost surely
\begin{align} \label{eq:mild_sol.continuity}
\langle \mu_t , \varphi \rangle
&=
\langle \mu_0 , P_t \varphi \rangle
+
\sum_k
\int_0^t \langle \mu_s , \sigma_k \cdot \nabla P_{t-s}\varphi \rangle \dd W^k_s,
\end{align}
where the stochastic integral is well-defined by \autoref{lem:stoch_integr} and \autoref{rem:bound_C_TV}.
\item
For every $\varphi \in C^\infty_c(\R^d)$ and $t \geq 0$ the random variable $\langle \mu_t, \varphi \rangle$ satisfies $\PP$-almost surely
\begin{align} \label{eq:wiener_sol.continuity}
\langle \mu_t ,\varphi \rangle
=
\langle \mu_0 , P_t \varphi \rangle
+
\sum_{n \geq 1}
I^n_t(\varphi,\mu_0),
\end{align}
where for every $n \geq 1$ the $n$-th stochastic integral is well-defined in $L^2_\omega$, and the sum over $n \geq 1$ converges in $L^2_\omega$.
\end{enumerate}

Moreover, any two processes $\mu$ and $\tilde{\mu}$ satisfying the properties above are a modification of each other.
We say that a stochastic process $\mu$ satisfying the above conditions is the unique  solution (up to modifications) of \eqref{eq:stochastic.continuity_diffusive}.
\end{prop}

A few comments on \autoref{prop:solution.continuity} are in order.

The proof of \autoref{prop:solution.continuity} unfolds similarly to that of \autoref{prop:solution} for the stochastic transport equation in \autoref{app:preliminaries}, and we omit it. 
Notice, however, that for the stochastic continuity equation we can accommodate initial conditions in $\mathcal{M}$.
This is possible for two reasons: 
First, we only need to use the first part of \autoref{lem:stoch_integr} to define the stochastic integrals above, as no divergence of the noise appears when formally integrating \eqref{eq:stochastic.continuity} by parts to transfer the derivatives on the test function $\varphi$;
Second, we can combine \eqref{eq:contraction_TV} with  \autoref{rem:bound_C_TV} applied to the vector-valued measures $\mathfrak{m} := \mu \nabla \varphi$ and $\mathfrak{m} := \mu \nabla P_{t-\cdot} \varphi$, giving the bound needed to satisfy the assumption of the first part of \autoref{lem:stoch_integr}.

In addition, we point out that \autoref{prop:solution.continuity} guarantees uniqueness of solutions to \eqref{eq:stochastic.continuity_diffusive}, but not existence.
This is contained in the next proposition, valid both in the case $\kappa>0$ and $\kappa=0$. For the lack of a suitable reference in the literature, we give a proof of it in \autoref{app:preliminaries}.

\begin{prop}
\label{prop:existence.SCE}
For any deterministic measure $\mu_0 \in \mathcal{M}$, there exists a progressively measurable solution $\mu$ to \eqref{eq:stochastic.continuity_diffusive}, 
with the properties that $\mathbb{P}$-a.s. $\mu \in C^{\gamma}_t H^{- d/2 - 2 - \eps}_x$ for all $\gamma < 1 / 2$ and $\eps > 0$, and such that \eqref{eq:contraction_TV} holds.
The solution $\mu$ is unique up to modifications, by the last part of \autoref{prop:solution.continuity}.
\end{prop}

Next, we introduce duality formulae linking the Le Jan-Raimond solution map $S_t$ and the unique solution of the continuity equation \eqref{eq:stochastic.continuity_diffusive}. We remark that these formulae are valid also in the inviscid case $\kappa=0$, when  \eqref{eq:stochastic.continuity_diffusive} reduces to \eqref{eq:stochastic.continuity}.

Notice that $I^n_t(\varphi,\mu_0)$ in \eqref{eq:wiener_sol.continuity} differs from the Wiener chaos of the Le Jan-Raimond solution map $S_t$ only by the order of the solution and the test function:
\begin{align*}
\langle \mu_t ,\varphi \rangle
=
\langle \mu_0 , P_t \varphi \rangle
+
\sum_{n \geq 1}
I^n_t(\varphi,\mu_0),
\quad
\mbox{whereas}
\quad
\langle S_t f , \varphi \rangle 
=
\langle f , P_t \varphi \rangle 
+
\sum_{n \geq 1}
I^n_t(f,\varphi).
\end{align*}
Thus we can formally identify the Le Jan-Raimond solution map $S_t$ and the (forward) stochastic continuity equation $\mu_t$ as adjoint of each other. More precisely we have the following:
\begin{lemma} \label{lem:duality}
For every $f \in C_c(\R^d)$, $\mu_0 \in L^2_x$, and $t \geq 0$ it holds $\PP$-almost surely
\begin{align*}
\langle S_t f, \mu_0 \rangle = \langle f, \mu_t \rangle.
\end{align*}
Furthermore, given $x \in \R^d$, denote $\mu^x$ the unique solution of \eqref{eq:stochastic.continuity_diffusive} with Dirac delta initial condition $\mu^x_0 = \delta_x$.
Then it holds for every $f \in L^2_x$ and $\PP \otimes \mathscr{L}^d$-almost every $(\omega,x) \in \Omega \times \R^d$ 
\begin{align*}
S_t f (x,\omega) =
\langle f, \mu^{x}_t \rangle.
\end{align*}
\end{lemma}

\subsection{Two-point self-correlation function and regularity of random distributions}\label{subsec:two_point_selfcor}

Let $\tilde{\theta}^\kappa$ be the unique solution of \eqref{eq:spde_viscous_approx} with $\kappa \in [0,1/2)$
% $\kappa = 0$ included. 
and denote $F^\kappa$ the two-point self-correlation function:
\begin{align} \label{eq:two-point_correlation}
F^\kappa(t,z) 
:= 
\mathbb{E}\left[ 
\int_{\R^d}\tilde{\theta}^\kappa_t(x+z)\tilde{\theta}^\kappa_t(x)\dd x \right]
= \EE[\langle \tilde{\theta}^\kappa_t(\cdot+z), \tilde{\theta}^\kappa_t\rangle_{L^2_x}]
=\EE[(\tilde\theta^\kappa_t\ast\tilde\theta^{\kappa-}_t)(z)],
\end{align}
where $f^-(x)=f(-x)$ denotes the reflection of a function (or distribution) $f$.
In the next lemma we derive a closed PDE for $F^\kappa$.
Recall that we have defined $Q(z)=C(0)-C(z)$.

\begin{lemma}\label{lem:PDE_twopoint_transport}
For every $\kappa \in [0,1/2)$ the two-point self-correlation function $F^\kappa$ is a weak solution of the parabolic PDE
\begin{align} \label{eq:PDE_F}
\partial_t F^\kappa 
= 
(1-\kappa) Q : D^2_z F^\kappa + \kappa C(0): D^2_z F^\kappa.
\end{align}

\end{lemma}
\begin{proof}
In view of the $L^2_{\omega,x}$ convergence $\tilde{\theta}^\kappa_t \to \theta_t$ of \autoref{lem:approx}, we can limit ourselves to the case $\kappa>0$.
The basic idea is to apply It\^o Formula to the process $t \mapsto \tilde{\theta}^\kappa_t(x+z) \tilde{\theta}^\kappa_t(x)$ for every fixed $x,z \in \R^d$. Rigorously, one should replace $\tilde{\theta}^\kappa$ by a smooth mollification $\tilde{\theta}^{\kappa,\eps}$, apply It\^o Formula to the mollifications and then pass to the limit,
%. Since this procedure has already been carried on rigorously in the proof of \autoref{prop:energy_balance_viscous_kappa}, 
but we omit the technical details and proceed formally.

The computation is similar to that of \cite[Proposition 2.1]{rowan2023}.
Let us denote for simplicity $y:=x+z$ and $r^\kappa_t(x,y) := \EE[\tilde{\theta}^\kappa_t(x) \tilde{\theta}^\kappa_t(y)] $. It\^o Formula gives
\begin{align*}
\frac{\dd}{\dd t} r^\kappa_t(x,y)
&= 
\frac12 \EE\left[ \tilde{\theta}^\kappa_t(x) \, C(0):D^2_y  \tilde{\theta}^\kappa_t(y) 
\right]
+ 
\frac12 \EE\left[  C(0):D^2_x \tilde{\theta}^\kappa_t(x) \, \tilde{\theta}^\kappa_t(y)\right]
\\
&\quad+ 
(1-\kappa)\EE\left[ \sum_k \sigma_k(x)\cdot\nabla_x \tilde{\theta}^\kappa_t(x)\, \sigma_k(y)\cdot\nabla_y \tilde{\theta}^\kappa_t(y)\right].
\end{align*}
Observing that
\begin{align*}
\sum_k \sigma_k(x)\cdot\nabla_x \tilde{\theta}^\kappa_t(x)\, \sigma_k(y)\cdot\nabla_y \tilde{\theta}^\kappa_t(y)
&= 
C(x-y) : \nabla_x \tilde{\theta}^\kappa_t(x)\otimes \nabla_y \tilde{\theta}^\kappa_t(y)
\end{align*}
and that
\begin{align*}
\partial_{x_i} r^\kappa_t(x,y)
= 
\EE[\partial_{x_i} \tilde{\theta}^\kappa_t(x)\tilde{\theta}^\kappa_t(y)], \quad 
\partial_{y_j} r^\kappa_t(x,y)
= 
\EE[\tilde{\theta}^\kappa_t(x) \partial_{y_j}\tilde{\theta}^\kappa_t(y)],
\end{align*}
we can collect the terms above to arrive at
\begin{equation}\label{eq:aux006}
\partial_t r^\kappa_t(x,y)
= 
\frac{1}{2} C(0) : (D^2_x+D^2_y) r^\kappa_t(x,y) 
+  
(1-\kappa)C(x-y): \nabla_x \nabla_y r^\kappa_t(x,y).
\end{equation}
Since 
\begin{align*}
\int_{\R^d}\partial_{x_i} r^\kappa_t (x,y) \dd x
= 
- \partial_{z_i} F^\kappa(t,z),
\quad
\int_{\R^d} \partial_{y_i} r^\kappa_t (x,y) \dd x
= 
\partial_{z_i} F^\kappa(t,z),
\end{align*}
we obtain the desired result upon integrating \eqref{eq:aux006} with respect to $x$ and rewriting the matrix $C(0)$ as $(1-\kappa)C(0) + \kappa C(0)$, so to isolate the coefficient $(1-\kappa)C(0) - (1-\kappa)C(x-y) = (1-\kappa) Q(z)$.
\end{proof}

\begin{remark}\label{rem:PDE_twopoint_continuity}
Replacing $\tilde{\theta}^\kappa$ with solutions of the diffusive stochastic continuity equation 
\begin{equation}\label{eq:viscous.stochastic.continuity}
\dd \tilde{\mu}^\kappa_t 
+ 
\sqrt{1-\kappa}\nabla \cdot (\tilde{\mu}^\kappa_t \dd W_t) 
= 
\frac{1}{2} C(0) : D^2 \tilde{\mu}^\kappa_t\dd t, 
\end{equation}
with similar arguments as above one obtains the closed PDE satisfied by the two-point self-correlation  $G^\kappa := \mathbb{E} \int_{\R^d} \tilde{\mu}^\kappa(\cdot+x) \tilde{\mu}^\kappa(x) \dd x$, namely
\begin{align} \label{eq:PDE_G}
\partial_t G^\kappa = (1-\kappa) D_z^2 : (Q G^\kappa) + \kappa C(0): D^2_z G^\kappa.
\end{align} 
\end{remark}

\begin{remark} \label{rmk:isotropy}
In the special case $C(0) = 2c_0 I_d$ for some $c_0>0$, valid for instance when the noise is isotropic, we obtain the simplified expressions for \eqref{eq:PDE_F} and \eqref{eq:PDE_G}: 
\begin{align*}
\partial_t F^\kappa 
&= 
(1-\kappa) Q : D^2_z F^\kappa + 2 c_0 \kappa \Delta  F^\kappa,
\\
\partial_t G^\kappa 
&= 
(1-\kappa) D_z^2 : (Q G^\kappa) + 2 c_0 \kappa \Delta G^\kappa.
\end{align*}
\end{remark}

The interest in the self-correlation function, besides the PDE \eqref{eq:PDE_F}, comes from the fact that its regularity is dictated exclusively by its local behaviour around $0$ and is linked to the regularity of the original random function from which it was generated, as the next results show.

\begin{lemma}\label{lem:Sobolev-Holder}
Let $f$ be a random function in $L^2(\Omega;L^2_x)$ and let $F=F[f]:=\EE[f\ast f^-]$ be the associated self-correction function. Then for every $s \in (0,1]$ the following hold:
\begin{enumerate}
      \item If there exist $l>0$ such that 
    \begin{align*}
        \llbracket F\rrbracket_{I^{2s}(l)} := \sup_{|z|< l} \frac{|F(z)-F(0)|}{|z|^{2s}}<\infty,
    \end{align*}
    then for any $\delta>0$ sufficiently small, $f\in L^2(\Omega;H^{s-\delta}_x)$ with $\EE[\| f\|_{H^{s-\delta}_x}^2]\lesssim_{l,s,\delta} \llbracket F\rrbracket_{I^{2s}(l)} + F(0)$.
    \item Conversely, if $f\in L^2(\Omega;H^s_x)$, then $F[f]\in C^{2s}_x$ with $\| F\|_{C^{2s}_x} \lesssim_s \EE[\|f\|_{H^{s}_x}^2]$.
\end{enumerate}
In particular, if $\llbracket F\rrbracket_{I^{2s}(l)}<\infty$, then $F\in C^{2s-\delta}_x$ for every small $\delta>0$.
\end{lemma}

\begin{proof}
To prove $(1)$, since by assumption $f\in L^2(\Omega;L^2_x)$, it suffices to check finiteness of the Gagliardo--Niremberg seminorm. It holds
\begin{align*}
\mathbb{E}[ \| f \|^2_{\dot{H}^{s-\delta}_x}]
& = \int_{\R^d}
\frac{ \mathbb{E} [\| f(\cdot+z) -f(\cdot) \|^2_{L^2_x}] }{|z|^{d+2s-2\delta}} \dd z
\\
& = 2 \int_{\R^d} \frac{ \mathbb{E} [\|f \|^2_{L^2_x} - \langle  f(\cdot+z), f\rangle_{L^2_x}] }{|z|^{d+2s-2\delta}} \dd z 
= 2 \int_{\R^d} \frac{ F(0) - F(z) }{|z|^{d+2s-2\delta}} \dd z.
\end{align*}
In the region $|z|\geq l$, we can estimate the above integral using that $|F(z)|\leq F(0)$ (see \eqref{eq:selfcorrelation.sup.bound} below). Instead, in the region $|z|<l$ by assumption we have
\begin{align*}
    \int_{|z|<l} \frac{ F(0) - F(z) }{|z|^{d+2s-2\delta}} \dd z
    \leq \llbracket F\rrbracket_{I^{2s}(l)} \int_{|z|<l} \frac{ 1}{|z|^{d-2\delta}} \dd z\lesssim \llbracket F\rrbracket_{I^{2s}(l)}
\end{align*}
yielding the desired estimate. 

Concerning $(2)$, notice that the convolution of two functions in $H^s_x$ belongs to $C^{2s}_x$, as can be seen by looking at its Fourier transform.
Thus
%, by Young-type inequalities we get
\begin{equation*}
    \| F\|_{C^{2s}_x} \leq \EE[\|f\ast f^-\|_{C^{2s}_x}] \lesssim_s \EE[\|f\|_{H^{s}_x}^2]. \qedhere
\end{equation*}
\end{proof}

Note that, if $f\in L^2(\Omega;L^p_x)$ for some $p\in [1,2]$, by Young's inequality $F[f]$ is still a well-defined function with
\begin{align*}
    \| F\|_{L^r_x}\leq \EE[\| f\ast f^-\|_{L^r_x}] \lesssim \EE[\| f\|_{L^p_x}^2],\quad\text{for } r=\frac{p}{2p-1}.
\end{align*}
Similarly, if $f\in L^2(\Omega;\mathcal{M})$, then $F\in \mathcal{M}$ with $\| F\|_{TV}\leq \EE[\|f\|_{TV}^2]$.
Recall that $F$ belongs to the Fourier--Lebesgue space $\mathcal{F}L^1_x$ if $\hat{F}\in L^1_x$.

\begin{lemma}\label{lem:L^2-criterion-selfcorrelation}
    Let $f\in L^2(\Omega;E)$ for either $E=L^p_x$ with $p\in [1,2]$, or $E=\mathcal{M}$, and let $F$ be the associated self-correlation function. Then $f\in L^2(\Omega;L^2_x)$ if and only if $F$ is locally bounded around the origin. Moreover in this case $F$ is continuous and bounded, $F\in \mathcal{F}L^1_x$ and
    \begin{equation}\label{eq:selfcorrelation.sup.bound}
        \| F\|_{C^0_x} = \| F\|_{L^\infty_x} = \| \hat F\|_{L^1_x} = \EE[\| f\|_{L^2_x}^2]= F(0).
    \end{equation}
\end{lemma}

\begin{proof}
    First assume $f\in L^2(\Omega;L^2_x)$. Then by Young's inequality for convolutions we have
    \begin{align*}
        \| F\|_{C^0_x} \leq \EE[\| f\ast f^-\|_{C^0_x}] \leq \EE[\| f\|_{L^2_x}^2]
    \end{align*}
    and by properties of Fourier transform 
    \begin{align} \label{eq:hat_xi=L2}
    \hat F(\xi)=\EE[\widehat{f\ast f^-}(\xi)]=\EE[|\hat f(\xi)|^2]\geq 0,
    \end{align}
    so that by Parseval identity
    \begin{align*}
        \int_{\R^d} |\hat F(\xi)| \dd \xi
        = \int_{\R^d} \EE[|\hat f(\xi)|^2] \dd \xi = \EE[\| f\|_{L^2_x}^2].
    \end{align*}
    The last equality in \eqref{eq:selfcorrelation.sup.bound} immediately follows from the definition of $F$ and implies the chain of equalities since $F(0) \leq \| F\|_{C^0_x}$.

    Next assume $f\in L^2(\Omega;E)$ for $E$ as above such that $F$ is locally bounded around $0$. Let $\chi$ be a compactly supported, radially symmetric smooth probability density, $\{\chi^\eps\}_{\eps>0}$ the associated standard mollifiers, and set $\psi=\chi\ast\chi$, so that $\psi^\eps=\chi^\eps\ast\chi^\eps$. By properties of convolutions
    \begin{align*}
        F^\eps:=\psi^\eps\ast F
        =\EE[(\chi^\eps\ast\chi^\eps)\ast (f \ast f^-)]
        =\EE[(\chi^\eps \ast f)\ast (\chi^\eps \ast f)^-]=F[f \ast \chi^\eps]. 
    \end{align*}
    By the assumptions, $f \ast \chi^\eps\in L^2(\Omega;L^2_x)$, and therefore \eqref{eq:selfcorrelation.sup.bound} applies to $F^\varepsilon$.
    Since $F$ is bounded around $0$ and $\psi^\eps$ are mollifiers with support of size $\sim \varepsilon$, we deduce the existence of $\eps_0>0$ small enough such that $\sup_{\eps<\eps_0} F^\eps(0) < \infty$, and in particular by \eqref{eq:selfcorrelation.sup.bound} we find
    \begin{align*}
    \sup_{\eps<\eps_0} \EE[\|f \ast \chi^\eps\|_{L^2_x}^2] = \sup_{\eps<\eps_0} F^\eps(0) <\infty.
    \end{align*}
   Since $f \ast \chi^\eps\to f$ in the sense of distributions, the conclusion follows by the lower semicontinuity of the $L^2_x$-norm.
\end{proof}

To compare more precisely the regularity of a random function $f$ and its self-correlation function $F=F[f]$, we need to introduce some notation.
On a fixed time horizon $[0,T]$, given $p\in [1,\infty)$, set
\begin{align*}
    \| f\|_{L^p_{\omega,t,x}}
    =\| f\|_{L^p(\Omega\times [0,T]\times\R^d)}
    :=\bigg(\int_0^T \EE[\| f_t\|_{L^p_x}^p] \dd t\bigg)^{1/p}.
\end{align*}
Given $z\in\R^d$, we define the increment operator $\delta_z$ by $(\delta_z g)(x):=g(x+z)-g(x)$.
For $s\in (0,1)$, we then define the Besov-type space $\tilde{L}^p_{t,\omega} \tilde B^s_{p,\infty}$ as the collection of random functions $f\in L^p_{\omega,t,x}$ such that
\begin{align*}
    \llbracket f\rrbracket_{\tilde{L}^p_{t,\omega} \tilde B^s_{p,\infty}}
    := \sup_{z\neq 0} \frac{1}{|z|^{s}} \| \delta_z f\|_{L^p_{\omega,t,x}}
    =\bigg( \sup_{z\neq 0} \frac{1}{|z|^{sp}} \int_0^T \EE[\| \delta_z f_t\|_{L^p_x}^p] \dd t \bigg)^{1/p}.
\end{align*}
Such spaces can be considered as random analogues of the ones naturally arising in regularity criteria for energy conservation in fluid dynamics, see for instance \cite{CCFS08,BGSTW2019}. They can also be regarded as a random version of the space-time Besov spaces $\tilde{L}^q_t \dot B^s_{p,r}$ presented in \cite[\S 2.6.3]{BaChDa2011}.

\begin{remark}\label{rem:lsc_besov_type}
    Being related to Besov-type spaces, the seminorms $\llbracket\, \cdot\,\rrbracket_{\tilde L^p_{t,\omega} \tilde B^s_{p,\infty}}$ naturally enjoy suitable lower-semicontinuity type properties. For instance, for any $s\in (0,1)$ and $p\in [1,\infty)$, if $\{f^n\}_n\subset L^p_{\omega,t,x}$ is a sequence such that $f^n$ converge weakly in $L^p_{\omega,t,x}$ to $f$, then
    \begin{align*}
        \llbracket f\rrbracket_{\tilde L^p_{t,\omega} \tilde B^s_{p,\infty}} \leq \liminf_{n\to\infty} \llbracket f^n\rrbracket_{\tilde L^p_{t,\omega} \tilde B^s_{p,\infty}}.
    \end{align*}
    A similar argument applies to $p=\infty$, where we may replace weak convergence by weak-$\ast$ convergence.
\end{remark}

The next technical lemma clarifies the nature of these spaces. Therein, we denote by $\{\dot\Delta_j\}_{j\in\ZZ}$ the homogeneous Littlewood--Paley blocks, cf. \cite{BaChDa2011}. 
Recall that $\dot\Delta_j \varphi=\psi_j\ast\varphi$, for $\psi_j(x)=2^{jd} \psi(2^j x)$, where $\psi$ is a radial Schwartz function with Fourier transform supported in an annulus.

\begin{lemma}\label{lem:besov_type_spaces}
    Let $p\in[1,\infty)$, $s\in (0,1)$. For any $f\in L^p_{\omega,t,x}$, the following quantities are equivalent:
    \begin{equation}\label{eq:besov_type_spaces}
        \llbracket f\rrbracket_{\tilde{L}^p_{t,\omega} \tilde B^s_{p,\infty}}^p
        \sim \sup_{j\in\ZZ} 2^{sjp} \| \dot\Delta_j f\|_{L^p_{\omega,t,x}}^p
        \sim \sup_{\eps>0} \frac{1}{\eps^{sp}} \fint_{\SS^{d-2}} \| \delta_{\eps \hat z} f\|_{L^p_{\omega,t,x}}^p \sigma(\dif \hat z).
    \end{equation}
    Moreover
    \begin{equation}\label{eq:besov_type_spaces2}\begin{split}
        \lim_{|z|\to 0} \frac{1}{|z|^{sp}} \int_0^T \EE[\| \delta_z f_t\|_{L^p_x}^p] \dd t=0
        &\Leftrightarrow \lim_{j\to+\infty} 2^{sj} \| \dot\Delta_j f\|_{L^p_{\omega,t,x}} = 0\\
        &\Leftrightarrow \lim_{\eps\to 0} \frac{1}{\eps^{sp}} \fint_{\SS^{d-2}} \| \delta_{\eps \hat z} f\|_{L^p_{\omega,t,x}}^p \sigma(\dif \hat z) =0,
    \end{split}\end{equation}
    and the above happens if and only if $f$ belongs to the closure in $\tilde{L}^p_{t,\omega} \tilde B^s_{p,\infty}$ of smooth-in-space random functions.
\end{lemma}

The proof of \autoref{lem:besov_type_spaces} is postponed to \autoref{app:proof.besov.spaces}. With the above preparation, we can relate fractional regularity of $f$ in $L^2_x$-scales to that of $F$ in $L^\infty_x$-scales.

\begin{lemma}\label{lem:refined.Besov.Holder}
    Let $f\in L^2_{\omega,t,x}$, $F_t=F[f_t]$ be its self-correlation function. Then for any $s\in (0,1)$, the following quantities are equivalent:
    \begin{equation}\label{eq:refined.Besov.Holder}
        \llbracket f\rrbracket_{\tilde{L}^2_{t,\omega} \tilde B^s_{2,\infty}}^2
        \sim \sup_{z\neq 0} \frac{1}{|z|^{2s}} \int_0^T |F_t(z)-F_t(0)| \dd t
        \sim \sup_{j\in\ZZ} 2^{2sj} \int_0^T \| \dot\Delta_j F_t\|_{L^\infty_x} \dd t.
    \end{equation}
    Moreover
    \begin{align*}
        \lim_{|z|\to 0} \frac{1}{|z|^{2s}} \int_0^T \EE[\| \delta_z f_t\|_{L^2_x}^2] \dd t=0
        \Leftrightarrow
        \lim_{|z|\to 0} \frac{1}{|z|^{2s}} \int_0^T |F_t(z)-F_t(0)| \dd t = 0.
    \end{align*}
\end{lemma}
Roughly speaking, \autoref{lem:refined.Besov.Holder} states that $f\in \tilde{L}^2_{t,\omega} \tilde B^s_{2,\infty}$ if and only if $F\in \tilde{L}^1_t \dot B^{2s}_{\infty,\infty}$ (in the sense of \cite{BaChDa2011}). 
Moreover, the regularity of $f$ is sharply in $\tilde{L}^2_{t,\omega} \tilde B^s_{2,\infty}$ if and only if $F$ displays a sharp power-law behaviour around the origin (after integrating in time).

\begin{proof}
    First note that, as in \autoref{lem:Sobolev-Holder},
    \begin{align*}
        \frac{1}{|z|^{2s}} \int_0^T \EE[\| \delta_z f_t\|_{L^2_x}^2] \dd t
        = \frac{2}{|z|^{2s}} \int_0^T (F_t(0)-F_t(z)) \dd t
        = \frac{2}{|z|^{2s}} \int_0^T |F_t(z)-F_t(0)| \dd t
    \end{align*}
    where in the second step we used \eqref{eq:selfcorrelation.sup.bound}. The above identity yields both the equivalence of the first two seminorms appearing in \eqref{eq:refined.Besov.Holder} and the last statement.

    Next assume that $f \in \tilde{L}^2_{t,\omega} \tilde B^s_{2,\infty}$. Then by properties of convolutions and Littlewood--Paley blocks,
    \begin{align*}
        \dot\Delta_j F_t
        = \EE[(\dot\Delta_j f_t)\ast f^-_t]
        = \sum_{|j'-j|\leq 1} \EE[(\dot\Delta_j f_t)\ast (\dot\Delta_{j'} f_t)^-],
    \end{align*}
    so that by Young's convolution inequalities, uniformly over $j\in \ZZ$ we find
    \begin{align*}
       \int_0^T \| \dot\Delta_j F_t\|_{L^\infty_x} \dd t
       & \lesssim \sum_{|j'-j|\leq 1} \int_0^T \EE[\|\dot\Delta_j f_t\|_{L^2_x} \|\dot\Delta_{j'} f_t\|_{L^2_x}] \dd t\\
       & \lesssim  \sum_{|j'-j|\leq 1} \| \dot\Delta_j f\|_{L^2_{\omega,t,x}} \| \dot\Delta_{j'} f\|_{L^2_{\omega,t,x}}
       \lesssim 2^{-2js} \llbracket f\rrbracket_{\tilde{L}^2_{t,\omega} \tilde B^s_{2,\infty}}^2,
    \end{align*}
    where we used \eqref{eq:besov_type_spaces}.
    Conversely, assume $F$ displays the above regularity; by $\dot\Delta_j\varphi=\psi_j\ast \varphi$ and properties of convolutions
    \begin{align*}
        \EE[\| \dot\Delta_j f_t\|_{L^2}^2]
        = \EE[\langle \psi_j\ast f_t,\psi_j\ast f_t\rangle]
        = \EE[\langle \psi_j\ast (f_t\ast f_t^-),\psi_j\rangle]
        = \langle \psi_j \ast F_t, \psi_j\rangle.
    \end{align*}
    It follows that
    \begin{align*}
        \EE[\| \dot\Delta_j f_t\|_{L^2}^2]\leq \| \psi_j \ast F_t\|_{L^\infty_x} \| \psi_j\|_{L^1_x} = \| \dot\Delta_j  F_t\|_{L^\infty_x} \| \psi\|_{L^1_x}
    \end{align*}
    and so
    \begin{align*}
        \sup_{j\in \ZZ} 2^{2js} \int_0^T \EE[\| \dot\Delta_j f_t\|_{L^2}^2] \dd t \lesssim  \sup_{j\in\ZZ} 2^{2sj} \int_0^T \| \dot\Delta_j F_t\|_{L^\infty_x} \dd t,
    \end{align*}
    from which we deduce \eqref{eq:refined.Besov.Holder} by applying \eqref{eq:besov_type_spaces}.
\end{proof}

\subsection{Dissipation measure}\label{subsec:dissipation.measure}

Next, we are interested in finding an expression for the local energy balance of solutions of \eqref{eq:spde_viscous_approx} when $\kappa = 0$ and $\kappa >0$. First, we derive an explicit stochastic PDE solved by the quantity $|\tilde{\theta}^\kappa|^2$ for $\kappa > 0$. This in hand, we take the limit $\kappa \downarrow 0$ and identify a dissipation measure in the local energy balance of the inviscid SPDE \eqref{eq:Kraichnan}. 

\begin{prop}\label{prop:energy_balance_viscous_kappa}
Suppose \autoref{ass:well_posed}.
Let $\theta_0\in L^2_x$ and let $\tilde{\theta}^\kappa$ be the unique solution of \eqref{eq:spde_viscous_approx} with $\kappa \in (0,1/2)$.
Then $|\tilde{\theta}^\kappa|^2$ satisfies the SPDE
\begin{equation}\label{eq:energy_balance_viscous_kappa}
\dd |\tilde{\theta}^\kappa|^2 
+ 
2\sqrt{1-\kappa}
\tilde{\theta}^\kappa \nabla \tilde{\theta}^\kappa \cdot \dd W 
=  
\frac12 C(0) :D^2 |\tilde{\theta}^{\kappa}|^2\dd t
- 
\kappa C(0) : \nabla \tilde{\theta}^{\kappa} \otimes \nabla \tilde{\theta}^{\kappa}
\dd t.
\end{equation}
\end{prop}

\begin{proof}
Mollifying $\tilde{\theta}^\kappa$ with a smooth mollifier $\chi^\eps$ we find the SPDE for $\tilde{\theta}^{\kappa,\eps}:= \tilde{\theta}^\kappa \ast \chi^\eps$
\begin{align*}
\dd \tilde{\theta}^{\kappa,\eps} 
+ 
\sqrt{1-\kappa}
\sum_k (\sigma_k\cdot\nabla \tilde{\theta}^\kappa)^\eps \dd W^k 
= 
\frac12 C(0) : D^2 \tilde{\theta}^{\kappa,\eps}\dd t,
\end{align*}
where $(\sigma_k\cdot\nabla \tilde{\theta}^\kappa)^\eps := (\sigma_k\cdot\nabla \tilde{\theta}^\kappa) \ast \chi^\eps$. The SPDE above has rigorous balance
\begin{align*}
\dd |\tilde{\theta}^{\kappa,\eps}|^2 
&+ 
\sqrt{1-\kappa}\sum_k 2 \tilde{\theta}^{\kappa,\eps}  (\sigma_k\cdot\nabla \tilde{\theta}^\kappa)^\eps \dd W^k 
\\
&=  
\tilde{\theta}^{\kappa,\eps} C(0):D^2 \tilde{\theta}^{\kappa,\eps}\dd t 
+ 
(1-\kappa)\sum_k |(\sigma_k\cdot\nabla \tilde{\theta}^\kappa)^\eps|^2\dd t
\\
&=
\frac12 C(0) :D^2 |\tilde{\theta}^{\kappa,\eps}|^2\dd t
- 
C(0) : \nabla \tilde{\theta}^{\kappa,\eps} \otimes \nabla \tilde{\theta}^{\kappa,\eps}
\dd t
+ 
(1-\kappa)\sum_k |(\sigma_k\cdot\nabla \tilde{\theta}^\kappa)^\eps|^2\dd t
.
\end{align*}
Next we pass to the limit as $\eps\to 0$. 
Clearly $| \tilde{\theta}^{\kappa,\eps}_t|^2 \to | \tilde{\theta}^{\kappa}_t|^2$ in $L^1_x$ and $C(0) :D^2| \tilde{\theta}^{\kappa,\eps}_t|^2 \to C(0) :D^2| \tilde{\theta}^{\kappa}_t|^2$ as distributions for every fixed $t \geq 0$, because of \eqref{eq:contraction_L2}.
Moreover, since for almost every $t\geq 0$ we have $\nabla \tilde{\theta}_t^\kappa\in L^2_x$ $\PP$-almost surely by \autoref{prop:wellposedness_viscous_SPDE}, we have as well $|\nabla \tilde{\theta}^{\kappa,\eps}_t|^2 \to |\nabla \tilde{\theta}^{\kappa}|^2$ and $C(0) : \nabla \tilde{\theta}_t^{\kappa,\eps} \otimes \nabla \tilde{\theta}_t^{\kappa,\eps}=\sum_k |\sigma_k\cdot\nabla \tilde{\theta}^{\kappa,\eps}_t|^2 \to \sum_k |\sigma_k\cdot\nabla \tilde{\theta}^{\kappa}_t|^2$ in $L^1_x$ $\PP$-almost surely.
The more challenging terms are the stochastic integral and the last term. For the last one, notice that for every $k \in \NN$ it holds for almost every $t \geq 0$ $\PP$-almost surely
\begin{align*}
	|(\sigma_k\cdot\nabla \tilde{\theta}^\kappa_t)^\eps|^2 \to |\sigma_k\cdot\nabla \tilde{\theta}^\kappa_t|^2 \text{ in } L^1_x.
\end{align*}
On the other hand, for every $T<\infty$ and $N \in \NN$ we have
\begin{align*}
\sup_{\eps \in (0,1)}
\int_0^T \int_{\R^d} \sum_{k>N} |(\sigma_k\cdot\nabla \tilde{\theta}^\kappa_t)^\eps|^2 \dd x\dd t	
&\leq 
\int_0^T \int_{\R^d} \sum_{k>N} |\sigma_k\cdot\nabla \tilde{\theta}^\kappa_t|^2 \dd x\dd t
\\
&= 
\int_0^T \int_{\R^d} \nabla \tilde{\theta}^\kappa_t(x) \cdot \left(\sum_{k>N} \sigma_k\otimes \sigma_k \right)\nabla \tilde{\theta}^\kappa_t(x) \dd x\dd t,
\end{align*}
and this term becomes $\PP$-almost surely infinitesimal as $N\to\infty$ by Dominated Convergence.
Therefore for almost every $t \geq 0$ it holds $\PP$-almost surely
\begin{align*}
- (1-\kappa)
C(0) : \nabla \tilde{\theta}_t^{\kappa,\eps} \otimes \nabla \tilde{\theta}_t^{\kappa,\eps}
+ 
(1-\kappa)\sum_k |(\sigma_k\cdot\nabla \tilde{\theta}_t^\kappa)^\eps|^2
\to 0 \text{ in } L^1_x.
\end{align*}

The stochastic term is quite tedious but also similar: for every $k$, the $\PP$-almost sure convergence holds
\begin{align*}
\tilde{\theta}^{\kappa,\eps}(\sigma_k\cdot\nabla \tilde{\theta}^\kappa)^\eps 
\to 
\tilde{\theta}^{\kappa}(\sigma_k\cdot\nabla \tilde{\theta}^\kappa)
\text{ in } L^2_t L^1_x,
\end{align*}
implying that the stochastic integral
\begin{align*}
\int_0^T
 \tilde{\theta}^{\kappa,\eps}_t  (\sigma_k\cdot\nabla \tilde{\theta}^\kappa_t)^\eps \dd W^k_t
\to
\int_0^T  \tilde{\theta}^{\kappa}_t (\sigma_k\cdot\nabla \tilde{\theta}^\kappa_t) \dd W^k_t
\end{align*}
in $L^2_\omega L^2_t L^1_x$ for every $k \in \NN$. In order to show the convergence in a space of distributions of the whole series
\begin{align*}
2 \sqrt{1-\kappa} \sum_k\int_0^T \tilde{\theta}^{\kappa,\eps}_t  (\sigma_k\cdot\nabla\tilde{\theta}^\kappa_t)^\eps \dd W^k_t
\to
2 \sqrt{1-\kappa} \sum_k\int_0^T  \tilde{\theta}^{\kappa}_t (\sigma_k\cdot\nabla\tilde{\theta}^\kappa_t) \dd W^k_t 
\end{align*}
it is sufficient to show that, for every $T<\infty$, as $N \to \infty$ it holds uniformly in $\eps \in (0,1)$
\begin{align*}
\mathbb{E} \left[ \left|
\sum_{k > N} 
\int_0^T \langle \tilde{\theta}^{\kappa,\eps}_t (\sigma_k \cdot \nabla \tilde{\theta}^\kappa_t)^\eps  ,\varphi_t\rangle \dd W^k_t 
\right|^2 \right] 
\to 
0
\end{align*}
for every test function $\varphi \in L^\infty_t L^\infty_x$.  
Let us expand the convolutions in the stochastic integral above
\begin{align*}
\sum_{k > N}& 
\int_0^T \langle \tilde{\theta}^{\kappa,\eps}_t (\sigma_k \cdot \nabla \tilde{\theta}^\kappa_t)^\eps  ,\varphi_t\rangle \dd W^k_t
\\
&=
\int_{\R^d}\int_{\R^d} 
\sum_{k > N} 
\int_0^T \int_{\R^d} \tilde{\theta}^{\kappa}_t(x-y) \sigma_k(x-z) \cdot \nabla \tilde{\theta}^\kappa_t(x-z)  \varphi_t(x) \dd x \dd W^k_t   \chi^\eps(y) \chi^\eps(z)\dd y\dd z 
\\
&=
\int_{\R^d}\int_{\R^d} 
\sum_{k > N} 
\int_0^T \langle \tilde{\theta}^{\kappa}_t(\cdot+z-y) \sigma_k \cdot \nabla \tilde{\theta}^\kappa_t , \varphi_t(\cdot+z) \rangle \dd W^k_t   \chi^\eps(y) \chi^\eps(z)\dd y\dd z. 
\end{align*}
By the analogue of \autoref{lem:stoch_integr} for the noise $\sum_{k > N} \sigma_k W^k$, we have that for every fixed $y,z \in \R^d$ the stochastic integral above is the value at time $T$ of a local martingale $M^f$ with quadratic variation 
\begin{align*}
\left[ M^f \right]_t
&\lesssim
\int_0^t 
\langle (C-C^N) \ast f_s , f_s \rangle\dd s,
\end{align*}
for $f_s := \tilde{\theta}^{\kappa}_s(\cdot+z-y) \nabla \tilde{\theta}^\kappa_s  \varphi_s(\cdot+z)$.
As a consequence,
\begin{align*}
\mathbb{E} &\left[ \left|
\sum_{k > N} 
\int_0^T \langle \tilde{\theta}^{\kappa,\eps}_t (\sigma_k \cdot \nabla \tilde{\theta}^\kappa_t)^\eps  ,\varphi_t\rangle \dd W^k_t 
\right|^2 \right] 
\\
&\lesssim
\int_{\R^d}\int_{\R^d} 
\mathbb{E} \left[ \left| M^f_T  \right|^2 \right]
\chi^\eps(y) \chi^\eps(z)\dd y \dd z
\\
&\lesssim
\int_{\R^d}\int_{\R^d} \mathbb{E} \left[\int_0^T  \langle (C-C^N) \ast f_t , f_t \rangle \right]\dd t \chi^\eps(y) \chi^\eps(z)\dd y\dd z 
\to 0
\end{align*}
as $N \to \infty$ by Dominated Convergence, since by \autoref{lem:stoch_integr_estim} and \autoref{prop:wellposedness_viscous_SPDE}
\begin{align*}
\mathbb{E} \left[ \int_0^T \langle C \ast f_t , f_t \rangle\dd t \right]
&\lesssim
\mathbb{E} \left[ \int_0^T \| f_t\|_{L^1_x}^2\dd t \right]
\lesssim
\|\varphi\|_{L^\infty_{t,x}}^2
\mathbb{E} \left[  
\sup_{t \in [0,T]}\|\tilde{\theta}^{\kappa}_t\|_{L^2_x}^2 
\int_0^T \| \nabla \tilde{\theta}^\kappa_t \|_{L^2_x}^2\dd t \right] < \infty. \qedhere
\end{align*}
\end{proof}

Under stronger integrability requirements on $\theta_0$, we are able to rigorously deduce a SPDE for the energy $|\theta|^2$ of solutions to \eqref{eq:Kraichnan}.

\begin{theorem}\label{thm:energy_balance_inviscid}
Suppose \autoref{ass:well_posed}, and
let $\theta_0\in L^2_x \cap L^4_x$, or alternatively let $\theta_0 \in L^2_x$ and $W$ be divergence-free. Then the unique solution $\theta$ to \eqref{eq:Kraichnan} is such that
\begin{equation}\label{eq:energy_balance_inviscid}
\dd |\theta|^2 
+ 
\nabla|\theta|^2 \cdot \dd W 
= 
\frac12 C(0) : D^2 |\theta|^2\dd t 
- 
\mathcal{D}[\theta]
\end{equation}
where the \emph{dissipation measure} $\mathcal{D}[\theta](\dif t,\dif x)$ is a nonnegative finite random measure, given by the limit (in probability with respect to $\omega$, in the sense of distributions with respect to $t,x$) as $\kappa\to 0^+$ of $ \kappa C(0) :  \nabla \tilde{\theta}^\kappa_t(x) \otimes  \nabla \tilde{\theta}^\kappa_t(x) \dif t \dif x$.
\end{theorem}

\begin{proof}
Apply \autoref{prop:energy_balance_viscous_kappa} to $\tilde\theta^\kappa$. We intend to take the limit as $\kappa \downarrow 0$ in each term in \eqref{eq:energy_balance_viscous_kappa}.
For the terms $\dd |\tilde{\theta}^\kappa|^2$ and $
\frac12 C(0) :D^2 |\tilde{\theta}^{\kappa}|^2\dd t$ we 
exploit \autoref{lem:approx}. 
As for the stochastic integral, we have
\begin{align*}
\sqrt{1-\kappa} \nabla |\tilde\theta^\kappa|^2 \cdot \dd W
&= 
\sqrt{1-\kappa}\nabla\cdot (|\tilde\theta^\kappa|^2 \dd W) 
- 
\sqrt{1-\kappa}|\tilde\theta^\kappa|^2 \dd({\rm div} W)
\\
&\to \nabla\cdot (| \theta|^2 \dd W) 
- 
|\theta|^2 \dd({\rm div} W).
\end{align*}
In fact, for the convergence of the first one needs to invoke \autoref{lem:stoch_integr}, \autoref{lem:stoch_integr_estim}, the bound \eqref{eq:contraction_L2}, and the strong $L^2_\omega L^2_x$ convergence $\tilde\theta^\kappa_t \to \theta_t$ given by \autoref{lem:approx}, which only require $\theta_0 \in L^2_x$.
This hypothesis also suffices when $\div W = 0$, as the last stochastic integral disappears. 
On the other hand, when $\div W \neq 0$, the term $|\tilde\theta^\kappa|^2 \dd({\rm div} W)$ remains and we need $L^4_x$ integrability of $\theta_0$ to rigorously define the last stochastic integral by \autoref{lem:stoch_integr}.
The convergence is then a consequence of \autoref{lem:approx} and \eqref{eq:contraction_Lp}, both applied with $p=4$.

Since three out of four terms in \eqref{eq:energy_balance_viscous_kappa} converge in probability with respect to $\omega$, in the sense of distributions with respect to $t,x$, so does the term $ \kappa C(0) :  \nabla \tilde{\theta}^\kappa_t(x) \otimes  \nabla \tilde{\theta}^\kappa_t(x) \dif t \dif x$.
We denote $\mathcal{D}[\theta]$ the limiting random distribution.
It is clear that $\mathcal{D}[\theta] \geq 0$ as the same is true for every $\kappa > 0$, and therefore $\mathcal{D}[\theta]$ is $\PP$-almost surely a non-negative measure. Moreover, by \autoref{prop:wellposedness_viscous_SPDE} the measure $\mathcal{D}[\theta]$ has $\PP$-almost surely finite mass.
\end{proof}

\begin{cor}\label{cor:dissipation_measure}
    Let $W$, $\theta_0$ be as in \autoref{thm:energy_balance_inviscid}. Then $\PP$-almost surely, in the sense of distributions with respect to time, it holds
    \begin{equation}\label{eq:stoch.energy.balance}
        \dif \| \theta_t\|_{L^2_x}^2 + \langle |\theta_t|^2, \dd( {\rm div} W_t)\rangle = - \mathcal{D}[\theta](\dif t,\R^d)
    \end{equation}
    and moreover for every $T\geq 0$ we have
    \begin{equation}\label{eq:mean.energy.balance}
        \|\theta_0\|_{L^2_x}^2-\EE[\|\theta_T\|_{L^2_x}^2]= \EE[\mathcal{D}([0,T]\times \R^d)].
    \end{equation}
\end{cor}

\begin{proof}
Test \eqref{eq:energy_balance_inviscid} against a test function $\varphi^R \in C^\infty_c(\R^d)$, taking values in $[0,1]$, with $\varphi^R \equiv 1$ on $B_R$, $\varphi^R \equiv 0$ on $B_{2R}$ for some $R>1$. Assume in addition $|\nabla \varphi^R| \lesssim R^{-1}$ and $|D^2 \varphi^R| \lesssim R^{-2}$.
Integration by parts gives (as distributions with respect to time)
\begin{align*}
\int_{\R^d} \dif |\theta_t(x)|^2 \varphi^R(x) \dif x
-
\int_{\R^d} |\theta_t(x)|^2 \nabla \varphi^R(x)  \dif W  \dif x
-
\int_{\R^d} |\theta_t(x)|^2 \varphi^R(x) \dif (\div W)  \dif x
\\
=
\frac12 \int_{\R^d} |\theta_t(x)|^2 C(0) : D^2 \varphi^R(x) \dif t  \dif x
-
\int_{\R^d} \varphi^R(x) \mathcal{D}[\theta](\dif t, \dif x).
\end{align*}
Taking the limit $R \to \infty$ thanks to \autoref{lem:tightness} we obtain \eqref{eq:stoch.energy.balance}.
In order to obtain \eqref{eq:mean.energy.balance}, fix $T \geq 0$ and define $\psi^n \in C^\infty_c(\R_+)$ such that $\psi^n(s) \equiv 1$ on $s \in[0,T]$ and $\psi^n(s) \equiv 0$ for $s > T+1/n$,  $\psi^n$ is non-increasing and $\psi^n \downarrow \mathbf{1}_{[0,T]}$ monotonically as $n \to \infty$. 
Integrate \eqref{eq:stoch.energy.balance} against $\psi^n$ and take expectations to get
\begin{align*}
\|\theta_0\|_{L^2_x}^2+ \int_{\R_+} \EE [\|\theta_t\|_{L^2_x}^2] \partial_t\psi^n(t) \dif t  
=
\mathbb{E}\int_{\R_+} \psi^n(t) \mathcal{D}[\theta](\dif t, \R^d).
\end{align*}
Since the map $t \mapsto \EE [\|\theta_t\|_{L^2_x}^2]$ is continuous by \autoref{prop:properties.inviscid.Kraichnan} and $\partial_t \psi^n(t) \to -\delta_{\{t=T\}}$ as distributions when $n \to \infty$, the left-hand side of the above converges to $ \|\theta_0\|_{L^2_x}^2-\EE[\|\theta_T\|_{L^2_x}^2]$. On the other hand, the left-hand side converges to $\EE[\mathcal{D}([0,T]\times \R^d)]$ by Monotone Convergence, concluding the proof.  
\end{proof}

We conclude this subsection with the following characterization of anomalous dissipation, in the setting of stochastic transport equation. The key feature consists in the fact that, by uniqueness of \eqref{eq:Kraichnan} and strong convergence $\tilde{\theta}^\kappa_t \to \theta_t$, persistence of mean energy dissipation along the vanishing diffusivity scheme $\kappa \downarrow 0$ can be read directly at the inviscid level, either by looking at $\mathbb{E}[\| \theta_t \|_{L^2_x}^2]$ or at $\mathbb{E}[\mathcal{D}([0,t] \times \R^d)]$.

\begin{prop} \label{prop:anomalous_dissipation_equivalence}
Let $W$, $\theta_0$ be as in \autoref{thm:energy_balance_inviscid}.  For any $s<t$, the following are equivalent:
\begin{enumerate}
    \item 
    $\mathbb{E}[\| \theta_t \|_{L^2_x}^2] < \mathbb{E}[\| \theta_s \|_{L^2_x}^2]$;
    \item 
    $\lim_{\kappa \downarrow 0} \kappa \int_s^t \int_{\R^d} \mathbb{E} \left[  C(0) : \nabla \tilde{\theta}_r^{\kappa}(x) \otimes \nabla \tilde{\theta}_r^{\kappa}(x) \right]
\dd x \dd r > 0$;
    \item 
    $\mathbb{E}[\mathcal{D}((s,t] \times \R^d)]>0$.
\end{enumerate}
Moreover any of the above implies that
\begin{itemize}
    \item[(4)]
    $\mathbb{E} \left[ \limsup_{\kappa \downarrow 0} \kappa \int_s^t \int_{\R^d}   C(0) : \nabla \tilde{\theta}_r^{\kappa}(x) \otimes \nabla \tilde{\theta}_r^{\kappa}(x)  
\dd x \dd r \right] > 0$.
\end{itemize}
\end{prop}

\begin{proof}
We claim that
\begin{equation}\label{eq:claim_equivalences}
    \mathbb{E}[\mathcal{D}((s,t] \times \R^d)]
    = \mathbb{E}[\| \theta_s \|_{L^2_x}^2] - \mathbb{E}[\| \theta_t \|_{L^2_x}^2]
    = \lim_{\kappa \downarrow 0} \kappa \int_s^t \int_{\R^d} \mathbb{E} \left[  C(0) : \nabla \tilde{\theta}_r^{\kappa}(x) \otimes \nabla \tilde{\theta}_r^{\kappa}(x) \right]
\dd x \dd r
\end{equation}
which readily implies $(1) \Leftrightarrow (2)\Leftrightarrow (3)$.
The first equality in \eqref{eq:claim_equivalences} follows immediately by taking the difference between  \eqref{eq:mean.energy.balance} evaluated at times $t$ and $s$, respectively.
% $(1) \Leftrightarrow (2)$.
By \autoref{prop:energy_balance_viscous_kappa}, for every $\kappa > 0$ we have 
\begin{align*}
\mathbb{E}[\| \tilde{\theta}^\kappa_s \|_{L^2_x}^2] 
- 
\mathbb{E}[\| \tilde{\theta}^\kappa_t \|_{L^2_x}^2]
=
\kappa \int_s^t \int_{\R^d} \mathbb{E} \left[  C(0) : \nabla \tilde{\theta}_r^{\kappa}(x) \otimes \nabla \tilde{\theta}_r^{\kappa}(x) \right]
\dd x \dd r.
\end{align*}
By \autoref{lem:approx}, the left-hand side converges to $\mathbb{E}[\| \theta_s \|_{L^2_x}^2] - \mathbb{E}[\| \theta_t \|_{L^2_x}^2]$ as $\kappa \downarrow 0$, implying existence of the limit of the right-hand side, as well as the second equality in \eqref{eq:claim_equivalences}.

Let us now show the implication $(2) \Rightarrow (4)$.
Integrating \eqref{eq:energy_balance_viscous_kappa} with respect to space and time we obtain
\begin{align*}
\| \tilde{\theta}^\kappa_s \|_{L^2_x}^2 
- 
 \| \tilde{\theta}^\kappa_t \|_{L^2_x}^2 + \sqrt{1-\kappa}\int_s^t \langle \nabla|\tilde\theta^\kappa_r|^2, \dd W_r\rangle 
=
\kappa \int_s^t \int_{\R^d}    C(0) : \nabla \tilde{\theta}_r^{\kappa}(x) \otimes \nabla \tilde{\theta}_r^{\kappa}(x) 
\dd x \dd r,
\quad
\PP\mbox{-a.s.}
\end{align*}
The stochastic integral on the left-hand side, which is only relevant if $\div W \neq 0$, is uniformly bounded in $L^2(\Omega)$ with respect to $\kappa$, thanks to estimate \eqref{eq:contraction_Lp} (with $p=4$) and \autoref{lem:stoch_integr}; it follows that the family of random variables on the left-hand side are uniformly integrable in $\kappa$, thus so are the ones on the right-hand side.
By Fatou's Lemma for uniformly integrable random variables (\cite[Thm. 4, p.188]{Shyriaev1996})
it follows that
\begin{align*}
    \mathbb{E} \left[ \limsup_{\kappa \downarrow 0} \kappa \int_s^t \int_{\R^d}   C(0) : \nabla \tilde{\theta}_r^{\kappa}(x) \otimes \nabla \tilde{\theta}_r^{\kappa}(x) \dd x \dd r \right]
    &\geq \limsup_{\kappa\downarrow 0} \EE[\| \tilde{\theta}^\kappa_s \|_{L^2_x}^2 - \| \tilde{\theta}^\kappa_t \|_{L^2_x}^2]\\
    & = \EE[\|\theta_s \|_{L^2_x}^2]-\EE[\|\theta_t \|_{L^2_x}^2],
\end{align*}
proving that $(2) \Rightarrow (4)$.
\end{proof}

\section{Anomalous dissipation of mean energy}
\label{sec:anomalous_dissipation}

According to \autoref{prop:properties.inviscid.Kraichnan} and \autoref{cor:dissipation_measure}, for any initial condition $\theta_0 \in L^2_x$ the associated solution $\theta$ to the inviscid transport equation \eqref{eq:Kraichnan} satisfies $\mathbb{E} [\| \theta_t \|_{L^2_x}^2]
\leq \| \theta_0 \|_{L^2_x}^2$, with exact equality if and only if $\EE[\mathcal{D}[\theta]([0,t]\times\R^d)]=0$, so that $\PP$-a.s. $\mathcal{D}[\theta]([0,t]\times\R^d)\equiv 0$.
Also in light of points (2)-(3) from \autoref{prop:anomalous_dissipation_equivalence}, we will say a solution $\theta$ displays \emph{anomalous dissipation} on an interval $[s,t]$ whenever $\mathbb{E} [\| \theta_t \|_{L^2_x}^2]
< \mathbb{E} [\| \theta_s \|_{L^2_x}^2]$.

This section is devoted to study anomalous dissipation for solutions of the stochastic transport equation perturbed by a space-homogeneous noise (\autoref{sec:dissipation}) and the interplay between anomalous dissipation and space regularity of solutions (\autoref{sec:rigidity}). 

\subsection{Continuous-in-time anomalous dissipation for every initial condition}\label{sec:dissipation}

It follows from \cite[Lemma 6.5]{LeJRai2002} that \emph{spontaneous stochasticity} happens (in the sense that solutions to \eqref{eq:Kraichnan} are not representable by a flow of maps, but only by a flow of kernels) if and only if there exist some $t>0$ and $\theta_0\in L^2_x$ such that anomalous dissipation happens on $[0,t]$. However, it is not clear from \cite{LeJRai2002} how generic this phenomenon should be, genericity being intended with respect to the initial condition and time.
Here we prove the following dichotomy, valid in the general framework of \eqref{eq:Kraichnan} perturbed by a Gaussian homogeneous noise $W$. Combined with \cite[Lemma 6.5]{LeJRai2002} and \cite[Theorem 10.2]{LeJRai2002}, it implies in particular \autoref{thm:dissipation_intro} from the Introduction.

\begin{theorem} \label{thm:dissipation}
Let $W$ satisfy \autoref{ass:well_posed} and consider the inviscid SPDE \eqref{eq:Kraichnan}, which is well-posed in $L^2_x$ in the sense of \autoref{prop:solution}.
Then exactly one of the following alternatives holds:
\begin{enumerate}
    \item 
 All solutions preserve the mean energy: for any $\theta_0\in L^2_x$ and $t\geq 0$, $\mathbb{E} [\| \theta_t \|_{L^2_x}^2] = \| \theta_0\|_{L^2_x}^2$.
    \item All non-zero solutions continuously dissipate the mean energy: for any $\theta_0\in L^2_x\setminus\{0\}$ and any $s<t$, it holds $\EE[\| \theta_t \|_{L^2_x}^2]<\EE[\| \theta_s \|_{L^2_x}^2]$.
\end{enumerate}
\end{theorem}

We will refer to case $(2)$ as ``continuous anomalous dissipation of mean energy'', as the resulting map $t\mapsto \EE[\| \theta_t\|_{L^2_x}^2]$ is strictly monotone and continuous thanks to \autoref{prop:properties.inviscid.Kraichnan}.

In particular, combining \cite[Lemma 6.5]{LeJRai2002} with \autoref{thm:richardson} in the Introduction, we find:

\begin{cor}\label{cor:dissipation.general.assumption}
Let $W$ be a noise whose covariance satisfies \autoref{assumption} and consider the associated inviscid SPDE \eqref{eq:Kraichnan}, then all non-zero solutions continuously dissipate the mean energy.
\end{cor}

It is worth comparing \autoref{cor:dissipation.general.assumption} with \cite[Theorem 1.1]{rowan2023} on the incompressible Kraichnan model on the torus $\TT^d$. There, a quantitative estimate of the form $\EE[\| \theta_t \|_{L^2_x}^2] \leq Ce^{-t/C} \| \theta_0\|_{L^2_x}^2$ is proved for all non-constant initial data $\theta_0$, implying anomalous dissipation on the interval $[0,t]$ for \emph{sufficiently large} $t$. However, the proofs therein don't clarify whether solutions continuously dissipate the mean energy.
On the other hand, our results do not directly transfer to $\TT^d$, see \autoref{rem:counterexample_torus} below.

In order to prove \autoref{thm:dissipation}, let us first notice that we can reduce ourselves to the case $s=0$ only in scenario $(2)$. Indeed, assuming the validity of the result for $s=0$, by the Markovianity of solutions (\autoref{prop:properties.inviscid.Kraichnan}), solving \eqref{eq:Kraichnan} on $[s,t]$ is equivalent to considering the solution $\tilde \theta_r$ on the time interval $[0,t-s]$, with random initial condition $\theta_s$ independent of $\tilde W_r=W_{r+s}-W_r$; 
upon conditioning on $\theta_s$, applying the result to $\tilde\theta$, one finds $\EE[\|\theta_t\|_{L^2_x}^2]=\EE[\|\tilde\theta_{t-s}\|_{L^2_x}^2]<\EE[\|\theta_s\|_{L^2_x}^2]$.
Notice that $\tilde\theta_0 = \theta_s\equiv 0$ $\PP$-almost surely is not possible by the explicit chaos decomposition \eqref{eq:wiener_sol} (for instance $\EE[\theta_s]=P_s\theta_0\neq 0$ whenever $\theta_0\neq 0$).

From now on, $t > 0$ is fixed without further specification. 
Let us introduce the following sets:
\begin{align*}
\mathfrak{C} 
&:= 
\Big\{ \theta_0 \in L^2_x : \mathbb{E} [\| \theta_t \|_{L^2_x}^2] = 
\| \theta_0\|_{L^2_x}^2 \Big\}, \qquad 
\mathfrak{D} 
:= 
\Big\{ \theta_0 \in L^2_x : \mathbb{E} [\| \theta_t \|_{L^2_x}^2] < 
\| \theta_0 \|_{L^2_x}^2 \Big\},
\end{align*}
where the letters $\mathfrak{C}$ and $\mathfrak{D}$ stand respectively for ``conservation'' and ``dissipation''. 
By the above observation, \autoref{thm:dissipation} boils down to showing that either $\mathfrak{C} = \{0\}$ or $\mathfrak{C} = L^2_x$.

\begin{lemma} \label{lem:Dt.open}
$\mathfrak{D}$ is an open subset of $L^2_x$.
\end{lemma}

\begin{proof}
Let $\theta_0 \in \mathfrak{D}$ and consider $g_0 \in L^2_x$ such that $\| g_0 - \theta_0 \|_{L^2_x} < \eps$, with $\eps>0$ to be fixed later.
Let $g$ be the unique solution to \eqref{eq:Kraichnan} starting from $g_0$; by linearity of \eqref{eq:Kraichnan}, $g-\theta$ is the solution to the equation with initial condition $g_0-\theta_0$.
Then by \eqref{eq:contraction_Lp} and triangular inequality it holds  
\begin{align*}
\| g_0 \|_{L^2_x} 
&> 
\| \theta_0 \|_{L^2_x} - \eps, 
\qquad 
\| g_t \|_{L^2_\omega L^2_x} 
\leq
\| \theta_t \|_{L^2_\omega L^2_x}
+ 
\eps 
\end{align*}  
from which it follows that $\| g_t \|_{L^2_\omega L^2_x} < \| g_0 \|_{L^2_x}$ as soon as we choose $\eps > 0$ small enough, so that $\| \theta_t \|_{L^2_\omega L^2_x} + 2 \eps \leq\| \theta_0 \|_{L^2_x}$.
\end{proof}

\begin{lemma}
\label{lem:Ct.subspace}
Suppose $\mathfrak{D}$ non-empty.
Then, there exists $h \in L^1_x \cap L^2_x$, $h \neq 0$,  such that
\[\mathfrak{C} \subset \{ \varphi \in L^2_x : \langle \varphi, h \rangle = 0 \} \] where $\langle \cdot, \cdot \rangle$ denotes the scalar product in $L^2_x$.
\end{lemma}

\begin{proof}
If $\mathfrak{D}$ is non-empty, then by \autoref{lem:Dt.open} there exists $\theta_0\in L^1_x \cap L^{\infty}_x$ belonging to it. By \autoref{prop:properties.inviscid.Kraichnan}, the associated solution $\theta_t$ enjoys the same integrability. 
Similarly to before, consider the evolution of initial datum $\theta_0 + g_0$. 
By linearity of the expectation, we have
\begin{align*}
\mathbb{E} [\| \theta_t + g_t \|_{L^2_x}^2] 
&=
\mathbb{E} [\| \theta_t \|_{L^2_x}^2] 
+
2\mathbb{E} [\langle \theta_t, g_t \rangle] 
+
\mathbb{E} [\| g_t \|_{L^2_x}^2],
\\
\| \theta_0 + g_0 \|_{L^2_x}^2 
&= 
\| \theta_0 \|_{L^2_x}^2 
+ 
2\langle \theta_0, g_0 \rangle
+ 
\| g_0 \|_{L^2_x}^2.
\end{align*}
Since $\mathbb{E} [\| \theta_t \|_{L^2_x}^2] < \| \theta_0 \|_{L^2_x}^2$ and $\mathbb{E} [\| g_t \|_{L^2_x}^2] \leq \| g_0 \|_{L^2_x}^2$, in order to show that $\theta_0 + g_0 \in \mathfrak{D}$ it suffices to require that
\begin{equation}
\mathbb{E} [\langle \theta_t, g_t \rangle] 
\leq
\langle \theta_0, g_0 \rangle .
\label{eq:cond1}
\end{equation}
Notice that for fixed $\theta_0$, the map $\psi \mapsto \mathbb{E} [\langle \theta_t,
\psi_t \rangle]$ is linear and bounded from $L^2_x$ to $\mathbb{R}$, in view of bound~\eqref{eq:contraction_Lp}. 
It follows from Riesz Representation Theorem that there exists $\tilde{h} \in L^2_x$ such that
\[ \mathbb{E} [\langle \theta_t, \psi_t \rangle] 
= \langle \tilde{h}, \psi \rangle,
\qquad 
\forall \, \psi \in L^2_x. 
\]
Moreover, since $\theta_0, \theta_t \in L^1_x \cap L^{\infty}_x$, again by \autoref{prop:properties.inviscid.Kraichnan} we have the estimate
\[ | \langle \tilde{h}, \psi \rangle | 
\leq 
\mathbb{E} [\| \theta_t \|_{L^1_x} \| \psi_t \|_{L^{\infty}_x}] 
\leq \| \theta_0 \|_{L^1_x} 
\| \psi \|_{L^{\infty}_x},
\qquad 
\forall \, \psi \in L^{\infty}_x, 
\]
which by standard density arguments implies that $\tilde{h} \in L^1_x $.
Condition~\eqref{eq:cond1} then becomes:
\[ \theta_0 + g_0 \in \mathfrak{D}
\quad 
\Longleftarrow 
\quad 
\langle \tilde{h}, g_0 \rangle 
\leq 
\langle \theta_0, g_0 \rangle 
\quad 
\Longleftrightarrow
\quad 
\langle \tilde{h} - \theta_0, g_0 \rangle 
\leq 
0. 
\]
Set $h := \tilde{h} - \theta_0$. 
Notice that, by the above facts, $h \in L^1_x \cap L^2_x$ and
\begin{align*}
    \| h \|_{L^2_x} & \geq  \| \theta_0 \|_{L^2_x} - \| \tilde{h} \|_{L^2_x}
    \geq  \| \theta_0 \|_{L^2_x} -\mathbb{E} [\| \theta_t \|_{L^2_x}] > 0,
\end{align*}
so that $h \neq 0$. By the change of variables $\varphi = \theta_0 + g_0$ we deduce that
\[ 
\varphi \in \mathfrak{D}
\quad\Longleftarrow \quad 
\langle h, \varphi \rangle
\leq
\langle h, \theta_0 \rangle 
=
\mathbb{E} [\| \theta_t \|_{L^2_x}^2] - \| \theta_0\|_{L^2_x}^2 %=: a 
< 0. \]
By linearity of \eqref{eq:spde_inviscid_approx}, we also know that $\varphi \in \mathfrak{D}$ if and only if $\lambda \varphi \in \mathfrak{D}$ for any $\lambda \in \mathbb{R}  \setminus \{  0 \}$. 
In particular, we can always find $\lambda$ such that $\langle h,  \lambda \varphi \rangle < 0$ if and only if $\langle h, \varphi \rangle \neq  0$. 
It follows that
\[ 
\{ \varphi \in L^2_x : \langle h, \varphi \rangle \neq 0 \} \subset \mathfrak{D}
\quad \Longrightarrow \quad 
\mathfrak{C} \subset \{ \varphi \in L^2_x : \langle h,\varphi \rangle = 0 \} . \qedhere
\]
\end{proof}

\begin{lemma}
\label{lem:closed}
$\mathfrak{C}$ is a linear closed subspace of $L^2_x$, and it is closed under translations $\theta_0\mapsto \theta_0(\cdot+l)$. 
Furthermore,  $\theta_0 \in \mathfrak{C}$ if and only if
\[
\lim_{\kappa \downarrow 0} 
\mathbb{E} [\|   \hat{\theta}^\kappa_t - \theta_t \|_{L^2_x}^2] 
= 
0.
\]
where $\hat{\theta}^\kappa$ denotes the solution to the SPDE with smoothed noise \eqref{eq:spde_inviscid_approx}, starting from the same initial condition $\theta_0$.
\end{lemma}

\begin{proof}
Closedness of $\mathfrak{C}$ in $L^2_x$ is clear by \autoref{lem:Dt.open}. 
Closure under translation
%of $\mathfrak{C}$
descends from the noise $W$ being space homogeneous.
Indeed, for every $\theta_0 \in \mathfrak{C}$ and $l \in \R^d$ it holds
\begin{align*}
\mathbb{E}[\| (\theta_0 (\cdot+l))_t \|_{L^2_x}^2]
=
\mathbb{E}[\| \theta_t \|_{L^2_x}^2]
=
\mathbb{E}[\| \theta_0 \|_{L^2_x}^2]
=
\mathbb{E}[\| \theta_0(\cdot+l) \|_{L^2_x}^2],
\end{align*}
therefore $\theta_0(\cdot+l) \in \mathfrak{C}$ as well.
  
Let us move to the convergence in $L^2_{\omega,x}$. Since $\mathbb{E} [\| \hat{\theta}^\kappa_t \|_{L^2_x}^2] = \| \theta_0 \|_{L^2_x}^2$ for every $t,\kappa>0$, clearly $\hat{\theta}^\kappa_t \to \theta_t$ in $L^2_{\omega,x}$ implies $\theta_0 \in \mathfrak{C}$.
Viceversa, if $\theta_0 \in \mathfrak{C}$, then $\mathbb{E} [\| \theta_t \|_{L^2_x}^2] = \| \theta_0 \|_{L^2_x}^2$ and, by \autoref{lem:approx}, $\theta_t$ is the weak limit in $L^2_{\omega,x}$ of $\hat{\theta}^\kappa_t$.
Since moreover 
\[
\lim_{\kappa \downarrow 0} 
\mathbb{E} [\| \hat{\theta}^\kappa_t \|_{L^2_x}^2] 
= 
\| \theta_0 \|_{L^2_x}^2 
=
\mathbb{E} [\| \theta_t \|_{L^2_x}^2],
\] 
we deduce that %the sequence actually converges strongly in $L^2_{\omega,x}$, i.e.
$\lim_{\kappa \downarrow 0} 
\mathbb{E} [\|   \hat{\theta}^\kappa_t - \theta_t \|_{L^2_x}^2] 
= 
0$, as desired. 

Finally, linearity of $\mathfrak{C}$ follows from the linearity of the SPDEs \eqref{eq:Kraichnan}-\eqref{eq:spde_inviscid_approx} and the characterization of $\mathfrak{C}$ above in terms of convergence of $\hat\theta^\kappa_t$.
\end{proof}

Let $\mathcal{C}_0$ denote the space of continuous functions from $\R^d$ to $\R$ that vanish at $ \infty$.
Next, consider the set of initial data in $\mathcal{C}_0$ that do not lead to anomalous dissipation, namely
\[ \mathfrak{C}_0 := \mathfrak{C} \cap \mathcal{C}_0 . \]

By a similar argument as the one employed in \autoref{lem:closed}, we have
the following.

\begin{lemma}
\label{lem:algebra}
$(\mathfrak{C}_0,+,\cdot)$ is an algebra.
\end{lemma}

\begin{proof}
Let $\theta_0$, $g_0\in\mathfrak{C}_0$.
Since $\mathcal{C}_0$ is an algebra and $\mathfrak{C}$ is a linear space by \autoref{lem:closed}, to conclude we only need to check that $\theta_0 g_0 \in \mathfrak{C}$.

By \autoref{lem:closed}, the solutions $\theta_t$, $g_t$ associated to $\theta_0$, $g_0$ are strong limits in $L^2_{\omega, x}$ of the solutions $\hat{\theta}^\kappa_t$, $\hat{g}^\kappa_t$ associated to \eqref{eq:spde_inviscid_approx}. 
At the same time, since $\theta_0, g_0 \in L^{\infty}_x$, the solutions $\hat{\theta}^\kappa_t$, $\hat{g}^\kappa_t$ satisfy uniform bounds in $L^{\infty}_{\omega, x}$; by interpolation, convergence holds in $L^p_{\omega, x}$ for any $p \in [2,\infty)$, in particular for $p = 4$. 
Again by the smooth transport structure, at fixed $\kappa > 0$ it holds 
\[
\hat{\theta}^\kappa_t \hat{g}^\kappa_t = \widehat{(\theta_0 g_0)}^\kappa_t.
\]
By the above, $\hat{\theta}^\kappa_t \hat{g}^\kappa_t \rightarrow \theta_t g_t$ in $L^2_{\omega, x}$, which again by \autoref{lem:closed} implies that $\theta_0 g_0 \in \mathfrak{C}$.  
\end{proof}

\begin{lemma} \label{lem:non_empty}
If $\mathfrak{C} \neq \{ 0 \}$, then there exists a strictly positive $\theta_0 \in \mathfrak{C}_0$.
\end{lemma}

\begin{proof}
Assume $\mathfrak{C}$ is not $\{0\}$ and let $f\in \mathfrak{C}$, $f \neq 0$.
As $\mathfrak{C}$ is closed, linear, and closed under translations (\autoref{lem:closed}), it is also closed under convolutions with smooth functions. 
Therefore, if $\chi$ is the Gaussian kernel on $\R^d$, we have $g :=  f  \ast \chi \in \mathfrak{C}_0$; $g \neq 0$ since $f \neq 0$ by assumption.
By \autoref{lem:algebra}, $|g|^2 \in \mathfrak{C}_0$ and by the same argument as above $\theta_0 := |g|^2 \ast \chi \in \mathfrak{C}_0$ as well.
By construction  $\theta_0(x) > 0$ for all $x \in \mathbb{R}^d$, proving the lemma. 
\end{proof}

We are finally ready to prove our \autoref{thm:dissipation}.

\begin{proof}[Proof of \autoref{thm:dissipation}]
It suffices to show that $\mathfrak{C}\setminus\{0\}$ and $\mathfrak{D}$ cannot be both non-empty.
To this end, we first show that, if $\mathfrak{C} \neq \{0\}$, then $\mathfrak{C}_0$ is dense in $\mathcal{C}_0$ with respect to the $\mathcal{C}_0$-topology.

By \autoref{lem:algebra}, $\mathfrak{C}_0$ is a subalgebra of $\mathcal{C}_0$. 
By Stone--Weierstrass Theorem, in order to prove its density in $\mathcal{C}_0$ it is sufficient to verify that:
\begin{itemize}  
\item 
$\mathfrak{C}_0$ vanishes nowhere, i.e. for any $x \in \mathbb{R}^d$ there exists $\theta_0 \in \mathfrak{C}_0$ such that $\theta_0 (x) \neq 0$.
\item 
$\mathfrak{C}_0$ separates points, i.e. for any distinct $x, y \in \mathbb{R}^d$ there exists $\theta_0 \in \mathfrak{C}_0$ such that $\theta_0 (x) \neq \theta_0(y)$.
\end{itemize}

By \autoref{lem:closed} and \autoref{lem:non_empty}, there exists a strictly positive $\theta_0 \in L^2_x \cap \mathcal{C}_0$ such that $\theta_0$ and all its translations belong to $\mathfrak{C}_0$. 
Since $\theta_0$ is strictly positive, clearly $\mathfrak{C}_0$ vanishes nowhere. 
Assume by contradiction that $\mathfrak{C}_0$ didn't separate points, then there would exist $\bar{x}, \bar{y} \in \mathbb{R}^d$ distinct such that
\begin{align*}
\theta_0 (\bar{x} + l) = \theta_0 (\bar{y} + l), 
\qquad 
\forall \, l \in \mathbb{R}^d.
\end{align*} 
In other words, $\theta_0$ would be ($\bar{x} - \bar{y}$)-periodic;  
but $\theta_0$ is vanishing at $\infty$, which clearly yields a contradiction.
Thus $\mathfrak{C}_0$ separates points and the Stone-Weierstrass Theorem applies: $\mathfrak{C}_0 \subset \mathfrak{C}$ is dense in $\mathcal{C}_0$ with respect to the $\mathcal{C}_0$-topology.

Suppose now by contradiction that $\mathfrak{C}\setminus\{0\}$ and $\mathfrak{D}$ are both non-empty. By \autoref{lem:Ct.subspace}, there exists a non-zero $h\in L^1_x\cap L^2_x$ such that $\langle \theta_0,h \rangle$ for all $\theta_0 \in \mathfrak{C}$; by density, the relation extends to any  $\theta_0\in\cC_0$. But then $h \equiv 0$ by the fundamental lemma of calculus of variations, yielding a contradiction.
\end{proof}

\begin{remark}\label{rem:counterexample_torus}
An analogous statement to \autoref{thm:dissipation} likely fails on the torus $\TT^d$.
To see this, let $d=4$, $z=(x,y)\in \TT^2\times\TT^2=\TT^4$.
For fixed $\alpha\in (0,1)$, consider two independent, sharply $\alpha$-regular Kraichnan noises $W^1(x)$ and $W^2(y)$ on $\TT^2$, with $W^1$ being strongly compressible and $W^2$ being incompressible; set $W(z) = (W^1(x), W^2(y))$ on $\TT^4$. 
Without loss of generality one may take $C^1(0)=C^2(0)= I_2$, so that $C(0)=I_4$; by construction, $W$ is still a space-homogeneous Gaussian noise, sharply $\alpha$-regular in all directions.  
Given an initial condition of the form $\theta_0(x, y) = f_0(x)g_0(y)$, one can verify - e.g., by approximation arguments - that the solution factorizes as $\theta_t(x, y) = f_t(x)g_t(y)$, where $f_t$ and $g_t$ solve the two-dimensional problems driven by $W^1$ and $W^2$, respectively.
The term $g_t$ dissipates energy over time as soon $g_0$ is non-constant, for instance by \cite{rowan2023}. On the other hand, in light of the results from \cite{LeJRai2002}, we expect $W^1$ to induce a flow of maps and $f_t$ to preserve energy.
As a consequence, we obtain a nontrivial subspace of energy-conserving solutions, associated to $\theta_0(x, y) = f_0(x)$ (i.e., $g_0 \equiv 1$). 
Such a scenario is not present in $\R^d$ since therein $\theta_0(x, y) = f_0(x)\notin L^2_x$ unless $f_0\equiv 0$.
Notice that $W$ as constructed is strongly anisotropic; it is reasonable to expect a variant of \autoref{thm:dissipation} to hold on $\TT^d$, under stronger isotropy-type conditions.
\end{remark}

\subsection{Anomalous dissipation implies irregularity} \label{sec:rigidity}

This subsection is devoted to the proof that solutions to the inviscid transport \eqref{eq:Kraichnan} enjoy limited spatial regularity, specifically condition \eqref{eq:sharpness_intro} of \autoref{thm:regularity_Kraichnan}, whenever the noise induces anomalous dissipation of energy.
We will actually show the converse: if the solution is regular enough, then dissipation of energy cannot take place; the conclusion will then follow by genericity of anomalous dissipation, namely from applying results in the style of \autoref{cor:dissipation.general.assumption}. 

Recall that $Q(z) = C(0) - C(z)$, where $C$ is the covariance of the noise $W$. We will work under the following assumption:

\begin{assumption} \label{ass:sharpness_reg}
\autoref{ass:well_posed} holds and there exists a neighbourhood $U$ of the origin %$0 \in \R^d$
such that $Q \in C^2 (U \setminus \{0\})$, with bounds
\begin{equation}
| Q (z) | \lesssim | z |^{2 \alpha}, 
\quad 
| \div Q (z) | \lesssim | z |^{2 \alpha - 1}, 
\quad 
| \nabla \cdot (\div Q) (z) | \lesssim 
| z |^{2 \alpha - 2}, 
\quad \forall  z \in U \setminus \{0\}.
    \label{eq:ass.energy.conservation}
  \end{equation}
\end{assumption}

\autoref{ass:sharpness_reg} covers the majority of the cases of interest. For example, for fixed $\alpha\in (0,1)$, it is verified by the Kraichnan model for any compressibility ratio $\eta\in [0,1]$, cf. \autoref{cor:assumptions_kraichnan}.
Recall the increment notation $\delta_z \theta$ introduced before \autoref{lem:besov_type_spaces}.

\begin{prop}\label{prop:regularity_implies_conservation}
Let $W$ satisfy \autoref{ass:sharpness_reg}, $\theta_0\in L^2_x$, and $\theta$ be the unique solution to \eqref{eq:Kraichnan}.
Then, on any finite time interval $I \subset \R_+$ it holds that
\[ 
\lim_{|z|\to 0} \frac{1}{|z|^{2(1-\alpha)}} \int_I \EE[\| \delta_z \theta_r\|_{L^2_x}^2] \dd r=0
\quad \Longrightarrow \quad 
\mathbb{E} [\| \theta_t \|_{L^2_x}^2] 
=
\mathbb{E} [\| \theta_s \|_{L^2_x}^2],
\quad \forall \, s, t \in I. \]
\end{prop}

Roughly speaking, if $W$ is $\alpha$-regular and the solution $\theta$ belongs to the closure of spatially smooth functions in $\tilde L^2_{\omega,t} \tilde B^{2(1-\alpha)}_{2,\infty}$, then anomalous dissipation cannot take place. In light of \autoref{thm:dissipation} and its consequences we obtain the following:

\begin{cor}\label{cor:generic_irregularity}
    For fixed $\alpha\in (0,1)$, let $W$ satisfy \autoref{assumption}. Then for any non-zero $\theta_0\in L^2_x$ and any $s<t$, the associated solution $\theta$ to \eqref{eq:Kraichnan} is such that
    \begin{equation}\label{eq:sharp_upper_regularity}
        \limsup_{|z|\to 0} \frac{1}{|z|^{2(1-\alpha)}} \int_s^t \EE[\| \delta_z \theta_r\|_{L^2_x}^2] \dd r>0, \quad
        \limsup_{j\to +\infty} 2^{2j(1-\alpha)} \int_s^t \EE[\| \dot\Delta_j \theta_r\|_{L^2_{x}}^2] \dd r>0.
    \end{equation}
\end{cor}

\begin{proof}
    Under \autoref{ass:sharpness_reg}, by \autoref{cor:dissipation.general.assumption}, all non-zero initial data $\theta_0$ display continuous anomalous dissipation of mean energy, therefore the assumptions of \autoref{prop:regularity_implies_conservation} must not hold on $I=[s,t]$. The second part of the statement then follows from \autoref{lem:besov_type_spaces}.
\end{proof}

\begin{remark}\label{rem:generic.irregularity}
    Under \autoref{ass:sharpness_reg}, it follows from \autoref{cor:generic_irregularity} and characterizations of function spaces that $\theta$ cannot belong to either $L^2(\Omega\times [s,t]; H^{1-\alpha}_x)$, nor $L^2(\Omega\times [s,t]; B^{1-\alpha+\eps}_{2,\infty})$ with $\eps>0$, otherwise both $\limsup$ appearing in \eqref{eq:sharp_upper_regularity} would be $0$. In particular, regardless of how smooth the initial condition $\theta_0$ is, the associated solution must instantaneously lose energy-critical regularity.
In particular this proves the claim \eqref{eq:sharpness_intro} in \autoref{thm:regularity_Kraichnan}.
    
    Similarly, by \autoref{lem:refined.Besov.Holder}, the associated two-point self-correlation $F[\theta]$ must also enjoy limited regularity, at best $F[\theta] \in \tilde L^1([s,t]; \tilde B^{2(1-\alpha)}_{\infty,\infty})$ but never better. 
\end{remark}

In order to prove \autoref{prop:regularity_implies_conservation}, we need some preparations.
Let $\theta_0\in L^2_x$, $\theta$ associated solution $\theta$ to \eqref{eq:Kraichnan}.
The following lemma  describes the evolution of the quantity $\mathbb{E} [\langle G \ast \theta, \theta
\rangle]$ for regular kernels $G$;
its proof can be easily obtained following
similar computations to those of \autoref{lem:PDE_twopoint_transport} and we omit it.

\begin{lemma}\label{lemma1}
Let $W$ satisfy \autoref{ass:sharpness_reg} for some neighbourhood $U$ of $0$ and let $G \in C^\infty_c(U\setminus\{0\})$; let $H$ be the matrix-valued kernel given by $H (x) := G (x) Q (x)$.
%for every $x \in \R^d$.
Then
\begin{equation}
\frac{\dd}{\dd t} \mathbb{E} [\langle G \ast \theta, \theta \rangle] 
%=
%-\mathbb{E} [\langle H \ast \nabla \theta, \nabla \theta\rangle]
=
\mathbb{E} [\langle (D^2 : H) \ast \theta,  \theta \rangle].
\label{eq:balance1}
\end{equation}
\end{lemma}

\begin{proof}[Proof of \autoref{prop:regularity_implies_conservation}]
Up to rescaling and shifting space and/or time, without loss of generality we may assume $I = [0, 1]$ and the open set $U$ from \autoref{ass:sharpness_reg} to contain the unit ball $B_1 \subset \R^d$.
Let $\chi \in C^\infty_c(\R^d)$ be a probability density supported on the annulus $B_1 \setminus B_{1 / 2}$,  $\{\chi^{\eps}\}_{\eps\in (0,1)}$ be the associated standard mollifiers. Applying \autoref{lemma1} with $G = \chi^{\eps}$, we obtain
  \begin{equation} 
    \mathbb{E} [\langle \chi^{\eps} \ast \theta_t, \theta_t \rangle]
    - \langle \chi^{\eps} \ast \theta_0, \theta_0\rangle
    =
    \int_0^1 \EE[\langle D^2  : (\chi^{\eps} Q) \ast \theta_r, \theta_r \rangle] \dd r
    =:
    \int_0^1 \mathbb{E} [\mathcal{A}^{\eps} (\theta_r, \theta_r)] \dd r.
    \label{eq:approximate.balance}
  \end{equation}
Notice that $Q \chi^{\eps} \in C^2_c$ implies $\int_{\mathbb{R}^d} D^2 : (\chi^{\eps} Q)(z)\dd z = 0$, giving
\begin{align*}
\mathcal{A}^{\eps} (\theta_r, \theta_r) 
&= 
\int_{\R^d}\int_{\R^d} D^2 : (\chi^{\eps} Q) (x - y) \theta_r (x) \theta_r (y) \dd x \dd y
\\
&= 
- \frac{1}{2} \int_{\R^d}\int_{\R^d} D^2 : (\chi^{\eps} Q) (x-y) | \theta_r (x) - \theta_r (y) |^2 \dd x \dd y.
\end{align*}  
Expanding
$D^2 : (\chi^{\eps} Q)$ and rescaling we find
\[ \mathcal{A}^{\eps} (\theta_r, \theta_r) 
\lesssim 
\int_{\R^d}\int_{\R^d}
\left( \frac{| D^2 \chi (z) Q (\eps z) |}{\eps^2} 
+ 
\frac{| \nabla \chi^{\eps} (z) \cdot \text{div} Q (\eps z)|}{\eps} 
+ 
| \chi (z) \nabla \cdot (\text{div} Q) (\eps z) | \right) 
\, 
| \delta_{\eps z} \theta_r(y)|^2\dd y\dd z.
\]
Applying assumption~\eqref{eq:ass.energy.conservation}, using that $\chi$ and its derivatives are supported on $B_1 \setminus B_{1 / 2}$, we get
\begin{align*}
    \mathcal{A}^{\eps} (\theta_r, \theta_r) \lesssim \int_{B_1 \times
    \mathbb{R}^d} \eps^{2 \alpha - 2} \, | \delta_{\eps z} \theta_r(y)|^2 \dd y\dd z = \eps^{2 \alpha - 2} \int_{B_1}
    \left\| \delta_{\eps z} \theta_r
    \right\|_{L^2_x}^2\dd z.
\end{align*}
Using properties of the mollifiers and plugging the expression above in \eqref{eq:approximate.balance} we get
\begin{align*}
 \mathbb{E} \| \theta_t \|_{L^2_x}^2 - \|\theta_0\|_{L^2_x}^2
 &=
 \lim_{\varepsilon \to 0}
 \left( \mathbb{E} [\langle \chi^{\eps} \ast \theta_t, \theta_t \rangle]
    - \langle \chi^{\eps} \ast \theta_0, \theta_0\rangle \right)
    \\
    &=
    \lim_{\varepsilon \to 0}
    \int_0^1 \mathbb{E} [\mathcal{A}^{\eps} (\theta_r, \theta_r)] \dd r
    \lesssim 
    \lim_{\varepsilon \to 0}\sup_{|z|\leq \eps}
    \frac{1}{|z|^{2(1-\alpha)}}  \int_0^1 \EE[\| \delta_{ z}\theta_r\|_{L^2_x}^2] \dd r. \qedhere
\end{align*}
\end{proof}

\section{Anomalous Sobolev regularization} \label{sec:regularity}

From now on, we will restrict ourselves to considering homogeneous, \emph{isotropic} noise $W$ on $\R^d$; the resulting class of covariance functions $C$ can be fully characterized, see \autoref{app:auxiliary} for more details.
Under mild assumptions, $C$ must be of the form
\begin{align} \label{eq:covariance}
C^{ij}(z)
=
B_L(|z|) \frac{z^i z^j}{|z|^2} + B_N(|z|) \left( \delta^{ij}-\frac{z^i z^j}{|z|^2}\right),
\quad
\forall\, 1 \leq i,j \leq d, 
\quad 
z \in \R^d \setminus \{0\}.
\end{align}
See \eqref{eq:covariance_BLBN} for the exact formula for $B_L$, $B_N$; they satisfy $B_L(0) = B_N(0)>0$, so that $C^{ij}(0) = B_N(0) \delta^{ij}$, and \autoref{rmk:isotropy} is valid with $c_0 = B_N(0)/2$. 
Correspondingly, let us define
\begin{align} \label{eq:definition_b}
b_L(r) 
:= 
B_L(0) - B_L(r) ,
\quad
b_N(r) 
:= 
B_N(0) - B_N(r),
\end{align}
so that $Q(z)=C(0)-C(z)$ admits a similar expansion as \eqref{eq:covariance}, with $B_L$, $B_N$ replaced by functions $b_L$, $b_N$ going like $b_L(r) \sim b_N(r) \sim r^{2\alpha}$ for small $r$, cf. \eqref{eq:covariance_generalQ}. In particular, in the isotropic case, $b_L$ and $b_N$ determine precisely the asymptotic behaviour of $Q$ around $0$.
This section is devoted to the proof of the Sobolev regularity results of \autoref{thm:regularity_Kraichnan} (in a more general version that is not restricted to the Kraichnan noise). Hereafter, we informally refer to this results as anomalous regularization. This terminology is justified by the fact that smooth transport does not yield any Sobolev regularization of solutions. In order to show anomalous regularization,
we study the zero-viscosity limits $\kappa \downarrow 0$ of \eqref{eq:PDE_F_intro_tilde} and prove that 
\begin{align} \label{eq:anom_reg_01}
\sup_{\kappa \in (0,1/2)} \int_0^t \sup_{|z| < l} \frac{| F^\kappa_s(z) - F^\kappa_s(0) |}{|z|^{2-2\alpha-\delta}}\dd s
\lesssim 
(1+t)\, \| \theta_0 \|_{L^2_x}^2.
\end{align}
By \autoref{lem:Sobolev-Holder} this yields uniform-in-$\kappa$ Sobolev regularity for approximations $\tilde{\theta}^\kappa$.
We deduce \eqref{eq:anom_reg_01} above under the additional assumption that the initial condition $F^\kappa_0$ is radially symmetric. 
We show how the general case is reduced to this one in \autoref{ssec:radial}.
Under this condition, $F^\kappa(h) =: f^\kappa(|h|)$ for some $f^\kappa:\R_+ \times \R_+ \to \R$ solving the parabolic PDE \begin{align} \label{eq:PDE_f}
\partial_t f^ \kappa 
=
(1-\kappa) b_L(r) \partial_r^2 f^ \kappa 
+
(1-\kappa) \frac{d-1}{r}b_N(r) \partial_r f^ \kappa 
+
2 c_0 \kappa \left( \partial_r^2 f^ \kappa
+ 
\frac{d-1}{r} \partial_r f^ \kappa \right),
\end{align}
and \eqref{eq:anom_reg_01} boils down to deriving analogous estimates for the behaviour of $f^\kappa$ close to $r=0$ that are uniform in $\kappa \in (0,1/2)$. 
We point out that, in the formal limit $\kappa \downarrow 0$, \eqref{eq:PDE_f} becomes a \textit{degenerate} parabolic PDE:
\begin{align} \label{eq:PDE_f_deg}
\partial_t f 
=
b_L(r) \partial_r^2 f 
+
\frac{d-1}{r}b_N(r) \partial_r f .
\end{align}
In particular, the non-degenerate parabolic term $2c_0\kappa \partial_r^2 f^\kappa$ in \eqref{eq:PDE_f} cannot be used in a standard way to recover regularity estimates for $f$ close to $r=0$.
However, our analysis suggests that the presence of the $\kappa$-dependent, non-degenerate parabolic term does indeed help also in the limit $\kappa \downarrow 0$.
More precisely, together with the fact that $f^\kappa$ is the radial part of a two-point self-correlation function, it acts by selecting a suitable Neumann boundary condition that allows to close uniform-in-$\kappa$ regularity estimates for $f^\kappa$.
As a consequence, being the limit $\kappa \downarrow 0$ of \eqref{eq:PDE_f} constitutes an additional physical constraint that in principle improves the regularity of solutions of the degenerate parabolic PDE \eqref{eq:PDE_f_deg}.
Furthermore, it is worth to mention the fact that our result is not based only on the degree of degeneracy of the coefficients $b_L$, $b_N$ as $r\downarrow 0$ (the parameter $\alpha$), but also on their relative intensity (the parameter $\beta$ below).
As far as we know, the best available regularity result for solutions of \eqref{eq:PDE_f_deg} is given by \cite{ChiSer1984}, proving local H\"older regularity of solutions, with unspecified H\"older exponent.

Let us move to the precise statements of the results contained in this section.
The main assumption on $W$ we will adopt in order to prove anomalous regularization is:
\begin{assumption} \label{assumption}
The noise $W$ admits an homogeneous, isotropic covariance $C$.
Let $b_L$, $b_N$ be given by \eqref{eq:definition_b}, then there exist $\alpha \in (0,1)$ and $c>0$ such that, as $r \downarrow 0$, it holds
\begin{align} \label{eq:assumption_b}
b_L(r) = c r^{2\alpha} + o(r^{2\alpha})&,
\quad
b_N(r) = \beta c r^{2\alpha} + o(r^{2\alpha}),
\\
\label{eq:assumption_db}
\partial_r b_L(r) = 2\alpha c r^{2\alpha-1} + o(r^{2\alpha-1})&,
\quad
\partial_r b_N(r) = 2\alpha \beta c r^{2\alpha-1} + o(r^{2\alpha-1}),
\end{align}
for some 
\begin{align} \label{eq:assumption_beta}
\beta > \frac{2\alpha-1}{d-1}.
\end{align}
In \eqref{eq:assumption_db} above, we are implicitly assuming that $b_L$, $b_N$ are differentiable in a neighbourhood of $r=0$, excluding at most the point $r=0$ when $\alpha<1/2$.
\end{assumption}

The Kraichnan model satisfies \autoref{assumption} if and only if $\alpha \in (0,1)$ and $\eta > 1-\frac{d}{4\alpha^2}$, see \autoref{cor:assumptions_kraichnan} in  \autoref{app:auxiliary}.
This condition corresponds exactly to the diffusive regime identified by Le Jan and Raimond \cite{LeJRai2002}.% for the Kraichnan model.

Next, let us define
\begin{align} \label{eq:assumption_delta}
 \delta_\star = \delta_\star(\alpha,\beta,d) := \min \{ 1-\alpha  , (d-1)\beta +1 -2\alpha\}>0,
\end{align}
and recall the definition of the increment ratio
\begin{align*}
\llbracket f \rrbracket_{I^{\gamma}(l)}
:=
\sup_{r \in (0,l)}
\frac{|f(r)-f(0)|}{r^{\gamma}},
\quad
\gamma \geq 0, \quad l>0.
\end{align*}

We shall restrict our analysis to solutions $f^\kappa$ satisfying for every $\kappa \in (0,1/2)$
\begin{align*}
\sup |f^\kappa| :=
\sup_{t,r \in \R_+}
|f^\kappa(t,r)|  < \infty,
\end{align*}
and the Neumann boundary condition $\partial_r f^\kappa(t,0)=0$ for every $t>0$.
These assumption are justified by the fact that, for $\kappa \in (0,1/2)$ and $t>0$, the two-point self-correlation function is a bounded $C^1_x$ function which is symmetric with respect the origin, and therefore its gradient vanishes at zero.

We have the following:
\begin{prop} \label{prop:regularity}
Suppose \autoref{assumption}. 
Then for every $\delta \in (0,\delta_\star)$ there exists $\eps = \eps(\alpha,\delta)>0$ arbitrarily small and $l>0$ with the following property.
For every $\kappa \in (0,1/2)$, let $f^\kappa$ satisfy \eqref{eq:PDE_f} and $f^\kappa_t \in  C^2_{loc}$, $\partial_r f^\kappa(t,0)=0$ for every $t>0$.
Then for every $\kappa \in (0,1/2)$ and $t > 0$ it holds  
\begin{align} \label{eq:thm_reg}
\int_0^t
\llbracket f^\kappa_s \rrbracket^{1-\eps}_{I^{2-2\alpha-\delta}(l)}
\dd s
\lesssim (1+t)
\sup|f^\kappa|^{1-\eps},
\end{align}
with implicit constant that does not depend upon $\kappa$.
\end{prop} 

The strategy of the proof of \autoref{prop:regularity} is the following. First, we find a suitable change of variables $r \mapsto \xi(r)$ so that the local H\"older-like regularity for $f^\kappa$ stated in the theorem corresponds to local Lipschitz regularity for the function $g^\kappa(t,\xi) := f(t,r)$.
Then, for $t>\tau>0$ we produce suitable bounds on the $L^2([\tau,t], Lip_x)$ norm of $g^\kappa$ by using the PDE satisfied by $g^\kappa$ and, critically, the Neumann boundary condition on $f^\kappa$ to cancel out boundary terms coming from integration by parts.
Finally, we obtain the desired control on the $L^{1-\eps}_t Lip_x$ norm of $g^\kappa$ by a localization argument.

Under extra regularity assumptions on $f^\kappa$, \autoref{prop:regularity} above implies the following:
\begin{cor} \label{cor:extra_regularity}
Under the same assumptions of \autoref{prop:regularity}, suppose in addition $f_0^\kappa \in C^2_b$. Then for every $\delta>0$ sufficiently small there exists $l > 0$ such that, for every $\kappa \in (0,1/2)$ and $t \geq 0$, it holds 
\begin{align} \label{eq:thm_reg2}
\sup_{s \in [0,t]} \llbracket f^\kappa_s \rrbracket_{I^{2-2\alpha-\delta}(l)} 
\lesssim
(1+t) \| f^\kappa_0 \|_{C^2_b}.
\end{align}
\end{cor} 
\begin{proof}
Since the coefficients of \eqref{eq:PDE_f} are time-independent, the function $h^\kappa := \partial_t f^\kappa$ satisfies the same PDE with initial condition in $C_b$. In particular, we can apply \autoref{prop:regularity} and interpolation with $L^\infty_{t,x}$ to conclude that $\llbracket h^\kappa \rrbracket_{I^{2-2\alpha-\delta}(l)}  \in L^1_{t,loc}$ with estimate proportional to $\| f_0^\kappa \|_{C^2_b}$. But then
\begin{align*}
\sup_{s\in [0,t]} \llbracket f^\kappa_s \rrbracket_{I^{2-2\alpha-\delta}(l)}  
&= 
\sup_{s\in [0,t]} \left\llbracket f^\kappa_0 + \int_0^s h^\kappa_r \dd r \right\rrbracket_{I^{2-2\alpha-\delta}(l)} 
\\
&\leq 
\| f_0^\kappa \|_{C^2_b} 
+ 
\int_0^t \llbracket h^\kappa_r \rrbracket_{I^{2-2\alpha-\delta}(l)} \dd r \lesssim
(1+t) \| f_0^\kappa \|_{C^2_b}. \qedhere
\end{align*} 
\end{proof}

In virtue of \autoref{lem:Sobolev-Holder}, the previous \autoref{prop:regularity} and its \autoref{cor:extra_regularity} yield Sobolev regularity for $\tilde{\theta}^\kappa$. Since all the estimates are uniform in $\kappa$, the $H^{1-\alpha-\delta}_x$ regularity gain in \autoref{thm:regularity_Kraichnan} is recovered by taking the limit $\kappa \downarrow 0$, as detailed below.
\begin{proof}[Proof of \autoref{thm:regularity_Kraichnan}, bounds \eqref{eq:regularization_intro} and \eqref{eq:regularization_intro2}]
Let $\tilde{\theta}^\kappa$ denote the solution of \eqref{eq:spde_viscous_approx} with $\kappa \in (0,1/2)$. As $\tilde{\theta}^\kappa_t \to \theta_t$ strongly in $L^2_{\omega,x}$ for every fixed $t \geq 0$ by \autoref{lem:approx},  to prove \eqref{eq:regularization_intro} it suffices to show
\begin{align*}
\sup_{\kappa \in (0,1/2)}
\int_0^t \mathbb{E} \| \tilde{\theta}^\kappa_s \|^2_{H^{1-\alpha-\delta}_x}\dd s
\lesssim 
(1+t)\, \| \theta_0 \|_{L^2_x}^2.
\end{align*}
By \autoref{lem:Sobolev-Holder} and \eqref{eq:contraction_L2}, for the above it is enough to prove \eqref{eq:anom_reg_01}.
By the results in \autoref{ssec:radial}, \emph{Step 4}, we can assume without loss of generality that $F^\kappa$ is radial.
Next, recall that for every $\kappa \in (0,1/2)$ and $t>0$ there exists $\tau \in (0,t)$ such that $\tilde{\theta}^\kappa_\tau \in L^2(\Omega;H^1_x)$, by \autoref{prop:wellposedness_viscous_SPDE}. In particular, by \autoref{lem:Sobolev-Holder} we have $F^\kappa_\tau \in C^{2}_x$, and since the PDE \eqref{eq:PDE_F_intro_tilde} is strictly parabolic, by \cite[Theorem 8.2.1 and Corollary 8.3.1]{Kr08} we have $F^\kappa_t \in C^2_{x}$ and $\nabla F^\kappa_t$ vanishes at the origin as $F^\kappa$ is symmetric, being it a two-point self-correlation function (by a change of variables in \eqref{eq:two-point_correlation}). Since $t>0$ is arbitrary, the radial part of $F^\kappa$ satisfies all the assumptions of \autoref{prop:regularity}. 
Furthermore, by interpolation with $L^\infty_{t,x}$ one can reach full $L^1$ integrability in time up to paying approximately $\eps$ regularity in space (which can be absorbed into the arbitrariness of $\delta$).
We deduce
\begin{align*}
\sup_{\kappa \in (0,1/2)} \int_0^t
\llbracket F^\kappa_s \rrbracket_{I^{2-2\alpha-\delta}(l)} \dd s
&=
\sup_{\kappa \in (0,1/2)} \int_0^t
\llbracket f^\kappa_s \rrbracket_{I^{2-2\alpha-\delta}(l)} 
\dd s
\\
&\lesssim
(1+t) 
\sup|f^\kappa|
=
(1+t)\, \| \theta_0 \|_{L^2_x}^2.
\end{align*}

As for the bound \eqref{eq:regularization_intro2}, observe that  since $f^\kappa$ comes from a two-point self-correlation function, a sufficient condition to guarantee $f_0^\kappa \in C^2_b$ is to take $\tilde{\theta}^\kappa_0 = \theta_0 \in H^1_x$. Then we can repeat the argument above using \eqref{eq:thm_reg2} instead of \eqref{eq:thm_reg}.
\end{proof}

\begin{remark}
In the authors' opinion, it would be interesting to understand if the regularity given by \eqref{eq:anom_reg_01} can be obtained even without assuming the additional structure given by being a zero-diffusivity limit of two-point self-correlation functions, namely if 
\begin{align*}
\int_0^t \sup_{|z| < l} \frac{|F_s(z) - F_s(0) |}{|z|^{2-2\alpha-\delta}}\dd s
< \infty
\end{align*}
holds in general for bounded solutions of \eqref{eq:PDE_F} with $\kappa = 0$. This is a purely deterministic PDE problem that we believe deserves further attention.
\end{remark}

The next corollary is about $L^p_x$-based Sobolev regularity of solutions when the initial condition $\theta_0$ has more than $L^2_x$ integrability. For simplicity we state the result for $\theta_0 \in L^2_x \cap L^\infty_x$ but the same argument can be used whenever $\theta_0 \in L^2_x \cap L^q_x$ for some $q>2$, up to modifying the regularity exponent accordingly.
\begin{cor}
    Let $\theta_0\in L^2_x\cap L^\infty_x$, then for any $p\in (2,\infty)$ and small $\delta>0$, the associated solution of \eqref{eq:Kraichnan} satisfies \begin{equation}\label{eq:anomalous_regularity_interpolation1}
 \int_0^t \mathbb{E} \| \theta_s \|^p_{W^{2(1-\alpha-\delta)/p,p}_x} \dd s
\lesssim (1+t)\, \| \theta_0\|_{L^\infty_x}^{p-2} \| \theta_0 \|_{L^2_x}^{2},
\qquad
\forall t \geq 0.
    \end{equation}
    If additionally $\theta_0\in H^1_x\cap L^\infty_x$, then
    \begin{equation}\label{eq:anomalous_regularity_interpolation2}
        \mathbb{E} \Big[ \| \theta_s \|^p_{W^{2(1-\alpha-\delta)/p,p}_x} \Big] \lesssim (1+t)\,  \| \theta_0\|_{L^\infty_x}^{p-2} \| \theta_0 \|_{H^1_x}^{2},
\qquad
\forall t \geq 0.
    \end{equation}
\end{cor}
\begin{proof}
    By assumption and inequality \eqref{eq:contraction_Lp} from \autoref{prop:properties.inviscid.Kraichnan}, 
    \begin{equation}\label{eq:contraction_Linfty}
        \sup_{t \geq 0}  \| \theta_t \|_{L^\infty_\omega L^\infty_x} \leq \| \theta_0 \|_{L^\infty_x}.
    \end{equation}
    Estimate \eqref{eq:anomalous_regularity_interpolation1} follows from interpolating \eqref{eq:contraction_Linfty} with \eqref{eq:regularization_intro}; similarly, \eqref{eq:anomalous_regularity_interpolation2} follows from interpolating \eqref{eq:contraction_Linfty} (now at fixed $t\geq 0$) with \eqref{eq:regularization_intro2}.
\end{proof}

\subsection{Proof of \autoref{prop:regularity}}

Let us notice that the precise values of $c$ in \eqref{eq:assumption_b} and \eqref{eq:assumption_db} and $c_0$ in \eqref{eq:PDE_f} do not affect the regularity stated in \autoref{prop:regularity} and \autoref{cor:extra_regularity}.
Hereafter, we will replace $c=c_0=1$ for the sake of simplicity. Moreover, by rescaling time and since $(1-\kappa)$ is bounded away from zero for $\kappa \in (0,1/2)$, we can equivalently study the regularity of
\begin{align*}
\partial_t \bar{f}^\kappa 
=
b_L(r) \partial_r^2 \bar{f}^\kappa 
+
\frac{d-1}{r}b_N(r) \partial_r \bar{f}^\kappa 
+
2\bar{\kappa} \left( \partial_r^2 \bar{f}^\kappa + \frac{d-1}{r} \partial_r \bar{f}^\kappa\right),
\quad
\bar{\kappa} := \frac{\kappa}{1-\kappa} \in (0,1). 
\end{align*}

In the following we rename $f:=\bar{f}^\kappa$ and $\kappa:=\bar{\kappa} \in (0,1)$ for notational convenience, and we study
\begin{align} \label{eq:PDE_f_new}
\partial_t f 
=
b_L(r) \partial_r^2 f
+
\frac{d-1}{r}b_N(r) \partial_r f 
+
2{\kappa} \left( \partial_r^2 f + \frac{d-1}{r} \partial_r f\right).
\end{align}

We divide the proof of \autoref{prop:regularity} into several steps.

\subsubsection{Step 1: Change of variables}

Let us consider a change of variables $r \mapsto \xi(r)$ and define $g$ as
\begin{align*}
g(t,\xi) := f(t,r).
\end{align*}
Taking partial derivatives in the expression above we obtain the identities
\begin{align} \label{eq:derivatives_g}
\partial_t f(t,r) &= \partial_t g(t,\xi),
\\
\partial_r f(t,r) &= \partial_r \xi(r) \partial_\xi g(t,\xi),
\\
\partial^2_r f(t,r) &= (\partial_r \xi(r))^2 \partial_\xi^2 g(t,\xi)
+
\partial_r^2 \xi(r) \partial_\xi g(r,\xi).
\end{align}
Using the relations above and \eqref{eq:PDE_f_new} we obtain a PDE for $g$, that is
\begin{align} \label{eq:PDE_g}
\partial_t g(t,\xi)
&=
b_L(r) (\partial_r \xi(r))^2 \partial_\xi^2 g(t,\xi)
+
\left( b_L(r) \partial^2_r \xi(r) + \frac{d-1}{r}b_N(r) \partial_r \xi(r) \right)
\partial_\xi g(t,\xi)
\\ \nonumber
&\quad+ 2\kappa \left( 
(\partial_r \xi(r))^2 \partial_\xi^2 g(t,\xi)
+
\left( \partial^2_r \xi(r) + \frac{d-1}{r} \partial_r \xi(r) \right)
\partial_\xi g(t,\xi)
\right).
\end{align}

To avoid any potential confusion, let us point out that the PDE above does not retain any dependence of $r$, even if we abuse notation and write the coefficients in front of the $\xi$ derivatives of $g$ as functions of $r$.
That is, by injectivity of the change of variable $r \mapsto \xi(r)$, one can rewrite all the $r$-dependent quantities in \eqref{eq:PDE_g} as functions of $\xi$, e.g. $\partial_r \xi(r) =: \Phi_1(\xi)$, $\partial_r^2 \xi(r) =: \Phi_2(\xi)$, or $b_L(r) =: \Phi_L(\xi)$.
In particular, one can takes derivatives with respect to $\xi$ as $\partial_\xi(\partial_r \xi(r)) = \partial_\xi \Phi_1 (\xi)$ et cetera.
In the following, we will often use the correspondence between the variables $r$ and $\xi$ without explicit mention.  

Next, we have to fix the change of variable $\xi$.
Let $q=q(r)= r^{-\delta}$ where $\delta \in (0,\delta_\star)$.
Define $\xi = \xi(r)$ by imposing $\xi(0)=0$ and 
\begin{align*}
\partial_r \xi(r)
:=
\int_0^r
\frac{q(\rho)}{b_L(\rho)}
\exp \left( 
-(d-1)\int_\rho^r \frac{b_N(u) \dd u}{b_L(u) u}
\right) \dd \rho.
\end{align*}

It is easy to check that $\xi$ solves the ODE
\begin{align} \label{eq:ODE_xi}
b_L(r) \partial_r^2 \xi(r)
+
\frac{d-1}{r}b_N(r) \partial_r \xi(r)
=
q(r),
\quad
\xi(0) = 0.
\end{align}

\begin{lemma} \label{lem:change_of_variables}
The following hold true.
\begin{enumerate}
\item[i)]
There exists $c_\xi>0$ such that for every $\eps>0$ there exists $\ell>0$ such that for every $r \in (0,\ell)$
\begin{align*}
(c_\xi-\eps) r^{1-2\alpha-\delta}
\leq
\partial_r \xi(r) 
\leq
(c_\xi+\eps) r^{1-2\alpha-\delta};
\end{align*}
\item[ii)]
Fix $\eps = c_\xi /2$ and take the corresponding $\ell>0$ given by point $i)$ above.
Then, for every $\kappa \in (0,1)$ and $t>0$ there exists $c_g >0$ such that such that for every $r \in (0,\ell)$
\begin{align*}
|\partial_\xi g(t,\xi)| 
\leq
c_g r^{2\alpha+\delta}.
\end{align*}
\end{enumerate}
\end{lemma}

\begin{proof}
By the estimates \eqref{eq:assumption_b}, for every $\eps'>0$ there exists $\ell>0$ such that for every $u \in (0,\ell)$
\begin{align*}
\frac{\beta - \eps'}{u}
\leq
\frac{b_N(u)}{b_L(u) u}
\leq
\frac{\beta + \eps'}{u}.
\end{align*} 
Therefore assuming $r<\ell$
\begin{align*}
\left(\frac{\rho}{r}\right)^{(d-1)(\beta+\eps')}
\leq
\exp \left( 
-(d-1)\int_ \rho ^r \frac{b_N(u) \dd u}{b_L(u) u}
\right)
\leq
\left(\frac{\rho}{r}\right)^{(d-1)(\beta-\eps')}.
\end{align*}
Moreover,
\begin{align*}
\left( 1 - \eps' \right) \rho ^{-2\alpha-\delta}
\leq
\frac{q(\rho)}{b_L(\rho)}
\leq
\left( 1 + \eps' \right) \rho ^{-2\alpha-\delta}.
\end{align*}
Next, without loss of generality we can change the value of $\eps'$ so that $(d-1)(\beta \pm \eps') + 1 - 2\alpha - \delta > 0$, which is always possible by \eqref{eq:assumption_delta} since $\delta \in (0,\delta_\star)$. 
We have   
\begin{align*}
\frac{\left( 1 - \eps' \right) r^{1-2\alpha-\delta}}{(d-1)(\beta+\eps')+1-2\alpha-\delta}
\leq
\partial_r \xi(r)
\leq
\frac{\left( 1 + \eps' \right) r^{1-2\alpha-\delta}}{(d-1)(\beta-\eps')+1-2\alpha-\delta}.
\end{align*}
This proves \textit{i)} with 
\begin{align*}
c_\xi = \frac{1}{(d-1)\beta+1-2\alpha-\delta} > 0.
\end{align*}

Let us move to \textit{ii)}.
Using $\partial_r \xi (r) >0 $ for every $r \in (0,\ell)$ we have
\begin{align*}
\partial_\xi g(t,\xi) = 
\frac{\partial_r f(t,r)}{\partial_r\xi(r)}.
\end{align*}
By assumption $\partial_r f(t,0)=0$ and $f_t \in C^2_{loc}$, therefore 
\begin{align*}
|\partial_r f(t,r)|
\leq
\sup_{u \in (0,r)} |\partial_r^2 f(t,u)| r.
\end{align*}
Combining with point \textit{i)} we get for every $r \in (0,\ell)$
\begin{align*}
|\partial_\xi g(t,\xi)|
\leq
\frac{\sup_{u \in (0,r)} |\partial_r^2 f(t,u)|}{c_\xi/2} r^{2\alpha+\delta}
\end{align*}
This proves \textit{ii)} with 
\begin{align*}
c_g =  \frac{\sup_{u \in (0,\ell)} |\partial_r^2 f(t,u)|}{c_\xi/2}.
\end{align*}
Notice that $c_g$ may depend on $\kappa$ and $t$ since $\partial_r^2 f_t$ does.
\end{proof}

\subsubsection{Step 2: A priori estimates}
Let us move back to the estimate for $g$. 
The ultimate goal of this subsection is to produce good estimates for the $C^0_{[0,\ell]}$ norm of $\partial_\xi g$. Here $\ell>0$ is a fixed, possibly extremely small parameter, and all implicit constants are allowed to depend on it.

Let us define
\begin{align*}
w(r) 
:=
b_L(r) (\partial_r \xi(r))^2,
\quad
\tilde{w}(r) 
:=
(\partial_r \xi(r))^2,
\quad
\tilde{q}(r)
:=
\partial^2_r \xi(r) + \frac{d-1}{r} \partial_r \xi(r).
\end{align*}

We are going to show a priori estimates for the derivative $y(t,\xi) := \partial_\xi g(t,\xi)$.
Since $\xi$ solves the ODE \eqref{eq:ODE_xi}, the PDE \eqref{eq:PDE_g} can be rewritten as
\begin{align*}
\partial_t g(t,\xi)
&=
w(r) \partial_\xi^2 g(t,\xi)
+
q(r)
\partial_\xi g(t,\xi)+ 2\kappa 
\tilde{w}(r) \partial_\xi^2 g(t,\xi)
+
2\kappa \tilde{q}(r)
\partial_\xi g(t,\xi).
\end{align*}
Therefore, $y$ satisfies
\begin{align} \label{eq:PDE_y}
\partial_t y(t,\xi)
&=
w(r) \partial_\xi^2 y(t,\xi)
+
\partial_\xi \left( w(r) \right) \partial_\xi y(t,\xi)
\\ \nonumber
&\quad+
q(r) \partial_\xi y(t,\xi)
+
\partial_\xi (q(r)) y(t,\xi)
\\ \nonumber
&\quad+ 
2\kappa \tilde{w}(r) \partial_\xi^2 y(t,\xi)
+
2\kappa \partial_\xi \left( \tilde{w}(r) \right) \partial_\xi y(t,\xi)
\\ \nonumber
&\quad+
2\kappa \tilde{q}(r) \partial_\xi y(t,\xi)
+
2\kappa \partial_\xi (\tilde{q}(r)) y(t,\xi).
\end{align}

The next lemma concerns monotonicity of the term $\tilde{q}$ as a function of $\xi$.
Being $r \mapsto \xi(r)$ strictly increasing, it is sufficient to check $\partial_r  \tilde{q} \leq 0$. 
\begin{lemma} 
Assume \eqref{eq:assumption_b}, \eqref{eq:assumption_db}, and \eqref{eq:assumption_beta}.
Then there exists $\ell>0$ sufficiently small such that
\label{lem:q_tilde_decreasing}
\begin{align*}
\partial_r \tilde{q}(r) \leq 0,
\quad
\forall r \in (0,\ell).
\end{align*}
\end{lemma}

\begin{proof}
Since $\xi$ solves the ODE \eqref{eq:ODE_xi}, we have
\begin{align*}
\partial_r \tilde{q}
&=
\partial_r \left( \frac{q}{b_L} \right)
+
\partial_r \left( \frac{d-1}{r} \left( 1-\frac{b_N}{b_L} \right) \partial_r \xi \right)
\\
&=
\partial_r \left( \frac{q}{b_L} \right)
+
\frac{d-1}{r} \left( 1-\frac{b_N}{b_L} \right)
\frac{q}{b_L}
-
\left(\frac{d-1}{r}\right)^2 \left( 1-\frac{b_N}{b_L} \right)\frac{b_N}{b_L} \partial_r \xi
\\
&\quad- 
\frac{d-1}{r^2} \left( 1-\frac{b_N}{b_L} \right) \partial_r \xi
-
\frac{d-1}{r} \partial_r \left( \frac{b_N}{b_L} \right) \partial_r \xi.
\end{align*}

Now we use $q(r) = r^{-\delta}$ and the following bounds, true for arbitrary $\eps>0$ and every $r \in (0,\ell)$ up to choosing $\ell>0$ small enough:
\begin{align*}
(1-\eps)r^{2\alpha} \leq 
b_L(r) 
&\leq (1+\eps)r^{2\alpha},
\\
(2\alpha-\eps)r^{2\alpha-1} \leq 
\partial_r b_L(r) 
&\leq (2\alpha+\eps)r^{2\alpha-1},
\\
(\beta-\eps)r^{2\alpha} \leq 
b_N(r) 
&\leq (\beta+\eps)r^{2\alpha},
\\
(2\alpha\beta-\eps)r^{2\alpha-1} \leq 
\partial_r b_N(r) 
&\leq (2\alpha\beta+\eps)r^{2\alpha-1},
\\
(c_\xi-\eps) r^{1-2\alpha-\delta} \leq
\partial_r \xi
&\leq (c_\xi+\eps) r^{1-2\alpha-\delta}.
\end{align*}
The estimates on $b_L$ and $b_N$ are nothing but the conditions \eqref{eq:assumption_b} and \eqref{eq:assumption_db} in \autoref{assumption}.
The last one comes from \autoref{lem:change_of_variables}.
Let us plug these bounds into the expression of $\partial_r \tilde{q}$ above. 
Up to changing the value of $\eps$ into a larger one (but still arbitrarily small), we get the following: for arbitrary $\eps>0$ and $r \in (0,\ell)$, $\ell$ small enough, it holds
\begin{align*}
\frac{\partial_r \tilde{q}(r)}{r^{-1-\delta-2\alpha}}
&\leq
-2\alpha-\delta
+
(d-1)(1-\beta)
-
(d-1)(1-\beta)c_\xi
\left( (d-1)\beta + 1 \right)
+
\eps
\\
&=
-\frac{(2\alpha+\delta) (d-2\alpha-\delta)}{(d-1)\beta+1-2\alpha-\delta}
+
\eps
<
0.
\end{align*}
It is interesting to observe that we have recovered the exact same compressibility threshold \eqref{eq:assumption_beta}.
\end{proof}

\begin{lemma} \label{lem:apriori_g}
For every $\ell>0$ sufficiently small there exists an implicit constant such that, for every $t > \tau \geq 0$ and $\kappa \in (0,1)$ it holds
\begin{align*}
\int_\tau^t \int_0^\ell b_L(r)(\partial_r\xi(r))^2 |\partial_\xi^2 g_s|^2 \dd\xi\dd s
\lesssim
\int_0^{2\ell} |\partial_\xi g_\tau|^2 \dd\xi
+
\int_\tau^t \int_{\ell}^{2\ell}
|\partial_\xi g_s|^2 \dd\xi\dd s.
\end{align*}
\end{lemma}

\begin{proof}
Let $\chi$ be a smooth cut-off function such that $\chi(\xi)=1$ for $\xi \leq \ell$ and $\chi(\xi)=0$ for $\xi \geq 2\ell$.

Testing \eqref{eq:PDE_y} against $\chi y$ and using integration by parts, we get
\begin{align*}
\frac12 &\int_0^\infty \chi(\xi) |y(t,\xi)|^2 \dd\xi
-
\frac12 \int_0^\infty \chi(\xi) |y(\tau,\xi)|^2 \dd\xi
\\
&=
-\int_\tau^t \int_0^\infty \chi(\xi)
w(r) |\partial_\xi y(s,\xi)|^2 \dd\xi\dd s
-
2\kappa \int_\tau^t \int_0^\infty \chi(\xi)
\tilde{w}(r) |\partial_\xi y(s,\xi)|^2 \dd\xi\dd s
\\
&\quad+ \frac12
\int_\tau^t \int_0^\infty
\Big( 
\partial_\xi ((w+2\kappa\tilde{w}) \partial_\xi\chi) 
- 
(q+2\kappa \tilde{q})\partial_\xi \chi + \partial_\xi (q+2\kappa\tilde q) \chi
\Big)
|y(s,\xi)|^2 \dd\xi\dd s
\\
&\quad+
\int_\tau^t
\left[ \chi(\xi) w(r) y(s,\xi) \partial_\xi y(s,\xi)  \right]_{\xi=0}^{\xi=\infty}\dd s
+
\frac12 \int_\tau^t \left[ \chi(\xi) q(r) |y(s,\xi)|^2 \right]_{\xi=0}^{\xi=\infty}\dd s
\\
&\quad+
2\kappa \int_\tau^t \left[ \chi(\xi) \tilde{w}(r) y(s,\xi) \partial_\xi y(s,\xi)  \right]_{\xi=0}^{\xi=\infty}\dd s
+
\kappa \int_\tau^t \left[ \chi(\xi) \tilde{q}(r) |y(s,\xi)|^2 \right]_{\xi=0}^{\xi=\infty}\dd s,
\end{align*}
where the terms in square brackets are boundary terms coming from the integration by parts. 

Next, we notice that  
\begin{align*}
&\frac12  \int_\tau^t \int_0^\infty
\Big( 
\partial_\xi ((w+2\kappa\tilde{w}) \partial_\xi\chi) 
- 
(q+2\kappa \tilde{q})\partial_\xi \chi 
\Big)
|y(s,\xi)|^2 \dd\xi\dd s
\\
&\quad=
\frac12 \int_\tau^t \int_\ell^{2\ell}
\Big( 
\partial_\xi ((w+2\kappa\tilde{w}) \partial_\xi\chi) 
- 
(q+2\kappa \tilde{q})\partial_\xi \chi 
\Big)
|y(s,\xi)|^2 \dd\xi\dd s
\lesssim
 \int_\tau^t \int_\ell^{2\ell}
|y(s,\xi)|^2 \dd\xi\dd s.
\end{align*}
Again, the coefficient in front of $|y(s,\xi)|^2$ in the integral above is controlled by an implicit constant, depending on $\ell>0$. Notice that here we also use the assumption \eqref{eq:assumption_db} on the differentiability of $b_L$ to bound the term with $\partial_\xi(w+2\kappa \tilde{w})$, at least up to choosing $\ell>0$ small enough.

The term involving the derivatives of $q$ and $\tilde{q}$ has a definite sign, more precisely:
\begin{align*}
&\frac12  \int_\tau^t \int_0^\infty
\partial_\xi (q+2\kappa \tilde{q}) \chi |y(s,\xi)|^2 \dd\xi\dd s
\leq
0,
\end{align*}
where we have used the monotonicity of $\xi$ and \autoref{lem:q_tilde_decreasing} to ensure $\partial_\xi (q+2\kappa \tilde{q}) \leq 0$.

In order to conclude, it only remains to check that the boundary terms vanish for every given $\kappa \in (0,1)$ and $s>0$. The boundary terms at $\xi=\infty$ clearly vanish, since $\chi$ is supported in $[0,2\ell]$. In addition, since $\xi:\R_+ \to \R_+$ is strictly increasing and bijective, it holds 
\begin{align*}
[\,\cdot\,]_{\xi=0}^{\xi=\infty} 
=  
-
\lim_{\xi \downarrow 0} [\,\cdot\,]
=
-
\lim_{r \downarrow 0} [\,\cdot\,],
\end{align*}
so we only need to study the behaviour of the terms inside square brackets as $r \downarrow 0$.

First, we consider $y$ and $\partial_\xi y$. By \eqref{eq:derivatives_g}, since $\xi$ solves the ODE \eqref{eq:ODE_xi}, we have
\begin{align*}
y(s,\xi) 
&= 
\frac{\partial_r f(s,r)}{\partial_r \xi(r)},
\\
\partial_\xi y(s,\xi)
%&=
%\frac{\partial^2_r f(s,r)}{(\partial_r \xi(r))^2}
%-
%\frac{\partial_r^2 \xi(r)}{(\partial_r \xi(r))^2} y(s,\xi)
%\\
&=
\frac{\partial^2_r f(s,r)}{(\partial_r \xi(r))^2}
+ 
\frac{d-1}{r} 
\frac{b_N(r)}{ b_L(r) \partial_r \xi(r)} y(s,\xi)
-
\frac{q(r)}{b_L(r)(\partial_r \xi(r))^2} y(s,\xi).
\end{align*}
for every $s>0$. Hence, by \autoref{lem:change_of_variables} we deduce the asymptotic behaviours
\begin{align*}
y(s,\xi) \sim r^{2\alpha+\delta},
\quad
\partial_\xi y(s,\xi) \sim r^{-2+4\alpha+2\delta},
\quad
\mbox{ as } r \downarrow 0
\end{align*}
Notice that the proportionality constants may depend on $\kappa \in (0,1)$ and $s>0$, but this will not affect the final result.
Moreover, for  the coefficients $w,\tilde{w},q,\tilde{q}$ we have as $r \downarrow 0$ 
\begin{align*}
& w(r) \sim r^{2-2\alpha-2\delta},
\qquad
&&\tilde{w}(r) \sim r^{2-4\alpha-2\delta},
\\
&q(r) \sim r^{-\delta},
&&\tilde{q}(r) \sim r^{-2\alpha-\delta}.
\end{align*}
Now it is easy to check that the boundary terms equal zero. Indeed, we have
\begin{align*}
& w(r) y(s,\xi) \partial_\xi y(s,\xi)
\sim
r^{4\alpha+\delta}
,\qquad\qquad\qquad
q(r) |y(s,\xi)|^2
\sim
r^{4\alpha+\delta},
\\
& \tilde{w}(r) y(s,\xi) \partial_\xi y(s,\xi)
\sim
r^{2\alpha+\delta}
,\qquad\qquad\qquad
\tilde{q}(r) |y(s,\xi)|^2
\sim
r^{2\alpha+\delta}. \qedhere
\end{align*}
\end{proof}

\begin{lemma} \label{lem:apriori_g_bis}
For every $\ell>0$ sufficiently small there exists an implicit constant such that, for every $t > \tau \geq 0$ and $\kappa \in (0,1)$ it holds
\begin{align*}
\int_\tau^t \int_{\ell}^{4\ell}
|\partial_\xi g_s|^2 \dd\xi\dd s
\lesssim
(1+t) \sup |g|^2.
\end{align*}
\end{lemma}
\begin{proof}
The proof is obtained testing the PDE \eqref{eq:PDE_g} against $\chi g$, where $\chi$ is a smooth cut-off such that $\chi(\xi) = 1$ for $\xi \in [\ell,4\ell]$ and $\chi(\xi) = 0$ for $\xi \notin [\ell/2,8\ell]$. 
\end{proof}

In the following, we define the local H\"older norm
\begin{align*}
\| g \|_{C^\gamma_{[0,\ell]}}
:=
[ g ]_{C^\gamma_{[0,\ell]}}
+
\sum_{k \leq \lfloor \gamma \rfloor}
[ \partial_\xi^k g ]_{C^0_{[0,\ell]}}
\end{align*}
for $\gamma\geq 0$ and $\ell > 0$.

\begin{lemma} \label{lem:apriori_g_tris}
For every $\ell>0$ sufficiently small, $\theta \in (0,1)$ and $\eps\in(0,\theta/(1+\theta))$, there exists an implicit constant such that, for every $t > \tau \geq 0$ and $\kappa \in (0,1)$ it holds
\begin{align*}
\int_\tau^t \int_{0}^{2\ell}
|\partial_\xi g_s|^2 \dd\xi\dd s
\lesssim
(1+t) \sup |g|^2
+
\sup |g|^{1+\eps}
\int_\tau^t \| g_s \|^{1-\eps}_{C^{1+\theta}_{[0,\ell]}}\dd s.
\end{align*} 
\end{lemma}
\begin{proof}
Let $\chi$ be a smooth cut-off such that $\chi(\xi) = 1$ for $\xi <2\ell$ and $\chi(\xi) = 0$ for $\xi >4\ell$. Test the PDE \eqref{eq:PDE_g} against $\chi (w+2\kappa\tilde{w})^{-1} g$ to get

\begin{align*}
\int_\tau^t \int_{0}^{2\ell}
|\partial_\xi g_s|^2 \dd\xi\dd s
&\lesssim
\int_0^{4\ell}
(w+2\kappa\tilde{w})^{-1} |g_\tau|^2 \dd\xi
+
\int_\tau^t \int_0^{4\ell}
(\partial_\xi \chi+|q+2\kappa \tilde{q}| (w+2\kappa\tilde{w})^{-1})  |g_s||\partial_\xi g_s| \dd\xi\dd s
\\
&\lesssim
\int_0^{4\ell}
w^{-1} |g_\tau|^2 \dd\xi
+
\int_\tau^t \int_0^{4\ell}
(1+qw^{-1} + |\tilde{q}|\tilde{w}^{-1})  |g_s||\partial_\xi g_s| \dd\xi\dd s,
\end{align*}

where in the second inequality we have used 
\begin{align*}
(w+2\kappa\tilde{w})^{-1} 
\leq 
w^{-1},
\qquad
(w+2\kappa\tilde{w})^{-1} 
\leq 
\frac{1}{2\kappa} \tilde{w}^{-1}.
\end{align*}

Now we observe that for $r$ sufficiently small it holds
\begin{align*}
w(r) \sim r^{2-2\alpha-2\delta} \sim \xi^{1-\delta'}
\end{align*}
where 
\begin{align*}
\delta' := \frac{\delta}{2-2\alpha-\delta}  \in (0,1).
\end{align*}
In particular the weight $w(r)^{-1}$ is integrable on $[0,4\ell]$ with respect to $\dd\xi$ and therefore
\begin{align*}
\int_0^{4\ell}
w^{-1} |g_\tau|^2 \dd\xi
\lesssim
\sup |g|^2.
\end{align*}
As for the other term, we notice that when $r \downarrow 0$
\begin{align*}
q(r) w(r)^{-1} \sim \tilde{q}(r) \tilde{w}(r)^{-1} \sim \xi^{-1}, 
\end{align*}
therefore up to taking $\ell$ sufficiently small
\begin{align*}
\int_\tau^t \int_0^{4\ell}
(1+qw^{-1} + |\tilde{q}|\tilde{w}^{-1})  |g_s||\partial_\xi g_s| \dd\xi\dd s
\lesssim
\int_\tau^t \int_0^{4\ell}
\xi^{-1} |g_s||\partial_\xi g_s| \dd\xi\dd s.
\end{align*}
We split
\begin{align*}
\int_\tau^t \int_0^{4\ell}
\xi^{-1} |g_s||\partial_\xi g_s| \dd\xi\dd s
=
\int_\tau^t \int_0^{\ell}
\xi^{-1} |g_s||\partial_\xi g_s| \dd\xi\dd s
+
\int_\tau^t \int_\ell^{4\ell}
\xi^{-1} |g_s||\partial_\xi g_s| \dd\xi\dd s,
\end{align*}
and use \autoref{lem:apriori_g_bis} and Cauchy's inequality to bound
\begin{align*}
\int_\tau^t \int_\ell^{4\ell}
\xi^{-1} |g_s||\partial_\xi g_s| \dd\xi\dd s
\lesssim
(1+t) \sup |g|^2.
\end{align*}
For the other term, take any $\theta' \in (0,\theta)$ and write
\begin{align*}
\int_\tau^t \int_0^{\ell}
\xi^{-1} |g_s||\partial_\xi g_s| \dd\xi\dd s
&\lesssim
\sup |g| \int_\tau^t \int_0^{\ell}
\xi^{-1} |\partial_\xi g_s| \dd\xi\dd s
\\
&\lesssim
\sup |g| \int_\tau^t  \llbracket \partial_\xi g_s \rrbracket_{I^{\theta'}(\ell)} \int_0^{\ell} 
\xi^{\theta'-1}  \dd\xi\dd s
\\
&\lesssim
\sup |g| \int_\tau^t \llbracket \partial_\xi g_s \rrbracket_{I^{\theta'}(\ell)}\dd s,
\end{align*}
where we have used the fact that $\partial_\xi g(s,0) = 0$. By interpolation there exists $\eps \in (0,1)$, depending only on $\theta$ and $\theta'$, such that
\begin{align*}
\llbracket \partial_\xi g_s\rrbracket_{I^{\theta'}(\ell)}
\lesssim
\sup |g|^{\eps}
\| g_s \|^{1-\eps}_{C^{1+\theta}_{[0,\ell]}}.
\end{align*}
More precisely $\eps$ is chosen as
\begin{align*}
(1-\eps)(1+\theta) = 1+\theta'
\quad
\Rightarrow
\quad
\eps = \frac{\theta-\theta'}{1+\theta},
\end{align*}
and can take any value between $0$ and $\frac{\theta}{1+\theta}$ up to choosing $\theta'$ properly.
\end{proof}

\subsubsection{Step 3: Proof of \autoref{prop:regularity} in the rotational-invariant case} \label{subsec:proof_reg_theorems}
Before moving to the proof of \autoref{prop:regularity}, let us comment on the implications on $f$ of having a control on the Lipschitz norm of $g$.
Recall that $\delta$ satisfies \eqref{eq:assumption_delta}.
Let us distinguish two cases.

If $\delta \in (0,1-2\alpha)$ then the function $\xi = \xi(r)$ is globally $C^2$ except at $r = 0$, where it is $2-2\alpha-\delta$ H\"older. 
Moreover, $\xi$ is strictly increasing since $\partial_r \xi > 0$. 
As a consequence, proving local Lipschitz regularity for $g$ close to zero, uniformly in $\kappa \in (0,1)$, gives a bound on the increments of $f$ close to zero, uniformly in $\kappa \in (0,1)$. Indeed, defining $l:=\xi^{-1}(\ell)$ it holds for every $r \in [0,l]$
\begin{align*}
|f(t,r)-f(t,0)|
&=
|g(t,\xi(r))-g(t,\xi(0))|\leq
\| \partial_\xi g(t,\cdot) \|_{C^0_{[0,\ell]}}
\, |\xi(r)-\xi(0)|
\\
&\lesssim
\| \partial_\xi g(t,\cdot) \|_{C^0_{[0,\ell]}}
\, r^{2-2\alpha-\delta}.
\end{align*}

If $\delta \in [1-2\alpha,1-\alpha)$ the previous argument applies, but since $2-2\alpha-\delta < 1$ one gains genuine local $C^{2-2\alpha-\delta}$ regularity of $f$ close to zero. 
Indeed, defining $l:=\xi^{-1}(\ell)$ it holds for every $r,r' \in [0,l]$
\begin{align*}
|f(t,r)-f(t,r')|
&=
|g(t,\xi(r))-g(t,\xi(r'))|\leq
\| \partial_\xi g(t,\cdot) \|_{C^0_{[0,\ell]}}
\, |\xi(r)-\xi(r')|
\\
&\lesssim
\| \partial_\xi g(t,\cdot) \|_{C^0_{[0,\ell]}}
\, |r-r'|^{2-2\alpha-\delta}.
\end{align*}

\begin{proof}
By the previous argument, it is sufficient to show
\begin{align*}
\int_{0}^{t}
\| \partial_\xi g_s \|^{1-\eps}_{C^0_{[0,\ell]}}\dd s
&\lesssim
\sup |g|^{1-\eps}.
\end{align*}
We will show something slightly better, that is a bound on the time integral of $\| \partial_\xi g_s \|^{1-\eps}_{C^\theta_{[0,\ell]}}$ for some $\theta>0$.
By \autoref{lem:change_of_variables}, the coefficient in front of $\partial_\xi^2 g$ in the PDE \eqref{eq:PDE_g} goes like
\begin{align*}
b_L(r) (\partial_r\xi(r))^2
\sim
r^{2-2\alpha-2\delta}
\sim
\xi^{1-\delta'},
\end{align*}
as $r \downarrow 0$, where 
\begin{align*}
\delta'
=
\frac{\delta}{2-2\alpha-\delta}
\in (0,1),
\end{align*}
since $\delta \in (0,1-\alpha)$.
By \autoref{lem:apriori_g} and \autoref{lem:apriori_g_bis}, we have for $t > \tau \geq 0$ and $\ell$ small enough
\begin{align*}
\int_\tau^t \int_0^\ell
\xi^{1-\delta'} |\partial_\xi^2 g_s|^2 \dd\xi\dd s
&\lesssim
\int_\tau^t \int_0^\ell
b_L(r) |\partial_r \xi(r)|^2 |\partial_\xi^2 g_s|^2 \dd\xi\dd s
\lesssim
(1+t) \sup|g|^2 +
\int_0^{2\ell} |\partial_\xi g_\tau|^2 \dd\xi.
\end{align*}

For any given $s>0$, we have the H\"older inequality
\begin{align*}
\| \partial_\xi^2 g_s \|_{L^p_{[0,\ell]}}
\lesssim
\| \xi^{\frac{1-\delta'}{2}}
\partial_\xi^2 g_s \|_{L^2_{[0,\ell]}}
\| \xi^{-\frac{1-\delta'}{2}}
\|_{L^q_{[0,\ell]}}
\lesssim
\| \xi^{\frac{1-\delta'}{2}}
\partial_\xi^2 g_s \|_{L^2_{[0,\ell]}},
\end{align*}
as soon as 
\begin{align*}
\frac{1}{p}
=
\frac{1}{2}
+
\frac{1}{q},
\quad
2<q<\frac{2}{1-\delta'}.
\end{align*}
Since $p>1$, by Sobolev embedding there exists $\theta>0$ (of order $\theta \sim 1-1/p$) such that
\begin{align*}
\| \partial_\xi g_s \|_{C^\theta_{[0,\ell]}}
\lesssim
\| \partial_\xi g_s \|_{L^p_{[0,\ell]}}
+
\| \partial_\xi^2 g_s \|_{L^p_{[0,\ell]}}
\lesssim
\| \partial_\xi^2 g_s \|_{L^p_{[0,\ell]}}.
\end{align*}
The second inequality here is justified since $\partial_\xi g(s,0) = 0$ for every $s>0$ and we can apply Poincaré inequality on $\partial_\xi g_s$. Therefore,
\begin{align} \label{eq:int_taut}
\int_\tau^t
\| g_s \|_{C^{1+\theta}_{[0,\ell]}}^2
\dd s
&\lesssim
(1+t) \sup|g|^2 
+
\int_\tau^t
\| \partial_\xi g_s \|^2_{C^\theta_{[0,\ell]}}
\dd s 
\\
&\lesssim
(1+t) \sup|g|^2 
+
\int_0^{2\ell}
| \partial_\xi g_\tau |^2 \dd\xi. 
\end{align}

From \autoref{lem:apriori_g_tris} we deduce the following: for every $n \in \mathbb{Z}$ there exists $\tau_n \in [0,2^{-n}]$ such that
\begin{align} \label{eq:g_taun}
\int_0^{2\ell}
| \partial_\xi g_{\tau_n} |^2 \dd\xi
\lesssim
2^n \left( 
(1+2^{-n}) \sup |g|^2
+
\sup |g|^{1+\eps}
\int_0^{2^{-n}} \| g_s \|^{1-\eps}_{C^{1+\theta}_{[0,\ell]}}\dd s
\right),
\end{align}
otherwise one could contradict the former lemma by integrating between $0$ and $2^{-n}$.

Therefore, given $n \in \ZZ$ one has by \eqref{eq:int_taut} applied between time $\tau_n<2^{-n}$ and $2^{1-n}<2t$ and \eqref{eq:g_taun}
\begin{align*}
\int_{2^{-n}}^{2^{1-n}}
\| g_s \|^2_{C^{1+\theta}_{[0,\ell]}}
\dd s
&\lesssim
\int_{\tau_n}^{2^{1-n}}
\| g_s \|^2_{C^{1+\theta}_{[0,\ell]}}
\dd s
\\
&\lesssim
(1+2^{-n})\sup|g|^2 
+
\int_0^{2\ell}
| \partial_\xi g_{\tau_n} |^2 \dd\xi
 \dd s
\\
&\lesssim
(1+2^{-n})\sup|g|^2 
+
2^n \left( 
(1+2^{-n}) \sup |g|^2
+
\sup |g|^{1+\eps}
\int_0^{2^{-n}} \| g_s \|^{1-\eps}_{C^{1+\theta}_{[0,\ell]}}\dd s
\right).
\end{align*}

As a consequence, given any $n_0 \in \mathbb{Z}$ such that $2^{-n_0} \leq t<2^{1-n_0}$, by H\"older inequality one has
\begin{align*}
\int_{0}^{t}
\| g_s \|^{1-\eps}_{C^{1+\theta}_{[0,\ell]}}\dd s
&\lesssim
\sum_{n \geq n_0}
\int_{2^{-n}}^{2^{1-n}}
\| g_s \|^{1-\eps}_{C^{1+\theta}_{[0,\ell]}}\dd s
\\
&\lesssim
\sum_{n \geq n_0}
2^{-n(\frac{1+\eps}{2})}
\left( \int_{2^{-n}}^{2^{1-n}}
\| g_s \|^2_{C^{1+\theta}_{[0,\ell]}}\dd s \right)^{\frac{1-\eps}{2}} 
\\
&\lesssim
\sum_{n \geq n_0}
2^{-n\eps} \left( 
(1+2^{-n})^2 \sup |g|^2
+
\sup |g|^{1+\eps}
\int_0^{2^{-n}} \| g_s \|^{1-\eps}_{C^{1+\theta}_{[0,\ell]}}\dd s
\right)^{\frac{1-\eps}{2}} 
\\
&\lesssim
(1+t) \sup |g|^{1-\eps}
+
t^\eps \left( 
\sup |g|^{1+\eps}
\int_0^t \| g_s \|^{1-\eps}_{C^{1+\theta}_{[0,\ell]}}\dd s
\right)^{\frac{1-\eps}{2}}.
\end{align*}
Next apply Young inequality with exponents $\frac{2}{1-\eps}$ and $\frac{2}{1+\eps}$ so that 
\begin{align*}
t^\eps \left( 
\sup |g|^{1+\eps}
\int_0^t \| g_s \|^{1-\eps}_{C^{1+\theta}_{[0,\ell]}}\dd s
\right)^{\frac{1-\eps}{2}} 
&\leq
t^{\frac{2\eps}{1+\eps}}C_{\eps'} 
\sup |g|^{1-\eps}
+
\eps'
\int_0^t \| g_s \|^{1-\eps}_{C^{1+\theta}_{[0,\ell]}}\dd s
\end{align*}
with $\eps'$ sufficiently small, so that the last term on the right-hand side can be absorbed in the left-hand side of the previous chain of inequalities.
Taking $\eps$ small enough so that $t^{\frac{2\eps}{1+\eps}} \lesssim 1+t$, we obtain
\begin{align*}
\int_{0}^{t}
\| g_s \|^{1-\eps}_{C^{1+\theta}_{[0,\ell]}}\dd s
&\lesssim
(1+t) \sup |g|^{1-\eps}. \qedhere
\end{align*}
\end{proof}

\subsubsection{Step 4: Reduction to rotational-invariant initial data}
\label{ssec:radial}
Here we show that we can reduce, without any loss of generality, to the case of rotationally invariant solutions.
The underlying idea is to ``randomly rotate'' initial data $\theta_0$, in the style of \cite{rowan2023}.

Let us consider the special orthogonal group $SO(d)$ endowed with its Haar measure $\mathfrak{m}$.
Let us recall that, since $SO(d)$ is a compact Lie group, its right- and left-invariant Haar measures coincide, see \cite{VonNeumann35}, end of page 112. 

If $\theta$ is a solution to \eqref{eq:Kraichnan}, then $\bar \theta_t(x):=\theta_t(R x)$, for a given $R \in SO(d)$, solves the SPDE
\begin{align*}
\dd \bar \theta_t
+ 
\circ \dd \bar W_t
\cdot
\nabla\bar\theta_t 
= 
0, 
\quad 
\bar W_t(x):= R^T W_t(R x).
\end{align*}
Since $W$ is isotropic, the ``rotated'' noise $\bar W$ has same covariance as $W$, and so $\bar \theta$ has the same law as the solution to \eqref{eq:Kraichnan} with noise $W$ and initial datum $\bar\theta_0(x) = \theta_0(R x)$. 
In other words, we have $Law( \theta_t \circ R)=Law((\theta_0\circ R)_t)$, where $(\theta_0\circ R)_t$ denotes the unique solution at time $t$ of \eqref{eq:Kraichnan} with initial condition $\theta_0\circ R$. 
Additionally, notice that composition by $R \in SO(d)$ cannot affect regularity or integrability of $\theta$, so that for any $s \geq 0$
\begin{align*}
\EE[\|\theta_t\|_{H^s_x}^2]
= 
\EE[\| \theta_t \circ R\|_{H^s_x}^2]
= 
\EE[\| (\theta_0\circ R)_t \|_{H^s_x}^2],\qquad \forall\, R \in SO(d).
\end{align*}

Given a deterministic $\theta_0$, we now proceed to randomize it by taking the initial datum $\bar\theta_0 = \theta_0\circ R$, where now $R$ is sampled with distribution $\mathfrak{m}$ and independently of $W$. It follows from the above that
\begin{equation}\label{eq:random_data_ok}
\EE[\| \bar\theta_t \|_{H^s_x}^2]
= 
\int_{SO(d)} \EE_W[\| (\theta_0\circ R)_t \|_{H^s_x}^2] \,\mathfrak{m}(\mathrm{d} R)
= 
\EE_W[\| \theta_t\|_{H^s_x}^2],
\end{equation}
so that proving regularity for $\theta_t$ is equivalent to proving it for $\bar\theta_t$. 

On the other hand, the function $\bar F(h)$ associated to $\bar\theta_t$ is $SO(d)$-invariant. 
Indeed, by the above it holds
\begin{align*}
\bar F_t(h)
&= 
\EE\left[\int_{\R^d} \bar\theta_t(h+x)  \bar \theta_t(x) \dd x \right]
\\
&= 
\int_{SO(d)} \EE_W \left[
\int_{\R^d} (\theta_0\circ R)_t(h+x) (\theta_0\circ R)_t(x)  \dd x \right] \mathfrak{m}(\mathrm{d} R)
\\
&= 
\int_{SO(d)} \EE_W\left[
\int_{\R^d} \theta_t(R(h+x)) \theta_t(Rx)\dd x \right] \mathfrak{m}(\mathrm{d} R).
 \end{align*}
Since any $R$ in $SO(d)$ has determinant $1$, change of variables gives
\begin{align*}
\int_{\R^d} \theta_t(R(h+x)) \theta_t(Rx) \dd x
=
\int_{\R^d} \theta_t(Rh+y)) \theta_t(y)\dd y.
\end{align*}
Plugging in the expression above, we get
\begin{align*}
\bar F_t(h)
&= 
\int_{SO(d)}\EE_W\left[ \int_{\R^d} \theta_t(Rh+y)) \theta_t(y)\dd y \right] \mathfrak{m}(\mathrm{d} R)
= 
\int_{SO(d)} F_t(R h) \mathfrak{m}(\mathrm{d} R).
\end{align*}
By the group structure of $SO(d)$ and since $\mathfrak{m}$ is right invariant, it holds for every fixed $R_0 \in SO(d)$
\begin{align*}
\bar F_t(R_0 h) 
= 
\int_{SO(d)} F_t(R R_0 h) \mathfrak{m}(\mathrm{d} R)
%= 
%\int_{SO(d)} F_t(R R_0 h) \dd \mathfrak{m}(R R_0)	
= 
\int_{SO(d)}  F_t(R' h) \mathfrak{m}(\mathrm{d} R')	
= 
\bar F_t(h),
\end{align*}
giving $SO(d)$-invariance of the function $\bar{F}$.

Overall, thanks to \eqref{eq:random_data_ok}, we conclude that given any $\theta_0$ we can randomize it so to enforce that $F(h)$ is rotationally symmetric (i.e. $SO(d)$-invariant), and still conclude the desired estimates for $\EE[\| \theta_t \|_{H^s_x}^2]$ even though $\theta$ itself is not symmetric. So we can assume that $F$ is rotationally symmetric in the PDE estimates. This concludes the proof of \autoref{prop:regularity}. 

\subsection{Fixed-time regularity results}
\label{ssec:consequences.regularization}
In this subsection we present some results on the regularity of $\theta$ at fixed time $t >0$ descending from the anomalous Sobolev regularity \eqref{eq:regularization_intro}, more specifically we prove \eqref{eq:regularization_intro3} and \eqref{eq:regularization_intro4} of \autoref{thm:regularity_Kraichnan}.

Recall that by \eqref{eq:regularization_intro}, the solution $\theta$ of \eqref{eq:Kraichnan} belongs to $L^2_{\omega,t}H^{1-\alpha-}_x$
The next statement guarantees that the equation instantaneously regularizes the initial datum and $\theta$ stays in $L^2_\omega H^{1-\alpha-}_x$ for \emph{every} time $t>0$, implying in particular \eqref{eq:regularization_intro3} of \autoref{thm:regularity_Kraichnan}.
\begin{prop}\label{prop:anomalous_regularity_fixedt}
    Let $\theta_0\in L^2_x$, then for any $\gamma\in (0,1-\alpha)$ and small $\tilde\delta>0$, the associated solution of \eqref{eq:Kraichnan} satisfies
    \begin{equation}\label{eq:anomalous_regularity_pointwise}
        \EE \| \theta_t\|_{H^\gamma_x}^2 \lesssim (1\wedge t)^{-\frac{\gamma}{(1-\alpha)^2}-\tilde\delta} \| \theta_0\|_{L^2_x}^2,\qquad \forall\, t\geq 0.
    \end{equation}
\end{prop}

\begin{proof}
    It suffices to give the proof for $t\in (0,1]$; indeed, in the case $t>1$, we can just run the SPDE up to time $t-1$, invoke \eqref{eq:contraction_L2} and the Markov property, and then use the bound \eqref{eq:anomalous_regularity_pointwise} to pass from $\EE \| \theta_{t-1}\|_{L^2_x}^2$ to $\EE \| \theta_t\|_{H^\gamma_x}^2$.
    
    Fix $\eps>0$ small and set $\sigma:=1-\alpha-\eps$. From \eqref{eq:contraction_L2} and \eqref{eq:regularization_intro2}, the linear operator $\theta_0\mapsto \theta_t$ satisfies
    \begin{align*}
        \EE\|\theta_t\|_{L^2_x}^2 \leq \|\theta_0\|_{L^2_x}^2, \quad
        \EE\|\theta_t\|_{H^\sigma_x}^2 \leq \|\theta_0\|_{H^1_x}^2,\qquad \forall\, t\in (0,1].
    \end{align*}
We can interpolate between these two bounds above, applying e.g. \cite[Thm C.3.3]{HVNVW1}. % on real interpolation of linear operators.
Indeed, using the notation therein, $H^{\sigma}_x$ can be obtained as real interpolation between $L^2_x$ and $H^1_x$, more precisely $H^\sigma_x = (L^2_x,H^1_x)_{\sigma,2}$ (\cite[Thm. 14.4.31]{HVNVW3}); by \cite[Thms. 2.2.10 and C.3.14]{HVNVW1}, $L^2_\omega H^\sigma_x = (L^2_\omega L^2_x,L^2_\omega H^1_x)_{\sigma,2}$ as well.    
    Therefore, by \cite[Thm. C.3.3]{HVNVW1}, for $\sigma_1:=\sigma^2$ it holds that
    \begin{equation}\label{eq:interpolation1}
        \EE\|\theta_t\|_{H^{\sigma_1}_x}^2 \lesssim \|\theta_0\|_{H^{\sigma}_x}^2,\qquad \forall\, t\in (0,1], \ \forall\, \theta_0\in H^\sigma_x,
    \end{equation}
    with implicit constant depending on $\tilde{\delta}$, but independent of $\theta_0$ and $t$.
    On the other hand, for fixed $t\in (0,1]$, by bound \eqref{eq:regularization_intro} and Markov's inequality, for any $\theta_0\in L^2_x$ we can find $u\in [0,t]$ such that
    \begin{equation}\label{eq:markov}
        \EE\|\theta_u\|_{H^{\sigma}_x}^2
        \leq \frac{1}{t} \int_0^t \EE\|\theta_s\|_{H^{\sigma}_x}^2 \dd s
        \lesssim \frac{1}{t} \| \theta_0\|_{L^2_x}^2.
    \end{equation}
    Restarting the dynamics at $u$, using the Markov property, and then applying the bound \eqref{eq:interpolation1} to pass from $u$ to $t$, we conclude that
    \begin{equation}\label{eq:interpolation2}
        \EE\|\theta_t\|_{H^{\sigma_1}_x}^2
        \lesssim \frac{1}{t} \| \theta_0\|_{L^2_x}^2,\qquad \forall\, t\in (0,1],\ \forall\, \theta_0\in L^2_x.
    \end{equation}
    We can now proceed iteratively: interpolating \eqref{eq:interpolation2} with \eqref{eq:regularization_intro2}, it holds
    \begin{equation*}
        \EE\|\theta_t\|_{H^{\sigma_2}_x}^2 \lesssim t^{-(1-\sigma)} \, \|\theta_0\|_{H^{\sigma}_x}^2, 
        \quad
        \sigma_2:=\sigma^2 + (1-\sigma)\sigma_1.
    \end{equation*}
    In turn, this bound can be concatenated with \eqref{eq:markov}, and using the Markov property of $\theta$ one gets $\EE\|\theta_t\|_{H^{\sigma_2}_x}^2 \lesssim t^{-1-(1-\sigma)} \, \|\theta_0\|_{L^2_x}^2$, with implicit constant independent of $\theta_0$ and $t$.
    Inductively, it's not hard to see that for any $n\in\NN$ one obtains the estimate
    \begin{equation*} 
        \EE\|\theta_t\|_{H^{\sigma_n}_x}^2
        \lesssim t^{-\sum_{i=0}^{n-1}(1-\sigma)^i} \| \theta_0\|_{L^2_x}^2,\qquad \forall\, t\in (0,1],\ \forall\, \theta_0\in L^2_x,
    \end{equation*}
    where $\sigma_n:=\sigma^2 \sum_{i=0}^{n-1}(1-\sigma)^i$, and with an implicit constant that does not depend upon $\theta_0$, $t$, nor $n$. Noting that $\sigma_n\to \sigma$ monotonically as $n\to\infty$, we conclude that for any $\gamma\in (0,\sigma)$ there exists $n=n_\gamma$ sufficiently large such that 
    \begin{equation*}
        \EE\|\theta_t\|_{H^{\gamma}_x}^2
        \lesssim 
        \EE\|\theta_t\|_{H^{\sigma_n}_x}^2
        \lesssim t^{-\sum_{i=0}^{n-1}(1-\sigma)^i} \| \theta_0\|_{L^2_x}^2
        \lesssim 
        t^{-\frac{\gamma}{\sigma^2}} \| \theta_0\|_{L^2_x}^2,\qquad \forall\, t\in (0,1],\ \forall\, \theta_0\in L^2_x.
    \end{equation*}
    The conclusion now follows by noting that we can have $\sigma$ get arbitrarily close to $1-\alpha$ upon choosing $\eps>0$ small enough.
\end{proof}

The next result is based on similar techniques as the previous one, with the only difference that we now focus on anomalous integrability gain rather than anomalous regularity gain. More specifically, solutions starting from $\theta_0 \in L^2_x$ become $L^\infty_x$-valued at any positive time $t>0$, giving \eqref{eq:regularization_intro4} of \autoref{thm:regularity_Kraichnan}.
This is clearly ``anomalous'' as a smooth transport cannot improve integrability of initial data, e.g. in the smooth incompressible case where it just rearranges the initial condition in a measure-preserving way.

\begin{prop} \label{prop:Linfty_theta}
	Let $W$ satisfy Assumptions \ref{ass:well_posed}-\ref{assumption} and $\delta>0$ small. Then for any $\theta_0\in L^2_x$ and any $q\in (2,\infty]$, the unique solution to \eqref{eq:Kraichnan} satisfies
	\begin{equation}\label{eq:Linfty_theta}
		\EE \| \theta_t\|_{L^q_x}^2 \lesssim \big(1\wedge t^{-\frac{d}{2(1-\alpha)} \frac{q}{q-2}-\delta}\big) \| \theta_0\|_{L^2_x}^2.
	\end{equation}
\end{prop}

\begin{proof}
It suffices to prove \eqref{eq:Linfty_theta} in the case $q=\infty$ (with the convention that $\frac{\infty}{\infty}=1$). The general case then follows by combining this bound with \eqref{eq:contraction_L2} and interpolation arguments, invoking the results from \cite{HVNVW1,HVNVW3} similarly to \autoref{prop:anomalous_regularity_fixedt}. Moreover, by Markovianity and \eqref{eq:contraction_L2}, we may assume $t\leq 1$.

Fix $\tilde\delta>0$ small and set $\sigma:=1-\alpha-\tilde\delta$; let $p_1\in (2,\infty)$ such that $1/p_1=1/2-\sigma/d$. Arguing as in \eqref{eq:markov}, using Sobolev embeddings, Markovianity and \eqref{eq:contraction_Lp}, it holds
\begin{equation}\label{eq:markov2}
	\EE\| \theta_t\|_{L^{p_1}_x}^2
	\leq \EE\Big[\| \theta_t\|_{L^{p_1}_x}^{p_1}\Big]^{\frac{2}{p_1}}
	\leq \EE\| \theta_u\|_{L^{p_1}_x}^2
	\lesssim \EE\| \theta_u\|_{H^\sigma_x}^2
	\lesssim \frac{1}{t} \| \theta_0\|_{L^2_x}^2.
\end{equation}
This provides a bound for the linear map $\theta_0\mapsto \theta_t$ from $L^2_x$ to $L^2_\omega L^{p_1}_x$. On the other hand, by \eqref{eq:contraction_Linfty} and Jensen's inequality, the same linear map is bounded (with constant $1$) from $L^\infty_x$ to $L^2_\omega L^\infty_x$. 
By \cite[Theorems 2.2.10 and C.3.14]{HVNVW1}, for $p \in (2,\infty)$ the space $L^p_x$ can be obtained as real interpolation between $L^2_x$ and $L^\infty_x$ with parameters
$L^p_x = (L^2_x, L^\infty_x)_{1-\frac{2}{p},p}$. On the other hand,
\begin{align*}
    L^2_\omega L^q_x = (L^2_\omega L^{p_1}_x, L^2_\omega L^\infty_x)_{1-\frac{2}{p},q}
    \qquad
    \mbox{for }
    q := \frac{p p_1}{2}.
\end{align*}
Since $q>p$, by \cite[Proposition C.3.2 and Theorem C.3.3]{HVNVW1} we deduce that the map $\theta_0\mapsto \theta_t$ is bounded from $L^p_x$ to $L^2_\omega L^q_x$, with
\begin{equation*}
	\EE \| \theta_t\|_{L^q_x}^2 \leq C^{\frac{2}{p}} t^{-{\frac{2}{p}}} \| \theta_0\|_{L^p_x}^2.
\end{equation*}
Here $C$ is the constant coming from inequality \eqref{eq:markov2} and we may assume $C\geq 1$ if needed.
Now for fixed $t>0$, consider the increasing sequence $\{t_n\}_{n \in \NN}$ defined by $t_0=0$ and $t_n=t_{n-1}+2^{-n} t$ for $n\geq 1$; set $p_n := p_1^n / 2^{n-1}$. Using the Markov property and concatenating the previous estimates to pass from $\EE\| \theta_{t_n}\|_{L^{p_n}_x}$ to $\EE\| \theta_{t_{n+1}}\|_{L^{p_{n+1}}_x}$, we obtain
\begin{align*}
	\sup_{n\in\NN} \EE\| \theta_t\|_{L^{p_n}_x}^2
	& \leq \| \theta_0\|_{L^2_x}^2 \prod_{k=0}^\infty (C (t_{k+1}-t_k)^{-1})^{\frac{2}{p_n}}\\
	&= \| \theta_0\|_{L^2_x}^2 C^{\sum_{k=0}^\infty (\frac{2}{p_1})^k} t^{- \sum_{k=0}^\infty (\frac{2}{p_1})^k} 2^{\sum_{k=0}^\infty (k+1)(\frac{2}{p_1})^k}\\
	& \lesssim t^{-\frac{1}{1-2/p_1}} \| \theta_0\|_{L^2_x}^2
	= t^{-\frac{d}{2\sigma}} \| \theta_0\|_{L^2_x}^2.
\end{align*}
By Fatou's lemma and \cite[Lem. A.2]{GalLuo2023} we deduce that
\begin{align*}
	\EE \|\theta_t\|_{L^\infty_x}^2
	\leq \liminf_{n\to\infty}\EE \| \theta_t\|_{L^{p_n}_x}^2
	\lesssim t^{-\frac{d}{2\sigma}} \| \theta_0\|_{L^2_x}^2.
\end{align*}
The conclusion then follows by the arbitrariness of $\tilde\delta>0$.
\end{proof}

\begin{proof}[Proof of \autoref{thm:regularity_Kraichnan}, bounds \eqref{eq:regularization_intro3} and \eqref{eq:regularization_intro4}]
     \eqref{eq:regularization_intro3} follows immediately from \autoref{prop:anomalous_regularity_fixedt}, upon taking $\gamma=1-\alpha-\delta$ and tuning the small parameter $\tilde\delta>0$ accordingly. \eqref{eq:regularization_intro4} corresponds to \eqref{eq:Linfty_theta} with $q=\infty$.
\end{proof}

%{\color{red} Many thanks!}

\subsection{Regularity theory for a degenerate parabolic PDE}
\label{sec:regularity_PDE}

Formally passing to the limit in \eqref{eq:PDE_F}, we obtain the \textit{degenerate} parabolic equation in non-divergence form
\begin{align} \label{eq:PDE_F_deg}
\partial_t F = Q : D_z^2 F.
\end{align}
Making use on the strong link between the degenerate PDE \eqref{eq:PDE_F_deg} and the SPDE \eqref{eq:Kraichnan}, we identify a relevant \emph{selection} criterion that additionally allows to show $C^{2-2\alpha-}_x$ H\"older regularity for $F$, which we believe is of independent interest. 
The regularity is sharp by \eqref{eq:sharpness_intro} and \autoref{lem:Sobolev-Holder}.
The criterion is based on a Fourier splitting of $F$ in four terms, each being a zero diffusivity limit of two-point self-correlation functions $F^{\kappa,j}$ associated to solutions to \eqref{eq:spde_viscous_approx}, for $j=1,\dots,4$. 

Before stating the main result of this section, recall that equation \eqref{eq:PDE_F} with $\kappa \in (0,1/2)$ is a strictly parabolic PDE with H\"older coefficients: by linearity and \cite[Corollary 8.3.1]{Kr08}, it enjoys uniqueness of bounded solutions.

\begin{theorem} \label{thm:reg_F_intro}
Suppose \autoref{ass:well_posed} and \autoref{assumption}. Let $F_0 \in \mathcal{F}L^1(\R^d) := \{ F \in C_b(\R^d) : \hat{F} \in L^1(\R^d) \}$ and let $F^\kappa$ be the unique bounded solution of \eqref{eq:PDE_F} with initial condition $F_0$. 
Then:
\begin{enumerate}
    \item 
    For every $\delta>0$ small enough and $t > 0$ we have
\begin{align*}
\sup_{\kappa \in (0,1/2)}
\left(
\sup_{s \leq t} \| F^\kappa_s \|_{\mathcal{F}L^1}
+
\int_0^t \|F^\kappa_s\|_{C^{2-2\alpha-\delta}_x} \dd s
\right)
\lesssim
(1+t) \|F_0\|_{\mathcal{F}L^1}.
\end{align*}
If in addition $D^2_z F_0 \in \mathcal{F}L^1(\R^d)$ then
\begin{align*}
\sup_{\kappa \in (0,1/2)}
\sup_{s \leq t}
\|F^\kappa_s\|_{C^{2-2\alpha-\delta}_x}
\lesssim
(1+t) \left( \|F_0\|_{\mathcal{F}L^1}+\|D^2_z F_0\|_{\mathcal{F}L^1} \right).
\end{align*}
\item 
For every $t>0$ the functions $F^\kappa_t$ converge in $\mathcal{F}L^1(\R^d)$ as $\kappa \downarrow 0$ towards 
a unique limit $F$, that is a solution of \eqref{eq:PDE_F_deg} with initial condition $F_0$, and retains the same regularity as $F^\kappa$:
\begin{align*}
\sup_{s \leq t} \| F_s \|_{\mathcal{F}L^1}
+
\int_0^t \|F_s\|_{C^{2-2\alpha-\delta}_x} \dd s
\lesssim
(1+t) \|F_0\|_{\mathcal{F}L^1},
\end{align*}
and if $D^2_z F_0 \in \mathcal{F}L^1(\R^d)$
\begin{align*}
\sup_{s \leq t}
\|F_s\|_{C^{2-2\alpha-\delta}_x}
\lesssim
(1+t) \left( \|F_0\|_{\mathcal{F}L^1}+\|D^2_z F_0\|_{\mathcal{F}L^1} \right).
\end{align*}

\item
The regularity is fundamentally sharp, in the sense that there exist arbitrarily regular initial conditions  (e.g. with $\hat{F}_0 \in C^\infty_c(\R^d)$) such that for every $s<t$ and $\delta>0$
\begin{align*}
\int_s^t \sup_{|z| \in (0,l)}
\frac{|F_r(0)-F_r(z)|}{|z|^{2-2\alpha+\delta}}\dd r
= \infty.
\end{align*}
\end{enumerate}
\end{theorem}
In the statement above $\mathcal{F}L^1(\R^d)$ denotes the space of functions with integrable Fourier transform, endowed with the norm $\| F \|_{\mathcal{F}L^1}:=\| \hat{F} \|_{L^1}$.
This is the relevant functional space to equation \eqref{eq:PDE_F_intro_tilde}: Indeed, recall the formula \eqref{eq:hat_xi=L2} relating the Fourier transforms of solutions to \eqref{eq:spde_viscous_approx} and their associated two point self-correlation functions: in this particular case,
\begin{align*}
\hat{F}^\kappa_t(\xi)
=
\mathbb{E}
[|\mathcal{F}(\tilde{\theta}^\kappa_t)(\xi)|^2],
\quad
\xi \in \R^d,
\end{align*}
which is integrable by Parseval identity.

The basic idea behind the proof of \autoref{thm:reg_F_intro} is to ``reverse engineer'' the proof of \autoref{thm:regularity_Kraichnan}. Namely, given a solution $F^\kappa$ of the PDE \eqref{eq:PDE_F}, not necessarily given by a self-correlation function, we split $F^\kappa= F^{\kappa,1}-F^{\kappa,2}+iF^{\kappa,3}-iF^{\kappa,4}$ and find $\tilde{\theta}^{\kappa,j}$ solving \eqref{eq:spde_viscous_approx} for $j=1,\dots,4$ and such that \eqref{eq:two-point_correlation} holds, then using the regularity of $\tilde{\theta}^{\kappa,j}$ to deduce the desired regularity of each $F^{\kappa,j}$.

\begin{proof}[Proof of \autoref{thm:reg_F_intro}] 
$(1)$
We preliminarily show that, in fact, we can restrict ourselves to prove the thesis under the additional condition that $\hat{F}_0$ is real-valued and $\hat{F}_0 \geq 0$. To see this, rewrite 
\begin{align*}
    \hat{F}_0 := \hat{F}_0^1 - \hat{F}_0^2 +i\hat{F}_0^3 - i\hat{F}_0^4,
\end{align*}
where each $\hat{F}^j_0$ is real-valued, non-negative, and explicitly given by:
\begin{align*}
    \hat{F}_0^1 &:= Re(\hat{F}_0) \mathbf{1}_{\{Re(\hat{F}_0) \geq 0\}}, \qquad 
    \hat{F}_0^2 := -Re(\hat{F}_0) \mathbf{1}_{\{Re(\hat{F}_0) < 0\}}, 
    \\
    \hat{F}_0^3 &:= Im(\hat{F}_0) \mathbf{1}_{\{Im(\hat{F}_0) \geq 0\}}, \qquad
    \hat{F}_0^4 := -Im(\hat{F}_0) \mathbf{1}_{\{Im(\hat{F}_0) < 0\}}. 
\end{align*}

Notice that the PDE \eqref{eq:PDE_F} is linear, with coefficients that are symmetric with respect to the origin and real-valued.
In particular, if $F^\kappa$ is the unique bounded solution of \eqref{eq:PDE_F} with initial condition $F_0$, then both
\begin{align*}
    F^{\kappa,+}(t,z) := \frac{F^\kappa(t,z)+F^\kappa(t,-z)}{2},
    \quad
    F^{\kappa,-}(t,z) := \frac{F^\kappa(t,z)-F^\kappa(t,-z)}{2},
\end{align*}
are the unique bounded solutions of the same equation, with initial condition respectively $F^+_0$ and $F^-_0$, and satisfy $\hat{F}^\kappa = \hat{F}^{\kappa,+} + \hat{F}^{\kappa,-}$.  
By symmetry, $\hat{F}^{\kappa,+}$ is purely real-valued and symmetric, and $\hat{F}^{\kappa,-}$ is purely imaginary-valued and antisymmetric, meaning that $Re(\hat{F}^\kappa)=\hat{F}^{\kappa,+}$ and  $Im(\hat{F}^\kappa)=-i\hat{F}^{\kappa,-}$ by uniqueness of solutions.
Moreover, \eqref{eq:PDE_F} is positivity-preserving on Fourier transforms in the following sense: if $\hat{F}_0^j$ is real-valued and non-negative, then
one can define $\theta_0^j := \mathcal{F}^{-1}((\hat{F}_0^j)^{1/2}) \in L^2(\R^d;\CC)$ and run the SPDE \eqref{eq:spde_viscous_approx} from initial condition $\theta_0^j$, using $\CC$-linearity of the equation, obtaining a unique solution $\tilde{\theta}^{\kappa,j}$;
then one can verify that \autoref{lem:PDE_twopoint_transport}, \eqref{eq:two-point_correlation}, and \eqref{eq:hat_xi=L2} still hold true for $\tilde{\theta}^{\kappa,j}$. Therefore, the unique bounded solution of \eqref{eq:PDE_F} with initial condition $F_0^j$ satisfies 
\begin{align*}
\hat{F}^{\kappa,j}_t 
=
\mathbb{E}
[|\mathcal{F}(\tilde{\theta}^{\kappa,j}_t) |^2] \geq 0,
\quad
\mbox{where } \theta_0^j := \mathcal{F}^{-1}((\hat{F}_0^j)^{1/2})
\end{align*}
for every $t \geq 0$. In particular, this means that (modulo taking the Fourier transform)
\begin{align*}
    \hat{F}^{\kappa,1} &:= \hat{F}^{\kappa,+} \mathbf{1}_{\{\hat{F}^{\kappa,+} \geq 0\}},\qquad \qquad 
    \hat{F}^{\kappa,2} := -\hat{F}^{\kappa,+} \mathbf{1}_{\{\hat{F}^{\kappa,+}< 0\}},
    \\
    \hat{F}^{\kappa,3} &:= -i\hat{F}^{\kappa,-} \mathbf{1}_{\{-i\hat{F}^{\kappa,-} \geq 0\}},\qquad 
    \hat{F}^{\kappa,4} := i\hat{F}^{\kappa,-} \mathbf{1}_{\{-i\hat{F}^{\kappa,-} < 0\}},
\end{align*}
are the unique bounded solutions of \eqref{eq:PDE_F} with initial conditions ${F}_0^j$, respectively for $j=1,\dots,4$.
Moreover, the decomposition
\begin{align*}
    F^{\kappa,+} = F^{\kappa,1}-F^{\kappa,2},
    \qquad
    F^{\kappa,-} = iF^{\kappa,3}-iF^{\kappa,4}, 
\end{align*}
holds for all times $t \geq 0$, giving in particular $F^{\kappa} = F^{\kappa,1}-F^{\kappa,2} +iF^{\kappa,3}-iF^{\kappa,4}$. 

Thus, in order to prove the desired regularity for $F^\kappa$, it suffices to show it holds for each $F^{\kappa,j}$, where now the advantage is that $\hat{F}_0^j$ is real-valued and non-negative for every $j$.
To do this, we proceed as described above: let us define $\theta_0^j \in L^2(\R^d;\CC)$ given by $\theta_0^j = \mathcal{F}^{-1}((\hat{F}_0^j)^{1/2})$ and run the SPDE \eqref{eq:spde_viscous_approx} from the initial condition $\theta_0^j$; then the desired estimates follow from \autoref{lem:Sobolev-Holder} and the uniform-in-$\kappa$ Sobolev regularity of $\tilde{\theta}^{\kappa,j}$ proved before in \autoref{thm:regularity_Kraichnan}. 

$(2)$
The convergence of ${F}^{\kappa,j}_t$ in $\mathcal{F}L^1$ for every $j=1,\dots,4$, and the identification of the limit, both follow from the strong convergence $\tilde{\theta}^{\kappa,j}_t \to \theta^j_t$ in $L^2_{\omega,x}$. 
Indeed by construction
\begin{align*}
\hat{F}^{\kappa,j}_t 
=
\mathbb{E}
[|\mathcal{F}(\tilde{\theta}^{\kappa,j}_t) |^2]
\to
\mathbb{E}
[|\mathcal{F}(\theta^j_t) |^2]
=:
\hat{F}^j_t,
\quad \mbox{ in } L^1 (\R^d),
\end{align*}
and $F^j$ solves \eqref{eq:PDE_F_deg} by \autoref{lem:PDE_twopoint_transport} applied with $\kappa = 0$. Therefore, so does $F = F^1-F^2+iF^3-iF^4$.
Moreover, by \autoref{thm:regularity_Kraichnan} we know that $\theta^j$ is Sobolev regular in space, either $L^1_t$ or $L^\infty_t$ depending upon the space regularity of $F_0^j$. By \autoref{lem:Sobolev-Holder}, each $F^j$ has the desired regularity, and so does $F$. 

$(3)$
To see that the regularity of point $(2)$ is sharp, take any $\hat{F}_0 \in C^\infty_c(\R^d)$, $\hat{F}_0 \geq 0$, $F_0 \neq 0$. Then the associated solution $\theta$ of the SPDE, constructed as above, must satisfy
\begin{align*}
\int_s^t \mathbb{E}[\| \theta_r\|_{H^{1-\alpha}_x}^2] \dd r  = \infty
\end{align*}
by \autoref{thm:regularity_Kraichnan}.
Thus we have by \autoref{lem:Sobolev-Holder} (with $\alpha-\delta$ replacing $\alpha$)
\begin{equation*}
\int_s^t \sup_{|z| \in (0,l)}
\frac{|F_r(0)-F_r(z)|}{|z|^{2-2\alpha+\delta}}\dd r
\gtrsim
\int_s^t \mathbb{E} \| \theta_r \|^2_{H^{1-\alpha}_x}\dd r
-
(t-s) \|\theta_0\|_{L^2_x}^2
= \infty. \qedhere
\end{equation*}
\end{proof}

We have seen above that vanishing-viscosity limits of solutions to the parabolic PDE \eqref{eq:PDE_F} satisfy sharp $C^{2-2\alpha-}_x$ regularity in space, either $L^1_t$ or $L^\infty_t$ with respect to time depending on the regularity of the initial condition.
Next, we focus on time regularity, in particular we are interested in obtaining H\"older estimates in time.
One could interpret the following results as an explicit, quantitative version of the locally Holder regularity result obtained in \cite{ChSe87} by Chiarenza and Serapioni.

For notationally simplicity, the estimates below are obtained directly for $\kappa=0$, but it is easy to check that they are all uniform with respect to the parameter $\kappa \in [0,1/2)$.
Let us preliminarly recall that, if $\theta_0\in H^1_x$, then $\partial_t F$ also solves the PDE with an initial datum in $C^0_b$, and therefore it enjoys $\partial_t F\in L^\infty_{t,x}\cap L^1_t C^{2-2\alpha-}_x$. It follows upon integrating in time that, for any $\gamma<1-\alpha$, we have
\begin{equation}\label{eq:space-time_reg_H1}
	|F_t(x)-F_s(y)|\lesssim (|t-s|+s|x-y|^{1\wedge 2\gamma} )\| \theta_0\|_{H^1_x}^2.
\end{equation}

\begin{lemma}
	Let $W$ satisfy Assumptions \ref{ass:well_posed}-\ref{assumption}, $\theta_0\in H^1_x$. Then for any $\gamma < 1-\alpha$ it holds
	\begin{equation}\label{eq:time_regularity_SPDE1}
		\EE \| \theta_t-\theta_s\|_{L^2_x}^2 \lesssim |t-s|^\gamma (1+s) \| \theta_0\|_{H^1_x}^2, \quad \forall\, s\leq t.
	\end{equation}
\end{lemma}

\begin{proof}
	%\blue{L: Some details still to be polished.}
	It suffices to consider $|t-s|\leq 1$, otherwise by the triangular inequality  with \eqref{eq:contraction_L2}, it holds
	\begin{align*}
		\EE \| \theta_t-\theta_s\|_{L^2_x}^2
		&= \EE \| \theta_t\|^2 -\EE\|\theta_s\|_{L^2_x}^2-2\EE\langle \theta_t-\theta_s,\theta_s\rangle
		\\&= [F_t(0)-F_s(0)]-2\EE\langle \theta_t-\theta_s,\theta_s\rangle
		=: I_1 + I_2.
	\end{align*}
	By \eqref{eq:space-time_reg_H1} we have $I_1 = |F_t(0)-F_s(0)|\lesssim |t-s| \| \theta_0\|_{H^1_x}^2$, so we only have to control the term $I_2$.
    By \autoref{prop:properties.inviscid.Kraichnan}(1),
    %the arguments in the proof of point (1) of (see \autoref{app:preliminaries}),
    we have $\EE[\theta_t|\cF_s]=P_{t-s} \theta_s$  and therefore
	\begin{align*}
		|\EE\langle \theta_t-\theta_s,\theta_s\rangle|
		= |\EE\langle P_{t-s}\theta_s-\theta_s,\theta_s\rangle|
		\leq \EE[ \|P_{t-s}\theta_s-\theta_s\|_{H^{-\gamma}_x}\|\theta_s\|_{H^\gamma_x}].
	\end{align*}
	Since the matrix $C(0)$ is non-degenerate by assumption, usual properties of heat semigroup give
	\begin{align*} 
        \|P_{t-s}\theta_s-\theta_s\|_{H^{-\gamma}_x}
		\lesssim |t-s|^{\gamma} \| \theta_s\|_{H^\gamma_x},
        %?}
	\end{align*}
	which combined with the above and \eqref{eq:regularization_intro2} yields
	\begin{equation*}
		|I_2| \lesssim |t-s|^\gamma \EE \| \theta_s\|_{H^\gamma_x}^2
		\lesssim |t-s|^\gamma(1+s) \| \theta_0\|_{H^1_x}^2.\qedhere
	\end{equation*}
\end{proof}

\begin{cor}
	Let $W$ satisfy Assumptions \ref{ass:well_posed}-\ref{assumption}. Then for any $\beta\in (0,1)$ and $\theta_0\in H^\beta_x$, it holds
	\begin{equation}\label{eq:time_regularity_SPDE2}
		\EE \| \theta_t-\theta_s\|_{L^2_x}^2 \lesssim |t-s|^{\beta\gamma} (1+s)^\beta \| \theta_0\|_{H^\beta_x}^2\quad\forall\,
		s\leq t.
	\end{equation}
	Moreover, for any $\theta_0\in L^2_x$ and any $u\in (0,1)$ and $\delta>0$ fixed, we have
	\begin{equation}\label{eq:time_regularity_SPDE3}
		\EE \| \theta_t-\theta_s\|_{L^2_x}^2 \lesssim u^{-\frac{\gamma}{(1-\alpha)^2}-\delta} |t-s|^{\gamma^2} (1+s)^\gamma\| \theta_0\|_{L^2_x}^2 \quad \forall s,t\in [u,\infty).
	\end{equation}
\end{cor}

\begin{proof}
	As the map $\theta_0\mapsto (\theta_t-\theta_s)$ is linear, the bound \eqref{eq:time_regularity_SPDE2} follows by interpolation, using the bound \eqref{eq:time_regularity_SPDE1} and the trivial estimate $\EE\|\theta_t-\theta_s\|_{L^2}^2 \lesssim \| \theta_0\|_{L^2_x}^2$ coming from \eqref{eq:contraction_L2}.
	Concerning \eqref{eq:time_regularity_SPDE3}, without loss of generality we may assume $|t-s|\leq 1$ and $s \leq t$.
	Estimate \eqref{eq:time_regularity_SPDE3} then follows by Markovianity and concatenating \eqref{eq:time_regularity_SPDE2} with \eqref{eq:anomalous_regularity_pointwise}, for the choice $\beta=\gamma$.
\end{proof}

\begin{cor}\label{cor:reg_F_improved}
	Let $W$ satisfy Assumptions \ref{ass:well_posed}-\ref{assumption}; let $\theta_0\in L^2_x$ and let $F$ be the associated two-point self-correlation function. Then for any $u\in (0,1)$ it holds that
	\begin{align*}
		|F_t(x)-F_s(y)| \lesssim u^{-\frac{\gamma}{(1-\alpha)^2}-\delta}(|t-s|^{\frac{\gamma^2}{2}}(1+s)^\gamma + |x-y|^{1\wedge 2\gamma}) \| F_0\|_{C^0_b} \quad
		\forall\, s,t\in [u,\infty],\, \forall\, x,y\in\R^d.
	\end{align*}
\end{cor}

\begin{proof}
    Let us decompose
    \begin{align*}
        |F_t(x)-F_s(y)| \leq |F_t(x)-F_s(x)| + |F_s(x)-F_s(y)| =: I_1+ I_2.
    \end{align*}
    By Young's convolution inequality and invariance of the $L^2_x$ norm with respect of the reflection map $\theta \mapsto\theta^-$, and recalling \eqref{eq:time_regularity_SPDE3} and \eqref{eq:selfcorrelation.sup.bound}, it holds
	\begin{align*}
		I_1 \leq \|F_t-F_s\|_{C^0_b}
		&\leq \|\EE[\theta_t\ast\theta^-_t-\theta_s\ast\theta^-_s]\|_{C^0_b}
		\\
        &\leq \EE[\| \theta_t-\theta_s\|_{L^2_x} \| \theta_t\|_{L^2_x}+\| \theta_t-\theta_s\|_{L^2_x} \| \theta_s\|_{L^2_x}]
		\\
        &\lesssim \| \theta_0\|_{L^2_x} \EE[\| \theta_t-\theta_s\|_{L^2_x}^2]^{1/2}
        \\
        &\lesssim
        u^{-\frac{\gamma}{2(1-\alpha)^2}-\frac{\delta}{2}} |t-s|^{\frac{\gamma^2}{2}} (1+s)^\frac{\gamma}{2}\| F_0\|_{C^0_b}.
	\end{align*}
As for the term $I_2$, we use \autoref{lem:Sobolev-Holder} and the space regularity coming from \eqref{eq:anomalous_regularity_pointwise} to deduce
\begin{align*}
I_2 \leq |x-y|^{1\wedge2\gamma} \| F_s\|_{C^{2\gamma}_x} 
&\lesssim
|x-y|^{1\wedge2\gamma} \mathbb{E}\| \theta_s\|_{H^{\gamma}_x}^2 \lesssim
u^{-\frac{\gamma}{(1-\alpha)^2}-\delta} |x-y|^{1\wedge2\gamma}  \| F_0\|_{C^0_b}. \qedhere
\end{align*}
\end{proof}

\section{Structure of the dissipation measure in the incompressible Kraichnan model} \label{section:dissipation_measure}

Recall the dissipation measure $\mathcal{D}[\theta]$ from \autoref{thm:energy_balance_inviscid}, which governs the anomalous loss of kinetic energy of solutions to the inviscid transport SPDE as part of the local energy balance \eqref{eq:energy_balance_inviscid}.
In this section, we obtain an exact formula for $\EE \mathcal{D}[\theta]$ in terms of increments of the solution $\theta$; as a consequence, boundedness and non-triviality of $\EE \mathcal{D}[\theta]$ are directly linked to the sharp regularity of solutions $\theta$ in Besov-type spaces $\tilde L^2_{\omega,t} \tilde B^{1-\alpha}_{2,\infty}$.
In fact, equation \eqref{eq:exact.formula} below can be regarded as an analogue of known formulas for the Duchon--Robert dissipation measure \cite{DuRo00}, in the context of exact laws of turbulence in hydrodynamics, see \cite{Novack2024}.

Compared to the previous sections, the forthcoming \autoref{thm:exact.formula.dissipation.measure} requires more stringent assumptions on the noise $W$. Throughout this section, we will always assume $W$ to be of Kraichnan type and divergence-free (so that $\eta=1$). For simplicity of exposition, we will also set the infra-red cut-off $m=1$, although this is not crucial.
In this case, by equation \eqref{eq:covariance_generalQ} and \autoref{lem:asymptotics_kraichnan} from \autoref{app:auxiliary}, the asymptotics of $Q$ around $0$ are given by
\begin{equation}\label{eq:covariance_recap}
    Q(z)= c |z|^{2\alpha} \Big[ P^\parallel_z + \Big(1+\frac{2\alpha}{d-1} \Big) P^\perp_z\Big] + O(|z|^2),
    \quad \text{for }|z|\ll 1,
\end{equation}
for a suitable constant $c=c(\alpha,d)$. We adopted the notation $P^\parallel_z=\hat z\otimes \hat z$ and $P^\perp_z=I_d-P^\parallel_z$.

The following result summarizes the main findings of this section. It particularly implies \autoref{thm:dissipation_measure_intro} and \eqref{eq:regularization_besov_intro} of \autoref{thm:regularity_Kraichnan} from the Introduction.

\begin{theorem}\label{thm:exact.formula.dissipation.measure}
Let $\alpha\in (0,1)$ and consider the associated incompressible Kraichnan noise $W$ ($\eta=1$, $m=1$). Let $\theta_0\in L^2_x$, $\theta$ be the unique solution to \eqref{eq:spde_inviscid_approx}, and $\mathcal{D}[\theta]$ be the associated dissipation measure. Then
\begin{equation}\label{eq:exact.formula}
    \mathbb{E} \mathcal{D}[\theta](\dif t, \dif x)
    = c\, d\, (1-\alpha)\, \lim_{\eps \rightarrow 0} 
\mathbb{E}h_t(\eps, x)\dd t \dd x
\end{equation}
in the sense of space-time distributions, where $c$ is as in \eqref{eq:covariance_recap} and
\begin{equation*}
%\label{eq:defn.h}
    h_t(\eps, x):= \oint_{\mathbb{S}^{d - 1}} \frac{| \delta_{\eps \hat{y}} \theta_t (x)|^2}{\eps^{2 - 2 \alpha}} \sigma ( \mathrm{d} \hat{y}).
\end{equation*}
Moreover 
\begin{equation}\label{eq:sharp_regularity}
    \lim_{\eps\to 0} \frac{1}{\eps^{2(1-\alpha)}} \oint_{\mathbb{S}^{d-1}} \| \delta_{\eps \hat{z}} \theta \|_{L^2_{\omega,t,x}}^2 \sigma ( \mathrm{d} \hat{z})
    = \frac{1}{c\, d\, (1-\alpha)} (\| \theta_0\|_{L^2_x}^2 - \EE[| \theta_T\|_{L^2_x}^2])>0
\end{equation}
\end{theorem}

\begin{remark}
    In light of \eqref{eq:sharp_regularity} and \autoref{lem:besov_type_spaces}, $\theta$ belongs sharply to $\tilde{L}^2_{\omega,t} \tilde B^{1-\alpha}_{2,\infty}$, in the sense that it does not belong to the closure of smooth functions in this space and therefore cannot enjoy higher space regularity. By \autoref{lem:refined.Besov.Holder}, a similar consideration applies to its two-point self-correlation $F[\theta]$, which belongs sharply to $\tilde L^1_t \dot B^{2-2\alpha}_{\infty,\infty}$. Note however that, while space regularity is sharp, time integrability can be improved: by \eqref{eq:regularization_intro2} we have $\sup_{t\in [0,T]} \EE[\| \theta_t\|_{H^{2-2\alpha-\delta}_x}^2]<\infty$ and $F[\theta]\in\tilde L^\infty_t \dot B^{2-2\alpha-2\delta}_{\infty,\infty}$ for each small $\delta>0$, whenever $\theta_0\in H^1_x$. 
\end{remark}

The proof of \autoref{thm:exact.formula.dissipation.measure} is reminiscent of the ones from \cite{Novack2024}; it also has its own specific features, like the use of \autoref{lem:identification_limits} below. We divide the proof into a few substeps.

\begin{prop}\label{prop:alternative_dissipation_measure}
    Let $W$, $\theta_0$, $\theta$, $\cD[\theta]$ be as in \autoref{thm:energy_balance_inviscid}; let $\chi$ be a compactly supported probability density such that $\chi$, $\nabla\chi\in L^\infty_x\cap \mathrm{BV}$ and let $\{\chi^\eps\}_\eps$ denote the associated mollifiers. Define
    \begin{equation}\label{eq:defn_Deps}
        \cD^\eps(t,x):=\frac{1}{2} \int_{\R^d} D^2\chi^\eps(y):Q(y) |\delta_y \theta(x)|^2 \dd y.
    \end{equation}
    Then $\EE\cD^\eps(t,x) \dd t \dd x\to \EE\cD[\theta](\dif t, \dif x)$ as $\eps\to 0^+$ in the sense of space-time distributions.
\end{prop}

\begin{proof}
    We divide the proof in several steps.

    \emph{Step 1.} First assume $\chi$ to be smooth. As in the proof of \autoref{prop:energy_balance_viscous_kappa}, $\theta^{\eps} := \theta \ast \chi^{\eps}$ satisfies
    \[ \dd \theta^{\eps}_t  + \sum_k \nabla\cdot (\sigma_k  \theta_t)^{\eps} \dd W^k_t
    = \frac{1}{2} C (0) : D^2 \theta^{\eps}_t\]
    where we used the fact that $\nabla\cdot \sigma_k=0$ since $W$ is divergence free.
    Applying It\^{o} Formula to $\theta \theta^{\varepsilon}$ (e.g. as an $H^{-1}_x$-valued semimartingale), we find
    \begin{equation} \label{eq:LEB_eps} \begin{split}
    \dd(\theta_t \theta^{\varepsilon}_t) 
    & = \frac{1}{2} C (0) : D^2 (\theta_t \theta^{\varepsilon}_t)\dd t 
    - C (0) \nabla \theta_t \cdot \nabla \theta^{\varepsilon}_t\dd t
    + \sum_k \nabla\cdot (\sigma_k \theta_t)^{\varepsilon} \nabla\cdot (\sigma_k \theta_t)\dd t \\
    &\quad - \sum_k \theta_t \nabla\cdot (\sigma_k \theta_t)^{\eps} \dd W^k_t
    - \sum_k \theta^\eps_t \nabla (\sigma_k \theta _t)\dd W^k_t.
    \end{split}\end{equation}
    For fixed $t \geq 0$, we can rewrite the second term of the first line above as
    \begin{align*}
        C (0) \nabla \theta_t \cdot \nabla \theta^{\varepsilon}_t
        & = \nabla\cdot [\theta_t C(0)\nabla\theta^\eps_t]- \theta_t C(0): D^2\theta^\eps_t\\
        & = \nabla\cdot \big( \theta_t [(C(0) \nabla\chi^\eps)\ast \theta_t]\big) - \theta_t [(C(0): D^2\chi^\eps)\ast \theta_t].
    \end{align*}
    For fixed $t \geq 0$, we can rewrite the last term of the first line of \eqref{eq:LEB_eps} as
    \begin{align*}
        \sum_k \nabla\cdot (\sigma_k \theta_t)^{\varepsilon} \nabla\cdot (\sigma_k \theta_t)
        = \nabla\cdot \bigg( \sum_k \nabla\cdot (\sigma_k \theta_t)^{\varepsilon} \sigma_k \theta_t \bigg)
        - \sum_k D^2: (\sigma_k \theta_t)^{\varepsilon} \sigma_k \theta_t
        =: \nabla \cdot I_1^\eps - I^\eps_2 ,
    \end{align*}
    where $I_1^\eps$ and $I^\eps_2$ are pointwise defined functions. $I_1^\eps$ is given by
    \begin{align*}
        I_1^\eps(x)
        & = \sum_k \nabla\cdot (\sigma_k \theta_t)^{\varepsilon}(x) \sigma_k(x) \theta_t(x)\\
        & = \theta_t (x) \sum_k \int_{\mathbb{R}^d} \sigma_k (y) \theta_t (y) \cdot \nabla \chi^{\varepsilon} (x - y) \sigma_k (x) \dd y\\
        & = \theta_t (x) \int_{\mathbb{R}^d} C(x-y) \nabla \chi^{\varepsilon} (x - y) \theta_t(y) \dd y
        = \theta_t (x) [(C \nabla\chi^\eps)\ast \theta_t](x).
    \end{align*}
    A similar computation yields
    \begin{align*}
        I_2^\eps(x)
        = \theta_t (x) [(C: D^2\chi^\eps)\ast \theta_t](x).
    \end{align*}
    Taking expectation in \eqref{eq:LEB_eps} removes the stochastic integrals, and inserting the above formulas while recalling that $Q(x)=C(0)-C(x)$ we arrive at the identity
    \begin{equation}\label{eq:LEB_eps2}
        \partial_t \EE[\theta_t \theta^{\varepsilon}_t] 
        = \frac{1}{2} C (0) : D^2 \EE[\theta_t \theta^{\varepsilon}_t]
        - \EE\nabla\cdot\big( \theta_t [(Q \nabla\chi^\eps)\ast \theta_t]\big)
        + \EE \theta_t [(Q: D^2\chi^\eps)\ast \theta_t]
    \end{equation}
    in the sense of distributions.
    For fixed $\eps>0$, the above derivation holds for smooth $\chi$, but notice that the resulting expression is meaningful as long as $\nabla\chi$, $D^2\chi$ are finite measures.
    Indeed in this case, since $Q$ is a continuous bounded function, by Young's inequality $(Q: D^2\chi^\eps)\ast \theta_t$ is a well-defined element of $L^2_x$, and similarly for $[(Q \nabla\chi^\eps)\ast \theta_t]$, therefore the right-hand side of \eqref{eq:LEB_eps2} is a well-defined distribution.
    Standard approximation arguments then allow to extend the validity of \eqref{eq:LEB_eps2}, at fixed $\eps>0$, to any $\chi$ as in our assumptions.

    {\em Step 2.} For $\eps>0$, for any fixed $t>0$, define the random distribution 
    \begin{equation}\label{eq:defn_cV_eps}
    \mathcal{V}^\eps_t := \cV^{1,\eps}_t + \cV^{2,\eps}_t
    = \nabla\cdot\big( \theta_t [(Q \nabla\chi^\eps)\ast \theta_t]\big) - \theta_t [(Q: D^2\chi^\eps)\ast \theta_t]
    \end{equation}
    By properties of mollifiers and dominated convergence, $\theta_t \theta_t^\eps\to |\theta_t|^2$ in $L^1_{\omega,t,x}$, and similarly $D^2(\theta_t \theta_t^\eps)\to D^2|\theta_t|^2$ as space-time distributions.
    Therefore, taking expectation in \eqref{eq:energy_balance_inviscid} and comparing it to \eqref{eq:LEB_eps2}, we find the following identity between space-time distributions:
    \begin{equation}\label{eq:LEB_eps3}
        \mathbb{E} \mathcal{D}[\theta] (\dif x,\dif t) = \lim_{\eps\to 0} \EE[\cV^\eps_t](\dif x) \dd t.
    \end{equation}

    \emph{Step 3.} We claim that $\EE[\cV^{1,\eps}_t]\to 0$ as a space-time distribution, for $\cV^{1,\eps}$ as defined in \eqref{eq:defn_cV_eps}, possibly thanks to the positive Sobolev regularity of $\theta$ coming from \autoref{thm:regularity_Kraichnan}.
    We will show something slightly stronger, namely that $\theta_t [(Q \nabla\chi^\eps)\ast \theta_t]$ converges to $0$ in $L^1_{t,x}$, both $\PP$-a.s. and in expectation.
    Indeed, exploiting the facts that ${\rm div} Q=0$ and that the integral of a divergence on full space vanishes, we have
    \begin{align*}
        \theta_t (x) (\nabla \cdot [Q\chi^\eps] \ast \theta_t) (x)
        & = \theta_t (x) \int_{\mathbb{R}^d} Q (x - y) \nabla \chi^{\varepsilon} (x - y) (\theta_t (y) - \theta_t (x))\dd y\\
        & = \theta_t (x) \int_{\mathbb{R}^d} \frac{Q (\eps z)}{\eps} \nabla \chi (z)\, \delta_{\eps z} \theta_t (x)\dd z.
    \end{align*}
    Let $\lambda>0$ to be chosen later.
    By \eqref{eq:covariance_recap} and the boundedness of $Q$, it holds $|Q(z)|\lesssim |z|^{2\alpha}$ for all $z\in\R^d$; together with Minkowski's inequality, this yields
    \begin{align*}
        \| \theta_t (\nabla \cdot [Q \nabla\chi^\eps] \ast \theta_t)\|_{L^1_x}
        & \lesssim \| \theta_t\|_{L^2_x}\, \eps^{2\alpha-1} \int_{\mathbb{R}^d} |z|^{2\alpha} |\nabla \chi (z)| \|\, \delta_{\eps z} \theta_t \|_{L^2_x} \dd z\\
        & \lesssim \eps^{\alpha-\lambda}\,\| \theta_t\|_{L^2_x}\, \| \theta_t\|_{H^{1-\alpha-\lambda}_x}
    \end{align*}
    whenever $\lambda<\alpha$ and $\lambda\leq 1-\alpha$; above we used the compact support of $\nabla\chi$. Note that such $\lambda$ can always be found thanks to \autoref{thm:regularity_Kraichnan}.
    Integrating in time, using that $\theta\in L^2_t H^{1-\alpha-\lambda}_x$ both $\PP$-a.s. and in $L^2_\omega$, the claim follows by sending $\eps\to 0$.
    It then follows from \eqref{eq:LEB_eps3} that
    \begin{equation}\label{eq:first.option}
        \mathbb{E} \mathcal{D}[\theta] (\dif x,\dif t) = \lim_{\eps\to 0} \EE[\cV^{2,\eps}_t(x)] \dd x \dd t
        = -\lim_{\eps\to 0} \EE\Big[\theta_t(x) [(Q: D^2\chi^\eps)\ast \theta_t](x)\Big] \dd x \dd t.
    \end{equation}

    \emph{Step 4.} Upon expanding $|\delta_y \theta(x)|^2$ in \eqref{eq:defn_Deps}, notice that
    \begin{align*}
        \cD^\eps(t,x)
        & = \frac{1}{2} \int_{\R^d} D^2\chi^\eps(y):Q(y) |\theta_t(x+y)|^2 \dd y
        - \int_{\R^d} D^2\chi^\eps(y):Q(y) \theta_t(x+y)\theta_t(x) \dd y\\
        & = \frac{1}{2} [(D^2\chi^\eps:Q) \ast |\theta_t|^2](x)
        + \cV^{2,\eps}_t(x)
    \end{align*}
    where $|\theta_t(x)|^2$ gives no contribution since $Q$ is divergence free.
    We are left to show that the first term in the above vanishes in the sense of distributions as $\eps\to 0$.
    Fix $\varphi \in C^{\infty}_c(\R^d)$; by $D^2\chi^\eps:Q=D^2:(\chi^\eps\,Q)$, integration by parts and rescaling, we have
\begin{align*}
\langle \varphi, D^2 : (\chi^{\eps} Q) \ast | \theta_t |^2 \rangle 
&=
\langle (D^2 \varphi) \ast (\chi^{\eps} Q), | \theta_t |^2 \rangle
\\
&=
\int_{\mathbb{R}^d}\int_{\mathbb{R}^d} 
D^2 \varphi (x + \eps y) Q (\eps y) \chi (y) |\theta_t(x)|^2\dd y \dd x \to 0
  \end{align*}
by dominated convergence, as $Q$ is continuous and bounded with $Q(0)=0$.
From the above and \eqref{eq:first.option} we conclude that
\begin{equation*}
    \mathbb{E} \mathcal{D}[\theta] (\dif x,\dif t) = \lim_{\eps\to 0} \EE[\cV^{2,\eps}_t(x)] \dd x \dd t
    =\lim_{\eps\to 0} \EE[\cD^{\eps}_t(x)] \dd x \dd t. \qedhere
\end{equation*}
\end{proof}

\begin{remark}\label{rem:required_regularity}
Step 3 above is the only place throughout the proof of \autoref{thm:exact.formula.dissipation.measure} where Sobolev regularity of $\theta$ is needed.
In particular, for $\alpha<1/2$, we need at least $\theta\in L^2_\omega L^2_t H^{1-2\alpha+}_x$; for $\alpha>1/2$, one can actually choose $\lambda=1-\alpha$, so that in this regime the derivation works for any $L^2_x$-valued solution.    
\end{remark}

\begin{cor}\label{cor:exact_diss_measure2}
Let $W$, $\theta_0$, $\theta$, $\cD[\theta]$ be as in \autoref{thm:energy_balance_inviscid} and define 
\begin{equation}\label{eq:def_h}
    h_t (r, x) := \oint_{\mathbb{S}^{d - 1}} \frac{| \delta _{r
    \hat{y}}\theta_t(x)) |^2}{r^{2 - 2 \alpha}} \sigma ( \mathrm{d} \hat{y}).
\end{equation}
Then for any $t> 0$ it holds 
\begin{equation}\label{eq:exact_diss_measure2}
    \lim_{\eps \rightarrow 0} \EE\left[ \bigg\| \mathcal{D}^\eps(t,x)- c\,\frac{d (d + 2)}{2}
    \left(  h_t(\eps, x) -  (d + 2 \alpha) \eps^{-d-2} \int_0^\eps r^{d + 1} h_t(r, x) \dd r \right) \bigg\|_{L^1_x} \right] = 0
\end{equation}
where $c$ the same constant as in \eqref{eq:covariance_recap}.
\end{cor}

\begin{proof}
By rescaling, we have
\begin{align*}
\cD^\eps(t,x) = \frac{1}{2} \int_{\mathbb{R}^d} Q
    (\eps y) : D^2 \chi (y)  \frac{| \delta_{\eps y} \theta_t (x)
    |^2}{\eps^2}\dd y.
\end{align*}
Let us now specialize the formula above to $\chi (y) := k_d (1 - | y |^2) \mathbf{1}_{| y | \leq 1}$, with 
\begin{align*}
k_d^{- 1} := \| (1 - | y |^2) \mathbf{1}_{| y | \leq 1}\|_{L^1_x}
= \omega_{d - 1} \left( \frac{1}{d} - \frac{1}{d + 2} \right)
= \frac{2\, \omega_{d - 1}}{d(d+2)},
\end{align*}
so that $\chi$ is a probability measure. 
It holds
\[ \nabla \chi (y) = - 2 k_d\, y\, \mathbf{1}_{\{| y | \leq 1\}}, \quad
     D^2 \chi (y) = 2 k_d \, \big(\hat y \otimes \hat y\, \sigma(\dif \hat y) - I_d \, \mathbf{1}_{\{| y | \leq 1\}}  \dd y\big). \]
Therefore, 
  \begin{align*}
\cD^\eps(t,x)
&= 
k_d \left[ 
\omega_{d-1}\oint_{\mathbb{S}^{d-1}} Q(\eps \hat y) : \hat y \otimes \hat y\, \frac{| \delta_{\eps \hat y} \theta_t (x)|^2}{\eps^{2 }} \sigma ( \mathrm{d} \hat y)
-   
\int_{|y| \leq 1} 
\mathrm{Tr} (Q(\eps y)) \,
\frac{| \delta_{\eps y} \theta_t (x)|^2}{\eps^{2 }}\dd y 
\right] \\
&= 
c\, k_d \left[ \omega_{d-1}\oint_{\mathbb{S}^{d-1}}  
\frac{| \delta_{\eps \hat y} \theta_t (x)|^2}{\eps^{2 - 2 \alpha}} \sigma ( \mathrm{d} \hat y) 
- (d + 2\alpha) 
\int_{| y | \leq 1} | y |^{2 \alpha}  \frac{|\delta_{\eps y} \theta_t (x) |^2}{\eps^{2 - 2 \alpha}}  \dd y
\right] + o_\eps(1)
\end{align*}
where we invoked \eqref{eq:covariance_recap} and we used the facts that
\begin{align*}
    P^\parallel_{\hat y} : \hat y \otimes \hat y=1, \quad
    P^\perp_{\hat y} : \hat y \otimes \hat y=0, \quad
    {\rm Tr}(P^\parallel_{ y})=1,\quad 
    {\rm Tr}(P^\perp_{ y})=d-1.
\end{align*}
The remainder $o_\eps(1)$ must be interpreted in the $L^1_{\omega,x}$-topology, as
\begin{align*}
   \EE[ \| o_\eps(1)\|_{L^1_x}]
   & \lesssim \EE\Big[\,\Big\| \oint_{\mathbb{S}^{d-1}}  | \delta_{\eps \hat y} \theta_t (x)|^2 \sigma ( \mathrm{d} \hat y) \Big\|_{L^1_x}\Big]
   + \EE\Big[\,\Big\| \int_{| y | \leq 1} | y |^{2 \alpha}  |\delta_{\eps y} \theta_t(x) |^2  \dd y\Big\|_{L^1_x}\Big]\\
   & \leq \oint_{\mathbb{S}^{d-1}} \EE[\| \delta_{\eps \hat y} \theta_t\|_{L^2_x}^2] \sigma ( \mathrm{d} \hat y)
   + \int_{| y | \leq 1} | y |^{2 \alpha}  \EE[\|\delta_{\eps y} \theta_t\|_{L^2_x}^2]  \dd y
\end{align*}
and $\theta_t\in L^2_x$, so that $\EE[\|\delta_{\eps y} \theta_t\|_{L^2_x}^2]\to 0$ as $\eps\to 0^+$.
Finally, by Cavalieri's principle we can rewrite 
  \begin{align*}
 \int_{| y | \leq 1} | y |^{2 \alpha} \frac{|
    \delta_{\eps y} \theta_t (x) |^2}{\eps^{2 - 2 \alpha}}\dd y &
    = \omega_{d-1}\int_0^1 r^{2 \alpha + d - 1} \oint_{\mathbb{S}^{d - 1}} \frac{|
    \delta_{r \eps \hat y} \theta_t (x) |^2}{\eps^{2 - 2 \alpha}}
    \sigma ( \mathrm{d} \hat{y})  \dd r\\
    & = \omega_{d-1}\int_0^1 r^{d + 1} \oint_{\mathbb{S}^{d - 1}} \frac{| \delta_{r \eps \hat{y}} \theta_t
    (x) |^2}{(r \eps)^{2 - 2 \alpha}} \sigma
    ( \mathrm{d} \hat{y})  \dd r\\
    & = \omega_{d-1} \int_0^1 r^{d + 1} h (r \eps, x)  \dd r;
  \end{align*}
changing variables $r \eps=r'$ and recalling the expression for $k_d$ yields the desired \eqref{eq:exact_diss_measure2}.
\end{proof}

The following general lemma allows to rewrite, in a cleaner way, the expression inside round parentheses in \eqref{eq:exact_diss_measure2}. 

\begin{lemma}\label{lem:identification_limits}
Denote by $(C^\infty_c)'$ the space of distributions on $\R^d$, with its usual topology.
Let $r\mapsto v(r)$ be a continuous map from $(0,1]$ to $(C^\infty_c)'$, such that $\lim_{r\to 0^+} r^{2-2\alpha} v(r)=0$ in $(C^\infty_c)'$. 
Further assume that there exists $\nu\in (C^\infty_c)'$ such that
\begin{equation}\label{eq:ass_distributional_limit}
    \nu = \frac{1}{2} \lim_{r \rightarrow 0} \left[ v(r) - (d + 2 \alpha)
    r^{- d - 2} \int_0^r u^{d + 1} v(u) \dd u \right],
\end{equation}
where both the identity and the limit are interpreted in $(C^\infty_c)'$. Then 
\begin{equation}\label{eq:identification_limits}
\nu = 
(1 - \alpha) \lim_{r \rightarrow 0}  r^{- d - 2} \int_0^r u^{d + 1} v(u) \dd u
= \frac{1 - \alpha}{d + 2} \lim_{r \rightarrow 0} v(r),
  \end{equation}
with both limits existing in $(C^\infty_c)'$.
\end{lemma}

\begin{proof}
Fix $\varphi \in C^\infty_c(\R^d)$ and set $g (r) := \langle \varphi, v(r) \rangle$; define 
  \[ \gamma_1 := d + 2 \alpha, \quad \gamma_2 : = 2 \langle \varphi, \nu
     \rangle, \quad G (r) := \int_0^r u^{d + 1} g (u) \dd u. \]
By \eqref{eq:ass_distributional_limit}, for all $\delta > 0$ sufficiently small and all $r \in (0, \bar{r}  (\delta))$, it holds
  \[ \gamma_2 - \delta \leq r^{- d - 1} G' (r) - \gamma_1 r^{- d - 2} G
     (r) \leq \gamma_2 + \delta, \]
  which can be rewritten as
  \begin{equation}
    (\gamma_2 - \delta) r^{d + 1 - \gamma_1} \leq (r^{- \gamma_1} G)'
    \leq (\gamma_2 + \delta) r^{d + 1 - \gamma_1} .
    \label{eq:differential_formula2}
  \end{equation}
  By assumption, $g(r)=r^{2\alpha-2}\, o_r(1)$ and so it's easy to check that $\lim_{r \rightarrow 0} r^{- \gamma_1} G (r) = 0$. 
  Integrating~\eqref{eq:differential_formula2} over $(0, r)$ for $r < \bar{r}
  (\delta)$, we get
  \[ \frac{\gamma_2 - \delta}{d + 2 - \gamma_1} r^{d + 2} \leq G (r)
     \leq \frac{\gamma_2 + \delta}{d + 2 - \gamma_1} r^{d + 2}, \]
  which amounts to
  \[ \lim_{r \rightarrow 0} r^{- d - 2} \int_0^r u^{d + 1}  g
     (u) \dd u = \lim_{r \rightarrow 0} r^{- d - 2} G (r) =
     \frac{\gamma_2}{d + 2 - \gamma_1} = \frac{\langle \varphi, \nu \rangle}{1
     - \alpha} . \]
As the argument holds for any $\varphi$, this identifies the distributional limit of the last term on the right-hand side of \eqref{eq:ass_distributional_limit}, from which~\eqref{eq:identification_limits} follows by basic algebraic
  manipulations.
\end{proof}

\begin{proof}[Proof of \autoref{thm:exact.formula.dissipation.measure}]
    Formula \eqref{eq:exact.formula} follows from concatenating \autoref{prop:alternative_dissipation_measure}, \autoref{cor:exact_diss_measure2} and \autoref{lem:identification_limits}. In particular, by \eqref{eq:exact_diss_measure2}, we can apply \autoref{lem:identification_limits} to $v(r)=c d (d+2) \EE h_t(r,x)$, for any $t>0$. Condition $\lim_{r\to 0^+} r^{2-2\alpha} v(r)=0$ is satisfied as $\theta_t\in L^2_{\omega,x}$ and continuity of translations in $L^2_x$:
    \begin{align*}
        \lim_{r\to 0^+} \| r^{2-2\alpha} \EE h_t(r,\cdot)\|_{L^1_x}
        \leq \lim_{r\to 0^+} \oint_{\mathbb{S}^{d - 1}} \EE\| \delta _{r
    \hat{y}}\theta_t \|_{L^2_x}^2 \sigma ( \mathrm{d} \hat{y}) = 0.
    \end{align*}

Let us move to showing \eqref{eq:sharp_regularity}.  
Since $c d (1-\alpha)\mathbb{E}h(\eps, x)\dd t \dd x$ and $\mathbb{E} \mathcal{D}[\theta]$ are both non-negative measures, \eqref{eq:exact.formula} implies narrow convergence of one to the other; by lower semicontinuity of total variation with respect to narrow convergence, we deduce that
    \begin{align*}
        c d (1-\alpha) \liminf_{\eps\to 0} \frac{1}{\eps^{2(1-\alpha)}} \oint_{\mathbb{S}^{d-1}} 
\| \delta_{\eps \hat{z}} \theta \|_{L^2_{\omega,t,x}}^2 \sigma ( \mathrm{d} \hat{z})
        & = \liminf_{\eps\to 0} c d (1-\alpha) \| \EE h(\eps,\cdot)\|_{L^1([0,T]\times\R^d)}\\
        & \geq \EE \cD[\theta] ([0,T]\times\R^d)]= \| \theta_0\|_{L^2_x}^2 - \EE[\| \theta_T\|_{L^2_x}^2]>0,
    \end{align*}
    where in the last inequality we applied \autoref{thm:dissipation_intro}.
    Concerning the limsup, note that by integrating \eqref{eq:LEB_eps2} on $[0,T]\times\R^d$ and using that divergence terms vanish, one finds
    \begin{align*}
        \EE[\langle\theta_T^\eps,\theta_T\rangle] - \langle\theta_0^\eps,\theta_0\rangle
        & = \EE \int_{[0,T]\times \R^d} \theta_t(x) [(Q: D^2\chi^\eps)\ast \theta_t](x) \dd x \dd t\\
        & = -\frac{1}{2} \EE\Big[\int_{[0,T]\times \R^d\times\RR^d} (Q: D^2\chi^\eps)(x-y)|\theta_t(x)-\theta_t(y)|^2 \Big] \dd x \dd y \dd t\\
        & = - \int_{[0,T]\times \RR^d} \EE\cD^\eps(t,x) \dd t \dd x,
    \end{align*}
    where in the intermediate passage we used again the fact that, expanding $|\theta_t(x)-\theta_t(y)|^2$, the terms different from $\theta_t(x)\theta_t(y)$ vanish upon integrating in $\dif x\dd y$.
   Rearranging terms and using $\lim_{\varepsilon \downarrow 0}\EE[\langle\theta_T^\eps,\theta_T\rangle] =\EE[\|\theta_T\|_{L^2_x}^2]$, $\lim_{\varepsilon \downarrow 0} \langle\theta_0^\eps,\theta_0\rangle = \|\theta_0\|_{L^2_x}^2$ one then finds
    \begin{align*}
        \lim_{\eps \downarrow 0} \int_{[0,T]\times \RR^d} \EE\cD^\eps(t,x) \dd t \dd x =
        \|\theta_0\|_{L^2_x}^2-\EE[\|\theta_T\|_{L^2_x}^2] = \EE\cD[\theta]([0,T]\times\R^d).
    \end{align*}
    Using the $L^1_x$-closedness between $\EE\cD^\eps(t,x)$ and the other expression appearing in \eqref{eq:exact_diss_measure2}, and going through the same computations as in  \autoref{lem:identification_limits}, it is not hard to see that it similarly holds
    \begin{align*}
      c\,d\,(1-\alpha) \limsup_{\eps \downarrow 0} \| \mathbb{E} h(\eps,\cdot)\|_{L^1([0,T]\times\R^d)} \leq \EE\cD[\theta]([0,T]\times\R^d).
    \end{align*}
    Combined with the previous inequality, this yields \eqref{eq:sharp_regularity}.
    \end{proof}

\begin{remark}\label{rem:proportionality_constant}
    The proportionality constant $\tilde c=\tilde c(\alpha,d):=d(1-\alpha) c(\alpha,d)$ appearing in \eqref{eq:exact.formula}, related to the constant $c=c(\alpha,d)$ appearing in the asymptotics \eqref{eq:covariance_recap}, can be tracked explicitly. Recalling the relations \eqref{eq:defn_alpha1_c_beta}, one can rewrite the terms $\alpha_1$, $a$, $b$ in function of $\eta$, ${\rm Tr} C(0)=\EE[|W_1(x)|^2]$ and special functions. For $\eta=1$, after some bookkeeping one finds 
    \begin{equation}\label{eq:proportionality_constant}\begin{split}
        \tilde c
        & = \frac{d}{d+2\alpha} \,{\rm Tr}(C(0))\, \frac{\pi^{d-1}}{\Gamma(d/2)}\, \frac{(1-\alpha) \cos(\alpha \pi)\Gamma(1-2\alpha)\Gamma(\alpha +1/2)}{ \Gamma(\alpha+1)}\\
        & =: \frac{d}{d+2\alpha}\,{\rm Tr}(C(0))\, K_1(d)\,K_2(\alpha)
    \end{split}\end{equation}
    where $\Gamma$ denotes the Gamma function. Formula \eqref{eq:proportionality_constant} allows to determine the asymptotic behaviour of $\tilde c$ as $\alpha\downarrow 0$ and $\alpha\uparrow 1$, using the properties of the Gamma function:
    \begin{align*}
        \lim_{\alpha\to 0^+} \tilde c(\alpha,d) = {\rm Tr}(C(0))\, K_1(d) \sqrt{\pi}, \quad
        \lim_{\alpha\to 1^-} \tilde c(\alpha,d)= \frac{d}{d+2}\,{\rm Tr}(C(0))\, K_1(d)\, \frac{\sqrt\pi}{4}.
    \end{align*}
    Interestingly, $\tilde c$ is uniformly bounded away from $0$ and $\infty$:
    \begin{align*}
        0<\inf_{\alpha\in (0,1)} \tilde c(\alpha,d) \leq \sup_{\alpha\in (0,1)} \tilde c(\alpha,d) <\infty.
    \end{align*}
\end{remark}

\section{Richardson's law and $L^2_x$ regularity of the stochastic continuity equation} \label{sec:richardson}
In this section we present a rigorous proof of  \autoref{thm:richardson} on the expected variance of solutions $\mu$ to the stochastic continuity equation. Furthermore, we rigorously establish the relation \eqref{eq:particle_dispersion} linking the expected variance of $\mu$ to the average ``particle'' dispersion expressed in terms of the Le Jan-Raimond two-point motion (cf. \autoref{ssec:two_point}), thus justifying the choice of the name ``Richardson's law'' for \eqref{thm:richardson}.  
Finally, we discuss $L^2_x$ regularity of $\mu$ in \autoref{sec:dirac}.

\subsection{Richardson's law} \label{ssec:Richardson's law}
The core idea is to estimate the growth over time of the variance of $\mu_t$ by using the It\^o Formula 
\begin{align} \label{eq:G_mu_mu}
\dd \langle G \ast \mu, \mu \rangle = 2 \langle (\nabla G \ast \mu) \,
     \mu, \dd W \rangle + \langle (D^2 G : Q) \ast \mu, \mu \rangle \dd
     t,
\end{align}
where $G$ is a sufficiently regular, symmetric kernel on $\R^d$ with good bounds on its growth at infinity.
For the upper bound in \eqref{thm:richardson}, we choose $G(x) = |x|^2/2$, which links the dynamics of the expected variance to the covariance structure of the noise. By assuming only a bound on the trace of the covariance operator $\mbox{Tr}Q(x) \lesssim |x|^{2\alpha}$, we show that the expected variance grows at most like $t^\frac{1}{1-\alpha}$.
For the lower bound, we select the kernel $G(x) = |x|^{2-2\alpha}$, which allows us to track the separation of ``particle trajectories'' over short timescales and demonstrate that the variance achieves a minimal growth rate dictated by the same power law, under \autoref{assumption}.
The precise value of the Richardson constant $K_{\mathsf{Ric}}$ is given by
\begin{align*}
K_{\mathsf{Ric}}
:=
\frac12 \left( c(2-2\alpha)(1-2\alpha+\beta(d-1))  \right)^{\frac{1}{1-\alpha}}
> 0.
\end{align*}

\begin{lemma}
Suppose \autoref{ass:well_posed} and let $\mu_0 \in \mathcal{P}(\R^d)$ with $\langle |x|^2,\mu_0\rangle<\infty$. Let $G$ be a $C^2$ symmetric kernel such that $G$, $\nabla G$, $D^2 G$ have at most quadratic growth. Then \eqref{eq:G_mu_mu} holds $\PP$-a.s.
\end{lemma}
\begin{proof}
This is a standard application of It\^o Formula.
Note that by assumption, estimate \eqref{eq:stoch_continuity_bound_variance} holds, therefore all terms appearing in \eqref{eq:G_mu_mu} are well-defined. For instance, for the stochastic integral, by \eqref{eq:covariance_TV_inequaity} and the assumptions it holds
\begin{align*}
\| \mathcal{C}^{1/2}((\nabla G \ast \mu_t)\mu_t) \|_{L^2_x}^2
& \lesssim
\| (\nabla G \ast \mu_t)\mu_t \|_{TV}^2
\lesssim
\| (1+|x|^2) \ast \mu_t)\mu_t \|_{TV}^2\\
& \lesssim
1 + \left( \int_{\R^d} \int_{\R^d} |x-y|^2 \mu_t(\mathrm{d} y) \mu_t (\mathrm{d}x) \right)^2
\lesssim 1 + \sup_{t\in [0,T]} \langle |x|^2,\mu_t\rangle^2,
\end{align*}
where the last quantity admits moments of any order.
\end{proof}

\begin{proof}[Proof of \autoref{thm:richardson}]

\textit{Upper bound}.
Denote for simplicity $V_t  := \var(\mu_t)$. We can specialize \eqref{eq:G_mu_mu} to $G(x) := | x |^2 / 2$. This yields  
\begin{equation}
\dd V_t 
= 
2  \langle (x \ast \mu) \, \mu, \dd W_t \rangle + 
  \langle \mbox{Tr} Q \ast \mu, \mu \rangle\dd t. \label{eq:variance.balance}
\end{equation}
Under the assumption that $\mbox{Tr} Q (x) \lesssim | x |^{2 \alpha}$ for every $x \in \R^d$, after taking expectation in \eqref{eq:variance.balance} and applying Jensen's inequality, we get
\begin{align*}
\frac{\dd}{\dd t} \mathbb{E} [V_t] 
&\lesssim
\mathbb{E} [\langle | x - y |^{2 \alpha} \ast \mu_t, \mu_t \rangle] 
\lesssim \mathbb{E} [V_t]^{\alpha}.
  \end{align*}
Next apply Bihari–LaSalle inequality for $\dot{v}_t \lesssim v^{\alpha}_t$ with $v_0 = \mathbb{E}[V_0]$ to get
\begin{align*}
\mathbb{E}[V_t]
\lesssim
\left( \mathbb{E}[V_0]^{1-\alpha}+ t\right)^{\frac{1}{1-\alpha}},
\quad
\forall t \geq 0.
\end{align*} 

\textit{Lower bound}. 
Let $\gamma:=2-2\alpha$, so that $2\alpha+\gamma-2 = 0$. 
Approximate the kernel $G(x) := |x|^\gamma$ with a sequence of functions $\{G^\eps\}_{\eps \in (0,1)}$ satisfying the assumptions of previous lemma and such that $D^2 G^\eps : Q$ is bounded uniformly in $\eps$, which is possible since $\mbox{Tr}Q(x) \lesssim |x|^{2\alpha}$.
Then applying \eqref{eq:G_mu_mu} with $G^\eps$ and taking the limit $\eps \downarrow 0$, by Dominated Convergence we have
\begin{align*}
\mathbb{E} &\left[ \langle G \ast \mu_t , \mu_t \rangle \right] 
=
\mathbb{E} \left[ \langle G \ast \mu_0 , \mu_0 \rangle \right] 
\\
&\quad+
\int_0^t
\mathbb{E} \left[ \int_{\R^d} \int_{\R^d}
|x-y|^{\gamma-2} \gamma
\left( (\gamma-1) b_L(|x-y|) + (d-1) b_N(|x-y|) \right)
\mu_s( \mathrm{d} x) \mu_s( \mathrm{d} y) \right]
ds.
\end{align*}
Next, by \autoref{assumption}, for every $\eps>0$ there exists $\ell > 0$ small enough such that 
\begin{align*}
|x-y|^{\gamma-2} \gamma
\left( (\gamma-1) b_L(|x-y|) + (d-1) b_N(|x-y|) \right)
\geq
(1-\eps) c\gamma (\gamma-1+\beta(d-1)) > 0
\end{align*}
for every $|x-y| \in (0,\ell)$.
Moreover, since $\hat{C} \in L^1$ we also have the global bound
\begin{align*}
|x-y|^{\gamma-2} \gamma
\left( (\gamma-1) b_L(|x-y|) + (d-1) b_N(|x-y|) \right)
\leq
K < \infty.
\end{align*}
Finally, since $\mathbb{E} \left[ \langle G \ast \mu_0 , \mu_0 \rangle \right] \geq 0$ we get 
\begin{align*}
\mathbb{E} \left[ \langle G \ast \mu_t , \mu_t \rangle \right] 
&\geq
(1-\eps) c \gamma (\gamma-1+(d-1)\beta)
\int_0^t
\mathbb{E} \left[ \int_{\R^d} \int_{\R^d}
\mathbf{1}_{\{|x-y| < \ell\}}
\mu_s( \mathrm{d} x) \mu_s( \mathrm{d} y) \right]
\dd s
\\
&\quad-
K \int_0^t
\mathbb{E} \left[ \int_{\R^d} \int_{\R^d}
\mathbf{1}_{\{|x-y| \geq  \ell\}}
\mu_s( \mathrm{d} x) \mu_s( \mathrm{d} y) \right]
\dd s.
\end{align*}
By Markov inequality and the upper bound proved previously, it holds for every $t < t_\eps$ small enough
\begin{align*}
\sup_{s \leq t}
\mathbb{E} \left[ \int_{\R^d} \int_{\R^d}
\mathbf{1}_{\{|x-y| \geq  \ell\}}
\mu_s( \mathrm{d} x) \mu_s( \mathrm{d} y) \right]
\leq 
\ell^{-2}
\sup_{s \leq t}
\mathbb{E} \left[ V_s \right]
< \eps
\end{align*}
for every $M$ and $t_\eps$ small enough.
Therefore, up to changing the value of $\eps$, we get
\begin{align*}
\mathbb{E}& \left[ \int_{\R^d} \int_{\R^d}
|x-y|^{\gamma} \mu_t( \mathrm{d} x) \mu_t( \mathrm{d} y) \right]
\geq 
(1-\eps')
c \gamma (\gamma-1+(d-1)\beta) t,
\quad
\forall t \in (0,t_\eps).
\end{align*}
Since $\gamma \in (0,2)$, by Jensen's inequality it holds
\begin{align*}
\left(2 \mathbb{E}[V_t] \right)^{\gamma/2}
&\geq
\mathbb{E} \left[ \int_{\R^d} \int_{\R^d}
|x-y|^{\gamma} \mu_t( \mathrm{d} x) \mu_t( \mathrm{d} y) \right]
\\
&\geq
(1-\eps')c\gamma(\gamma-1+\beta(d-1)) t,
\qquad
\forall t \in (0,t_\eps),
\end{align*}
implying the desired result as
\begin{equation*}
\frac12 \left( c\gamma(\gamma-1+\beta(d-1)) t \right)^{2/\gamma}
=
\frac12 \left( c(2-2\alpha)(1-2\alpha+\beta(d-1)) t \right)^{\frac{1}{1-\alpha}}
=
K_{\mathsf{Ric}} t^{\frac{1}{1-\alpha}}. \qedhere
\end{equation*}
\end{proof}

The upper bound in \autoref{thm:richardson} can be improved in the following functional sense.
\begin{prop}\label{prop:richardson_upper_functional}
Suppose \autoref{ass:well_posed} and $\mathrm{Tr}\, Q(x) \lesssim |x|^{2\alpha}$ for every $x \in \R^d$; let $\mu_0=\delta_x$ for some $x\in\R^d$. Then for any $p \in (0, 1)$ and $T \in (0,\infty)$ it holds
\[ 
\mathbb{E} \left[ \left( \sup_{t \in [0, T]} t^{- \frac{1}{1 - \alpha}} \var(\mu_t) \right)^{(1 - \alpha) p} \right]
< 
\infty. 
\]
\end{prop}
In other words, the random variable $N := \sup_{t \in [0, T]} t^{- \frac{1}{1 -
\alpha}} \var(\mu_t)$ exists and is $\PP$-almost surely finite, which gives the pathwise bound 
\begin{align*}
\var(\mu_t)
\leq
N t^\frac{1}{1 - \alpha},
\quad
\forall t \in [0,T].
\end{align*}
\begin{proof}
Without loss of generality, we may assume $T=1$. 
Then by the same passages as in the proof of \autoref{thm:richardson}, for any $s < t$ we have 
\[ V_t \leq V_s + \int_s^t V_r^{\alpha} \dd r + M_t - M_s, \quad M_r
   := \int_0^r 2 \langle (x \ast \mu_u) \, \mu_u, \dd W_u \rangle . \]
In this situation, since $M$ is a nicely defined, continuous martingale, one
can apply the stochastic versions of the Bihari-LaSalle inequality (\cite[Theorem 3.1]{Ge23}) to find
\[ \mathbb{E} \left[ \left(
\sup_{r \in [s, t]} V_r^{1 - \alpha} \right)^p\right] 
\lesssim 
 \mathbb{E} \left[ \left(
\sup_{r \in [s, t]} V_r \right)^p\right]^{1 - \alpha} 
\lesssim 
\mathbb{E}[V_s]^{1-\alpha} + t-s,
\]
 for any fixed $p \in (0, 1)$ and $s<t$.
Subdiving the interval $[0, 1]$ in a dyadic way and using the upper bound in expectation coming from \autoref{thm:richardson}, one then finds
\begin{align*}
\mathbb{E} \left[\left( \sup_{t \in [0, 1]} t^{- 1} V_t^{1 - \alpha}\right)^p\right] 
&\leq
\sum_{n \geq  0} \mathbb{E} \left[\left( \sup_{t \in [2^{- n - 1}, 2^{- n}]} t^{-1} V_t^{1 - \alpha}\right)^p\right]
\\
&\lesssim 
\sum_{n \geq  0} 2^{n p} \left(\mathbb{E} \left[V_{2^{- n}}^{1 -\alpha}\right] + 2^{- n} \right) 
\lesssim
  \sum_{n \geq  0} 2^{- n (1 - p)} < \infty. \qedhere
\end{align*}
\end{proof}

\subsection{Relation to two-point motion}  \label{ssec:two_point}

Recall previous \autoref{prop:solution_LJR}, stating that the solution $\theta$ of the forward stochastic transport equation can be obtained by time reversal of the solution map $S$ given by \cite[Theorem 3.2]{LeJRai2002}, which instead solves the backward stochastic transport equation. 
Because of the latter, the solution map $S$ constructed by Le Jan and Raimond is in a natural duality with the solution $\mu$ of the (forward) stochastic continuity equation \eqref{eq:stochastic.continuity}.
In this subsection we aim at clarifying the link between $\mu$ and the two-point motion associated to the backward stochastic transport equation introduced in \cite{LeJRai2002}. 

Let us recall the Le Jan-Raimond construction of two-point motion.
By \cite[Equation (3.23)]{LeJRai2002}, the solution map $S$ satisfies for every $t \geq 0$ and $\PP \otimes \mathscr{L}^d$-almost every $(\omega,x) \in \Omega \times \R^d$
\begin{align} \label{eq:conditional_law}
S_t f (\omega,x) = \int_{C_t \R^d} f(e_t(\omega')) P_{x,\omega} ( \mathrm{d} \omega'),
\end{align}
where $P_{x,\omega} ( \mathrm{d} \omega')$ is a family of conditional probabilities on the space of continuous functions $C_t \R^d$ defined for $\PP \otimes \mathscr{L}^d$-almost every $(\omega,x)$. 
Here $\omega' \in C_t \R^d$ and $e_t$ is the evaluation at time $t$.

Let $\lambda_x , \lambda_y \ll \mathscr{L}^d$ be two probability measures on $\R^d$ absolutely continuous with respect to the Lebesgue measure.
The two-point motion (with initial distribution $\lambda_x \otimes \lambda_y$) is defined as the law on $C_t(\R^d \times \R^d)$ given by
\begin{align} \label{eq:2point}
P^2_{\lambda_x \otimes \lambda_y}( \mathrm{d} \omega', \mathrm{d}\omega'') 
&:=
\int_{\Omega} 
P_{\lambda_x,\omega} ( \mathrm{d} \omega') P_{\lambda_y,\omega} ( \mathrm{d} \omega'') \mathbb{P}( \mathrm{d} \omega),
\quad
\omega',\omega'' \in C_t \R^d,
\end{align}
where the conditional laws $P_{\lambda_x,\omega}$ and $P_{\lambda_y,\omega}$ are given by the disintegration formula
\begin{align} \label{eq:2point_x}
P_{\lambda_x,\omega} ( \mathrm{d} \omega') 
&:= 
\int_{\R^d} P_{x,\omega} ( \mathrm{d} \omega') \lambda_x( \mathrm{d} x),
\qquad
P_{\lambda_y,\omega} ( \mathrm{d} \omega'') 
:= 
\int_{\R^d} P_{y,\omega} ( \mathrm{d} \omega') \lambda_y( \mathrm{d} y),
\end{align}
and are well defined since $\lambda_x , \lambda_y \ll \mathscr{L}^d$.
The two-point motion defines a Markov process $(X,Y)$ on $\R^d \times \R^d$, still denoted two-point motion with slight abuse of notation, that has continuous trajectories $\PP$-almost surely and transition semigroup
\begin{align*}
P_t^2 (f \otimes g) := \mathbb{E} [S_t f \otimes S_t g],
\quad
\forall f, g \in L^2_x.
\end{align*}

\begin{lemma} \label{lem:variance_mu}
Let $\mu$ be the unique solution of \eqref{eq:stochastic.continuity} with initial condition $\mu_0 \ll \mathscr{L}^d$ with density $\frac{\dif \mu_0}{\dif x}\in L^2_x$, and let $(X,Y)$ be the two-point motion with initial distribution $\mu_0 \otimes \mu_0$. Then for every $t \geq 0$ it holds
\begin{align} \label{eq:variance_mu}
\mathbb{E}[\var(\mu_t)] 
&=
\frac12 
\int_\Omega \int_{\R^d} \int_{\R^d} 
|x'-y'|^2 \mu_t(\omega,\dd x') \mu_t(\omega,\dd y')
\PP ( \mathrm{d} \omega)
=
\frac12 \mathbb{E} [|X_t-Y_t|^2].
\end{align}
\end{lemma}
\begin{proof}
Let us compute the quantity $\mathbb{E}[|X_t-Y_t|^2]$ using the definition of two-point process above.
We start the two-point motion from an arbitrary initial condition $\lambda_x \otimes \lambda_y$. By \eqref{eq:2point}, \eqref{eq:2point_x} above and change of variables, we get
\begin{align*}
\mathbb{E}[|X_t-Y_t|^2]
&:=
\int_{C_t \R^d} \int_{C_t \R^d} 
|e_t(\omega')-e_t(\omega'')|^2 P^2_{\lambda_x \otimes \lambda_y}( \mathrm{d} \omega' \dd \omega'')\\
&=
\int_{\Omega} 
\int_{\R^d} \int_{\R^d}
|x'-y'|^2
\left( \int_{\R^d} ((e_t)_\sharp P_{x,\omega}) ( \mathrm{d} x') \lambda_x( \mathrm{d} x) \right)
\left( \int_{\R^d}  ((e_t)_\sharp P_{y,\omega}) ( \mathrm{d} y') \lambda_y( \mathrm{d} y) \right) \mathbb{P}( \mathrm{d} \omega) 
.
\end{align*}
Next, let us rewrite solutions of the forward stochastic continuity equation.
The key identities are $\langle S_t f, \mu_0 \rangle = \langle f, \mu_t \rangle$, given by \autoref{lem:duality} for any $f \in C^\infty_c(\R^d)$, and equation \eqref{eq:conditional_law}. We have
\begin{align*}
\int_{\R^d} f(x') \mu_t(\omega,\dd x')
=
\int_{\R^d} S_t f (x, \omega) \mu_0 ( \mathrm{d} x) 
=
\int_{\R^d} \int_{\R^d} f(x') ((e_t)_\sharp P_{x,\omega}) ( \mathrm{d} x') \mu_0( \mathrm{d} x)
.
\end{align*}
As a consequence, since $f \in C^\infty_c(\R^d)$ is arbitrary, we deduce the following expression for the solution of the stochastic continuity equation \eqref{eq:stochastic.continuity}:
\begin{align*} 
\mu_t(\omega,\dd x') = \int_{\R^d} ((e_t)_\sharp P_{x,\omega}) ( \mathrm{d} x') \mu_0( \mathrm{d} x),
\qquad
\PP\mbox{-almost surely}. 
\end{align*}
Together with the formula for $\mathbb{E}[|X_t-Y_t|^2]$ above,  \eqref{eq:variance_mu} follows after imposing $\lambda_x = \lambda_y = \mu_0$. 
\end{proof}

Next, we extend the formula \eqref{eq:variance_mu} in \autoref{lem:variance_mu} to a Dirac delta initial condition $\mu_0 = \delta_x$, for almost every $x \in \R^d$.
Recall that the conditional law $P_{x,\omega}( \mathrm{d} \omega')$ is well-defined for $\PP \otimes \mathscr{L}^d$-almost every $(\omega,x)$. 
Let us introduce the family $P_{x,\omega}( \mathrm{d} \omega')
P_{x,\omega}( \mathrm{d} \omega'')$ of conditional probabilities on $C_t \R^d \times C_t \R^d \simeq C_t(\R^d \times \R^d)$, defined for the same full-measure set of $(\omega,x)$.
Hence, by Markovianity of $S_t$ and \eqref{eq:conditional_law}, for almost every $x \in \R^d$ the transition densities $t \mapsto \int_\Omega ((e_t)_\sharp P_{x,\omega}) ( \mathrm{d} x') ((e_t)_\sharp P_{x,\omega}) ( \mathrm{d} x'') \PP(\mathrm{d}\omega)$ satisfy the Chapman-Kolmogorov equations, with initial distribution $\delta_x \otimes \delta_x$.
By the Kolmogorov extension theorem, for almost every $x \in \R^d$ this defines a Markov process $(X^{\delta_x},Y^{\delta_x})$ on $\R^d \times \R^d$, with law
\begin{align*}
P^2_{\delta_x \otimes \delta_x}( \mathrm{d} \omega'\dd \omega'')
=
\int_\Omega
P_{x,\omega}( \mathrm{d} \omega')
P_{x,\omega}( \mathrm{d} \omega'')
\PP( \mathrm{d} \omega),
\end{align*}
and thus $(X^{\delta_x},Y^{\delta_x})$ is the two-point motion with initial distribution $\delta_x \otimes \delta_x$.

By \autoref{lem:duality} and \eqref{eq:conditional_law}, for every fixed $t \geq 0$ and for $\PP \otimes \mathscr{L}^d$-almost every $(\omega,x)$ we have
\begin{align*}
(e_t)_\sharp P_{x,\omega}( \mathrm{d} x')
=
\mu_t^x(\omega,\dd x'),
\end{align*} 
where $\mu^x$ is the unique solution of the stochastic continuity equation \eqref{eq:stochastic.continuity} starting from initial condition $\mu_0 = \delta_x$.
By the same computation as above, we deduce that for almost every $x \in \R^d$:
\begin{align} \label{eq:particle_dispersion}
\frac12 \mathbb{E}[|X_t^{\delta_x}-Y_t^{\delta_x}|]
=
\mathbb{E}[\var(\mu_t^x)]
=
\mathbb{E}[\var(\mu_t)],
\end{align}
where with the second equality we refer to the fact that the law of $\var(\mu_t^x)$ is independent of the initial location $x$ of the Dirac delta measure. We have seen that the latter quantity can be controlled for short times with $\sim t^{\frac{1}{1-\alpha}}$, under the assumptions of \autoref{thm:richardson}.  This observation justifies the choice of the terminology ``Richardson's law'' when talking about \eqref{thm:richardson}.

\subsection{Instantaneous $L^2_x$ regularization in the incompressible case} \label{sec:dirac}

In this subsection, we present some estimates for the stochastic continuity equation \eqref{eq:stochastic.continuity}, allowing to show instantaneous $L^2_x$-regularization of solutions with initial condition $\mu_0=\delta_x$, ultimately due the diffusive and spontaneously stochastic nature of the inviscid system. 
Mathematically, our approach is based on uniform-in-$\kappa$ bounds for the fundamental solution of the parabolic PDEs \eqref{eq:PDE_F}-\eqref{eq:PDE_G} satisfied by the two-point self-correlation functions.
Our results are valid in the incompressible regime $\div W = 0$, corresponding to %$\beta= \frac{d+2\alpha-1}{d-1}$ (and
$\eta=1$ for the Kraichnan model.

To prove the result, we leverage \autoref{lem:L^2-criterion-selfcorrelation}: any $L^\infty_x$-bound on the two-point self-correlation can be linked to an $L^2_{\omega,x}$-bound on the solution to the SPDE.
For convenience, we will constantly work with the viscous approximation $\tilde\mu^\kappa$, solving \eqref{eq:viscous.stochastic.continuity}; its two-point self-correlation function, denoted $G^\kappa$, solves the PDEs \eqref{eq:PDE_G}.
Moreover, since $\div Q = 0$ by assumption, $G^\kappa$ also solves \eqref{eq:PDE_F}. The above can be compactly written as
\begin{equation}\label{eq:dual_PDEs}
    %\partial_t F^\kappa = - \cH_\kappa F^\kappa, \quad
    \partial_t G^\kappa = -\cH_\kappa G^\kappa,\quad
    \text{with}\quad \cH_\kappa = \cH^\ast_\kappa:= -(1-\kappa) Q : D^2 - \kappa C(0):D^2. 
\end{equation}

A general strategy to deal with (possibly degenerate) second order parabolic PDEs, usually in symmetric, divergence form, is to perform energy estimates and Nash-type inequalities, see e.g. \cite{Davies1990}; useful predecessors in application to homogeneous isotropic turbulence are \cite{Hakulinen,rowan2023}.

Going forward we shall suppose the following condition on the covariance matrix $Q$;
note that we do not need $W$ to be isotropic here.

\begin{assumption} \label{ass:sharpness_reg_bis}
\autoref{ass:well_posed} holds, $\div Q = 0$ and there exist $\alpha\in (0,1)$, $K>0$ such that
\begin{equation}\label{eq:sharpness_reg_bis_ii}
         Q(z) \geq K (1\wedge |z|^{2\alpha}) I_d, \quad \forall\, z\in\R^d.
\end{equation}
\end{assumption}

\begin{theorem}\label{thm:main.thm.dirac}
    Let $W$ satisfy \autoref{ass:sharpness_reg_bis} 
    for some $\alpha\in (0,1)$.
    Then for any $\mu_0\in \mathcal{M}(\R^d)$, the associated solutions $\tilde \mu^\kappa$ to \eqref{eq:viscous.stochastic.continuity} satisfy
    \begin{equation}\label{eq:main.thm.dirac}
        \sup_{\kappa\in [0,1/2)} \EE[\| \tilde\mu^\kappa_t\|_{L^2_x}^2] \lesssim t^{-\frac{d}{2(1-\alpha)}} \| \mu_0\|_{TV}^2, \quad \forall\, t\in (0,1).
    \end{equation}
\end{theorem}

\begin{remark}\label{rem:ass_bis_kraichnan}
In the Kraichnan model, \autoref{ass:sharpness_reg_bis} is satisfied for $\alpha\in (0,1)$ and $\eta=1$. See \autoref{cor:assumptions_kraichnan} and \autoref{lem:nondegeneracy_kraichnan}.
\end{remark}

To prove \autoref{thm:main.thm.dirac}, we start with the following lemma:

\begin{lemma} \label{lem:Davies}
Under the assumptions of \autoref{thm:main.thm.dirac}, let $\cH_\kappa$ as defined in \eqref{eq:dual_PDEs} and consider
\begin{align*}
    \gamma:=\frac{d}{1-\alpha}, \qquad r := \frac{2\gamma}{\gamma-2}.
\end{align*}
%$Q$ and $\cH_\kappa$ as defined above.
Then, uniformly in $\kappa\in [0,1/2)$, for all %radially symmetric
$G\in C^\infty_c(\R^d)$ we have 
\begin{align}\label{eq:generator_estimate}
\| G \|_{L^r_x}^2 
&\lesssim 
\langle \cH_\kappa G,G\rangle%_{L^2}
+ 
\| G \|_{L^2_x}^2.
\end{align}
\end{lemma}

\begin{proof}
For simplicity, let us take $\kappa=0$, the other cases being simpler thanks to the presence of the non-degenerate operator $\kappa C(0):D^2$.
Integrating by parts we have
\begin{align*}
    \langle \cH_0 G,G\rangle
    & = \langle \nabla G, {\rm div} (Q G) \rangle
    = \langle \nabla G, Q \nabla G \rangle .
\end{align*}
In order to conclude, it thus suffices to show that
\begin{align*} 
    \| F \|_{L^r_x}^2 & \lesssim \langle \nabla F, Q \nabla F \rangle + \| F \|_{L^2_x}^2.
\end{align*}
Let $\psi_1\in C^\infty_c$ be such that $\psi_1\equiv 1$ for $|z|\leq 1$, $\psi_1\equiv 0$ for $|z|\geq 2$ and set $\psi_2=1-\psi_1$. Then by \autoref{ass:sharpness_reg_bis} it holds
\begin{align*}
    \langle \nabla G, Q \nabla G \rangle
    & \gtrsim \int_{\R^d} (1\wedge |z|^{2\alpha}) |\nabla G (z)|^2 \dd z\\
    & \gtrsim \int_{\R^d} |z|^{2\alpha} |\nabla G (z)|^2 |\psi_1(z)|^2 \dd z + \int_{\R^d} |\nabla G (z)|^2 |\psi_2(z)|^2 \dd z\\
    & \gtrsim \int_{\R^d} |z|^{2\alpha} |\nabla ( \psi_1 G) (z)|^2 \dd z + \int_{\R^d} |\nabla ( \psi_2 G)(z)|^2 \dd z - \| G\|_{L^2_x}^2.
\end{align*}
Set $G_i = \psi_i G$ for $i=1,2$. By Sobolev embeddings and interpolation (note that $r\in (2,2^\ast)$ for any $d\geq 2$) it holds
\begin{align*}
    \| G_2\|_{L^r_x}^2 \lesssim \| \nabla G_2\|_{L^2_x}^2  + \| G\|_{L^2_x}^2.
\end{align*}
Concerning $G_1$, we may write $\nabla G_1 (z) = |z|^{-\alpha}|z|^\alpha \nabla G_1 (z)$ and use H\"older inequality in Lorentz spaces $L^{p,q}_x$ (\cite[Thm. 3.4]{ONeil1963}) to bound
\begin{equation}\label{eq:lorentz-inequality}
\| \nabla G_1\|_{L^{p,2}_x}
\lesssim 
\| |z|^\alpha \nabla G_1\|_{L^{2,2}_x}
\| |z|^{-\alpha} \|_{L^{d/\alpha,\infty}_x}
\lesssim
\| |z|^\alpha \nabla G_1\|_{L^{2}_x},
\qquad
\frac{1}{p}:=\frac{1}{2} + \frac{\alpha}{d}.
\end{equation}
where we used $L^2_x=L^{2,2}_x$ and $|z|^{-\alpha}\in L^{d/\alpha,\infty}_x$.
By Talenti's inequality \cite[Thm. 3]{talenti1992} we deduce
\begin{align*}
\| G_1\|_{L^{r,2}_x}
\lesssim 
\| \nabla G_1\|_{L^{p,2}_x}.
\end{align*}
Noting that $\gamma>2$, so that $r>2$, Lorentz embedding yields
$\| G_1\|_{L^{r}_x} = \| G_1\|_{L^{r,r}_x} \lesssim \| G_1\|_{L^{r,2}_x}$.
Combining the above estimates overall yields the proof.
\end{proof}

\autoref{lem:Davies} yields time-dependent $L^\infty_x$-bounds on the two-point self-correlation functions:
\begin{prop}
Under the assumptions of \autoref{thm:main.thm.dirac}, uniformly in $\kappa \in (0,1/2)$ and $t \in (0,1)$ we have
\begin{equation}\label{eq:semigroup_full_1}
\| e^{-\cH_\kappa t} G\|_{L^\infty_x} 
\lesssim 
t^{-\frac{\gamma}{4}} \|G\|_{L^2_x}, 
\qquad 
\forall  G\in L^2_x.
\end{equation}
\begin{equation}\label{eq:semigroup_full_2}
\| e^{-\cH_\kappa t} G\|_{L^\infty_x} 
\lesssim 
t^{-\frac{\gamma}{2}} \|G\|_{TV}, 
\qquad 
\forall  G\in \cM(\R^d).
\end{equation}
\end{prop}

\begin{proof}
Note that, for $\kappa>0$, the operator $\cH_\kappa$ is associated to a strictly elliptic, bounded, H\"older continuous matrix $(1-\kappa)Q(0) + \kappa C(0)$; classical PDE results thus guarantee the existence of the semigroup $e^{- \cH_\kappa t}$, as well as the existence of an associated Green function $K^\kappa_t(x,y)$. In particular, 
\begin{align*}
    (e^{-\cH_\kappa t} G)(x)=\int_{\R^d} K^\kappa_t(x,y) G(y)\dd y, %\quad
\end{align*}
and in order to prove \eqref{eq:semigroup_full_1}, it then suffices to show that
\begin{equation}\label{eq:semigroup_full_goal}
    \sup_{x\in\R^d} \| K^\kappa_t(x,\cdot)\|_{L^2_y} \lesssim t^{-\frac{\gamma}{4}}, 
\end{equation}
uniformly in $\kappa$. Further note that, since $\div Q=0$, $\cH_\kappa$ is self-adjoint and
\begin{align*}
    K^\kappa_t(x,\cdot) = e^{-\cH_\kappa^\ast t} \delta_x = e^{-\cH_\kappa t} \delta_x.
\end{align*}
Starting from the above and \autoref{lem:Davies}, getting estimates of the type \eqref{eq:semigroup_full_goal} is a classical argument by Nash, see also \cite{Davies1990} and \cite[Prop. 3.7]{rowan2023}; we briefly sketch it.

Let $\varphi^\eps_0\in \cP(\R^d)\cap C^\infty_c(\R^d)$ be a smooth approximation of $\delta_x$ and let $\varphi^\eps_t=e^{-\cH_\kappa t} \varphi^\eps_0$. As the PDE associated to $\cH_\kappa$ is of Fokker-Planck type, $\|\varphi^\eps_t\|_{L^1_y}=1$ for all $t\geq 0$;
applying \eqref{eq:generator_estimate}, one finds
\begin{align*}
    \frac{\dif}{\dif t} \frac{\| \varphi^\eps_t\|_{L^2_y}^2}{2}
    = -\langle \cH_\kappa \varphi^\eps_t, \varphi^\eps_t\rangle
    \lesssim - \| \varphi^\eps_t\|_{L^r_y}^2 + \|  \varphi^\eps_t\|_{L^2_y}^2
    \lesssim -\| \varphi^\eps_t\|_{L^2_y}^{2+\frac{4}{\gamma}} + \|  \varphi^\eps_t\|_{L^2_y}^2,
\end{align*}
where we used the interpolation estimate $\| \varphi^\eps_t\|_{L^2_y} \leq \| \varphi^\eps_t\|_{L^1_y}^{1-\theta} \| \varphi^\eps_t\|_{L^r_y}^{\theta} = \| \varphi^\eps_t\|_{L^r_y}^{\theta}$ with $\theta=\frac{\gamma}{\gamma+2}$.
From here, standard ODE type estimate yield
\begin{align*}
    \| \varphi^\eps_t\|_{L^2_y}^2 \lesssim t^{-\frac{\gamma}{2}}
\end{align*}
uniformly in $t\in (0,1)$, $\eps>0$ and $\kappa \in (0,1/2)$; we can then take $\eps\to 0^+$ to obtain the same estimate for $\| K^\kappa_t(x,\cdot)\|_{L^2_y}^2$.
Once \eqref{eq:semigroup_full_1} is proved, we can obtain \eqref{eq:semigroup_full_2} by duality. Indeed, for any $G\in L^2_x$ and $t>0$, it holds
\begin{align*}
|\langle G , e^{-\cH_\kappa t} G' \rangle |
=|\langle e^{-\cH_\kappa t} G , G' \rangle |
\leq
\| e^{-\cH_\kappa t} G \|_{C^0_x} \|G'\|_{TV}
\lesssim 
t^{-\frac{\gamma}{4}} \|G\|_{L^2_x} \|G'\|_{TV}
\end{align*}
where we used that, by parabolic smoothing, $e^{-\cH_\kappa t} G$ is a continuous bounded function, with supremum norm coinciding with the $L^\infty_x$-one.
Therefore $\| e^{-\cH_\kappa t} G' \|_{L^2_x}
\lesssim
t^{-\frac{\gamma}{4}} \|G'\|_{TV}$.
Then \eqref{eq:semigroup_full_2} is proved by concatenating this estimate with \eqref{eq:semigroup_full_1}.
\end{proof}

\begin{proof}[Proof of \autoref{thm:main.thm.dirac}]
    For given $\kappa>0$, $t\in (0,1)$ and $\mu_0\in\cM(\R^d)$, in light of \autoref{lem:L^2-criterion-selfcorrelation} and \eqref{eq:semigroup_full_2}, we have
    \begin{align*}
        \EE[\| \tilde \mu^\kappa_t\|_{L^2_x}^2]
        = \| G^\kappa_t\|_{L^\infty_x}
        = \| e^{-\cH_\kappa t} G_0\|_{L^\infty_x}
        \lesssim t^{-\frac{\gamma}{2}} \| G_0\|_{TV}
        = t^{-\frac{\gamma}{2}} \| \mu_0^-\ast\mu_0\|_{TV}
        \leq t^{-\frac{\gamma}{2}} \|\mu_0\|_{TV}^2,
    \end{align*}
    where the hidden constant does not depend on $\kappa$. The case $\kappa=0$ follows by taking $\kappa\to 0^+$ and using lower semicontinuity of the $L^2_{\omega,x}$-norm.
\end{proof}

\autoref{thm:main.thm.dirac} can be linked to Richardson's law (\autoref{thm:richardson}), yielding optimality of the estimate \eqref{eq:main.thm.dirac}, through the following deterministic lemma. 

\begin{lemma}
\label{lem:deterministic.lower.bound}
For any $d \geq  1$ there exists a constant $\gamma_d > 0$ such that, for any $\mu \in L^1_x \cap L^2_x$ with $\| \mu \|_{L^1_x}=1$, it holds
\begin{equation}\label{eq:deterministic.lower.bound}
\var (\mu) \| \mu \|_{L^2_x}^{\frac{4}{d}} \geq  \gamma_d .
\end{equation}
\end{lemma}

\begin{proof}
For any $\delta \in (0, 1/2)$, by Cauchy-Schwartz and Jensen inequalities it holds
\begin{equation}\label{eq:deterministic_variance_proof}\begin{split}
1 
& = 
\int_{\R^d} \int_{\R^d} \mu ( \mathrm{d} x) \mu ( \mathrm{d} y)\\
&\leq
\left( \int_{\R^d} \int_{\R^d} | x-y |^{2 \delta} \mu ( \mathrm{d} x) \mu ( \mathrm{d} y) \right)^{1 / 2} 
\left( \int_{\R^d} \int_{\R^d} | x-y |^{-2 \delta} \mu( \mathrm{d} x) \mu( \mathrm{d} y) \right)^{1/ 2}
\\
&\lesssim 
\var(\mu)^{\delta/2} 
\left( \int_{\R^d} \int_{\R^d} | x-y |^{-2\delta}\mu( \mathrm{d} x)\mu ( \mathrm{d} y) \right)^{1/2}.
\end{split}\end{equation}
Take $p \in (1,2)$ such that $1 /  p = 1 - \delta / d$. Hardy--Littlewood--Sobolev inequality gives
\[ 
\int_{\R^d} \int_{\R^d} | x-y |^{-2\delta}\mu( \mathrm{d} x)\mu ( \mathrm{d} y)
= 
\langle | \cdot |^{- 2 \delta} \ast \mu, \mu \rangle 
\lesssim 
\| \mu \|_{L^p_x}^2 
\lesssim
\| \mu \|_{L^2_x}^{\frac{4 \delta}{d}}
\]
where the last inequality comes by interpolation between $L^1_x$ and $L^2_x$. 
Inserting the above inequality in \eqref{eq:deterministic_variance_proof} and elevating both sides to power $2 / \delta$ then yields the conclusion.
\end{proof}

\begin{cor}
Under the same assumptions of the upper bound of \autoref{thm:richardson}, for $\mu_0=\delta_x$ and $\mu$ solution to \eqref{eq:stochastic.continuity}, it holds
\[\mathbb{E} [\| \mu_t \|_{L^2_x}^2] \gtrsim t^{-
  \frac{d}{2 (1 - \alpha)}},
  \quad
  \forall t \in (0,1).\]
In particular, the dependence on $t$ in estimate \eqref{eq:main.thm.dirac} is optimal.
\end{cor}

\begin{proof}
    By \eqref{eq:deterministic.lower.bound} and Jensen inequality it holds
    \begin{equation*}
        \EE[\| \mu_t\|_{L^2_x}^2] \gtrsim \EE[\var(\mu_t)^{-d/2}] \geq \EE[\var(\mu_t)]^{-d/2}
    \end{equation*}
    and the conclusion follows from \eqref{eq:richardson_upper}.
\end{proof}

We can strengthen the previous bound \eqref{eq:main.thm.dirac} to a functional one, up to the price of logarithmic corrections.
Correspondingly, we obtain a functional lower bound for Richardson's law; it nicely pairs the functional upper bound from \autoref{prop:richardson_upper_functional}.

\begin{cor}\label{cor:richardson_lower_functional}
Under the assumptions of \autoref{thm:main.thm.dirac}, let $\mu_0=\delta_x$ for some $x \in \mathbb{R}^d$ and $\mu$ the solution to \eqref{eq:stochastic.continuity}.
Then for any $\eps > 0$ it holds
\begin{equation}\label{eq:richardson_lower_functional1}
    \mathbb{E} \left[ \sup_{t \in (0, 1)} t^{\frac{d}{2 (1 - \alpha)}} | \log
    t -1|^{- 1 - \eps} \| \mu_t \|_{L^2_x}^2 \right] < \infty .
\end{equation}
In particular, there exists an $L^{d/2}_\omega$-integrable random variable $N$ such that $\PP$-a.s.
\begin{equation}\label{eq:richardson_lower_functional2}
    \var (\mu_t) \gtrsim \frac{1}{N}\, t^{\frac{1}{1 - \alpha}} | \log t-1
    |^{-\frac{2(1+\eps)}{d}} \quad \forall\,
    t\in (0,1).
\end{equation}
\end{cor}

\begin{proof}
Let us first prove \eqref{eq:richardson_lower_functional2} starting from \eqref{eq:richardson_lower_functional1}.
Setting $X:= \sup_{t \in (0, 1)} t^{\frac{d}{2 (1 - \alpha)}} | \log t-1 |^{- 1 - \eps} \| \mu_t \|_{L^2_x}^2$, so that $X\in L^1_\omega$, by \eqref{eq:deterministic.lower.bound} we get the $\PP$-a.s. estimate
\begin{align*}
    \var (\mu_t) \gtrsim \Big( X\,  t^{-\frac{d}{2 (1 - \alpha)}} | \log t -1|^{1+\eps} \Big)^{-\frac{2}{d}}
\end{align*}
which yields \eqref{eq:richardson_lower_functional2} for $N=X^{2/d}\in L^{d/2}_\omega$.

To prove \eqref{eq:richardson_lower_functional1}, note that in the divergence-free case $\mu$ is also a solution to the stochastic transport equation, which is Markovian and satisfies \autoref{rem:Lp_bound_divergence_free}.
In particular, for any fixed $s>0$, we have
 \[ \mathbb{E} \Big[\sup_{t \in [s, s + 1]} \| \mu_t \|_{L^2_x}^2\Big] \lesssim
    \mathbb{E} [\| \mu_s \|_{L^2_x}^2] . \]
Therefore, by virtue of \eqref{eq:main.thm.dirac}, for any $\eps>0$ we have
\begin{align*}
\mathbb{E} \left[ \sup_{t \in (0, 1)} t^{\frac{d}{2 (1 - \alpha)}} | \log t-1 |^{- 1 - \eps} \| \mu_t \|_{L^2_x}^2 \right] 
&\leq
\sum_{n\geq  0} \mathbb{E} \left[ \sup_{t \in [2^{- n - 1}, 2^{- n}]} t^{\frac{d}{2 (1 - \alpha)}} | \log t-1 |^{- 1 - \eps} \| \mu_t \|_{L^2_x}^2 \right]
\\
&\lesssim 
\sum_{n \geq  0} 2^{- n \frac{d}{2 (1 - \alpha)}} \, \mathbb{E} [\| \mu_{2^{- n - 1}} \|_{L^2_x}^2] \, (1+n)^{- 1 - \eps} 
\\
&\lesssim 
\sum_{n \geq  0} (1+n)^{- 1 - \eps} < \infty. \qedhere
 \end{align*}
\end{proof}

\begin{remark}
    In the incompressible Kraichnan model, using the Markov property, we can further concatenate \autoref{thm:main.thm.dirac} with \autoref{thm:regularity_Kraichnan}.
    In particular, starting from any $\mu_0\in \mathcal{M}(\R^d)$, the associated solution to \eqref{eq:viscous.stochastic.continuity} instantaneously becomes $L^\infty_x\cap H^{1-\alpha-}_x$-valued at positive times.
\end{remark}

\appendix
\section{Covariance of the Kraichnan model}\label{app:auxiliary}

Recall the characterization of the covariance $C$ associated to homogeneous isotropic Gaussian velocity fields on $\RR^d$ from \cite[section 12]{MoYa75}: there must exist two finite, non-negative measures $F_{\rm sol}$ and $F_{\rm grad}$ on $\R_+$ such that the Fourier transform of $C$ is given by
\begin{align*}
\hat{C}( \mathrm{d} \xi)
=
I_d\, \sigma( \mathrm{d} u) F_{\rm sol}( \mathrm{d} r) 
+
u\otimes u\, \sigma( \mathrm{d} u) (F_{\rm grad}( \mathrm{d} r)-F_{\rm sol}( \mathrm{d} r)),
\end{align*}
where $I_d$ is the identity matrix on $\RR^d$, $r := |\xi|$, $u := \xi/|\xi|$ and $\sigma$ is the normalized Haar measure on $\mathbb{S}^{d-1}=\{u\in\RR^d:|u|=1\}$.
The measures $F_{\rm sol}$ and $F_{\rm grad}$ correspond, respectively, to the divergence-free and gradient part of the noise. If they do not charge the singleton $\{0\}$, then for any $z \in \R^d \setminus \{0\}$ 
\begin{align}\label{eq:covariance.app}
C(z)
=
B_L(|z|) \hat{z}\otimes \hat z + B_N(|z|) \left( I_d-\hat{z}\otimes \hat z\right),\quad \hat{z}:=\frac{z}{|z|}
\end{align}
where the longitudinal and normal projections $B_L,B_N : \R_+ \to \R_+$ of the covariance are given by
\begin{equation}\label{eq:covariance_BLBN}\begin{split}
B_L(r)
:=
\int_0^\infty \int_{\mathbb{S}^{d-1}} \cos(\rho u_1 r) (1-u_1^2) \sigma( \mathrm{d} u) F_{\rm sol}( \mathrm{d} \rho)
+ 
\int_0^\infty \int_{\mathbb{S}^{d-1}} \cos(\rho u_1 r) u_1^2 \sigma( \mathrm{d} u) F_{\rm grad}( \mathrm{d} \rho),
\\
B_N(r)
:=
\int_0^\infty \int_{\mathbb{S}^{d-1}} \cos(\rho u_1 r) (1-u_2^2) \sigma( \mathrm{d} u) F_{\rm sol}( \mathrm{d} \rho)
+
\int_0^\infty \int_{\mathbb{S}^{d-1}} \cos(\rho u_1 r) u_2^2 \sigma( \mathrm{d} u) F_{\rm grad}( \mathrm{d} \rho).
\end{split}\end{equation}
It follows that $B_L(0) = B_N(0)$, implying that $C(0) = B_N(0) I_d$. Upon defining 
\begin{align} \label{eq:definition_b.app}
b_L(r) 
:= 
B_L(0) - B_L(r) ,
\quad
b_N(r) 
:= 
B_N(0) - B_N(r),
\end{align}
from \eqref{eq:covariance.app} we find
\begin{align}\label{eq:covariance_generalQ}
Q(z):= C(0)-C(z)
= b_L(|z|) \hat{z}\otimes \hat z + b_N(|z|) \left( I_d-\hat{z}\otimes \hat z\right)\quad \forall\,z \in \R^d \setminus \{0\}.
\end{align}
Within the above general framework, the Kraichnan model corresponds to the choice
\begin{align*}
F_{\rm grad}( \mathrm{d} r) := a F( \mathrm{d} r),
\qquad
F_{\rm sol}( \mathrm{d} r) := \frac{b}{d-1} F( \mathrm{d} r), \qquad F( \mathrm{d} r)
:=
\frac{r^{d-1}}{(r^2+m^2)^{\frac{d}{2}+\alpha}} \dd r.
\end{align*}
where $a,b\geq 0$ are such that $a+b>0$ and $m>0$. The \emph{incompressibility ratio} of the noise is
\begin{equation}\label{eq:incompressibility.ratio}
     \eta := \frac{b}{a+b}\in [0,1].
\end{equation}
Here $\eta=1$ corresponds to $a=0$ (divergence-free noise) and $\eta=0$ to $b=0$ (gradient noise).

The next lemma characterizes sharply the behaviour of $Q$ around $0$, guaranteeing that the Kraichnan model satisfies the assumptions imposed in the main body of the paper.

\begin{lemma}\label{lem:asymptotics_kraichnan}
Let $\alpha\in (0,1)$, $a$, $b$, $m$, $\eta$ as above and let $Q$ given by the Kraichnan model; define:
\begin{equation}\label{eq:defn_alpha1_c_beta}\begin{split}
\alpha_1& := \omega_{d-1} \bigg(\int_0^\infty (1-\cos x) x^{-1-2\alpha} \dd x\bigg) \bigg(\int_0^{\pi/2} (\cos \theta)^{2\alpha} (\sin \theta)^{d-2} \dd \theta \bigg),\\
c & := \alpha_1 (a+b) \frac{2\alpha+1-2\alpha \eta}{d+2\alpha} > 0,\\
\beta 
& :=
\frac{d-1+2\alpha \eta}{(d-1)(2\alpha+1-2\alpha \eta )},
\end{split}\end{equation}
where $c_d$ is as in \cite[Lemma 10.2]{LeJRai2002}. Then, as $r \downarrow 0$, for all $k\in \{0,1,2\}$ we have
\begin{equation}\label{eq:asymp123} 
\partial_r^{k} b_L(r) = c \, \partial_r^{k} (r^{2\alpha}) + m^{2-2\alpha}O(r^{2-k}),
\quad
\partial_r^{k} b_N(r) = \beta c \, \partial_r^{k} r^{2\alpha} + m^{2-2\alpha}O(r^{2-k})
\end{equation}
Moreover, away from $r=0$, $b_L$ and $b_N$ are smooth with bounded derivatives: for any $r_0>0$ and $k\in\NN$ it holds
\begin{equation}\label{eq:kraichnan_smooth_large_r}
    \sup_{r\geq r_0} \big[ |\partial^k_r b_L(r)| + |\partial^k_r b_N(r)|\big]<\infty.
\end{equation}
\end{lemma}

\begin{cor}\label{cor:assumptions_kraichnan}
    Let $C$ be the covariance given by the Kraichnan model, for given $\alpha\in (0,1)$, $\eta\in [0,1]$ and $m>0$. Then $C(0) = B_N(0) I_d$ with $B_N(0) > 0$. Moreover:
    \begin{itemize}
        \item Assumptions \ref{ass:well_posed} and \ref{ass:sharpness_reg} are satisfied for any choice of $\alpha$, $\eta$, $m$.
        \item Assumptions \ref{assumption} and \ref{ass:sharpness_reg_bis} are satisfied if and only if $\eta>1-\frac{d}{4\alpha^2}$.
    \end{itemize}
\end{cor}

\begin{proof}
The identity $C(0) = B_N(0) I_d$ has been discussed before, and the fact that $B_N(0) > 0$ descends from the requirement $a+b>0$. 
    Assumption \ref{ass:well_posed} immediately follows from the explicit formulae for $\hat C$, $F_{\rm sol}$, and $F_{\rm grad}$.
    By formula \eqref{eq:covariance_generalQ}, it holds
    \begin{align*}
        {\rm div}\, Q(z)& = \Big[ \partial_r b_L(r) + (b_L(r)-b_N(r)) \frac{d-1}{r} \Big]\, \hat z\\
        \nabla\cdot {\rm div}\, Q(z)& =\partial_r^2 b_L(r) + (2\partial_r b_L(r)-\partial_r b_N(r)) \frac{(d-1)}{r} + (b_L(r)-b_N(r)) \frac{(d-1)(d-2)}{r^2}
    \end{align*}
    which verifies \autoref{ass:sharpness_reg} thanks to \eqref{eq:asymp123}.
    Condition \eqref{eq:assumption_beta} in \autoref{assumption} is equivalent to
    \begin{align*}
        \frac{d-1+2\alpha \eta}{(d-1)(2\alpha+1-2\alpha \eta )}>\frac{2\alpha-1}{d-1}
    \end{align*}
    which after some algebraic manipulations can be seen equivalent to $\eta>1-\frac{d}{4\alpha^2}$.
    Having verified Assumptions \ref{assumption} and \ref{ass:sharpness_reg}, the remaining conditions from \autoref{ass:sharpness_reg_bis} follow from \eqref{eq:kraichnan_smooth_large_r} (uniform boundedness of $D^2:Q$ away from the origin) and \autoref{lem:nondegeneracy_kraichnan} below (uniform ellipticity of $Q$ away from the origin).    
\end{proof}

\begin{proof}[Proof of \autoref{lem:asymptotics_kraichnan}]
By \cite[Lemma 10.2]{LeJRai2002},
for $\alpha_1$ as defined above\footnote{Compared to our convention, \cite{LeJRai2002} considers $\alpha'=2\alpha\in (0,2)$.}, it holds
\begin{align} 
\label{eq:bL}
b_L(r) 
&=
\frac{(a+b)\alpha_1}{d+2\alpha}
\left( 2\alpha+1-2\alpha \eta \right) r^{2\alpha} + o(r^{2\alpha}),
\\
\label{eq:bN}
b_N(r) 
&=
\frac{(a+b)\alpha_1}{d+2\alpha}
\left(\frac{d-1+2\alpha \eta}{d-1}\right) r^{2\alpha} + o(r^{2\alpha}).
\end{align}
This identifies uniquely the parameters $c$ and $\beta$ above.
Recall that
\begin{align*}
b_L(r)
=
\frac{b}{d-1}
\int_0^\infty \int_{\mathbb{S}^{d-1}} (1-\cos(\rho u_1 r)) (1-u_1^2) \sigma( \mathrm{d} u) F( \dif \rho)
+
a 
\int_0^\infty \int_{\mathbb{S}^{d-1}} (1-\cos(\rho u_1 r)) u_1^2 \sigma( \mathrm{d} u) F( \dif \rho),
\\
b_N(r)
=
\frac{b}{d-1}
\int_0^\infty \int_{\mathbb{S}^{d-1}} (1-\cos(\rho u_1 r)) (1-u_2^2) \sigma( \mathrm{d} u) F( \mathrm{d} \rho)
+
a 
\int_0^\infty \int_{\mathbb{S}^{d-1}} (1-\cos(\rho u_1 r)) u_2^2 \sigma( \mathrm{d} u) F( \mathrm{d} \rho),
\end{align*}
where $F( \mathrm{d} \rho) = \frac{\rho^{d-1}}{(\rho^2+m^2)^{\frac{d}{2}+\alpha}} \dd\rho$.
It is hence sufficient to study the quantities
\begin{equation}\label{eq:bL_decomposition}\begin{split}
I^1(r)&:=\int_0^\infty \int_{\mathbb{S}^{d-1}} (1-\cos(\rho u_1 r)) \sigma( \mathrm{d} u) F( \mathrm{d} \rho),
\\
I^2(r)&:=\int_0^\infty \int_{\mathbb{S}^{d-1}} (1-\cos(\rho u_1 r)) u_1^2 \sigma( \mathrm{d} u) F( \mathrm{d} \rho),
\\
I^3(r)&:=\int_0^\infty \int_{\mathbb{S}^{d-1}} (1-\cos(\rho u_1 r)) u_2^2 \sigma( \mathrm{d} u) F( \mathrm{d} \rho),
\end{split}\end{equation}
and their derivatives with respect to $r$. For the sake of simplicity, we only study $I^1$, the other terms being similar.
Following the proof of \cite[Lemma 10.2]{LeJRai2002}, the change of variables $x = \rho u_1 r$ gives for some dimensional constant $c_d >0$
\begin{align*}
I^1(r)=
c_d r^{2\alpha}
\int_0^1
\left( 
\int_0^\infty
(1-\cos x)
\frac{x^{d-1}}{(x^2 + r^2 u_1^2 m^2)^{\frac{d}{2}+\alpha}} \dd x
\right)
u_1^{2\alpha}(1-u_1^2)^{\frac{d-2}{2}} \dd u_1.
\end{align*}
The derivative with respect to $r$ of the term above can be computed as
\begin{align*}
&2\alpha
c_d r^{2\alpha-1}
\int_0^1
\left( 
\int_0^\infty
(1-\cos x)
\frac{x^{d-1}}{(x^2 + r^2 u_1^2 m^2)^{\frac{d}{2}+\alpha}} \dd x
\right)
u_1^{2\alpha}(1-u_1^2)^{\frac{d-2}{2}} \dd u_1
\\
&\quad+
c_d r^{2\alpha}
\int_0^1
\left( 
\int_0^\infty
(1-\cos x)
\partial_r
\frac{x^{d-1}}{(x^2 + r^2 u_1^2 m^2)^{\frac{d}{2}+\alpha}} \dd x
\right)
u_1^{2\alpha}(1-u_1^2)^{\frac{d-2}{2}} \dd u_1
\\
&=
\frac{2\alpha}{r}
\int_0^\infty \int_{\mathbb{S}^{d-1}} (1-\cos(\rho u_1 r)) \sigma( \mathrm{d} u) F( \mathrm{d} \rho)
\\
&\quad-
c_d r^{2\alpha}
\int_0^1
\left( 
\int_0^\infty
(1-\cos x)
\frac{(d+2\alpha) u_1^2m^2 r x^{d-1}}{(x^2 + r^2 u_1^2 m^2)^{\frac{d}{2}+\alpha+1}} \dd x
\right)
u_1^{2\alpha}(1-u_1^2)^{\frac{d-2}{2}} \dd u_1,
\end{align*}
where taking the derivative inside the integral is justified by the fact that the derivative is absolutely integrable with respect to $\dd x \dd u_1$ for every $r>0$:
\begin{align*}
\int_0^1
\left( 
\int_0^\infty
|1-\cos x|
\frac{r x^{d-1}}{(x^2 + r^2 u_1^2 m^2)^{\frac{d}{2}+\alpha+1}}
\dd x\right)
u_1^{2+ 2\alpha}(1-u_1^2)^{\frac{d-2}{2}} \dd u_1
\\
\lesssim
\int_0^1
\left( 
\int_0^\infty
(1 \wedge x^2)
\frac{r x^{d-1}}{(x^2 + r^2 u_1^2 m^2)^{\frac{d}{2}+\alpha+1}} \dd x
\right) \dd u_1
< \infty.
\end{align*}
We can quantify the dependence of $r$ in the latter more precisely. 
This can be done by splitting the inner integral into the regions $x < ru_1m$ and $x> ru_1m$. Indeed, using $1 \wedge x^2 \leq 1$ in the first case we find
\begin{align*}
\left| 
\int_0^1\int_0^{ru_1m}
\frac{r x^{d+1}}{(x^2 + r^2 u_1^2 m^2)^{\frac{d}{2}+\alpha+1}} \dd x\, \dd u_1
\right|\lesssim
\left|\int_0^1\int_0^{ru_1m} \frac{r x^{d+1}}{(ru_1m)^{d+2\alpha+2}} \dd x\, \dd u_1
\right|\lesssim 
m^{-2\alpha} r^{1-2\alpha}. 
\end{align*}
In the second region we use instead 
\begin{align*}
\left| \int_0^1\int_{r u_1 m}^\infty
\frac{r x^{d+1}}{(x^2 + r^2 u_1^2 m^2)^{\frac{d}{2}+\alpha+1}} \dd x\, \dd u_1
\right|\lesssim
\left|  \int_0^1\int_{r u_1 m}^\infty \frac{rx^{d+1}}{ x^{d+2\alpha+2}} \dd x\, \dd u_1
\right|\lesssim 
m^{-2\alpha} r^{1-2\alpha}.
\end{align*}
Thus, putting all together we obtain
\begin{align*}
\partial_r b_L 
=
\frac{2\alpha}{r} b_L
+
m^{2-2\alpha}O(r),
\qquad
\partial_r b_N 
=
\frac{2\alpha}{r} b_N
+
m^{2-2\alpha}O(r),
\end{align*}
and \eqref{eq:asymp123} with $k=1$ descends from \eqref{eq:bL} and \eqref{eq:bN}.
From the previous line we deduce
\begin{align*}
\partial_r (b_L(r) r^{-2\alpha})
=
m^{2-2\alpha}O(r^{1-2\alpha}),
\qquad
\partial_r (b_N(r) r^{-2\alpha})
=
m^{2-2\alpha}O(r^{1-2\alpha}),
\end{align*}
Integrating in time, using \eqref{eq:bL}, \eqref{eq:bN}, and that $O(r^{1-2\alpha})$ is integrable at the origin since $1-2\alpha>-1$, one finds
\begin{align*}
b_L(r) - cr^{2\alpha}
=
m^{2-2\alpha}O(r^2),
\qquad
b_N(r) - \beta cr^{2\alpha} 
=
m^{2-2\alpha}O(r^2),
\end{align*}
giving \eqref{eq:asymp123} with $k=0$.

Let us move to estimate \eqref{eq:asymp123} with $k=2$, involving the second derivatives of $b_L$ and $b_N$.
As before, we can reduce ourselves to the study of the  function
\[ 
c_d r^{2\alpha}
\int_0^1
\left( 
\int_0^\infty
(1-\cos x)
\frac{x^{d-1}}{(x^2 + r^2 u_1^2 m^2)^{\frac{d}{2}+\alpha}} \dd x
\right)
u_1^{2\alpha}(1-u_1^2)^{\frac{d-2}{2}} \dd u_1. \]
In particular, also in view of the previous passages and avoiding unnecessary repetitions, it suffices to show that 
\[  
\left| \partial_r^2  \int_0^1
\left(  \int_0^\infty (1-\cos x)
\frac{x^{d-1}}{(x^2 + r^2 u_1^2 m^2)^{\frac{d}{2}+\alpha}} \dd x
\right)
u_1^{2\alpha}(1-u_1^2)^{\frac{d-2}{2}} \dd u_1
\right| \lesssim
m^{2-2\alpha} r^{- 2 \alpha} . \]
After some elementary passages, the latter boils down to proving
\begin{align*}
\left|
\int_0^1 \int_0^{\infty}  
\frac{r x^{d + 1}}{(x^2 + r^2
    u_1^2 m^2 )^{\frac{d}{2} + \alpha + 2}} \dd x \dd u_1 \right|
    \lesssim
  m^{-2\alpha} r^{- 2 \alpha - 1},
  \end{align*}
which follows, as before, by splitting the inner integral into the regions $x < ru_1m$ and $x> ru_1m$.

Finally, let us prove \eqref{eq:kraichnan_smooth_large_r}. As before, we only show bounds for $\partial^k_r I^1$, the other terms being similar.
Starting from \eqref{ass:sharpness_reg_bis}, applying the change of variables $x = \rho r$, it holds
\begin{align*}
    I^1(r)= \int_0^\infty \int_{\mathbb{S}^{d-1}} (1-\cos(x u_1)) x^{d-1} \frac{r^{2\alpha}}{(x^2+m^2 r^2)^{\frac{d}{2}+\alpha}}\sigma( \mathrm{d} u) \mathrm{d} x.
\end{align*}
For fixed $k\in\NN$, uniformly over $r\geq r_0$, one has (for some hidden constant depending on $d$, $\alpha$, $m$)
\begin{align*}
    \bigg| \partial^k_r \bigg( \frac{r^{2\alpha}}{(x^2+m^2 r^2)^{\frac{d}{2}+\alpha}}\bigg)\bigg|
    \lesssim r_0^{-k} \frac{r^{2\alpha}}{(x^2+m^2 r^2)^{\frac{d}{2}+\alpha}}
\end{align*}
and so that
\begin{align*}
    |\partial^k_r I^1(r)|
    & \lesssim \int_0^\infty x^{d-1} \bigg| \partial^k_r \bigg( \frac{r^{2\alpha}}{(x^2+m^2 r^2)^{\frac{d}{2}+\alpha}}\bigg)\bigg| \mathrm{d} x\\
    & \lesssim r_0^{-k} \int_0^\infty x^{d-1} \frac{r^{2\alpha}}{(x^2+m^2 r^2)^{\frac{d}{2}+\alpha}} \mathrm{d} x
    = r_0^{-k} \int_0^\infty  \frac{y^{d-1}}{(y^2+m^2)^{\frac{d}{2}+\alpha}} \mathrm{d} y<\infty. \qedhere
\end{align*}
\end{proof}

The next result gives uniform ellipticity of $Q$ far from the origin.

\begin{lemma}\label{lem:nondegeneracy_kraichnan}
Let $Q$ be given by the Kraichnan model.
Then for any $r_0>0$ there exists a constant $\kappa=\kappa(d,m,\alpha,a,b,r_0)>0$ such that
\begin{equation}\label{eq:nondegeracy_kraichnan}
Q(z)\geq \kappa I_d \quad \forall\, |z|\geq r_0.
\end{equation}
\end{lemma}

\begin{proof}
Recall that $Q(z)=C(0)-C(z)$, where $C(0)$ is a multiple of $I_d$ and by the Riemann--Lebesgue lemma $C(z)\to 0$ as $|z|\to\infty$; therefore the statement is automatically satisfied for $|z|\geq R$ with $R$ large enough.
It remains to cover the case of $r=|z|\in (r_0,R)$; we will show a more explicit estimate:
\begin{equation}\label{eq:nondegeneracy_claim}
	Q(z) \gtrsim |z|^{2\alpha} (1 + |z|^2)^{-\frac{d}{2}-\alpha+1} I_d \quad \forall\, z\in\R^d.
\end{equation}
As in \autoref{lem:asymptotics_kraichnan}, $b_L$, $b_N$ can be written in function of $I^i$ given by \eqref{eq:bL_decomposition}, therefore we will only show a lower bound for $I^1$, the others being similar; up to rescaling, we can assume $m=1$. Since $u_1^2\leq 1$, it holds
\begin{align*}
	I^1(r)
	& \sim r^{2\alpha} \int_0^1 \int_0^\infty (1-\cos x) \frac{x^{d-1} }{(x^2 + r^2 u_1^2)^{\frac{d}{2}+\alpha}} \dd x\,
u_1^{2\alpha}(1-u_1^2)^{\frac{d-2}{2}} \dd u_1\\
	& \geq r^{2\alpha} \int_0^\infty (1-\cos x) \frac{x^{d-1} }{(x^2 + r^2)^{\frac{d}{2}+\alpha}} \dd x\,
 \int_0^1 u_1^{2\alpha}(1-u_1^2)^{\frac{d-2}{2}} \dd u_1\\
 	& \sim \int_0^\infty (1-\cos (ry)) \frac{y^{d-1} }{(y^2 + 1)^{\frac{d}{2}+\alpha}} \dd y
\end{align*}
where we changed variables $y=rx$. Restricting the integral on the interval $ry\in (\pi/3, 2\pi/3)$, we find
\begin{align*}
	I^1(r)
	\gtrsim \int_{\frac{\pi}{3r}}^{\frac{2\pi}{3r}} \frac{y^{d-1} }{(y^2 + 1)^{\frac{d}{2}+\alpha}} \dd y
	\gtrsim \frac{\pi}{3r} \Big( \frac{\pi}{3r} \Big)^{d-1} \Big(1 + \frac{4\pi^2}{9 r^2} \Big)^{-\frac{d}{2}-\alpha};
\end{align*}
upon rearranging terms this yields $I^1(r)\gtrsim r^{2\alpha} (1+r^2)^{-d/2-\alpha}$ and thus \eqref{eq:nondegeneracy_claim}.
\end{proof}

\section{Complements to \autoref{sec:preliminaries}}
\label{app:preliminaries}
In this section we collect some complementary results to \autoref{sec:preliminaries}.

\subsection{Complements to \autoref{ssec:preliminaries_transp}}

\begin{proof}[Proof of \autoref{prop:solution}]

$1) \Leftrightarrow 2)$ amounts to the classical equivalence of weak and mild solutions, see \cite[Thm. 6.5]{DaPZab2014}.

$2) \Rightarrow 3)$.
We readapt the argument from \cite{Ma11}.
Since $\langle \theta_t ,\varphi \rangle \in L^2_\omega$ is $\mathcal{F}^W_t$-measurable, it admits a unique Wiener chaos decomposition; we only have to identify each term in the expansion.
Fix $m \in \NN$ and consider $\mathbf{K}=(K_1,\ldots,K_m)\in \NN^m$, with $K_i$ not necessarily distinct; let $K := \{ K_1,\dots,K_m \}$.
Let $\Pi^\mathbf{K}$ denote the projection on the Wiener chaos generated by iterated integrals with respect to $\dd W^{K_1}_{t_1} \dots \dd W^{K_m}_{t_m}$.

Since $\theta$ satisfies \eqref{eq:mild_sol}, for every $t \geq 0$ we have $\PP$-almost surely  
\begin{align*}
\langle \theta_t , \varphi \rangle
&=
\langle \theta_0 , P_t \varphi \rangle
+
\sum_{k_1\in\NN}
\int_0^t \langle \theta_{t_1} , \div (\sigma_{k_1} P_{t-{t_1}}\varphi) \rangle \dd W^{k_1}_{t_1}.
\end{align*}
Applying the projection $\Pi^\mathbf{K}$ to both sides, by properties of stochastic integrals we obtain
\begin{align*}
\Pi^\mathbf{K}\langle \theta_t , \varphi \rangle
&=
\Pi^\mathbf{K}\langle \theta_0 , P_t \varphi \rangle
+
\sum_{k_1 \in K}
\Pi^\mathbf{K}\int_0^t \langle \theta_{t_1} , \div (\sigma_{k_1} P_{t-{t_1}}\varphi) \rangle \dd W^{k_1}_{t_1},
\end{align*}
where now the sum over $k_1$ is finite. As such, we can iterate the formula in the integrand, and the resulting stochastic integral will be well-defined at each iteration by \autoref{lem:stoch_integr}. Therefore, recalling the definition of the operators $\mathcal{J}^{k}_t$ above, for every $N \geq 2$ and $t \geq 0$ we have $\PP$-almost surely 
\begin{align*}
\Pi^\mathbf{K}\langle \theta_t ,\varphi \rangle
&=
\Pi^\mathbf{K}\langle \theta_0 , P_t \varphi \rangle
+
\sum_{n = 1}^{N-1}
\Pi^\mathbf{K}J^n_t(\theta_0,\varphi)
\\
&\quad+
\sum_{k_1,\dots,k_N \in K}
\Pi^\mathbf{K}\int_0^t \int_0^{t_1} \dots \int_0^{t_{N}}
\langle
\theta_s  , \mathcal{J}^{k_N}_{t_{N-1}-t_N} \dots \mathcal{J}^{k_1}_{t-t_1} \varphi\rangle \,
\dd W^{k_1}_{t_1} \dots \dd W^{k_N}_{t_N}.
\end{align*}
The last iterated integral is orthogonal to any Wiener chaos of order $\leq N-1$ by direct computation. Since $N$ is arbitrary, 
\begin{align*}
\Pi^\mathbf{K}\langle \theta_t ,\varphi \rangle
&=
\Pi^\mathbf{K}\langle \theta_0 , P_t \varphi \rangle
+
\sum_{n = 1}^{m}
\Pi^\mathbf{K}J^n_t(\theta_0,\varphi).
\end{align*}
As $\mathbf{K}$ is arbitrary, the Wiener chaos decomposition of $\langle \theta_t ,\varphi \rangle$ is uniquely determined and \eqref{eq:wiener_sol} follows. Moreover, the series converges in $L^2_\omega$ as $N \to \infty$ by orthogonality of Wiener chaoses and the bound \eqref{eq:contraction_L2}. 

$3) \Rightarrow 2)$.
By \eqref{eq:wiener_sol} it holds $\PP$-almost surely 
\begin{align*}
\langle &\theta_t ,\varphi \rangle
-
\langle \theta_0 , P_t \varphi \rangle
\\
&=
\sum_{n \geq 1}
\sum_{k_1,\dots,k_n}
\int_0^t \int_0^{t_1} \dots \int_0^{t_{n-1}}
\langle 
\theta_0 , P_{t_n} \mathcal{J}^{k_n}_{t_{n-1}-t_n}  \dots \mathcal{J}^{k_1}_{t-t_1} \varphi  \rangle  \,
\dd W^{k_1}_{t_1} \dots \dd W^{k_n}_{t_n}
\\
&=
\sum_{k_1} \int_0^t 
\left( 
\langle \theta_0 , P_{t_1} \mathcal{J}^{k_1}_{t-t_1} \varphi  \rangle
+
\sum_{n \geq 2}
\sum_{k_2,\dots,k_n}
\int_0^{t_1} \dots \int_0^{t_{n-1}}
\langle 
\theta_0 , P_{t_n} \mathcal{J}^{k_n}_{t_{n-1}-t_n}  \dots \mathcal{J}^{k_1}_{t-t_1} \varphi \rangle  \,
\dd W^{k_2}_{t_2} \dots \dd W^{k_n}_{t_n} \right) \dd W^{k_1}_{t_1} 
\\
&=
\sum_{k_1} \int_0^t 
\langle \theta_{t_1} , \mathcal{J}^{k_1}_{t-t_1} \varphi  \rangle \, \dd W^{k_1}_{t_1},
\end{align*}
giving \eqref{eq:mild_sol}.

\textit{Uniqueness up to modifications}.
Let $\theta$ and $\tilde{\theta}$ two processes satisfying the properties above. Let $\{ \varphi_n \}_{n \in \NN} \subset C^\infty_c(\R^d)$ be a dense sequence in $C^\infty_c(\R^d)$. By point $3$), for every $t \geq 0$,  it holds $\PP$-almost surely $\langle \theta_t - \tilde{\theta}_t ,\varphi_n \rangle = 0$ for every $n \in \mathbb{N}$. 
Therefore $\theta_t = \tilde{\theta}_t$ as $L^2_x$-valued random variables
$\PP$-almost surely.
\end{proof}

\begin{lemma}\label{lem:stoch_integr_estim}
Let $\hat{C} \in L^1(\R^d) \cap L^\infty(\R^d)$.
Then for any $\eps>0$ and any $\gamma>0$ there exists $K := K(\eps,\gamma,\hat C)>0$ such that
\begin{equation}\label{eq:stoch_integr_estim}
	\| \cC^{1/2} g\|_{L^2_x}^2 
	\leq 
	\eps \| g\|_{TV}^2 
	+ 
	K \| g\|_{H^{-\gamma}_x},
	\quad \forall\, g\in \mathcal{M}\cap H^{-\gamma}_x.
\end{equation}
\end{lemma}

\begin{proof}
By Parseval's formula, for any $R>0$ it holds that
\begin{align*}
\| \cC^{1/2} g\|_{L^2_x}^2
& \sim 
\int_{\R^d} \hat C(\xi)\hat g(\xi)\cdot \hat g(\xi) \dd\xi
\\
& \leq 
\| \hat g\|_{L^\infty}^2 \int_{|\xi|>R} |\hat C(\xi)| \dd\xi
+ 
\| \hat C\|_{L^\infty} \int_{|\xi|\leq R} |\hat g(\xi)|^2 \dd\xi 
\\
& \lesssim 
\| g\|_{TV}^2 \int_{|\xi|>R} |\hat C(\xi)| \dd\xi
+
\| \hat C\|_{L^\infty} (1+R)^{2\gamma} \| g\|_{H^{-\gamma}_x}^2
.
\end{align*}
Since $\hat C\in L^1$, the first term can be made arbitrarily small up to taking $R>0$ large enough.
\end{proof}

\begin{prop}\label{prop:wellposedness_viscous_SPDE}
Suppose \autoref{ass:well_posed}, then for every $\kappa>0$ there exists a unique $\theta$ solution of \eqref{eq:spde_viscous} with initial condition $\theta_0\in L^2_x$,
and for any $m\in [1,\infty)$ and $T<\infty$ we have
\begin{equation}\label{eq:apriori_viscous_2}
\EE\left[ 
\left( \sup_{t\in [0,T]} \| \theta_t\|_{L^2_x}^2 
+ 
2\kappa \int_0^T \| \nabla \theta_t\|_{L^2_x}^2 \dd t \right)^m \right] 
\lesssim 
\| \theta_0\|_{L^2_x}^{2m},
\end{equation}
with implicit constant possibly depending upon ${m,T,\kappa}$. 
Moreover, for any $\theta_0 \in L^2_x \cap L^p_x$ with $p\in [1,\infty)$ it holds
\begin{equation}\label{eq:apriori_viscous_1}
\sup_{t\geq 0} \EE[\| \theta_t\|_{L^p_x}^p] 
\leq 
\| \theta_0\|_{L^p_x}^p.
\end{equation}
\end{prop}

\begin{proof}
The result can be proved with well-known techniques, here we only give a sketch of the proof. 

\textit{Step 1}.
Let us take a spatially smooth approximations of the noise $W^n$, and consider solutions to  $\{ \theta^{\kappa,n} \}$ to the approximate equation with noise $W^n$.
By the stochastic Lions Magenes Lemma \cite[Theorem 2.13]{RozLot2018} the approximate solutions satisfy
\begin{align*}
N_t 
:= 
\| \theta^{\kappa,n}_t\|_{L^2}^2 
+ 
2 \int_0^ t\kappa \|\nabla \theta^{\kappa,n}_s\|_{L^2_x}^2\dd s 
= 
2 \int_0^t \dd W^n_s \cdot \theta^{\kappa,n}_s \nabla\theta^\kappa_s  =: 
M_t.
\end{align*}
By Burkholder-Davies-Gundy inequality and \autoref{lem:stoch_integr}, for any $m\in [1,\infty)$ it holds
\begin{align*}
\EE\left[ \sup_{s\leq t} N_s^m \right]
\lesssim_m 
\| \theta_0\|_{L^2_x}^{2m} 
+ 
\EE\left[\left( 
\int_0^t \| \cC^{1/2} (\theta^{\kappa,n}_s \nabla\theta^{\kappa,n}_s)\|_{L^2_x}^2 \dd s \right)^{m/2} \right].
\end{align*}
Notice that $2\theta^{\kappa,n}\nabla\theta^{\kappa,n}=\nabla( |\theta^{\kappa,n}|^2)$ and that for $\gamma>d+1$ it holds
\begin{align*}
\| 2\theta^{\kappa,n}_t\nabla\theta^{\kappa,n}_t\|_{H^{-\gamma}_x}^2 
\leq 
\| |\theta^{\kappa,n}_t|^2\|_{H^{1-\gamma}_x}^2 
\leq 
\| |\theta^{\kappa,n}_t|^2\|_{L^1_x}^2 
= 
\| \theta^{\kappa,n}_t\|_{L^2_x}^4 \leq N_t^2, 
\end{align*}
while
\begin{align*}
\int_0^t \| \theta^{\kappa,n}_s \nabla\theta^{\kappa,n}_s\|_{L^1_x}^2\dd s
\leq 
\int_0^t \| \theta^{\kappa,n}_s\|_{L^2_x}^2 \|\nabla\theta^{\kappa,n}_s\|_{L^2_x}^2\dd s
\leq 
\sup_{s\leq t} \| \theta^{\kappa,n}_s\|_{L^2_x}^2 \int_0^t \|\nabla\theta^{\kappa,n}_s\|_{L^2_x}^2\dd s
\lesssim _\kappa
\sup_{s\leq t} N_s^{2}.
\end{align*}
Therefore, we can apply \autoref{lem:stoch_integr_estim} to find
\begin{align*}
\EE\left[\left( \int_0^t \| \cC^{1/2} (2\theta^{\kappa,n}_s \nabla\theta^{\kappa,n}_s)\|_{L^2_x}^2 \dd s \right)^{m/2} \right]
\leq 
\eps^{m/2} \EE \left[\sup_{s\leq t} N_s^{m}\right] + C_\kappa \EE\left[\left(\int_0^t N_s^2\dd s\right)^{m/2}\right].
\end{align*}
Choosing $\eps$ small enough one can reabsorb the first term on the left-hand side and arrive at
\begin{align*}
	\EE\left[\sup_{s\leq t} N_s^m\right]
	\lesssim_{m,\kappa} \| \theta_0\|_{L^2}^{2m} +  \EE\left[\left(\int_0^t N_s^2\dd s\right)^{m/2}\right],
\end{align*}
from which follows by Gr\"onwall lemma 
\begin{align} \label{eq:aux002}
\EE\left[ 
\left( \sup_{t\in [0,T]} \| \theta^{\kappa,n}_t\|_{L^2_x}^2 
+ 
2\kappa \int_0^T \| \nabla \theta^{\kappa,n}_t\|_{L^2_x}^2 \dd t \right)^m \right] 
\lesssim_{m,T,\kappa} 
\| \theta_0\|_{L^2_x}^{2m}
\end{align}
uniformly in $n \in \NN$.

\textit{Step 2}.
Let us move to the $L^p_x$ bounds, $p \in [1,\infty)$.
Given $f(r):=|r|^p$, take a sequence $f_m \uparrow f$ monotonically satisfying $f_m \in C^2(\R)$, $f_m'' \geq 0$ for every $m \in \NN$. Standard manipulations and It\^o Formula give for every $m,n \in \NN$
\begin{align*}
\EE\left[
\int_{\R^d} f_m(\theta^{\kappa,n}_t(x)) \dd x \right] 
\leq 
\int_{\R^d} f_n(\theta_0(x))\dd x
\leq 
\|\theta_0\|_{L^p_x}^p.
\end{align*}
Taking the supremum over $m \in \NN$ first, and then over $t \geq 0$, we get for every fixed $n \in \NN$
\begin{align} \label{eq:aux004}
\sup_{t \geq 0} \mathbb{E} \left[\| \theta_t^{\kappa,n} \|_{L^p_x}^p \right]
\leq
\|\theta_0\|_{L^p_x}^p.
\end{align}  

\textit{Step 3}.
Existence of progressively measurable solutions follows by compactness arguments. In particular, taking spatially smooth approximations of the noise $W^n$, one has a sequence of progressively measurable approximate solutions $\{ \theta^{\kappa,n} \}$ that are uniformly bounded in the spaces $L^{2m}_\omega L^2_{t,loc} H^1_x$ and $L^{2m}_\omega L^\infty_{t,loc} L^2_x$ for every $m \in \NN$ because of \eqref{eq:aux002}, and in the spaces $L^\infty_t L^p_\omega L^p_x$ for every $p \in \NN$ because of \eqref{eq:aux004}. By Banach-Alaoglu Theorem and a diagonal argument, one can extract a subsequence converging in the weak topology of $L^{2m}_\omega L^2_{t,loc} H^1_x$, in the weak-$\ast$ topology of $L^{2m}_\omega L^\infty_{t,loc} L^2_x$, and in the weak-$\ast$ topology of $L^{\infty}_t L^p_\omega L^p_x$ for every $m ,p \in \NN$ towards a process $\tilde{\theta}^\kappa$ satisfying 
\begin{align} \label{eq:aux003}
\EE\left[ 
\left( \mbox{ess}\sup_{t\in [0,T]} \| \tilde{\theta}^\kappa_t\|_{L^2_x}^2 
+ 
2\kappa \int_0^T \| \nabla \tilde{\theta}^\kappa_t\|_{L^2_x}^2 \dd t \right)^m \right] 
&\lesssim_{m,T,\kappa} 
\| \theta_0\|_{L^2_x}^{2m},
\\
\mbox{ess} \sup_{t \geq 0} \mathbb{E} \left[\| \theta_t^{\kappa,n} \|_{L^p_x}^p \right]
&\leq
\|\theta_0\|_{L^p_x}^p. \label{eq:aux005}
\end{align}

The limit is progressively measurable as the space of $\{\mathcal{F}_t \}_{t \geq 0}$-progressively measurable processes is weakly closed in $L^2_\omega L^2_{t,loc}$. 
As the stochastic integral
$f\mapsto \int_0^\cdot \langle f,\dd W\rangle$
is a bounded linear operator from progressively measurable $f \in L^2(\Omega \times [0,T] \times \R^d)$ to $L^2(\Omega \times [0,T] )$ for every $T<\infty$, they are also continuous with respect to the weak convergence and therefore we conclude that for every test function $\varphi \in C^\infty_c(\R^d)$ the following identity holds for almost every $(\omega,t) \in \Omega \times \R_+$
\begin{align} \label{eq:aux001}
\langle \tilde{\theta}^\kappa_t , \varphi \rangle
=
\langle \theta_0 , \varphi \rangle
-
\int_0^t
\langle \varphi \nabla \tilde{\theta}^\kappa_s , \dd W_s \rangle 
+
\int_0^t
\langle \tilde{\theta}^\kappa_s , A \varphi  \rangle\dd s.
\end{align} 
In order to improve the inequality above to every $t \geq 0$, we argue as follows.
Let $\{\varphi_n\}_{n \in \mathbb{N}}$ be a dense subsequence in $C^\infty_c(\R^d)$.
Notice that, for every $\varphi=\varphi_n$, the right-hand side of the expression above defines a stochastic process with continuous trajectories $\PP$-almost surely. 
Since $\tilde{\theta}^\kappa \in L^2_\omega L^\infty_{t,loc} L^2_x$, there exists a full measure probability set $\Omega' \subset \Omega$ such that for every $\omega \in \Omega'$ there exists a full measure set $\mathcal{T} \subset \R_+$ such that the maps $\{ \tilde{\theta}^\kappa_t(\omega)\}_{t \in \mathcal{T} \cap [0,T]}$ are uniformly bounded in $L^2_x$ for every $T<\infty$ and for every $n \in \NN$ the equation \eqref{eq:aux001} holds with $\varphi = \varphi_n$ for every $\omega \in \Omega'$ and $t \in \mathcal{T}$.
Given arbitrary $\omega \in \Omega'$ and $t \geq 0$, let $\{t_m\}_{m \in \NN} \subset \mathcal{T}$ be such that $t_m \to t$. 
Since $\{ \tilde{\theta}^\kappa_{t_m}(\omega)\}_{m \in \NN}$ are uniformly bounded in $L^2_x$, there exists a weakly convergent subsequence towards some $\theta^\kappa_t$. Moreover, up to relabelling the converging subsequence, for every $\omega \in \Omega'$, $t \geq 0$ and $n \in \NN$ it holds
\begin{align*}
\langle \theta^\kappa_t , \varphi_n \rangle
&=
\lim_{m \to \infty}
\langle \tilde{\theta}^\kappa_{t_m} , \varphi_n \rangle
\\
&=
\lim_{m \to \infty}
\left(
\langle \theta_0 , \varphi_n \rangle
-
\int_0^{t_m}
\langle \varphi_n \nabla \tilde{\theta}^\kappa_s , \dd W_s \rangle 
+
\int_0^{t_m}
\langle \tilde{\theta}^\kappa_s , A \varphi_n  \rangle\dd s
\right)
\\
&=
\langle \theta_0 , \varphi_n \rangle
-
\int_0^t
\langle \varphi_n \nabla \tilde{\theta}^\kappa_s , \dd W_s \rangle 
+
\int_0^t
\langle \tilde{\theta}^\kappa_s , A \varphi_n  \rangle\dd s,
\end{align*}
where the last equality comes by continuity of the right-hand side of \eqref{eq:aux001}. Hence $\theta^\kappa_t$ does not depend of the choice of the subsequence $\{t_m\}_{m \in \NN}$. In particular, $\theta^\kappa_t(\omega) = \tilde{\theta}^\kappa_t(\omega)$ for every $\omega \in \Omega'$ and $t \in \mathcal{T}$, and therefore for every $\omega \in \Omega'$ and $n \in \NN$ it holds
\begin{align*}
\langle \theta^\kappa_t , \varphi_n \rangle
&=
\langle \theta_0 , \varphi_n \rangle
-
\int_0^t
\langle \varphi_n \nabla \tilde{\theta}^\kappa_s , \dd W_s \rangle 
+
\int_0^t
\langle \tilde{\theta}^\kappa_s , A \varphi_n  \rangle\dd s
\\
&=
\langle \theta_0 , \varphi_n \rangle
-
\int_0^t
\langle \varphi_n \nabla \theta^\kappa_s , \dd W_s \rangle 
+
\int_0^t
\langle \theta^\kappa_s , A \varphi_n  \rangle\dd s,
\quad
\forall t \geq 0.
\end{align*}
Since $\{\varphi_n\}_{n \in \NN}$ are dense in $C^\infty_c(\R^d)$, the equality above extends to every $\varphi \in C^\infty_c(\R^d)$.
Moreover, $\theta^\kappa$ is $\PP$-almost surely weakly continuous (since the right-hand side of the equation above is $\PP$-almost surely continuous) and weakly $\{\mathcal{F}_t\}_{t \geq 0}$-progressively measurable (since the right-hand side of the equation above is $\{\mathcal{F}_t\}_{t \geq 0}$-adapted). By Pettis measurability theorem, $\theta^\kappa$ is progressively measurable.
We conclude that $\theta^\kappa$ is a weak solution of \eqref{eq:spde_viscous}. 
In addition, since $\theta^\kappa_t(\omega) = \tilde{\theta}^\kappa_t(\omega)$ for every $\omega \in \Omega'$ and $t \in \mathcal{T}$ and $\tilde{\theta}^\kappa$ satisfies \eqref{eq:aux003} and \eqref{eq:aux005}, by weak continuity $\theta^\kappa$ satisfies \eqref{eq:apriori_viscous_2} and \eqref{eq:apriori_viscous_1}.

\textit{Step 4}.
For pathwise uniqueness, by linearity it suffices to check that solutions starting at $0$ stay $0$. One can verify that the assumptions of the stochastic Lions-Magenes lemma are satisfied (cf. \cite[Thm 2.13]{RozLot2018}) and so the solution satisfies
\begin{align*}
	\frac{\dd}{\dd t} \EE[\| \theta^\kappa_t\|_{L^2}^2] + 2\kappa \EE[\|\nabla \theta^\kappa_t\|_{L^2}^2] = 0
\end{align*}
by which we can conclude by Gr\"onwall lemma.
\end{proof}

\begin{proof}[Proof of \autoref{lem:approx} for $p=2$]
For every $T<\infty$, the random variables $\hat{\theta}^\kappa$ are bounded in $L^2([0,T] \times \Omega \times \R^d)$, uniformly in $\kappa \in (0,1)$, and progressively measurable.
By Banach-Alaoglu, there exists a weakly converging subsequence in $L^2_{t,loc} L^2_\omega L^2_x$ towards a limit which retain the progressive measurability.
The limit satisfies condition $1$) of \autoref{prop:solution} by direct verification, therefore coincides with $\theta$ up to taking a modification. By uniqueness of the limit, it descends immediately that the whole sequence converges weakly as $\kappa \downarrow 0$. The weak convergence at fixed $t \geq 0$ then follows using the characterization \eqref{eq:weak_sol} of solutions, the density of $L^2_\omega \otimes C^\infty_c(\R^d)$ in $L^2_{\omega,x}$, and the bound \eqref{eq:contraction_L2}.

Let us move to the strong convergence in $L^2_\omega L^2_x$ of $\tilde{\theta}^\kappa_t$.
By \autoref{prop:solution_LJR}, for every $\varphi \in L^2_x$ and $t \geq 0$ we have $\PP$-almost surely 
\begin{align*}
\langle \tilde{\theta}^\kappa_t ,\varphi \rangle
=
\langle \theta_0 , P_t \varphi \rangle
+
\sum_{n \geq 1}
(1-\kappa)^{n/2}
J^n_t(\theta_0,\varphi).
\end{align*}
For every $n \in \NN$ fixed, each term of the series above converges pointwise as $\kappa \downarrow 0$ towards the corresponding $n$-th order chaos of $\langle \theta_t ,\varphi \rangle$. Moreover, by orthogonality of Wiener chaoses and since $(1-\kappa)^{n/2} \leq 1$ for every $\kappa \in (0,1/2)$ and $n \in \NN$, the convergence is dominated and therefore $\mathbb{E} \langle \tilde{\theta}^\kappa_t ,\varphi \rangle^2 \uparrow \mathbb{E}\langle \theta_t ,\varphi \rangle^2$ monotonically.
In order to conclude, it is sufficient to show convergence of the $L^2_\omega L^2_x$ norms as $\kappa \downarrow 0$.   
Take a complete orthonormal system $\{ e_n \}_{n \in \NN} \subset L^2_x$. By the above convergence with $\varphi = e_n$ we have
$\langle \tilde{\theta}^\kappa_t ,e_n \rangle^2  \to \langle \theta_t ,e_n \rangle^2$ in $L^1_\omega$, moreover for every $\eps>0$ there exists $N \in \NN$ sufficiently large such that uniformly in $\kappa \in (0,1/2)$
\begin{align*}
\sum_{n > N} \mathbb{E} \langle \tilde{\theta}^\kappa_t , e_n \rangle^2
\leq
\sum_{n > N} \mathbb{E} \langle \theta_t , e_n \rangle^2
< \eps
\end{align*}
since $\theta_t \in L^2_{\omega} L^2_x$, from which we deduce $\mathbb{E}[\| \tilde{\theta}^\kappa_t \|_{L^2_x}^2] \to \mathbb{E}[\| \theta_t \|_{L^2_x}^2]$.
\end{proof}

\begin{proof}[Proof of \autoref{prop:properties.inviscid.Kraichnan}]

(1) Let $\theta$ be the unique $\cF$-progressively measurable solution to \eqref{eq:spde_viscous} satisfying \eqref{eq:contraction_L2} and fix $s<t$; define $\tilde W^s_r:=W_{r+s}-W_s$, so that $\tilde W^s$ is independent of $\cF_s$, thus in particular of $\theta_s$.
We claim that, for any $\varphi\in C^\infty_c(\R^d)$, $\PP$-a.s. it holds
\begin{align}\label{eq:claim_markovianity}
\langle \theta_t ,\varphi \rangle
=
\langle \theta_s , P_{t-s} \varphi \rangle
+
\sum_{n \geq 1}
\tilde J^n_{t-s}(\theta_s,\varphi),
\end{align}
where $\tilde J^n_r$ are defined as $J^n_r$ but the stochastic integrals are with respect to $\tilde W^s$.
\eqref{eq:claim_markovianity} implies that $\EE[\langle\theta_t,\varphi\rangle|\cF_s]=\langle \theta_s,P_{t-s}\varphi\rangle$ for all $\varphi\in C^\infty_c(\R^d)$, thus $\EE[\theta_t|\cF_s]=P_{t-s}\theta_s$.
Similarly, \eqref{eq:claim_markovianity} and measure-theoretic arguments (Doob-Dynkin lemma, Pettis measurability theorem) imply the existence of a measurable $F$ such that $\theta_t=F(\theta_s,\tilde W^s)$; since $\tilde W^s$ is independent of $\theta_s$, Markovianity of $\theta$ follows.

Let us show \eqref{eq:claim_markovianity}. By first conditioning on $\theta_s$ and using independence of $\tilde W^s$, for $\PP$-a.e. realization of $\theta_s$ in $L^2_x$, we can apply \autoref{prop:solution_LJR} and \autoref{prop:solution}(3) to deduce the existence of a solution $\tilde \theta$ to the SPDE \eqref{eq:spde_viscous} driven by $\tilde W^s$, with random initial condition $\theta_s$, such that, for any $r\geq 0$ and $\varphi\in C^\infty_c(\R^d)$
\begin{align}\label{eq:claim_markovianity2}
    \langle \tilde\theta_r ,\varphi \rangle
=
\langle \theta_s , P_r \varphi \rangle
+
\sum_{n \geq 1}
\tilde J^n_r(\theta_s,\varphi).
\end{align}
Moreover by construction
\begin{equation}\label{eq:energy_markovianity}
    \EE[\|\tilde\theta_r\|_{L^2_x}^2 | \theta_s]\leq \| \theta_s\|_{L^2_x}^2 \quad \PP\text{-a.s.}
    \quad\Rightarrow\quad
    \EE\| \tilde\theta_r\|_{L^2_x}^2 \leq \EE\| \theta_s\|_{L^2_x}^2\leq \|\theta_0\|_{L^2_x}^2.
\end{equation}
Define $\theta'_r:= \theta_r \mathbf{1}_{r\leq s} + \tilde \theta_{r+s} \mathbf{1}_{r> s}$; by construction, since $\theta$ was a solution to \eqref{eq:spde_viscous} driven by $W$ and $\rmd \tilde W^s_r=\rmd W_{r+s}$, $\theta'$ is also a solution to \eqref{eq:spde_viscous} driven by $W$ starting at $\theta_0$. Furthermore, $\tilde\theta$ is $\cF^s_r$-progressively measurable, for $\cF^s_r:=\sigma(\theta_s,\{\tilde W^s_u\}_{u\leq r})\subset \cF_{s+r}$, so $\theta'$ is $\cF$-progressively measurable; by \eqref{eq:energy_markovianity}, $\theta'$ satisfies \eqref{eq:contraction_L2}.
By \autoref{prop:solution} we deduce that $\PP$-a.s. $\theta_t=\theta'_t=\tilde\theta_{t-s}$, so that \eqref{eq:claim_markovianity} follows from \eqref{eq:claim_markovianity2}.

$(2)$
Let $\theta_0 \in L^2_x \cap L^p_x$ with $p \in [1,\infty)$.
By \autoref{lem:approx} applied with $p=2$, for every $t \geq 0$ the function $\theta_t$ can be obtained as limit in $L^2_\omega L^2_x$ of $\tilde{\theta}_t^\kappa$, as $\kappa \downarrow 0$. 
Since the functions $\{\tilde{\theta}^\kappa_t\}_{\kappa \in (0,1/2)}$ are uniformly bounded in $L^p_\omega L^p_x$ by $\| \theta_0 \|_{L^p_x}$ for every $p \in [1,\infty)$ in virtue of \autoref{prop:wellposedness_viscous_SPDE}, the same bound holds true for $\theta_t$. 

For $p=\infty$ we can argue as follows.
By Markov inequality we have for any $\delta>0$ and $q \in (2,\infty)$
\begin{align*}
(\mathbb{P}\otimes Leb)& 
\Big\{ \theta_t (\omega,x) > (1+\delta)(\|\theta_0\|_{L^\infty_x}+\delta\|\theta_0\|_{L^2_x}) \Big\}
\\
&\leq
\frac{\mathbb{E} [\|\theta_t\|_{L^q_x}^q]}{ (1+\delta)^q(\|\theta_0\|_{L^\infty} +\delta\|\theta_0\|_{L^2_x})^q}. 
\end{align*}
Now we apply the $L^p$-estimate previously proved for $p=q$ and Young inequality with exponents $q/(q-2)$ and $q/2$ to get
\begin{align*}
\mathbb{E} [\|\theta_t\|_{L^q_x}^q]
\leq
\|\theta_0\|_{L^q_x}^q
\leq
\left(
\|\theta_0\|_{L^\infty_x}^{\frac{q-2}{q}} \|\theta_0\|_{L^2_x}^{\frac{2}{q}} \right)^q
\leq
\left(\frac{q-2}{q} \|\theta_0\|_{L^\infty_x}
+ 
\frac{2}{q} \|\theta_0\|_{L^2_x} \right)^q.
\end{align*} 
Taking $q > 2/\delta$ we deduce from the above 
\begin{align*}
(\mathbb{P}\otimes \mathscr{L}^d)& 
\Big\{ \theta_t (\omega,x) > (1+\delta)(\|\theta_0\|_{L^\infty_x}+\delta\|\theta_0\|_{L^2_x}) \Big\}
\leq
(1+\delta)^{-q},
\end{align*}
giving $\theta_t \leq \|\theta_0\|_{L^\infty_x}$ for $(\mathbb{P}\otimes \mathscr{L}^d)$ a.e. $(\omega,x)$ since $q<\infty$ and $\delta>0$ are arbitrary.
This gives all the desired a priori estimates.

$(3)$ 
The case $p=2$ descends immediately by $(1)$ and \cite[Theorem 3.2, point (b)]{LeJRai2002}.
For $p \neq 2$ we argue as follows.
Rewrite $\theta_0 = \Theta_0 + (\theta_0 - \Theta_0)$ with $\Theta_0 \in L^1_x \cap L^\infty_x$ and $\|\theta_0-\Theta_0\|_{L^p_x} \leq \eta$, for some $\eta$ to be chosen later.
By linearity of the equation and point $(3)$, we have $\sup_{t \geq 0}\mathbb{E}[\|\theta_t-\Theta_t\|_{L^p_x}^p] \leq \eta^p$. Furthermore $t \mapsto \mathbb{E}[\|\Theta_t\|_{L^p}^p]$ is uniformly continuous by uniform continuity of $t \mapsto\mathbb{E}[\|\Theta_t\|_{L^2}^2]$, the fact that $\Theta_0 \in L^1_x \cap L^\infty_x$, point $(3)$ and interpolation.
Therefore for every $s,t$
\begin{align*}
 |\mathbb{E}[\| \theta_{t}\|_{L^p_x}^p] - \mathbb{E}[\| \theta_s\|_{L^p_x}^p]|
 &\leq
 |\mathbb{E}[\| \theta_t\|_{L^p_x}^p] - \mathbb{E}[\| \Theta_t\|_{L^p_x}^p]|
 +
 |\mathbb{E}[\| \Theta_s\|_{L^p_x}^p] - \mathbb{E}[\| \theta_s\|_{L^p_x}^p]|
 \\
 &\quad+
 |\mathbb{E}[\| \Theta_t\|_{L^p_x}^p] - \mathbb{E}[\| \Theta_s\|_{L^p_x}^p]|
 \\
& \leq 
C(p,\|\theta_0\|_{L^p_x}, \| \Theta_0\|_{L^p_x}) \eta^p 
+ |\mathbb{E}[\| \Theta_t\|_{L^p_x}^p] - \mathbb{E}[\| \Theta_s\|_{L^p_x}^p]|.
\end{align*}
Now let $\varepsilon>0$ be given, and choose $\eta$ such that $C(p,\|\theta_0\|_{L^p_x}, \| \Theta_0\|_{L^p_x}) \eta^p < \varepsilon/2$. For this value of $\eta$, the associated $t \mapsto\mathbb{E}[\|\Theta_t\|_{L^2}^2]$ is uniformly continuous and therefore there exists $\delta>0$ such that $|\mathbb{E}[\| \Theta_t\|_{L^p_x}^p] - \mathbb{E}[\| \Theta_s\|_{L^p_x}^p]| \leq \varepsilon/2$ for every $|s-t|<\delta$, giving
\begin{equation*}
     |\mathbb{E}[\| \theta_{t}\|_{L^p_x}^p] - \mathbb{E}[\| \theta_s\|_{L^p_x}^p]| \leq \varepsilon,
     \quad\forall |t-s|<\delta. \qedhere
\end{equation*}
\end{proof}

\begin{proof}[Proof of \autoref{lem:approx} for $p \neq 2$]
Weak convergence of $\hat{\theta}^\kappa_t$ to $\theta_t$ in $L^p_{\omega} L^p_x$ descends from the weak convergence in $L^2_\omega L^2_x$ combined with the bound from \autoref{prop:properties.inviscid.Kraichnan}(4).
Let us move to the strong convergence of diffusive approximations.
Fix $\delta>0$ and rewrite $\theta_0 = \Theta_0 + (\theta_0 - \Theta_0)$ with $\Theta_0 \in L^1_x \cap L^\infty_x$ and $\|\theta_0-\Theta_0\|_{L^p_x} \leq \delta$.
By interpolation and \autoref{lem:approx} with $p \neq 2$ it holds
$\|\tilde{\Theta}^\kappa_t - {\Theta}_t \|_{L^p_\omega L^p_x} \to 0$ as $\kappa \downarrow 0$.
Thus, by linearity of \eqref{eq:spde_viscous_approx} and the stability estimate \autoref{prop:properties.inviscid.Kraichnan}, point $(4)$, we have
\begin{align*}
\limsup_{\kappa \downarrow 0}\|\tilde{\theta}^\kappa_t -{\theta}_t \|_{L^p_\omega L^p_x}
\lesssim
\limsup_{\kappa \downarrow 0} \left( \|\tilde{\theta}^\kappa_t - \tilde{\Theta}^\kappa_t \|_{L^p_\omega L^p_x}
+
\|\tilde{\Theta}^\kappa_t - {\Theta}_t \|_{L^p_\omega L^p_x}
+
\|{\Theta}_t- {\theta}_t \|_{L^p_\omega L^p_x} \right)
\lesssim
\delta. 
\end{align*}
Since $\delta>0$ is arbitrary we conclude that $\|\tilde{\theta}^\kappa_t -{\theta}_t \|_{L^p_\omega L^p_x} \to 0$.
\end{proof}

\begin{lemma} \label{lem:tightness}
Suppose \autoref{ass:well_posed}. For any $T,\eps>0$ there exists $R>0$ large enough such that
\begin{align*}
\sup_{\kappa \in (0,1/2)}
\left(
\sup_{t\in [0,T]} 
\EE[\| \tilde{\theta}^\kappa_t \mathbf{1}_{\{|x|>R\}}\|_{L^2_x}^2] 
+ 
\EE\left[\int_0^T \kappa \| \mathbf{1}_{\{|x|>R\}} C(0) :  \nabla \tilde{\theta}^\kappa_t(x) \otimes  \nabla \tilde{\theta}^\kappa_t(x) \|_{L^2_x}^2 
\dd t\right] 
\right)
<\eps.
\end{align*}
\end{lemma}

\begin{proof}
Let $\psi$ be a smooth nonnegative function such that $\psi\equiv 0$ on $B_{1/2}$ and $\psi\equiv 1$ on $B_1^c$ and set $\psi^R:= \psi(\cdot/R)$. Then testing \eqref{eq:spde_viscous_approx} against $\tilde{\theta}^\kappa \psi^R$, integrating in space and taking expectation 
\begin{align*}
	\EE&\left[ \int_{\R^d} |\tilde{\theta}^\kappa_t(x)|^2 \psi^R(x)\dd x + \kappa \int_0^t \int_{\R^d} C(0) :  \nabla \tilde{\theta}^\kappa_t(x) \otimes  \nabla \tilde{\theta}^\kappa_t(x) \psi^R(x)\dd x\right]
	\\
	&\lesssim \int_{\R^d} |\theta_0(x)|^2 \psi^R(x)\dd x + \frac{1}{R^2} \int_0^T \EE[\| \tilde{\theta}^\kappa_t\|_{L^2_x}^2]\dd t, 
\end{align*}
with implicit constant independent of $\kappa \in (0,1/2)$.
Taking supremum over $t \in [0,T]$ and $\kappa \in (0,1/2)$ on the left-hand side and sending $R\to\infty$ gives the conclusion.
\end{proof}

\subsection{Complements to \autoref{ssec:continuity}}

\begin{lemma}
\label{lem:estim.stoch.integrals}
Let $f$ be a progressively measurable process with values in $\mathcal{M}$, satisfying the bound $\sup_{r \geq  0} \| f_r \|_{\mathrm{TV}} \leq K$ $\mathbb{P}$-almost surely for some deterministic $K$. 
Then for any $\gamma < 1 / 2$, $\eps > 0$, $T \in (0, + \infty)$ and $p \in [2,\infty)$ we have
\[ \mathbb{E} \left[ \left\| \int_0^{\cdot} f_r \dd W_r \right\|^p_{C^{\gamma} ([0, T] ; H^{- d / 2 - \eps}_x)} \right]
\lesssim 
\| \hat{C} \|_{L^1}^{1 / 2} K. 
\]
\end{lemma}

\begin{proof}
Recall the by assumption the covariance of the noise $C$ is such that $| \hat{C} | \in L^1 (\mathbb{R}^d) \cap L^{\infty} (\mathbb{R}^d)$.
In particular, the associated convolution operator $\mathcal{C}$ satisfies
\begin{align*}
\| \mathcal{C}^{1/2} \varphi \|_{L^2_x}^2 
= 
\langle \mathcal{C} \varphi, \varphi \rangle_{L^2_x} 
\lesssim 
\| \varphi \|_{\mathrm{TV}}^2  
\| \hat{C} \|_{L^1},
\quad
\forall \varphi \in \mathcal{M} (\mathbb{R}^d) .
\end{align*}
Set $I_t \mapsto \int_0^t f_r \dd W_r$, and let $\varphi \in C^\infty_c(\R^d)$ be a smooth test  function. By the above considerations and \autoref{lem:stoch_integr}, the martingale
\[ M^{\varphi}_t \mapsto \langle \varphi, I_t \rangle 
= 
\int_0^t \langle \varphi f_r, \dd W_r \rangle \]
has quadratic variation satisfying the $\PP$-almost sure bound
\[ \frac{\dd}{\dd t} [M^{\varphi}]_t 
= 
\| \mathcal{C}^{1 / 2} (\varphi f_r) \|_{L^2}^2 \lesssim 
\| \hat{C} \|_{L^1} \| \varphi f_r\|_{\mathrm{TV}}^2 
\leq
\| \varphi \|_{C^0}^2 \| \hat{C} \|_{L^1} K^2. \]
Let now take any $p \in [2, \infty)$. By Burkholder-Davis-Gundy inequality, it follows that
  \[ \| M^{\varphi}_t - M^{\varphi}_s \|_{L^p (\Omega)} \lesssim_p \| \varphi
     \|_{C^0} \| \hat{C} \|^{1 / 2}_{L^1} K | t - s |^{1 / 2} . \]
  Taking $\varphi (x) = e^{i \xi \cdot x}$,  the above estimate
  becomes (dropping $\| \hat{C} \|^{1 / 2}_{L^1} K$ henceforth for
  simplicity)
  \[ \| \hat{I}_t (\xi) - \hat{I}_s (\xi) \|_{L^p (\Omega)} \lesssim | t - s
     |^{1 / 2} . \]
  Interpreting the $H^{- s}_x$-norm as the $L^2 (\mathbb{R}^d, (1 + | x  |^2)^{- s} \dd x)$ of the Fourier transform, by Minkowski's inequality we get
  \begin{align*}
    \| \| I_t - I_s \|_{H^{- d / 2 - \eps} (\mathbb{R}^d)} \|_{L^p
    (\Omega)} & \leq\left( \int \| \hat{I}_t (\xi) - \hat{I}_s (\xi)
    \|_{L^p (\Omega)}^2  (1 + | \xi |^2)^{- d / 2 - \eps} \dd\xi
    \right)^{1 / 2} \lesssim  | t - s |^{1 / 2} .
  \end{align*}  
As the estimate holds for any $p \in [2, \infty)$, conclusion follows by Kolmogorov's continuity theorem. 
\end{proof}

\begin{proof}[Proof of \autoref{prop:existence.SCE}]
For simplicity, we present the construction on a fixed finite interval $[0,T]$; the general case follows from a standard diagonal argument.
Let $W^n = \Pi_n W$, where $\Pi_n$ denotes the Fourier  projection on modes $| \xi | \leq n$. As $W^n$ is a spatially smooth  we can consider the unique strong solution
  $\mu^n$ to the stochastic continuity-diffusion equation
  \[ \dd \mu^n + \nabla \cdot (\mu^n \circ \dd W^n) = \kappa \tilde{C}:D^2 \mu^n \dd t, \quad \mu^n
     |_{t = 0} = \mu_0. \]
By linearity of the equation we can assume without loss of generality that $\mu_0 \in \mathcal{M}_+$.
Interpreting the above as a Fokker-Planck equation with random drift $\dd W^n$, one has that $\mu^n_t$ corresponds to the conditional law, with respect to an auxiliary $\R^d$-valued Brownian motion $w$ independent of $W^n$, of the unique solution of the SDE
\begin{align*}
    \dd X^n_t = \dd W^n(t,X_t) + \sqrt{2\kappa} \tilde{C}^{1/2}\dd w_t,
    \quad
    (X_0^n)_\sharp \mathbb{P}^w = \mu_0.
\end{align*}
Notice that the Fokker-Planck SPDE above has equivalent It\^{o}
  form given by
  \[ \dd \mu^n + \nabla \cdot (\mu^n \dd W^n) = \left( \kappa \tilde{C} + \frac{c_n}{2} I_d \right) : D^2
     \mu^n\dd t, \quad \mu^n |_{t = 0} = \mu_0, \]
  where $c_n = \mathrm{Tr}\, (C^n (0))$, where $C^n$ is the convolution kernel
  associated to $\mathcal{C}^n = \Pi_n  \mathcal{C}$, $\widehat{C^n} (\xi) =
  \hat{C} (\xi) \mathbf{1}_{| \xi | \leq n}$. By construction, $c^n
  \rightarrow c$. Moreover, by standard properties of pushforward maps, we have the
  $\mathbb{P}$-almost sure estimate
  \begin{equation} 
    \| \| \mu^n_t \|_{\mathrm{TV}} \|_{L^{\infty} ([0, T])} \leq\| \mu_0
    \|_{\mathrm{TV}} . \label{eq:apriori.estim.1}
  \end{equation}
  
  We can rewrite the equation in integral form as
  \begin{equation}
    \mu^n_t = \mu_0 + \nabla \cdot I^n_t + J^n_t = \mu_0 + \nabla \cdot \left(
    \int_0^t \mu^n_r \dd W^n_r \right) + \left( \kappa \tilde{C} + \frac{c_n}{2} I_d \right) : D^2 \int_0^t     \mu^n_r \dd r. \label{eq:approx.spde}
  \end{equation}
  Since $\mu^n \in L^{\infty}_t \mathrm{TV}$, $\sup_n  (c_n + \| \widehat{C^n}
  \|_{L^1}) < \infty$ and $\mathrm{TV} \hookrightarrow H^{- d / 2 -
  \eps}_x$ we can now apply the previous bounds to see that, for any
  fixed $p \in [1, \infty)$, $\eps > 0$, $\gamma < 1 / 2$ it holds
  \[ \sup_n \| \| \nabla \cdot I^n_t \|_{C^{\gamma} ([0, T] ; H^{- d / 2 - 1 -
     \eps}_x)} \|_{L^p (\Omega)} + \| \| J^n \|_{C^1 ([0, T] ; H^{- d / 2
     - 2 - \eps}_x)} \|_{L^p (\Omega)} \lesssim \| \hat{C} \|_{L^1}^{1 /
     2} \| \mu_0 \|_{\mathrm{TV}}. \]
In particular,
  \begin{equation}
    \sup_n \| \| \mu^n \|_{C^{\gamma} ([0, T] ; H^{- d / 2 - 2 -
    \eps}_x)} \|_{L^p (\Omega)} \lesssim \| \hat{C} \|_{L^1}^{1 / 2} \|
    \mu_0 \|_{\mathrm{TV}} < \infty . \label{eq:apriori.estim.2}
  \end{equation}
  For $\eps > 0$ as above, it will be convenient to introduce the
  Hilbert space $\mathcal{H} = \mathcal{C}^{- 1 / 2} (L^2_x) \cap H^{- d / 2 -
  \eps}_x$, endowed with the norm $\| f \|_{\mathcal{H}}^2 = \| \mathcal{C}^{1 /
  2} f \|_{L^2_x}^2 + \| f \|_{H^{- d / 2 - \eps}_x}^2$, and the space
  $\widetilde{\mathcal{H}} = L^2 (\Omega \times [0, T] ; \mathcal{H})$. As the
  sequence $\{\mu^n\}_{n \in \NN}$ is bounded in $\widetilde{\mathcal{H}}$, by Banach-Alaoglu Theorem we can extract a (not relabelled) subsequence
  such that $\mu^n \rightharpoonup \mu$ for some $\mu \in
  \widetilde{\mathcal{H}}$; as the processes $\mu^n$ are progressively
  measurable, the same holds for $\mu$. Notice that for any constants
  $K_1$ and $K_2$, the set
  \[ E = \left\{ \mu \in \widetilde{\mathcal{H}} : \, \| \| \mu \|_{\mathrm{TV}}
     \|_{L^{\infty} ([0, T])} \leq K_1, \, \| \| \mu \|_{C^{\gamma}
     ([0, T] ; H^{- d / 2 - 2 - \eps}_x)} \|_{L^p (\Omega)} \leq
     K_2 \right\} \]
  is convex and closed in the strong topology of $\widetilde{\mathcal{H}}$ (by
  standard lower-semicontinuity arguments) and therefore also closed with respect to
  weak convergence. This implies that we can take a representative of $\mu$
  which has continuous trajectories (in the $H^{- d / 2 - 2 -  \eps}_x$-topology) and such that estimates
  \eqref{eq:apriori.estim.1}-\eqref{eq:apriori.estim.2} still hold. Moreover,
  again by weak continuity of trajectories and lower-semicontinuity of the
  TV-norm, we can upgrade~\eqref{eq:apriori.estim.1} by replacing
  $\mathrm{esssup}_{t \in [0, T]}$ with $\sup_{t \in [0, T]}$. In other words we
 have found
  \begin{equation*}
    \sup_{t \in [0, T]} \| \mu_t \|_{\mathrm{TV}} \leq\| \mu_0
    \|_{\mathrm{TV}}, \quad \mathbb{P} \text{-almost surely}, 
\end{equation*}
and
\begin{equation*}
\| \| \mu
    \|_{C^{\gamma} ([0, T] ; H^{- d / 2 - 2 - \eps}_x)} \|_{L^p (\Omega)}
    \lesssim \| \hat{C} \|_{L^1}^{1 / 2} \| \mu_0 \|_{\mathrm{TV}} .
  \end{equation*}
  It is clear that $\mu$ is adapted, since each
  $\mu^n$ is so. It remains to show that $\mu$ is indeed a solution to the SPDE. 
For any $\varphi \in C^{\infty}_c$, by~\eqref{eq:apriori.estim.2} we have
  \[ \mathbb{E} \left[ \int_0^T \| \mathcal{C}^{1 / 2} (\nabla \varphi  \mu^n_r) \|_{L^2_x}^2 \dd r \right] 
\leq
\| \nabla \varphi  \|_{L^{\infty}_x}  \| \hat{C} \|_{L^1} \mathbb{E} \left[ \int_0^T \|
     \mu^n_r \|_{\mathrm{TV}}^2 \dd r \right] \leq 
 T \| \nabla \varphi \|_{L^{\infty}_x} \| \hat{C} \|_{L^1} \| \mu_0 \|_{\mathrm{TV}}, \]
  uniformly in $n \in \NN$, and similarly for $\mu$. By a similar argument, one checks
  that $\{\nabla \varphi \mu^n\}_{n \in \NN}$ is uniformly bounded in $L^2 (\Omega \times
  [0, T] ; H^{- d / 2 - \eps}_x)$ and so in $\widetilde{\mathcal{H}}$.
  Combined with weak compactness and the fact that $\mu^n \rightharpoonup \mu$
  in $\widetilde{\mathcal{H}}$, this implies that $\varphi \mu^n
  \rightharpoonup \varphi \mu$ in $\widetilde{\mathcal{H}}$ as well. But since
  the map $f \mapsto \int_0^{\cdot} \langle f_r, \dd W_r \rangle$ is linear
  and strongly continuous from $\widetilde{\mathcal{H}}$ to $L^2_{\omega}
  C_t$, it is also weakly continuous, so that
  \[ \bigg\langle \varphi, \nabla \cdot \left( \int_0^{\cdot} \mu^n_r \dd W_r
     \right) \bigg\rangle = - \int_0^{\cdot} \langle \varphi \mu^n_r, \dd W_r
     \rangle \rightharpoonup - \int_0^{\cdot} \langle \varphi \mu_r, \dd
     W_r \rangle . \]
On the other hand, for any $\varphi \in C^\infty_c(\R^d)$, by contruction and Doob's inequality it holds
  \begin{align*}
    \mathbb{E} \left[ \sup_{t \in [0, T]} \left| \bigg\langle \varphi, \nabla \cdot
    \left( \int_0^t \mu^n_r \dd(W - W^n)_r \right) \bigg\rangle \right|^2
    \right] & =\mathbb{E} \left[ \sup_{t \in [0, T]} \left| \int_0^t \langle
    \nabla \varphi \mu^n_r, \dd(W - W^n)_r \rangle \right|^2 \right]\\
    & \lesssim \mathbb{E} \left[ \int_0^T \| (\mathcal{C} - \mathcal{C}^n)^{1
    / 2}  (\nabla \varphi \mu^n_r) \|_{L^2_x}^2 \dd r \right]\\
    & \lesssim \| \hat{C} - \widehat{C^n} \|_{L^1} \| \nabla \varphi
    \|_{L^\infty_x}^2 \| \mu_0 \|_{\mathrm{TV}}^2 \rightarrow 0
  \end{align*}
since $\hat{C}^n \rightarrow \hat{C}$ in $L^1$. Combining these facts we deduce that
  \[ \bigg\langle \varphi, \nabla \cdot \left( \int_0^{\cdot} \mu^n_r \dd W^n_r
     \right) \bigg\rangle = - \int_0^{\cdot} \langle \varphi \mu^n_r, \dd W^n_r
     \rangle \rightharpoonup - \int_0^{\cdot} \langle \varphi \mu_r, \dd
     W_r \rangle. \]
A similar (much simpler) argument shows that
  \[\int_0^{\cdot} \Big\langle  \left( \kappa \tilde{C} + \frac{c_n}{2} I_d \right) : D^2  \varphi, \mu^n_r \Big\rangle
     \dd r \rightharpoonup  \int_0^{\cdot} \Big\langle \left( \kappa \tilde{C} + \frac{c}{2} I_d \right) : D^2
     \varphi, \mu_r \Big\rangle \dd r \]
  in the weak topology of $L^2_{\omega} C_t$ as well. But the this implies
  that, whenever testing against $\varphi$ in~\eqref{eq:approx.spde}, at any
  fixed $t$ all terms on the right-hand side are converging to the respective integrals
  with $(\mu^n, W^n)$ replaced by $(\mu, W)$. Therefore also the right-hand side must
  be converging; but by construction $\mu^n \rightharpoonup \mu$ and so we can
  conclude that $\mu$ is a desired solution to
  \eqref{eq:stochastic.continuity_diffusive}.
\end{proof}
 
We have the following proposition concerning some useful properties of solutions to \eqref{eq:stochastic.continuity_diffusive}.

\begin{prop}
\label{prop:existence.SCE.properties}
For any deterministic measure $\mu_0 \in \mathcal{M}$, the unique solution $\mu$ to \eqref{eq:stochastic.continuity_diffusive} satisfies:
\begin{enumerate}
    \item 
    If $\mu_0 \geq 0$ then $\mu_t \geq 0$ for every $t \geq 0$ $\PP$-almost surely. 
    \item 
    For every fixed $t \geq 0$,  $\PP$-almost surely we have $\mu_t (\mathbb{R}^d) = \mu_0 (\mathbb{R}^d)$ and $\| \mu_t \|_{\mathrm{TV}} \leq\| \mu_0
    \|_{\mathrm{TV}}$.
    In particular, if $\mu_0 \in \mathcal{P}(\mathbb{R}^d)$ (the set of probability measures on $\mathbb{R}^d$), then $\mu$ is a process with values in $\mathcal{P} (\mathbb{R}^d)$,
    \item 
    If $\mu_0\in \mathcal{P} (\mathbb{R}^d)$ with $\langle |x|^2,\mu_0\rangle=\int_{\R^d} |x|^2 \mu_0(\mathrm{d}x)<\infty$, then for any $m\in [1,\infty)$ it holds
    \begin{equation}\label{eq:stoch_continuity_bound_variance}
        \EE\bigg[\,\sup_{t\in [0,T]} \langle |x|^2,\mu_t\rangle^m\bigg]^{\frac{1}{m}} \lesssim_m \langle |x|^2,\mu_0 \rangle + T\,[{\rm Tr}(C(0))+\kappa {\rm Tr}(\tilde C)],\quad \forall\, T\geq 0.
    \end{equation}
\end{enumerate}
\end{prop}
\begin{proof}

$(1)$
The argument in the proof of \autoref{prop:existence.SCE} above in fact also shows that, for any fixed $t$, we have $\mu^n_t \rightharpoonup \mu_t$ (e.g. in $L^2_{\omega} H^{- d/2 - \eps}_x$).
If $\mu_0 \geq 0$, then by properties of the pushforward $\mathbb{P}$-almost surely this property also holds for $\mu^n_t$ for all $t \geq  0$, and by properties of weak convergence (up to modifying the definition of the convex set $E$ in the above proof) must then hold for $\mu_t$ as well. 

$(2)$
Notice that by construction and the previous point, $\mathbb{P}$-almost surely we have the identity $\mu^n_t (\mathbb{R}^d) = \mu_0 (\mathbb{R}^d)$ for all $t \geq  0$, and by weak convergence the same relation holds for $\mu^n_t$ replaced by $\mu_t$ as well. 
In particular, if $\mu_0 \in \mathcal{P}(\mathbb{R}^d)$ (the set of probability measures on $\mathbb{R}^d$), then $\mu$ is a process with values in $\mathcal{P} (\mathbb{R}^d)$.

Similarly, the $L^\infty_t$ bound \eqref{eq:apriori.estim.1}, together with $\PP$-almost sure time continuity of trajectories and the property at point (1), leads to $\| \mu_t \|_{\mathrm{TV}} \leq \| \mu_0 \|_{\mathrm{TV}}$.

$(3)$
For simplicity, we assume $W$ to be spatially smooth, as the general case follows by the above approximation procedure; by Jensen, we may assume $m\geq 2$.
Taking $\varphi=|x|^2$ in \eqref{eq:weak_sol.continuity} one finds
\begin{align*}
    \dif\langle |x|^2,\mu_t\rangle
    = 2 \langle x \mu_t, \dif W_t\rangle + \Big[{\rm Tr}(C(0))+\frac{\kappa}{2} {\rm Tr}(\tilde C)\Big] \dif t.
\end{align*}
By writing the above in integral form, applying Minkowski and Burkholder-Davies-Gundy inequality,
\begin{align*}
    \EE\bigg(\,\sup_{t\in [0,T]} \langle |x|^2,\mu_t\rangle^m\bigg)^{\frac{1}{m}} & \lesssim_m \langle |x|^2,\mu_0 \rangle + T\,[{\rm Tr}(C(0))+\kappa {\rm Tr}(\tilde C)] + \EE\bigg(\,\Big[\int_0^\cdot \langle x\mu_t, \dif W_t\rangle\Big]_T^{m/2}\bigg)^{\frac{1}{m}}\\
    & \leq \langle |x|^2,\mu_0 \rangle + T\,[{\rm Tr}(C(0))+\kappa {\rm Tr}(\tilde C)] + {\rm Tr}(C(0))^{1/2} \EE\bigg(\,\Big(\int_0^T \langle |x|,\mu_t\rangle^2 \dd t \Big)^{m/2}\bigg)^{\frac{1}{m}}\\
    & \leq \langle |x|^2,\mu_0 \rangle + T\,[{\rm Tr}(C(0))+\kappa {\rm Tr}(\tilde C)] + {\rm Tr}(C(0))^{1/2} T^{1/2}\EE\bigg(\,\sup_{t\in [0,T]} \langle |x|^2,\mu_t\rangle^{m/2}\bigg)^{\frac{1}{m}}
\end{align*}
where we applied \autoref{lem:stoch_integr} , estimate \eqref{eq:covariance_TV_inequaity} and Jensen's inequality.
By Jensen's and Young's inequalities, for any $\delta>0$ it holds
\begin{align*}
    {\rm Tr}(C(0))^{1/2} T^{1/2}\EE\bigg(\,\sup_{t\in [0,T]} \langle |x|^2,\mu_t\rangle^{m/2}\bigg)^{\frac{1}{m}}
    \leq \frac{1}{4\delta} {\rm Tr}(C(0)) T + \delta \,\EE\bigg(\,\sup_{t\in [0,T]} \langle |x|^2,\mu_t\rangle^m\bigg)^{\frac{1}{m}};
\end{align*}
combined with the previous estimate, up to choosing $\delta$ small enough in function of the hidden constant, this yields the desired \eqref{eq:stoch_continuity_bound_variance}.
\end{proof}

\begin{proof}[Proof of \autoref{lem:duality}]
When $f \in C^\infty_c(\R^d)$ and $\mu_0 \in L^2_x$, we have by \eqref{eq:wiener_sol.continuity} 
\begin{align*}
\langle \mu_t , f \rangle
=
\langle \mu_0 , P_t f \rangle
+
\sum_{n \geq 1}
I^n_t(f,\mu_0)
=
\langle f , P_t \mu_0  \rangle
+
\sum_{n \geq 1}
I^n_t(f,\mu_0)
=
\langle S_t f, \mu_0 \rangle.
\end{align*}
By density of $C^\infty_c(\R^d)$ in $C_c(\R^d)$, continuity of the map $S_t$ from $L^2_x$ to $L^2_{\omega,x}$, and since $\mu_t \in \mathcal{M}$ with $\|\mu_t\|_{\mathrm{TV}} \leq \|\mu_0\|_{\mathrm{TV}}$ $\PP$-almost surely, the duality formula can be extended to every $f \in C_c(\R^d)$. 

Let us move to the second statement.
Let $\{\chi^n(x-\cdot)\}_{n \in \NN}$ be a collection of rescaled smooth mollifiers approximating the Dirac delta $\delta_x$.
Identifying an absolutely continuous measure with its density, we can apply the first part of the lemma to $\chi^n(x-\cdot)$ and get for every $f \in C_c(\R^d)$
\begin{align*}
(S_t f \ast \chi^n)(x,\omega)
=
\langle f , \mu^{\chi^n(x-\cdot)}_t \rangle.
\end{align*} 
For every $x \in \R^d$ the right-hand side converges to $\langle f, \mu^{x}_t \rangle$ $\PP$-almost surely as $n \to \infty$. As for the left-hand side, it converges in $L^2_{\omega,x}$ as $n \to \infty$ towards $S_tf(x,\omega)$, implying the thesis when $f \in C_c(\R^d)$.  

In order to extend the relation to every $f \in L^2_x$ we argue as follows. Let $\{ f^n \}_{n \in \NN} \subset C_c(\R^d)$ be a sequence converging to $f$ in $L^2_x$ and such that $\| f^{n + 1} - f^n \|_{L^2_x} \leq 2^{-n}$ for all $n$. Then for any fixed $M$, by linearity of $S_t$ and the previous step we have 
\begin{equation*}
S_t \left( \sum_{n = 1}^M | f^{n + 1} - f^n | \right) (x,\omega) 
= 
\left\langle \sum_{n = 1}^M | f^{n + 1} - f^n |,\mu^x_t \right\rangle. 
\end{equation*}
Since $\mu_t^x$ is a non-negative measure by point $(2)$ of \autoref{prop:existence.SCE.properties}, by Monotone Convergence we can take $M \rightarrow \infty$ and deduce that for $\PP \otimes \mathscr{L}^d$-almost every $(\omega,x)$
\[ S_t \left( \sum_{n = 1}^{+ \infty} | f^{n + 1} -  f^n | \right) (x,\omega) 
= 
\left\langle \sum_{n = 1}^{+ \infty} | f^{n + 1} - f^n |, \mu^x_t \right\rangle. \]
For any such  $(x,\omega)$, since $\langle \sum_{n = 1}^{+ \infty} | f^{n + 1} -
  f^n |, \mu^x_t \rangle < \infty$ and $\mu_t^x$ is a non-negative measure, it follows that $\{ \langle f^n, \mu^x_t \rangle \}_{n \in \NN}$ is a Cauchy sequence.
In order to conclude the proof, we have to prove that $\langle f^n, \mu^x_t \rangle \to \langle f, \mu^x_t \rangle$. Notice that the latter is an element of $L^2_{\omega,x}$ since by Jensen inequality 
\begin{align*}
\mathbb{E} \int_{\R^d} \langle f , \mu^x_t \rangle^2 \dd x
\leq
\int_{\R^d} \langle |f|^2 , \mathbb{E} [\mu^x_t ]\rangle \dd x
=
\int_{\R^d} P_t(|f|^2)(x) \dd x
=
\| f \|_{L^2}^2 < \infty,
\end{align*}
where we have used the fact that $\mathbb{E}[\mu_t^x] = P_t \delta_x$, as follows by taking the expectation in \eqref{eq:wiener_sol.continuity}. 
In order to prove the convergence it suffices to show $\langle | f^n - f |, \mu^x_t \rangle \rightarrow 0$ in $L^2_{\omega,x}$ as $n \rightarrow \infty$.
Arguing as before we get for $\PP \otimes \mathscr{L}^d$-almost every $(\omega,x)$
\begin{align*}
\langle | f^n - f |, \mu^x_t \rangle
=
\langle | \sum_{m=n}^\infty (f^{m+1} - f^m) |, \mu^x_t \rangle
=
S_t \left(|\sum_{m=n}^\infty (f^{m+1} - f^m) | \right) (x,\omega),
\end{align*}
which goes to zero in $L^2_{\omega,x}$ since
\begin{equation*}
\left\| S_t \left(| \sum_{m=n}^\infty (f^{m+1} - f^m) | \right)  \right\|_{L^2_{\omega,x}}
\leq
\sum_{m=n}^\infty \left\| f^{m+1} - f^m \right\|_{L^2_{x}}
\leq
\sum_{m=n}^\infty 2^{-m}
=
2^{1-n} \to 0. \qedhere
\end{equation*}
\end{proof}

\subsection{Proof of \autoref{lem:besov_type_spaces}}\label{app:proof.besov.spaces}

%\begin{proof}[Proof of \autoref{lem:besov_type_spaces}]
For convenience, let us define
\begin{align*}
    I_1:= \sup_{j\in\ZZ} 2^{sjp} \| \dot\Delta_j f\|_{L^p_{\omega,t,x}}^p,\qquad I_2:= \sup_{\eps>0} \frac{1}{\eps^{sp}} \fint_{\SS^{d-2}} \| \delta_{\eps \hat z} f\|_{L^p_{\omega,t,x}}^p \sigma(\dif \hat z)
\end{align*}
First assume that $f\in \tilde L^p_{\omega,t} B^s_{p,\infty}$; by definition, since $|\hat z|=1$, we immediately have
\begin{align*}
    \frac{1}{\eps^{sp}} \fint_{\SS^{d-2}} \| \delta_{\eps \hat z} f\|_{L^p_{\omega,t,x}}^p \sigma(\dif \hat z)
    \leq \llbracket f\rrbracket_{\tilde{L}^p_{t,\omega} \tilde B^s_{p,\infty}}^p \quad \Rightarrow \quad I_2 \leq \llbracket f\rrbracket_{\tilde{L}^p_{t,\omega} \tilde B^s_{p,\infty}}^p.
\end{align*}
Next assume $I_1<\infty$ and fix $z\in\R^d$. By Bernstein estimates, for any fixed $j\in \ZZ$ we have
\begin{align*}
    \| \dot\Delta_j \delta_z f\|_{L^p_x} \leq \| \delta_z \dot\Delta_j f\|_{L^p_x}\lesssim |z| 2^j \| \dot\Delta_j f\|_{L^p_x},
\end{align*}
which by integrating also in $(\omega,t)$ implies that
\begin{align*}
    \| \dot\Delta_j \delta_z f\|_{L^p_{\omega,t,x}} \lesssim \min\{|z| 2^j,1\} \| \dot\Delta_j f\|_{L^p_{\omega,t,x}}
\end{align*}
where the term with $1$ comes from triangular inequality. Therefore, choosing $K\in\ZZ$ such that $|z|\sim 2^{-K}$ we obtain for every $z \in \R^d$:
\begin{align*}
    \| \delta_z f \|_{L^p_{\omega,t,x}}
    & \leq \sum_{j\in\ZZ} \| \delta_z \dot\Delta_j f \|_{L^p_{\omega,t,x}}
    \lesssim \sum_{j\leq K} |z| 2^j \| \dot\Delta_j f\|_{L^p_x} + \sum_{j> K} \| \dot\Delta_j f\|_{L^p_x}\\
    & \lesssim I_1^{1/p} \bigg(\sum_{j\leq K} |z| 2^{j(1-s)} +  \sum_{j> K} 2^{-js} \bigg)
    \\
    &\lesssim I_1^{1/p} ( |z| 2^{K(1-s)} + 2^{-K s})\lesssim I_1^{1/p} |z|^s,
\end{align*}
giving $\llbracket f\rrbracket_{\tilde{L}^p_{t,\omega} \tilde B^s_{p,\infty}}^p \lesssim I_1$.
Finally, assume $I_2<\infty$ and recall that $\dot\Delta_j f=f\ast \psi_j$, where $\psi$ is a radially symmetric smooth function with mean zero. Therefore
\begin{align*}
    \| \dot\Delta_j f\|_{L^p_x}^p
    & = \Big\| \int_{\R^d} \delta_z f \psi_j(z) \dd z\Big\|_{L^p_x}^p
    \leq \Big(  \int_{\R^d} \|\delta_z f\|_{L^p_x} |\psi_j(z)| \dd z\Big)^p\\
    & \leq \| \psi\|_{L^1}^{p-1} \int_{\R^d} \|\delta_z f\|_{L^p_x}^p |\psi_j(z)| \dd z
    \sim \int_{\R^d} \|\delta_z f\|_{L^p_x}^p 2^{jd} |\psi(2^j z)| \dd z,
\end{align*}
where in the intermediate step we used Jensen's inequality. Integrating also with respect to $(\omega,t)$, using that $\psi(z)=\varphi(|z|)$ and integrating over spheres of fixed radii $|z|=r$, we arrive at
\begin{align*}
    \| \dot\Delta_j f\|_{L^p_{\omega,t,x}}^p
    \lesssim I_2 \int_0^{+\infty} r^{sp}\, r^{d-1}\, 2^{jd} \varphi(2^j r) \dd r
    = 2^{-jsp} I_2 \int_0^{+\infty} r^{sp+d-1} \varphi(r) \dd r,
\end{align*}
where the last integral is finite since $\varphi$ is of fast decay. Overall this shows that
\begin{align*}
    I_2\leq \llbracket f\rrbracket_{\tilde{L}^p_{t,\omega} \tilde B^s_{p,\infty}}^p \lesssim I_1\lesssim I_2,
\end{align*}
yielding the desired \eqref{eq:besov_type_spaces}.
The proof of \eqref{eq:besov_type_spaces2} is almost identical and therefore omitted.
Concerning the last statement, it follows from \eqref{eq:besov_type_spaces} that \eqref{eq:besov_type_spaces2} holds if and only if the sequence $\{f^n\}_n$ converges to $f$ in $\tilde L^p_{\omega,t} \tilde B^s_{p,\infty}$, for the spatially smooth sequence given by
$
    f^n=\sum_{j\leq n} \dot\Delta_j f. 
$\qed
%
%\end{proof}

\bibliography{biblio}{}
\bibliographystyle{alpha}

\end{document}